\documentclass[a4paper,11pt,twoside, eqno]{article}
\usepackage{mathrsfs}
\usepackage{amssymb}
\usepackage{amsmath}
\usepackage{amsthm}
\usepackage{amsfonts}
\usepackage{color}
\usepackage{latexsym}
\usepackage{indentfirst}
\usepackage{amscd}
\usepackage{bbm}
\usepackage{graphicx}
\usepackage[all]{xy}
\usepackage[mathscr]{eucal}
\usepackage{subfigure}

\newcommand{\N}{{\mathbb N}}

\newcommand{\R}{{\mathbb R}}

\newcommand{\Z}{{\mathbb Z}}

\allowdisplaybreaks
\newtheorem{theorem}{Theorem}[section]

\newtheorem{corollary}[theorem]{Corollary}
\newtheorem{thm}[theorem]{Theorem}
\newtheorem{definition}[theorem]{Definition}

\newtheorem{remark}[theorem]{Remark}

\newtheorem{lemma}[theorem]{Lemma}

\newtheorem{proposition}[theorem]{Proposition}
\newtheorem{claim}[theorem]{Claim}

\newtheorem{assumption}[theorem]{Assumption}

\numberwithin{equation}{section}

\begin{document}

\title{\bf\Large Bifurcations in Lagrangian systems and geodesics II
\footnotetext{\hspace{-0.35cm} 2020
{\it Mathematics Subject Classification}.
Primary 37J20, 34C23, 53C22.
\endgraf
{\it Key words and phrases.}
Bifurcation, Lagrangian system, geodesic, Finsler metric, exponential map, Morse index,
critical point, focal point.
}}
\date{}
\author{Guangcun Lu\footnote{
E-mail: \texttt{gclu@bnu.edu.cn}}}

\maketitle

\vspace{-0.7cm}

\begin{center}

\begin{minipage}{13cm}
{\small {\bf Abstract}\quad
This is the second part of a two--part series investigating bifurcation phenomena in autonomous Lagrangian systems and geodesic flows on Finsler and Riemannian manifolds. Building upon the abstract bifurcation theorems established  in earlier work and the results of Part I,
this study makes contributions in two main directions.
In Part A, we focus on bifurcations of generalized periodic solutions in autonomous Lagrangian systems. By employing Morse index and nullity techniques within the normal space to the $\mathbb{R}$-orbits of solutions, we derive necessary and sufficient conditions for bifurcation, encompassing scenarios of both Fadell--Rabinowitz and Rabinowitz type.
In Part B, we extend these results to the geometric setting of geodesic bifurcations in Finsler and Riemannian manifolds. 
A principal achievement is the significant refinement
 of the classical Gauss lemma and its generalizations by Morse-Littauer and Savage,
providing a precise description of geodesic behavior near critical points of the exponential map. The sharpness of these theoretical results is rigorously tested and confirmed through explicit counterexamples, such as the round sphere.
The work is technically rigorous, leveraging a specialized technique developed by the author to establish novel bifurcation theorems. These findings have profound theoretical implications and potential applications in related fields such as the Zermelo navigation problem and the study of stationary spacetimes.
}
\end{minipage}
\end{center}

\vspace{0.2cm}

\tableofcontents

\vspace{0.2cm}

\section*{Introduction}

Bifurcation theory provides a fundamental framework for understanding qualitative changes in the structure of solutions to nonlinear systems as parameters vary. At the intersection of the calculus of variations and differential geometry,  a central challenge is to characterize how families of solutions--such as periodic trajectories of Lagrangian systems or geodesics on manifolds--emerge, vanish, or change their structure in response to variations in parameters like energy, boundary conditions, or the underlying metric itself.

This paper constitutes the second part of a series dedicated to developing a systematic bifurcation theory for autonomous Lagrangian systems and for geodesics in Finsler and Riemannian geometry.
Although bifurcation theory is well-developed in contexts such as Hamiltonian systems, this is not the case for Lagrangian systems with symmetries (e.g., translation invariance) and for the behavior of geodesics in Finsler geometry--a setting characterized by a non--smooth Lagrangian (the Finsler metric),  for which a precise theoretical description is still lacking. 
Seminal results concerning the non-injectivity of the exponential map, such as the classical Gauss lemma and its subsequent generalizations (e.g., by Morse-Littauer and Savage in the setting of Finsler spaces), cannot capture finer bifurcation scenarios.

The primary objectives and contributions of this work are threefold:

\textsf{Bifurcations in Lagrangian systems (Part A)}: We conduct a detailed investigation into the bifurcation of generalized periodic solutions for autonomous Lagrangian systems. A significant complication arises from the autonomy of the system, which implies that every nonconstant solution generates an entire
$\mathbb{R}$--orbit of solutions. To precisely characterize bifurcations near these orbits, one must work in the normal space to the $\mathbb{R}$--orbit, introducing substantial technical challenges. We successfully overcome these obstacles to establish necessary and sufficient conditions for bifurcation.

\textsf{Bifurcations of geodesics (Part B)}: We systematically extend the analytical framework from Part A to study geodesic bifurcations on Finsler and Riemannian manifolds. A key advancement is a substantial strengthening of the classical Morse--Littauer theorem.
A key advancement (Theorem~\ref{th:MorseLittauer1}) is a substantial strengthening of the classical Gauss lemma and its generalizations by Morse-Littauer and Savage.
We demonstrate that the behaviour of geodesics near a critical point of the exponential map is far richer than the classical conclusion of mere non-injectivity, providing a complete classification of possible scenarios (as detailed in Theorem~\ref{th:MorseLittauer1}).

\textsf{A unified framework with applications}: This paper develops a unified technical framework,
based on the abstract bifurcation theory of \cite{Lu6, Lu8, Lu11, Lu10} and a novel technique from
\cite{Lu4, Lu12-, Lu12}, to study these geometric variational problems. The approach formulates
them as critical point problems on Hilbert manifolds. With the precise computation of Morse indices and nullities, we obtain verifiable bifurcation criteria. This provides a powerful tool, simultaneously
establishing the sharpness of the results via explicit counterexamples and demonstrating its
applicability to areas such as Zermelo navigation and the geometry of stationary spacetimes.

\textsf{Structure of the Paper}:
The current paper is divided into two main sections: Part A details the analysis for autonomous
Lagrangian systems, while Part B is devoted to geodesic bifurcations and includes a discussion of their implications.

We adopt the notation and conventions of \cite{Lu12-}, particularly those listed in \cite[\S 2.1]{Lu12-}.
For the reader's convenience, we also state \cite[Assumption 1.0]{Lu12-} as follows.\\

\noindent\textbf{Basic assumptions and conventions}(\cite[Assumption~1.0]{Lu12-}).
Let $M$ be a $n$-dimensional, connected $C^7$ submanifold of $\R^N$.
Its tangent bundle $TM$ is a $C^6$-smooth manifold of dimension $2n$,
whose points are denoted by $(x,v)$, with $x\in M$ and $v\in T_xM$.
The bundle projection $\pi:TM\to M,\,(x,v)\mapsto x$ is $C^6$.
Let $g$ be a $C^6$ Riemannian metric and $\mathbb{I}_g$ a $C^7$ isometry  on $(M, g)$,
i.e., $\mathbb{I}_g:M\to M$ is $C^7$ and satisfies $g((\mathbb{I}_g)_\ast(u), (\mathbb{I}_g)_\ast(v))=g(u,v)$
for all $u,v\in TM$.
(Thus the Christoffel symbols $\Gamma^i_{jk}$ and the exponential map $\exp:TM\to M$ are $C^5$.)\\

\part{Bifurcations of generalized periodic solutions  in autonomous Lagrangian systems}\label{part:A}

The study of bifurcations for autonomous Lagrangian systems presents a significantly
higher level of complexity compared to the non-autonomous case discussed in \cite{Lu12-}.
 The reason is that  each nonconstant solution $\bar\gamma$ of (\ref{e:Lagr10*}) generates
 an $\mathbb{R}$-orbit, $\R\cdot \bar\gamma:=\{\bar\gamma(\theta+\cdot)\,|\,\theta\in\R\}$,
 of solutions of (\ref{e:Lagr10*}). Consequently,
  every point $(\lambda, \bar\gamma)$ in $\Lambda\times\{\bar\gamma\}$
 constitutes a bifurcation point according to \cite[Def.1.12]{Lu12-}
(cf. Definition~\ref{def:Bifur}).  To precisely characterize bifurcations near
  $\R\cdot \bar\gamma$,   one must work in the normal space to $\R\cdot \bar\gamma$ to apply abstract bifurcation theorems, which requires refining the techniques from \cite{Lu12-} and handling substantial technical complexities. In Section~\ref{sec:LgrResults}, we state the main results, along with the necessary definitions and assumptions. Their proofs are deferred to Sections~\ref{sec:Lagr3} and~\ref{sec:Lagr4}.

\section{Statement of main results}\label{sec:LgrResults}
\setcounter{equation}{0}

For a real $\tau>0$, we introduced the following spaces in [26, (1.31)],
    \begin{equation}\label{e:BanachPerMani}
\mathcal{X}^i_{\tau}(M, \mathbb{I}_g):=\{\gamma\in C^i(\mathbb{R}, M)\,|\, \mathbb{I}_g(\gamma(t))=\gamma(t+\tau)\;\forall t\},\quad
i=0,1,2,\cdots.
\end{equation}
    The tangent space of the $C^4$ Banach manifold $\mathcal{X}^1_{\tau}(M, \mathbb{I}_g)$
 at $\gamma\in \mathcal{X}^1_{\tau}(M, \mathbb{I}_g)$ is
$$
 C^1_{\mathbb{I}_g}(\gamma^\ast TM):=T_{\gamma}\mathcal{X}^1_{\tau}(M, \mathbb{I}_g)=\{\xi\in C^1(\gamma^\ast TM)\,|\, d\mathbb{I}_g(\gamma(t))[\xi(t)]=\xi(\tau+t)\;\forall t\}.
$$
It is a dense subspace in the Hilbert space
$$
W^{1,2}_{\mathbb{I}_g}(\gamma^\ast TM):=\{\xi\in H^1(\gamma^\ast TM)\,|\,
 d\mathbb{I}_g(\gamma(t))[\xi(t)]=\xi(\tau+t)\;\forall t\}
$$
with inner product given by
\begin{equation}\label{e:1.1}
\langle\xi,\eta\rangle_{1,2}=\int^\tau_0\langle\xi(t),\eta(t)\rangle_g dt+
\int^\tau_0\langle \frac{D\xi}{dt}(t), \frac{D\eta}{dt}(t)\rangle_g dt,
\end{equation}
where $\frac{D\xi}{dt}$ is the $W^{1,2}$ covariant
derivative of $\xi$ along $\gamma$,
defined via the Levi-Civita connection  of the metric $g$.
Hereafter $\langle u,v\rangle_g=g(u,v)$ for $u,v\in TM$.

\begin{assumption}\label{ass:Lagr6II}
{\rm  Let $(M, g, \mathbb{I}_g)$ be as in ``Basic assumptions and conventions'' in Introduction.
For  a topological space $\Lambda$,
let $L:\Lambda\times TM\to\R$ be a continuous function satisfying the following conditions:
 \begin{enumerate}
\item[\rm (i)] All $L_\lambda=L(\lambda, \cdot)$ are $\mathbb{I}_g$-invariant in the
following sense:
\begin{equation}\label{e:Lagr8}
L(\lambda, \mathbb{I}_g(x), d\mathbb{I}_g(x)[v])=L(\lambda, x,v)\quad\forall (\lambda,x,v)\in\Lambda\times TM.
\end{equation}
\item[\rm (ii)] For each $C^3$ chart $\alpha:U_\alpha\to\alpha(U_\alpha)\subset\mathbb{R}^n$ and the induced bundle
chart $T\alpha:TM|_{U_\alpha}\to \alpha(U_\alpha)\times\mathbb{R}^n\subset\mathbb{R}^n\times\mathbb{R}^n$ the function
$$
L^\alpha:\Lambda\times \alpha(U_\alpha)\times\mathbb{R}^n\to\mathbb{R},\;
(\lambda, q,v)\mapsto L(\lambda, (T\alpha)^{-1}(q,v))
$$
is $C^2$ with respect to $(q,v)$ and strictly convex with respect to $v$, and
all its partial derivatives also depend continuously on $(\lambda, q, v)$.
 \end{enumerate}}
\end{assumption}
Under Assumption~\ref{ass:Lagr6II}, for a real $\tau>0$, let us consider the problem
\begin{equation}\label{e:Lagr10*}
\left.\begin{array}{ll}
&\frac{d}{dt}\big(\partial_vL_\lambda(\gamma(t), \dot{\gamma}(t))\big)-\partial_q L_\lambda(\gamma(t), \dot{\gamma}(t))=0\;\;\forall t\in\mathbb{R},\\
&\mathbb{I}_g(\gamma(t))=\gamma(t+\tau)\quad\forall t\in\mathbb{R}.
\end{array}\right\}
\end{equation}
Solutions of (\ref{e:Lagr10*}) are also called  \textsf{$\mathbb{I}_g$-periodic trajectories with period $\tau$} (\cite{Dav}).
When $\mathbb{I}_g$ generates a cyclic group, that is, it is of finite order $p\in\mathbb{N}$,
every $\mathbb{I}_g$-periodic trajectory is $p\tau$-periodic.
  Clearly, for a nonconstant solution $\bar\gamma$ of (\ref{e:Lagr10*})
 with a fixed $\lambda\in\Lambda$, each element in $\R\cdot \bar\gamma:=\{\bar\gamma(\theta+\cdot)\,|\,\theta\in\R\}$ (\textsf{$\R$-orbit})
 also satisfies (\ref{e:Lagr10*}) with this $L_\lambda$.
  Thus each point
 $(\lambda, \bar\gamma)$ in $\Lambda\times\{\bar\gamma\}$ is a bifurcation point of (\ref{e:Lagr10*})
in the following sense.

\begin{definition}[\hbox{\cite[Def.1.12]{Lu12-}, \cite[Def.1.12]{Lu12}}]\label{def:Bifur}
{\rm  Let  $X=\mathcal{X}^1_{\tau}(M, \mathbb{I}_g)$ (or $\mathcal{X}^2_{\tau}(M, \mathbb{I}_g)$).
For $\mu\in\Lambda$,
we call $(\mu, \gamma_\mu)$  a \textsf{bifurcation point} of the problem (\ref{e:Lagr10*})
in $\Lambda\times X$ with respect to the branch $\{(\lambda,\gamma_\lambda)\,|\,\lambda\in\Lambda\}$
  if  there exists a point  $(\lambda_0, \gamma_0)$ in any neighborhood of $(\mu,\gamma_\mu)$ in $\Lambda\times X$
    such that $\gamma_0\ne\gamma_{\lambda_0}$
  is a solution of (\ref{e:Lagr10*}) with $\lambda=\lambda_0$.
  Moreover, $(\mu, \gamma_\mu)$ is said to be a \textsf{bifurcation point along sequences} of the problem (\ref{e:Lagr10*})
in $\Lambda\times X$ with respect to the branch $\{(\lambda,\gamma_\lambda)\,|\,\lambda\in\Lambda\}$
  if  there exists a sequence  $\{(\lambda_k, \gamma^k)\}_{k\ge 1}$ in $\Lambda\times X$,
   converging to $(\mu,\gamma_\mu)$ such that each $\gamma^k\ne\gamma_{\lambda_k}$
  is a solution of (\ref{e:Lagr10*}) with $\lambda=\lambda_k$, $k=1,2,\cdots$.
  (These two notions are equivalent if $\Lambda$ is first countable.)  }
\end{definition}

In order to give an exact description for bifurcation pictures of solutions of
(\ref{e:Lagr10*}) near $\R\cdot \bar\gamma$,
 two solutions $\gamma_1$ and $\gamma_2$  of (\ref{e:Lagr10*})
(for a fixed $\lambda$)
are called \textsf{$\mathbb{R}$-distinct} if they belong to different $\R$-orbits, i.e.,  $\gamma_1(\theta+\cdot)\ne\gamma_2$ for all $\theta\in\mathbb{R}$.

%

 \begin{definition}\label{def:orbitBifur}
{\rm
The $\R$-orbits of solutions to (\ref{e:Lagr10*})
are said to be \textsf{sequentially bifurcating
from the $\R$-orbit $\R\cdot \bar\gamma$ at $\mu$}
if there exists a sequence of parameters $\lambda_k\to\mu$ in $\Lambda$
 and, correspondingly,  solutions $\gamma^k$ of (\ref{e:Lagr10*}) with  $\lambda=\lambda_k$, satisfying:
 \begin{enumerate}
\item[\rm (i)] $\gamma^k\notin\R\cdot \bar\gamma\;\forall k$,
\item[\rm (ii)] the  solutions $\gamma^k$ are pairwise $\R$-distinct,
\item[\rm (iii)] $\gamma^k|_{[0,\tau]}\to \bar\gamma|_{[0,\tau]}$ in $C^1([0,\tau];M)$.
\end{enumerate}
}
\end{definition}
[Clearly, passing to a subsequence (i) is implied in (ii).]

Consider the functional $\mathfrak{E}_\lambda$ on $\mathcal{X}^1_{\tau}(M, \mathbb{I}_g)$ defined by
 \begin{equation}\label{e:Lagr9}
\mathfrak{E}_\lambda(\gamma)=\int^\tau_0L_\lambda(t, \gamma(t), \dot{\gamma}(t))dt.
\end{equation}
It follows from the proof of [24, Proposition~4.2] that
$\mathfrak{E}_\lambda$ is of class $C^2$.

 \cite[Proposition~4.2]{BuGiHi} shows that
$\gamma\in \mathcal{X}^1_{\tau}(M, \mathbb{I}_g)$ is
 a critical point of $\mathfrak{E}_\lambda$
if and only if it belongs to $\mathcal{X}^3_{\tau}(M, \mathbb{I}_g)$
and satisfies  (\ref{e:Lagr10*}).
According to \cite{Du}, the second-order differential
 $D^2\mathfrak{E}_\lambda(\gamma)$ of $\mathfrak{E}_\lambda$
 at such a critical point $\gamma$
can be extended into a continuous symmetric bilinear form on $W^{1,2}_{\mathbb{I}_g}(\gamma^\ast TM)$.
Moreover, the associated Morse index and nullity, defined as
\begin{equation}\label{e:Lagr12}
m^-_\tau(\mathfrak{E}_\lambda,\gamma)\quad\hbox{and}\quad m^0_\tau(\mathfrak{E}_\lambda,\gamma),
\end{equation}
are finite. These quantities are well-defined by \cite[\S4]{Dav}
and are called the \textsf{Morse index and nullity of $\mathfrak{E}_\lambda$ at $\gamma$}, respectively.

\subsubsection{Bifurcations of (\ref{e:Lagr10*}) starting at constant solutions}\label{sec:auto.1.1}

\begin{theorem}[\textsf{Alternative bifurcations of Fadell-Rabinowitz's type}]\label{th:bif-per3LagrMan}
Under Assumption~\ref{ass:Lagr6II}  with $\Lambda$ being a real interval, suppose also that
 $\mathbb{I}_g$ satisfies $\mathbb{I}_g^l=id_M$ for some $l\in\mathbb{N}$. Let
\begin{equation}\label{e:constLagrMan}
\left.\begin{array}{ll}
&\Lambda\ni\lambda\to \gamma_\lambda\in {\rm Fix}(\mathbb{I}_g)\subset M\;\hbox{be continuous and}\\
&\partial_xL_{\lambda}(\gamma_\lambda, 0)=0\;\forall \lambda\in\Lambda.
\end{array}\right\}
\end{equation}
{\rm (}Therefore $\gamma_\lambda$ is a constant solution of (\ref{e:Lagr10*}).
\textsf{Hereafter the points $\gamma_\lambda$ are also
understood as constant value maps from $\R$ to $M$ without special statements}.{\rm )}
Suppose that for some $\mu\in{\rm Int}(\Lambda)$ and $\tau>0$,
\begin{enumerate}
\item[\rm (a)]  $\partial_{vv}L_\mu\left(\gamma_\mu, 0\right)$ is positive definite;
\item[\rm (b)] $\partial_{xx}L_\mu\left(\gamma_\mu, 0\right)u=0$ and
 $d\mathbb{I}_g(\gamma_\mu)u=u$ have only the zero solution in $T_{\gamma_\mu}M$;
\item[\rm (c)]   $m^0_\tau(\mathfrak{E}_{\mu}, \gamma_\mu)\ne 0$,
$m^0_\tau(\mathfrak{E}_{\lambda}, \gamma_\mu)=0$
 for each $\lambda\in\Lambda\setminus\{\mu\}$ near $\mu$, and
 $m^-_\tau(\mathfrak{E}_{\lambda}, \gamma_\mu)$  take, respectively, values
 $m^-_\tau(\mathfrak{E}_{\mu}, \gamma_\mu)$ and
  $m^-_\tau(\mathfrak{E}_{\mu}, \gamma_\mu)+m^0_\tau(\mathfrak{E}_{\mu}, \gamma_\mu)$
 as $\lambda\in\Lambda$ varies in two deleted half neighborhoods  of $\mu$.
 \end{enumerate}
 Then one of the following alternatives occurs:
 \begin{enumerate}
\item[\rm (i)] The problem  (\ref{e:Lagr10*}) with $\lambda=\mu$ has a sequence of $\R$-distinct solutions, $\gamma^k$, $k=1,2,\cdots$,
which are $\R$-distinct with $\gamma_\mu$ and  converges to $\gamma_\mu$ on any compact interval $I\subset\R$ in $C^2$-topology.
\item[\rm (ii)] There exist left and right  neighborhoods $\Lambda^-$ and $\Lambda^+$ of $\mu$ in $\Lambda$
and integers $n^+, n^-\ge 0$, such that $n^++n^-\ge m^0_\tau(\mathfrak{E}_{\mu}, \gamma_\mu)/2$,  and for $\lambda\in\Lambda^-\setminus\{\mu\}$ {\rm (}resp. $\lambda\in\Lambda^+\setminus\{\mu\}${\rm )},
(\ref{e:Lagr10*}) with parameter value $\lambda$  has at least $n^-$ {\rm (}resp. $n^+${\rm )} $\R$-distinct solutions
 solutions, $\gamma_\lambda^i$, $i=1,\cdots,n^-$ {\rm (}resp. $n^+${\rm )},
which are $\R$-distinct with $\gamma_\mu$ and converge to  $\gamma_\mu$ on any compact interval $I\subset\R$ in $C^2$-topology as $\lambda\to\mu$.
\end{enumerate}
Moreover, if $m^0_\tau(\mathfrak{E}_{\mu}, \gamma_\mu)\ge 3$,
then at least one of (i), (iii) and (iv) holds, where (i) is as stated above, and the assertions of (iii) and (iv) are as follows:
\begin{enumerate}
\item[\rm (iii)]  For every $\lambda\in\Lambda\setminus\{\mu\}$ near $\mu$ there is a
 solution $\alpha_\lambda\notin\R\cdot \gamma_\lambda$  of  (\ref{e:Lagr10*})
 with parameter value $\lambda$, such that $\alpha_\lambda-\gamma_\lambda$ converges to zero
on any compact interval $I\subset\R$ in $C^2$-topology as $\lambda\to\mu$.

\item[\rm (iv)] For a given small $\epsilon>0$ there is a one-sided  neighborhood $\Lambda^0$ of $\mu$ in $\Lambda$ such that for any $\lambda\in\Lambda^0\setminus\{\mu\}$, (\ref{e:Lagr10*})
    with parameter value $\lambda$ has either infinitely many $\mathbb{R}$-distinct solutions $\bar{\alpha}_\lambda^k\notin\mathbb{R}\cdot \gamma_\lambda$ such that $\|\bar{\alpha}_\lambda^k|_{[0,\tau]}-\gamma_\lambda|_{[0,\tau]}\|_{C^2([0,\tau];\mathbb{R}^N)}<\epsilon$, $k=1,2,\cdots$, or
at least two $\mathbb{R}$-distinct solutions $\beta_\lambda^1\notin\mathbb{R}\cdot \gamma_\lambda$ and $\beta_\lambda^2\notin\mathbb{R}\cdot \gamma_\lambda$
such that $\|{\beta}_\lambda^i|_{[0,\tau]}-\gamma_\lambda|_{[0,\tau]}\|_{C^2([0,\tau];\mathbb{R}^N)}<\epsilon$, $i=1,2$,
 and that $\mathfrak{E}_{\lambda}(\beta^1_\lambda)\ne \mathfrak{E}_{\lambda}(\beta^2_\lambda)$. {\rm (}Recall that we have assumed $M\subset\mathbb{R}^N$.{\rm )}
\end{enumerate}
\end{theorem}

\subsubsection{Bifurcations of (\ref{e:Lagr10*}) starting at nonconstant solutions}\label{sec:auto.1.2}

We need make stronger:

\begin{assumption}\label{ass:Lagr10}
{\rm
Let $(M, g, \mathbb{I}_g)$ be as in ``Basic assumptions and conventions'' in Introduction.
For a real $\tau>0$ and a topological space $\Lambda$,
let $L:\Lambda\times TM\to\R$ be a continuous function such that each
$L(\lambda,\cdot): TM\to\R$, $\lambda\in\Lambda$, is $C^6$ and all its partial derivatives of order no more than two depend continuously on
 $(\lambda, x, v)\in\Lambda\times TM$. Each $L_\lambda(\cdot)=L(\lambda,\cdot)$ is fiberwise strictly convex,
 and  $\mathbb{I}_g$-invariant (i.e.,
$L(\lambda, \mathbb{I}_g(x), d\mathbb{I}_g(x)[v])=L(\lambda,  x,v)$ for all $(x,v)\in TM$).
Let  $\bar\gamma:\mathbb{R}\to M$ be a nonconstant $C^2$ map satisfying
(\ref{e:Lagr10*}) with this $L$ for all $\lambda\in\Lambda$. ($\bar\gamma$ is actually $C^6$ by \cite[Proposition~4.3]{BuGiHi}.)}
\end{assumption}

Under this assumption, note that $m^0_\tau(\mathscr{E}_{\lambda}, \bar\gamma)\ne 0$ for all $\lambda\in\Lambda$ because all elements in
the $\R$-orbit $\R\cdot \bar\gamma$  also satisfies (\ref{e:Lagr10*}) with this $L$ for all $\lambda\in\Lambda$.

\begin{theorem}[\textsf{Necessary condition}]\label{th:bif-ness-orbitLagrMan}
Under Assumption~\ref{ass:Lagr10}, suppose that
$\R$-orbits of solutions of the problem (\ref{e:Lagr10*})
with a parameter $\lambda\in\Lambda$
sequentially bifurcate from the $\R$-orbit $\R\cdot \bar\gamma$ at $\mu$.
Then $m^0_\tau(\mathfrak{E}_{\mu}, \bar\gamma)\ge 2$.
\end{theorem}

\begin{theorem}[\textsf{Sufficient condition}]\label{th:bif-suffict1-orbitLagrMan}
Under Assumption~\ref{ass:Lagr10}, suppose that $\Lambda$ is first countable, $\mu\in\Lambda$ and:
\begin{description}
\item[\rm (a)] $\bar\gamma$ is periodic, and $m^0_\tau(\mathfrak{E}_{\mu}, \bar\gamma)\ge 2$;
\item[\rm (b)] there exist two sequences in  $\Lambda$ converging to $\mu$, $(\lambda_k^-)$ and
$(\lambda_k^+)$,  such that for each $k\in\mathbb{N}$,
 $$
 [m^-_\tau(\mathfrak{E}_{\lambda_k^-}, \bar\gamma), m^-_\tau(\mathfrak{E}_{\lambda_k^-}, \bar\gamma)+
 m^0_\tau(\mathfrak{E}_{\lambda_k^-}, \bar\gamma)-1]\cap[m^-_\tau(\mathfrak{E}_{\lambda_k^+}, \bar\gamma),
 m^-_\tau(\mathfrak{E}_{\lambda_k^+}, \bar\gamma)+m^0_\tau(\mathfrak{E}_{\lambda_k^+}, \bar\gamma)-1]=\emptyset
 $$
 and either $m^0_\tau(\mathfrak{E}_{\lambda_k^-}, \bar\gamma)=1$ or $m^0_\tau(\mathfrak{E}_{\lambda_k^+}, \bar\gamma)=1$;
 \item[\rm (c)] for any solution $\gamma$ of (\ref{e:Lagr10*}) with $\lambda=\mu$,
 if there exists a sequence $(s_k)$ of reals such that  $s_k\cdot\gamma$
  converges to  $\bar\gamma$ on any compact interval $I\subset\mathbb{R}$ in $C^1$-topology, then $\gamma$ is periodic.
  {\rm (}Clearly, this holds if $(\mathbb{I}_g)^l=id_M$ for some $l\in\mathbb{N}$.{\rm )}
   \end{description}
  Then  there exists a sequence $(\lambda_k)\subset\hat{\Lambda}=\{\mu,\lambda^+_k, \lambda^-_k\,|\,k\in\mathbb{N}\}$ converging to $\mu$ and
 $C^6$ solutions $\gamma_k$ of the corresponding problem (\ref{e:Lagr10*}) with $\lambda=\lambda_k$, $k=1,2,\cdots$,
  such that any two of these $\gamma_k$ are $\mathbb{R}$-distinct and that
  $(\gamma_k)$ converges to  $\bar\gamma$ on any compact interval $I\subset\mathbb{R}$ in $C^2$-topology as $k\to\infty$.
  \end{theorem}

\begin{theorem}[\hbox{\textsf{Existence for bifurcations}}]\label{th:bif-existence-orbitLagrMan}
Under Assumption~\ref{ass:Lagr10}, suppose that $\Lambda$ is path-connected,
$(\mathbb{I}_g)^l=id_M$ for some $l\in\mathbb{N}$,
and the following is satisfied:
\begin{enumerate}
\item[\rm (d)] There exist two  points $\lambda^+, \lambda^-\in\Lambda$ such that
$$
 [m^-_\tau(\mathfrak{E}_{\lambda^-}, \bar\gamma), m^-_\tau(\mathfrak{E}_{\lambda^-}, \bar\gamma)+
 m^0_\tau(\mathfrak{E}_{\lambda^-}, \bar\gamma)-1]\cap[m^-_\tau(\mathfrak{E}_{\lambda^+}, \bar\gamma),
 m^-_\tau(\mathfrak{E}_{\lambda^+}, \bar\gamma)+m^0_\tau(\mathfrak{E}_{\lambda^+}, \bar\gamma)-1]=\emptyset
 $$
 and either $m^0_\tau(\mathfrak{E}_{\lambda^-}, \bar\gamma)=1$ or $m^0_\tau(\mathfrak{E}_{\lambda^+}, \bar\gamma)=1$.
  \end{enumerate}
 Then for any path $\alpha:[0,1]\to\Lambda$ connecting $\lambda^+$ to $\lambda^-$ there exists
  a sequence $(\lambda_k)$ in $\alpha([0,1])$ converging to $\mu\in \alpha([0,1])\subset\Lambda$,
and $C^6$ solutions $\gamma_k$ of the corresponding problem (\ref{e:Lagr10*}) with $\lambda=\lambda_k$, $k=1,2,\cdots$,
   such that any two of these $\gamma_k$ are $\R$-distinct and that
  $(\gamma_k)$ converges to  $\bar\gamma$ on any compact interval $I\subset\R$ in $C^2$-topology as $k\to\infty$.
   Moreover, this $\mu$ is not equal to $\lambda^+$ {\rm (}resp. $\lambda^-${\rm )} if
   $m^0_\tau(\mathfrak{E}_{\lambda^+}, \bar\gamma)=1$ {\rm (}resp. $m^0_\tau(\mathfrak{E}_{\lambda^-}, \bar\gamma)=1${\rm )}.
 \end{theorem}

\begin{theorem}[\textsf{Alternative bifurcations of Rabinowitz's type}]\label{th:bif-suffict-orbitLagrMan}
Under Assumption~\ref{ass:Lagr10} with $\Lambda$ being a real interval,
let $\mu\in{\rm Int}(\Lambda)$, $\mathbb{I}_g=id_M$ and $\bar\gamma$
 has least period $\tau$.  Suppose that
\begin{center}
$m^0_\tau(\mathfrak{E}_{\mu}, \bar\gamma)\ge 2$, \quad $m^0_\tau(\mathfrak{E}_{\lambda}, \bar\gamma)=1$
 for each $\lambda\in\Lambda\setminus\{\mu\}$ near $0$,
 \end{center}
  and that
 $m^-_\tau(\mathfrak{E}_{\lambda}, \bar\gamma)$  take, respectively, values $m^-_\tau(\mathfrak{E}_{\mu}, \bar\gamma)$ and
 $m^-_\tau(\mathfrak{E}_{\mu}, \bar\gamma)+ m^0_\tau(\mathfrak{E}_{\mu}, \bar\gamma)-1$
 as $\lambda\in\Lambda$ varies in two deleted half neighborhoods  of $0$.
 Then  one of the following alternatives occurs:
\begin{enumerate}
\item[\rm (i)] The corresponding problem (\ref{e:Lagr10*}) with $\lambda=\mu$ has a sequence of pairwise $\R$-distinct
$C^6$ solutions, $\gamma_k$, $k=1,2,\cdots$,
such that $(\gamma_k)$ converges to $\bar\gamma$ on any compact interval $I\subset\R$ in $C^2$-topology as $k\to\infty$.

\item[\rm (ii)]  For every $\lambda\in\Lambda\setminus\{\mu\}$ near $\mu$ there is a $C^6$ solution $\gamma_\lambda$ of
(\ref{e:Lagr10*}) with parameter value $\lambda$, which is $\R$-distinct with $\bar\gamma$ and
converges to $\bar\gamma$ on any compact interval $I\subset\R$ in $C^2$-topology as $\lambda\to \mu$.

\item[\rm (iii)] For a given neighborhood $\mathcal{W}$ of $\bar\gamma$ in $\mathcal{X}^1_{\tau}(M, \mathbb{I}_g)$, there exists
 a one-sided  neighborhood $\Lambda^0$ of $\mu$ such that
for any $\lambda\in\Lambda^0\setminus\{\mu\}$, (\ref{e:Lagr10*}) with parameter value $\lambda$
has at least two $\R$-distinct $C^6$ solutions in $\mathcal{W}$, $\gamma_\lambda^1\notin\mathbb{R}\cdot\bar\gamma$ and $\gamma_\lambda^2\notin\mathbb{R}\cdot\bar\gamma$,
which can also be chosen to satisfy  $\mathfrak{E}_{\lambda}(\gamma_\lambda^1)\ne\mathfrak{E}_{\lambda}(\gamma_\lambda^2)$
provided that $m^0_\tau(\mathfrak{E}_{\mu}, \bar\gamma)\ge 3$ and (\ref{e:Lagr10*}) with parameter value $\lambda$ has
only finitely many $\mathbb{R}$-distinct solutions in $\mathcal{W}$ which are  $\mathbb{R}$-distinct from $\bar\gamma$.
\end{enumerate}
\end{theorem}

  In the above Theorems~\ref{th:bif-ness-orbitLagrMan},~\ref{th:bif-suffict1-orbitLagrMan},~\ref{th:bif-existence-orbitLagrMan},~\ref{th:bif-suffict-orbitLagrMan},
  if $M$ is an open subset $U$ of $\mathbb{R}^n$ and $\mathbb{I}_g$ is an  orthogonal matrix $E$ of order $n$
  which maintain $U$ invariant,  Assumption~\ref{ass:Lagr10}
   can be replaced by a weaker Assumption~\ref{ass:BasiAssLagrS}, see Section~\ref{sec:AutoLagr5}.

\section{Proof of Theorem~\ref{th:bif-per3LagrMan}}\label{sec:Lagr3}
\setcounter{equation}{0}

\noindent{\bf Step 1} ({\it Reduction of the problem (\ref{e:Lagr10*}) to one on open subsets of $\mathbb{R}^n$}).
We can assume $\Lambda=[\mu-\varepsilon,\mu+\varepsilon]$ for some $\varepsilon>0$.
(\ref{e:constLagrMan}) implies that each $\gamma_\lambda$
is a constant solution of (\ref{e:Lagr10*}) for any $\tau>0$.
Since $\mathbb{I}_g^l=id_M$,  each solution of (\ref{e:Lagr10*}) is $l\tau$-periodic. All solutions of
(\ref{e:Lagr10*})  near $\gamma_\mu$ sit in a compact neighborhood of $\gamma_\mu\in M$.
Let $e_1,\cdots, e_n$ be a unit orthogonal
 frame at $T_{\gamma_\mu}M$. Then
 $(\mathbb{I}_{g\ast}e_1,\cdots,
\mathbb{I}_{g\ast}e_n)=(e_1,\cdots, e_n)E_{{\gamma}_\mu}$
for a unique orthogonal matrix  $E_{{\gamma}_\mu}$. Clearly, $E_{{\gamma}_\mu}^l=I_{n}$.
Let $B^n_{2\iota}(0):=\{x\in\R^n\,|\, |x|<2\iota\}$ and $\exp$ denote
 the exponential map of $g$.  Then
\begin{equation}\label{e:FixChart}
\phi: B^n_{2\iota}(0)\to
M,\;x\mapsto\exp_{{\gamma}_\mu}\left(\sum^n_{i=1}x_ie_i\right)
\end{equation}
 is a $C^5$ embedding of codimension zero and  satisfies
 $\phi(E_{\gamma_\mu}x)=\mathbb{I}_g\phi(x)$  and $d\phi(0)[y]=\sum^n_{i=1}y_ie_i$
 for any $y\in\mathbb{R}^n$. Shrinking $\Lambda$ toward $\mu$ (if necessary) we may assume
 that $\gamma_\lambda\in\phi(B^n_{2\iota}(0))$ for all $\lambda\in\Lambda$.
 (This is possible because $\Lambda\times\mathbb{R}\ni (\lambda,t)\mapsto\gamma_\lambda(t)\in M$ is continuous.)
 Therefore there exists a unique  $x_\lambda\in B^n_{2\iota}(0)$ such that $\phi(x_\lambda)=\gamma_\lambda$
 for all $\lambda\in\Lambda$. Clearly, $x_\mu=0$, $E_{\gamma_\mu}x_\lambda=x_\lambda\;\forall\lambda $ and
 $\Lambda\ni\lambda\to x_\lambda\in B^n_{2\iota}(0)$ is continuous.   Define
  $$
  L^\ast: \Lambda\times  B^n_{2\iota}(0)\times\R^{n}\to \R,\;
 (\lambda, x, y)\mapsto L(\lambda, \phi(x), d\phi(x)[y]).
 $$
 It is $C^2$ with respect to $(x,y)$ and strictly convex with respect to $y$, and
all its partial derivatives also depend continuously on $(\lambda, x, y)$.
Moreover, $L^\ast_\lambda=L^\ast(\lambda,\cdot)$ is also $E_{\gamma_\mu}$-invariant.
From (\ref{e:constLagrMan}) we derive that $L^\ast_\lambda(x,0)=L^\ast(\lambda, x,0)$ has the differential at $x_\lambda\in\mathbb{R}^n$,
$$
\partial_xL^\ast_\lambda(x_\lambda,0)[y]=\partial_qL(\lambda, \gamma_\lambda, 0)[d\phi(0)[y]]=0\quad\forall\lambda\in\Lambda.
$$
These show that $L^\ast$ satisfies
\begin{equation}\label{e:constLagr}
\left.\begin{array}{ll}
&\Lambda\ni\lambda\to x_\lambda\in U\cap {\rm Ker}(E_{{\gamma}_\mu}-I_{n})\;\hbox{be continuous and}\\
&\partial_qL_{\lambda}(x_\lambda, 0)=0\;\forall \lambda\in\Lambda.
\end{array}\right\}
\end{equation}

Recall that the Banach spaces
$$
\mathcal{X}^i_{\tau}(\mathbb{R}^n, E_{\gamma_\mu}):=\{\gamma\in C^i(\mathbb{R}, \mathbb{R}^n)\,|\, E_{\gamma_\mu}(\gamma(t))=\gamma(t+\tau)\;\forall t\}
$$
 with the induced norm $\|\xi\|_{C^i}$ from $C^1([0,\tau],\mathbb{R}^n)$, $i=0,1,2,\cdots$.
Consider the functional defined on the open subset $\mathcal{X}^1_{\tau}(B^n_{2\iota}(0)), E_{\gamma_\mu})$ of  $\mathcal{X}^1_{\tau}(\mathbb{R}^n, E_{\gamma_\mu})$,
\begin{equation}\label{e:Lagr9**}
x\mapsto\mathfrak{E}^\ast_\lambda(x)=\int^\tau_0L^\ast_\lambda(x(t), \dot{x}(t))dt.
\end{equation}
Note that the map $\phi$ induces an $C^2$ embedding of codimension zero
$$
\Phi_{\gamma_\mu}:\mathcal{X}^1_{\tau}(B^n_{2\iota}(0)), E_{\gamma_\mu})\to
\mathcal{X}^1_{\tau}(M, \mathbb{I}_g),\;x\mapsto\phi\circ x.
$$
It is not hard to see that $\Phi_{\gamma_\mu}(x_\lambda)=\gamma_\lambda$ and
 \begin{equation}\label{e:Lagr9***}
D^2\mathfrak{E}^\ast_\lambda({x}_\lambda)[\xi,\eta]=
D^2\mathfrak{E}_\lambda(\gamma_\lambda)[d\Phi_{{\gamma}_\mu}({x}_\lambda)\xi,
d\Phi_{{\gamma}_\mu}({x}_\lambda)\eta],\quad\forall \xi,\eta\in \mathcal{X}^1_\tau(\mathbb{R}^n, E_{\gamma_\mu}).
\end{equation}
We conclude that the conditions (a),(b) and (c) in Theorem~\ref{th:bif-per3LagrMan} are, respectively, equivalent to
the following three conditions:
\begin{enumerate}
\item[\rm (a')] $\partial_{xx}L^\ast_\mu(x_\mu, 0)$ is positive definite.
\item[\rm (b')] The system of equations
$
\begin{cases}
\partial_{xx}{L}^\ast_\mu(x_\mu, 0)(a_1,\cdots,a_n)^T=0,\\ E_{\gamma_\mu}(a_1,\cdots,a_n)^T=(a_1,\cdots,a_n)^T
\end{cases}
$
has only the zero solution in $\mathbb{R}^n$.

\item[\rm (c')]  $m^0_\tau(\mathfrak{E}^\ast_{\mu}, 0)\ne 0$,
$m^0_\tau(\mathfrak{E}^\ast_{\lambda}, x_\lambda)=0$
 for each $\lambda\in\Lambda\setminus\{\mu\}$ near $\mu$, and
 $m^-_\tau(\mathfrak{E}^\ast_{\lambda}, x_\lambda)$  takes, respectively, values
 $m^-_\tau(\mathfrak{E}^\ast_{\mu}, 0)$ and
  $m^-_\tau(\mathfrak{E}^\ast_{\mu}, 0)+m^0_\tau(\mathfrak{E}^\ast_{\mu}, 0)$
 as $\lambda\in\Lambda$ varies in two deleted half neighborhoods  of $\mu$.
 \end{enumerate}

In fact, let $\mathscr{H}_{\lambda,\gamma_\lambda}$ be the Hessian bilinear form of the function $M\ni q\mapsto L(\lambda,q,0)$ at $\gamma_\lambda\in M$.
Then
$$
\mathscr{H}_{\lambda,\gamma_\lambda}(u,v)=g_{\gamma_\lambda}(\partial_{qq}L_\mu\left(\gamma_\lambda, 0\right)u,v)\quad\forall
u,v\in T_{\gamma_\lambda}M.
$$
View $(x_1,\cdots,x_n)\in B^n_{2\iota}(0)$ as local coordinates at $\phi(x)\in M$ and write
$u=\sum^n_{i=1}a_i\frac{\partial}{\partial x_i}|_{\gamma_\mu}$ and
$v=\sum^n_{i=1}b_i\frac{\partial}{\partial x_i}|_{\gamma_\mu}$. We have
\begin{eqnarray*}
&\mathbb{I}_{g\ast}u=u\quad\Longleftrightarrow\quad E_{\gamma_\mu}(a_1,\cdots,a_n)^T=(a_1,\cdots,a_n)^T,\\
&g_{\gamma_\mu}\left(\partial_{qq}L_\mu(\gamma_\mu, 0)u,v\right)=\mathscr{H}_{\mu,\gamma_\mu}(u,v)=\sum^n_{i,j=1}a_ib_j
\frac{\partial^2L^\ast_\mu}{\partial x_i\partial x_j}(0,0)\\
&=\left(\partial_{xx}L^\ast_\mu(x_\mu, 0)(a_1,\cdots,a_n)^T, (b_1,\cdots,b_n)^T\right)_{\mathbb{R}^n}.
\end{eqnarray*}
Hence the conditions (a) and (b) in Theorem~\ref{th:bif-per3LagrMan}
are equivalent to (a') and (b'), respectively.
Clearly, (\ref{e:Lagr9***}) implies the equivalence between (c) and (c').

Since $\lambda\mapsto x_\lambda$ is continuous and $x_\mu=0$
we can shrink $\varepsilon>0$ in $\Lambda=[\mu-\varepsilon,\mu+\varepsilon]$ so that $x_\lambda\in B^n_{\iota}(0)$
for all $\lambda\in\Lambda$.
Define $\tilde L^\ast:\Lambda\times B^n_{\iota}(0)\times\R^n\to\R$ by
\begin{equation*}
\tilde L^\ast(\lambda, x,y)=\tilde L^\ast_{\lambda}(x,y)= {L}^\ast(\lambda, x+ x_\lambda, y).
\end{equation*}
For a given positive $\rho_0>0$, by Lemma~2.4 in \cite{Lu12-} or \cite{Lu12} 
we have a  continuous function
 $\tilde{L}^\ast:\Lambda\times B^n_{\iota}(0)\times\mathbb{R}^n\to\R$ and a constant $\kappa>0$ satisfying the following properties:
 \begin{enumerate}
  \item[\rm (i)] $\tilde{L}^\ast$  is equal to $L^\ast$ on $\Lambda\times B^n_{\iota}(0)\times B^n_{\rho_0}(0)$.
  \item[\rm (ii)] $\tilde{L}^\ast$ is $C^2$ with respect to $(x,y)$ and strictly convex with respect to $y$, and
all partial derivatives of $\tilde{L}^\ast$ also depend continuously on $(\lambda, x, y)$.
Moreover, each $\tilde{L}^\ast_\lambda=\tilde{L}^\ast(\lambda,\cdot)$ is also $E_{\gamma_\mu}$-invariant.
      \item[\rm (iii)] There exists a constant $C>0$ such that
    $$
    \tilde{L}^\ast_\lambda(x, y)\ge \kappa|y|^2-C,\quad\forall (\lambda, x,y)\in \Lambda\times B^n_{3\iota/4}(0)\times\mathbb{R}^n.
    $$
    \item[\rm (iv)] $\partial_{xx}\tilde{L}^\ast_\mu(0, 0)$ is positive definite;
\item[\rm (v)] The system of equations
$
\begin{cases}
\partial_{xx}\tilde{L}^\ast_\mu(0, 0)(a_1,\cdots,a_n)^T=0,\\ E_{\gamma_\mu}(a_1,\cdots,a_n)^T=(a_1,\cdots,a_n)^T
\end{cases}
$
has only the zero solution in $\mathbb{R}^n$.
     \end{enumerate}

(\textsf{{Note}}: Applying \cite[Theorem~8.12]{Lu12}
to $\tilde{L}^\ast$ we can also complete the required proof.)

 Since $\Lambda=[\mu-\varepsilon,\mu+\varepsilon]$ is compact, by shrinking $\varepsilon>0$ (if necessary)
as in Lemma~3.8 in \cite{Lu12-} or \cite{Lu12}
we can modify $\tilde{L}^\ast$ to get a  continuous function $\check{L}:\Lambda\times B^n_{3\iota/4}(0)\times\mathbb{R}^n\to\R$
 satisfying the following properties  for some constants $\check\kappa>0$ and $0<\check{c}<\check{C}$:
\begin{enumerate}
 \item[\rm (L0)] $\tilde{L}^\ast$ in the above (ii) and (iv)-(v) is changed into $\check{L}$.
\item[\rm (L1)] $\check{L}$ and $\tilde{L}^\ast$ are equal in $\Lambda\times B^n_{3\iota/4}(0)\times B^n_{\rho_0}(0)$.

\item[\rm (L2)] $\partial_{yy}\check{L}_\lambda(x,y)\ge\check{c}I_n,\quad \forall (\lambda, x,y)\in \Lambda\times B^n_{3\iota/4}(0)\times\mathbb{R}^n$.

\item[\rm (L3)] $\Bigl| \frac{\partial^2}{\partial x_i\partial
x_j}\check{L}_\lambda(x,y)\Bigr|\le \check{C}(1+ |y|^2),\quad \Bigl|
\frac{\partial^2}{\partial x_i\partial y_j}\check{L}_\lambda(x,y)\Bigr|\le\check{C}(1+
|y|),\quad\hbox{and}\\
 \Bigl| \frac{\partial^2}{\partial y_i\partial y_j}\check{L}_\lambda(x,y)\Bigr|\le \check{C},
 \quad\forall (\lambda, x,y)\in\Lambda\times B^n_{3\iota/4}(0)\times\mathbb{R}^n$.

\item[\rm (L4)]  $\check{L}(\lambda, x, y)\ge \check{\kappa}|y|^2-\check{C},\quad\forall (\lambda, x, y)\in\Lambda\times B^n_{3\iota/4}(0)\times\R^{n}$.
\item[\rm (L5)] $|\partial_{q}\check{L}(\lambda, x,y)|\le \check{C}(1+ |y|^2)$ and $|\partial_{y}\check{L}(\lambda,x,y)|\le \check{C}(1+ |y|)$
for all $(\lambda, x, y)\in \Lambda\times B^n_{3\iota/4}(0)\times\R^{n}$.
\item[\rm (L6)] $|\check{L}_\lambda(x,y)|\le \check{C}(1+|y|^2),\quad \forall (\lambda, x,y)\in \Lambda\times {B}^n_{3\iota/4}(0)\times\mathbb{R}^n$.
 \end{enumerate}

 Consider the Hilbert space
$$
{\bf H}_{E_{\gamma_\mu}}:=\{\xi\in W^{1,2}_{\rm loc}(\mathbb{R}; \mathbb{R}^n)\,|\, E_{\gamma_\mu}^T(\xi(t))=\xi(t+\tau)\;\forall t\in\mathbb{R}\}
$$
equipped with $W^{1,2}$-inner product
\begin{eqnarray}\label{e:innerP2}
(u,v)_{1,2}=\int^\tau_0[(u(t),v(t))_{\mathbb{R}^{2n}}+ (\dot{u},\dot{v})_{\mathbb{R}^{n}}]dt.
\end{eqnarray}
Since $E_{{\gamma}_\mu}^l=I_{n}$, the spaces
${\bf X}_{E_{\gamma_\mu}}:=\mathcal{X}^1_{\tau}(\mathbb{R}^n, E_{\gamma_\mu})$
and ${\bf H}_{E_{\gamma_\mu}}$  carry a natural
 $S^1$-action with $S^1=\mathbb{R}/(lT\mathbb{Z})$  given by
 \begin{equation}\label{e:S^1-action}
 (\theta\cdot x)(t)=x(t+\theta),\quad\theta\in\mathbb{R},
 \end{equation}
  and  have the following $S^1$-invariant open subsets
\begin{eqnarray*}
\mathcal{U}_{E_{\gamma_\mu}}:&=&\{\xi\in W^{1,2}_{\rm loc}(\mathbb{R}; B^n_{3\iota/4}(0))\,|\, {E}^T_{\bar\gamma}(\xi(t))=\xi(t+\tau)\;\forall t\in\mathbb{R}\},\\
\mathcal{U}^X_{E_{\gamma_\mu}}:&=&\mathcal{U}_{E_{\gamma_\mu}}\cap{\bf X}_{E_{\gamma_\mu}}=\mathcal{X}^1_{\tau}(B^n_{3\iota/4}(0)), E_{\gamma_\mu})
\end{eqnarray*}
respectively. For each $\lambda\in\Lambda$, let us define  functionals
\begin{eqnarray}\label{e:Fixcheck*}
&&\mathfrak{E}^\ast_\lambda:\mathcal{X}^1_{\tau}(B^n_{\iota}(0)), E_{\gamma_\mu})\to\mathbb{R},\;x\mapsto\tilde{\mathfrak{E}}^\ast_\lambda(x)=
\int^\tau_0\tilde{L}^\ast_\lambda(x(t), \dot{x}(t))dt, \\
&&\check{\mathscr{L}}_\lambda:\mathcal{U}_{E_{\gamma_\mu}}\to\R,\;
x\mapsto\check{\mathscr{L}}_{\lambda}(x)=
\int^{\tau}_0\check{L}_\lambda(x(t),\dot{x}(t))dt.\label{e:Fixcheck*+}
\end{eqnarray}
They are invariant for the above $S^1$-action.

Corresponding to Propositions~4.5, 4.6, 4.7 in \cite{Lu12-} or \cite{Lu12}
we, respectively, have:

\begin{proposition}\label{prop:Fix-Pfunct-analy}
\begin{enumerate}
\item[\rm (i)] Each $\check{\mathcal{L}}_{\lambda}$ is
$C^{2-0}$ and twice G\^ateaux-differentiable, and $d\check{\mathcal{L}}_{\lambda}(0)=0$ and
 \begin{equation}\label{e:MorseindexEu}
 m^\star(\mathfrak{E}^\ast_\lambda, 0)=m^\star(\check{\mathcal{L}}_{\lambda}|_{\mathcal{U}^X},0)=
m^\star(\check{\mathscr{L}}_{\lambda},0),\quad\star=-,0.
\end{equation}
\item[\rm (ii)] Each critical point of $\check{\mathcal{L}}_\lambda$ sits in
$C^2\big([0,\tau]; B^n_{\iota/2}(0)\big)\cap\mathcal{U}^X$, and satisfies
the boundary problem:
\begin{equation}\label{e:PPerLagrorbit}
\left.\begin{array}{ll}
&\frac{d}{dt}\Big(\partial_v\check{L}_\lambda(x(t), \dot{x}(t))\Big)-\partial_q\check{L}_\lambda(x(t), \dot{x}(t))=0\;\forall t\in\mathbb{R}\\
&E_{\gamma_\mu}x(t)=x(t+\tau)\quad\forall t\in\mathbb{R}.
\end{array}\right\}
\end{equation}

\item[\rm (iii)]  The gradient of $\check{\mathcal{L}}_{\lambda}$ at $x\in\mathcal{U}$, denoted by
 $\nabla\check{\mathcal{L}}_{\lambda}(x)$, is given by
\begin{eqnarray}\label{e:4.14}
\nabla\check{\mathcal{L}}_{\lambda}(x)(t)&=&e^t\int^t_0\left[
e^{-2s}\int^s_0e^{r}f_{\lambda,x}(r)dr\right]ds  + c_1(\lambda,x)e^t+
c_2(\lambda,x)e^{-t}\nonumber\\
&&\qquad +\int^t_0 \partial_v \check{L}_\lambda(s, x(s),\dot{x}(s))ds ,
\end{eqnarray}
where $c_1(\lambda,x), c_2(\lambda,x)\in\R^n$ are suitable constant vectors and 
\begin{eqnarray}\label{e:4.13}
 f_{\lambda,x}(t)= - \partial_q \check{L}_\lambda(t,x(t),\dot{x}(t))+ \int^t_0
\partial_v \check{L}_\lambda(s, x(s),\dot{x}(s))ds.
\end{eqnarray}
\item[\rm (iv)] $\nabla\check{\mathcal{L}}_{\lambda}$ restricts to a $C^1$ map $A_\lambda$
from $\mathcal{U}^X$ to ${\bf X}_{E_{\gamma_\mu}}$.
\item[\rm (v)] $\nabla\check{\mathcal{L}}_{\lambda}$
 has the G\^ateaux derivative  $B_\lambda(\zeta)\in{\mathscr{L}}_s({\bf H}_{E_{\gamma_\mu}})$ at $\zeta\in\mathcal{U}$ given by
 \begin{eqnarray}\label{e:gradient4Lagr+}
 (B_\lambda(\zeta)\xi,\eta)_{1,2}
   = \int_0^{\tau} \Bigl(\!\! \!\!\!&&\!\!\!\!\!\partial_{vv}
     \check{L}_{{\lambda}}\bigl(t, \zeta(t),\dot{\zeta}(t)\bigr)
\bigl[\dot{\xi}(t), \dot{\eta}(t)\bigr]
+ \partial_{qv}
  \check{L}_{{\lambda}}\bigl(t, \zeta(t), \dot{\zeta}(t)\bigr)
\bigl[\xi(t), \dot{\eta}(t)\bigr]\nonumber \\
&& + \partial_{vq}
  \check{L}_{{\lambda}}\bigl(t,\zeta(t),\dot{\zeta}(t)\bigr)
\bigl[\dot{\xi}(t), \eta(t)\bigr] \nonumber\\
&&+  \partial_{qq} \check{L}_{{\lambda}}\bigl(t,\zeta(t),
\dot{\zeta}(t)\bigr) \bigl[\xi(t), \eta(t)\bigr]\Bigr) \, dt
\end{eqnarray}
for any  $\xi,\eta\in {\bf H}_{E_{\gamma_\mu}}$.
$B_\lambda(\zeta)$  is a self-adjoint Fredholm operator and
  has a decomposition
${B}_{\lambda}(\zeta)=\textsl{{P}}_{{\lambda}}(\zeta)+\textsl{{Q}}_{{\lambda}}(\zeta)$, where
$\textsl{P}_{{\lambda}}(\zeta)\in{\mathscr{L}}_s({\bf H}_{E_{\gamma_\mu}})$ is a positive
definitive linear operator defined by
\begin{eqnarray}\label{e:3.17}
(\textsl{P}_{{\lambda}}(\zeta)\xi, \eta)_{1,2}
   = \int_0^\tau \big(\partial_{vv}\check{L}_{{\lambda}}\bigl(t,\zeta(t),\dot\zeta(t)\bigr)
[\dot\xi(t), \dot\eta(t)]+ \bigl(\xi(t), \eta(t)\bigr)_{\R^n}\big)
\, dt,
\end{eqnarray}
and $\textsl{Q}_{{\lambda}}(\zeta)\in\check{\mathscr{L}}_s({\bf H}_{E_{\gamma_\mu}})$ is a compact self-adjoint linear operator.
Moreover, (L2) in Lemma~3.8 in \cite{Lu12-} or \cite{Lu12} 
implies that  $(\textsl{P}_{{\lambda}}(\zeta)\xi, \xi)_{1,2}\ge\min\{\check{c},1\}\|\xi\|_{1,2}^2$
for all $\zeta\in \mathcal{U}$ and $\xi\in{\bf H}_{E_{\gamma_\mu}}$.
 \item[\rm (vi)] If $(\lambda_k)\subset\hat\Lambda$ and $(\zeta_k)\subset\mathcal{U}$ converge to $\mu\in\hat\Lambda$ and $0$, respectively,
then $\|\textsl{P}_{{\lambda_k}}(\zeta_k)\xi-\textsl{P}_{{\mu}}(0)\xi\|_{1,2}\to 0$ for each $\xi\in {\bf H}_{E_{\gamma_\mu}}$.
\item[\rm (vii)] $\mathcal{U}\ni\zeta\mapsto \textsl{Q}_{{\lambda}}(\zeta)\in\mathscr{L}_s({\bf H}_{E_{\gamma_\mu}})$
is uniformly continuous at $0\in\mathcal{U}$ with respect to $\lambda\in\hat\Lambda$ and
$\|\textsl{Q}_{\lambda_k}(0)-\textsl{Q}_{\mu}(0)\|\to 0$ as $(\lambda_k)\subset\Lambda$  converges to $\mu\in\hat\Lambda$.
\end{enumerate}
\end{proposition}

\begin{proposition}\label{prop:Fix-PcontinA}
Both maps $\Lambda\times \mathcal{U}^X\ni (\lambda, x)\mapsto
\check{\mathscr{L}}_{\lambda}(x)\in\R$ and
$\Lambda\times \mathcal{U}^X\ni (\lambda, x)\mapsto A_\lambda(x)\in{\bf X}_{E_{\gamma_\mu}}$ are continuous.
\end{proposition}

\begin{proposition}\label{prop:Fix-PsolutionLagr}
For any given $\bar\epsilon>0$ there exists $\bar\varepsilon>0$ such that
if a critical point $x$ of $\check{\mathscr{L}}_\lambda$ satisfies
$\|x\|_{1,2}<\bar\varepsilon$ then $\|x\|_{C^2}<\bar\epsilon$. (\textsf{Note}: $\varepsilon$ is independent of $\lambda\in\hat\Lambda$.)
Consequently, if $0\in \mathcal{U}^X_{E_{\gamma_\mu}}$ is an isolated critical point of $\check{\mathscr{E}}_{\lambda}|_{\mathscr{U}^X}$
then $0\in \mathcal{U}_{E_{\gamma_\mu}}$ is also an isolated critical point of $\check{\mathscr{L}}_{\lambda}$.
\end{proposition}

 \noindent{\bf Step 2} ({\it Prove  that \cite[Theorem~C.8]{Lu10} (\cite[Theorem~3.7]{Lu11} or \cite[Theorem~5.12]{Lu8}) can be used for $\check{\mathscr{L}}_{\lambda}$}).
As before, we have $m^\star_\tau(\mathfrak{E}_{\lambda}, \gamma_\lambda)=m^\star_\tau(\mathfrak{E}^\ast_{\lambda}, x_\lambda)=m^\star_\tau(\tilde{\mathfrak{E}}^\ast_{\lambda}, 0)
=m^\star_\tau(\check{\mathscr{L}}_{\lambda},0)$ for $\star=-, 0$.
Because of the assumption (c), we obtain
\begin{enumerate}
\item[\rm (c'')]  $m^0_\tau(\check{\mathscr{L}}_{\mu}, 0)\ne 0$,
$m^0_\tau(\check{\mathscr{L}}_{\lambda}, x_\lambda)=0$
 for each $\lambda\in\Lambda\setminus\{\mu\}$ near $\mu$, and
 $m^-_\tau(\check{\mathscr{L}}_{\lambda}, x_\lambda)$  takes, respectively, values
 $m^-_\tau(\check{\mathscr{L}}_{\mu}, 0)$ and
  $m^-_\tau(\check{\mathscr{L}}_{\mu}, 0)+m^0_\tau(\check{\mathscr{L}}_{\mu}, 0)$
 as $\lambda\in\Lambda$ varies in two deleted half neighborhoods  of $\mu$.
 \end{enumerate}
Next, let us  prove that
\begin{equation}\label{e:FixedPS}
\hbox{the fixed point set of the induced $S^1$-action on
$({\bf H}_{E_{\gamma_\mu}})_\mu^0:={\rm Ker}(\check{\mathscr{L}}''_{\mu}(0))$ is $\{0\}$, }
\end{equation}
Note that $\xi\in {\bf H}_{E_{\gamma_\mu}}$ belongs to
$({\bf H}_{E_{\gamma_\mu}})_\mu^0$ if and only if it is $C^2$ and satisfies
\begin{equation}\label{e:PPerLagrLine}
\partial_{xy}\check{L}_\lambda(0, 0)\dot{\xi}+
\partial_{yy}\check{L}_\lambda(0, 0)\ddot{\xi}
-\partial_{xx}\check{L}_\lambda(0, 0)\xi-\partial_{yx}\check{L}_\lambda(0, 0)\dot{\xi}=0.
\end{equation}
Suppose that $\xi$ is also a fixed point for the action in (\ref{e:S^1-action}).
Then it is equal to a constant vector in $\mathbb{R}^n$
and $E_{\gamma_\mu}\xi=\xi$. By (\ref{e:PPerLagrLine}) we obtain
$\partial_{xx}\check{L}_\lambda(0, 0)\xi=0$ and hence $\xi=0$ because
$\partial_{xx}\check{L}_\mu(0, 0)$ is positive definite by (L0). (\ref{e:FixedPS}) is proved.

By \cite[Theorem~C.8]{Lu10} (\cite[Theorem~3.7]{Lu11} or \cite[Theorem~5.12]{Lu8}) and \cite[Remark~3.9]{Lu11} (or \cite[Remark~5.14]{Lu8})
one of the following alternatives occurs:
\begin{enumerate}
\item[(I)] $(\mu, 0)$ is not an isolated solution  in  $\{\mu\}\times \mathcal{U}_{E_{\gamma_\mu}}$
 of $\nabla\check{\mathscr{L}}_{\mu}=0$.

 \item[(II)] There exist left and right  neighborhoods $\Lambda^-$ and $\Lambda^+$ of $\mu$ in $\Lambda$
and integers $n^+, n^-\ge 0$, such that $n^++n^-\ge\frac{1}{2}\dim ({\bf H}_{E_{\gamma_\mu}})_\mu^0$,
and that for $\lambda\in\Lambda^-\setminus\{\mu\}$ (resp. $\lambda\in\Lambda^+\setminus\{\mu\}$)
the functional $\check{\mathscr{L}}_{\lambda}$ has at least $n^-$ (resp. $n^+$) distinct critical
$S^1$-orbits disjoint with $0$, which converge to  $0$ in $\mathcal{U}^X_{E_{\gamma_\mu}}$  as $\lambda\to \mu$.
\end{enumerate}
Moreover, if $\dim ({\bf H}_{E_{\gamma_\mu}})_\mu^0\ge 3$, then
\cite[Theorem~C.9]{Lu10} (or \cite[Theorem~3.10]{Lu11})
shows that  (II) may be replaced by the following alternatives:
\begin{enumerate}
\item[\rm (III)]  For every $\lambda\in\Lambda\setminus\{0\}$ near $0\in\Lambda$ there is a
 $S^1$-orbit $S^1\cdot \bar{w}_\lambda\ne\{0\}$ near $0\in\mathcal{U}^X_{E_{\gamma_\mu}}$
 such that $\nabla\check{\mathscr{L}}_{\lambda}(\bar{w}_\lambda)=0$ and that
 $S^1\cdot\bar{w}_\lambda\to 0$ in $\mathcal{U}^X_{E_{\gamma_\mu}}$  as $\lambda\to \mu$.

\item[\rm (IV)] For any small $S^1$-invariant neighborhood $\mathcal{N}$ of $0$ in $\mathcal{U}^X_{E_{\gamma_\mu}}$
there is a one-sided deleted neighborhood $\Lambda^0$ of $\mu\in\Lambda$ such that
for any $\lambda\in\Lambda^0$, $\nabla\check{\mathscr{L}}_{\lambda}=0$
 has either infinitely many  $S^1$-orbits  of solutions in $\mathcal{N}$,
 $S^1\cdot\bar{w}^j_\lambda$, $j=1,2,\cdots$,
  or  at least two $S^1$-orbits of solutions in $\mathcal{N}$,
 $S^1\cdot\hat{w}^1_\lambda\ne\{0\}$ and $S^1\cdot\hat{w}^2_\lambda\ne\{0\}$, such that $\check{\mathscr{L}}_{\lambda}(\hat{w}^1_\lambda)\ne\check{\mathscr{L}}_{\lambda}(\hat{w}^2_\lambda)$.
Moreover, these orbits converge to $0$ in $\mathcal{U}^X_{E_{\gamma_\mu}}$  as $\lambda\to \mu$.
\end{enumerate}

\noindent{\bf Step 3} ({\it Complete the proof of Theorem~\ref{th:bif-per3LagrMan}}).

\underline{{In the case of (I)}}, we have a sequence  $(w_j)\subset {\bf H}_{E_{\gamma_\mu}}\setminus\{0\}$
such that $\|w_j\|_{1,2}\to 0$ as $j\to\infty$ and that $\nabla\check{\mathscr{L}}_{\mu}(w_j)=0$ for each $j\in\N$.
By Proposition~\ref{prop:Fix-PsolutionLagr} these $w_j$ are $C^2$ and $\|w_j\|_{C^2}\to 0$.
Because $0$ is a fixed point for the $S^1$-action, the $S^1$-orbits are compact and different $S^1$-orbits
are not intersecting, by passing to a subsequence we can assume that any two of $w_j$, $j=0,1,\cdots$,
do not belong the same $S^1$-orbit. Using the chart $\phi$ in (\ref{e:FixChart})
we define $\mathbb{R}\ni t\mapsto\gamma^k(t):=\phi(w_k(t))$, $k=1,\cdots$. They satisfy (i) of Theorem~\ref{th:bif-per3LagrMan}.

\underline{In the case of (II)}, note firstly that $\dim ({\bf H}_{E_{\gamma_\mu}})_\mu^0$ is equal to
$m^0_\tau(\check{\mathscr{L}}_{\mu},0)=m^0_\tau(\mathfrak{E}_{\mu}, \gamma_\mu)$.
For $\lambda\in\Lambda^-\setminus\{\mu\}$ (resp. $\lambda\in\Lambda^+\setminus\{\mu\}$) let
$S^1\cdot w_\lambda^i$,\; $i=1,\cdots,n^-$ (resp. $n^+$)
be  distinct critical $S^1$-orbits of $\check{\mathscr{L}}_{\lambda}$
disjoint with $0$,  which converge to  $0$ in $\mathcal{U}^X_{E_{\gamma_\mu}}$  as $\lambda\to \mu$.
 Proposition~\ref{prop:Fix-PsolutionLagr} implies that  $\|w_\lambda^i\|_{C^2}\to 0$ as $\lambda\to \mu$.
Then $\mathbb{R}\ni t\mapsto\gamma_\lambda^i(t)=\phi(x_\lambda(t)+w_\lambda^i(t))$, $i=1,\cdots,n^-$ (resp. $n^+$) are the required solutions.

\underline{In the case of (III)}, let $\alpha_\lambda(t)=\phi(x_\lambda(t)+\bar{w}_\lambda(t))$ for $t\in\mathbb{R}$.
It satisfies  (\ref{e:Lagr10*})  with parameter value $\lambda$ and Proposition~\ref{prop:Fix-PsolutionLagr}
implies that $\alpha_\lambda-\gamma_\lambda$ converges to zero
on any compact interval $I\subset\R$ in $C^2$-topology as $\lambda\to\mu$.
Since $S^1\cdot \bar{w}_\lambda\ne\{0\}$, for any $t\in\mathbb{R}$ we have $\bar{w}_\lambda(t)\ne 0$ and so $\alpha_\lambda(t)\ne\gamma_\lambda(t)$.
Note that all $\gamma_\lambda$ are constant solutions. Hence $\alpha_\lambda\notin\R\cdot \gamma_\lambda$.

\underline{In the case of (IV)}, for the first case
let $\alpha^j_\lambda(t)=\phi(x_\lambda(t)+\bar{w}^j_\lambda(t))$ for $t\in\mathbb{R}$ and $j=1,2,\cdots$,
and for the second case let $\beta^i_\lambda(t)=\phi(x_\lambda(t)+\hat{w}^i_\lambda(t))$ for $t\in\mathbb{R}$ and $i=1,2$.
They satisfy (\ref{e:Lagr10*})  with parameter value $\lambda$.
For a given small $\epsilon>0$, by Proposition~\ref{prop:Fix-PsolutionLagr}
there is an one-sided  neighborhood $\Lambda^0$ of $\mu$ in $\Lambda$ such that
for any $\lambda\in\Lambda^0\setminus\{\mu\}$,
\begin{description}
\item[$\bullet$] for the first case $\|\bar{\alpha}_\lambda^k|_{[0,\tau]}-\gamma_\lambda|_{[0,\tau]}\|_{C^2([0,\tau];\mathbb{R}^N)}<\epsilon$, $k=1,2,\cdots$,
\item[$\bullet$] for the second case $\|{\beta}_\lambda^i|_{[0,\tau]}-\gamma_\lambda|_{[0,\tau]}\|_{C^2([0,\tau];\mathbb{R}^N)}<\epsilon$, $i=1,2$.
\end{description}
Moreover, each $\gamma_\lambda$ is constant, orbits
$\mathbb{R}\cdot\bar{\alpha}_\lambda^k=S^1\cdot\bar{\alpha}_\lambda^k$ are pairwise distinct, and $\bar{\alpha}_\lambda^k\notin\mathbb{R}\cdot \gamma_\lambda$ as above.
Similarly, $\beta_\lambda^1$ and $\beta_\lambda^2$ are  $\mathbb{R}$-distinct, and
 $\beta_\lambda^1\notin\mathbb{R}\cdot \gamma_\lambda$ and $\beta_\lambda^2\notin\mathbb{R}\cdot \gamma_\lambda$.
Finally, $\mathfrak{E}_\lambda(\beta_\lambda^1)\ne\mathfrak{E}_\lambda(\beta_\lambda^2)$ because  $\check{\mathscr{L}}_{\lambda}(\hat{w}^1_\lambda)\ne\check{\mathscr{L}}_{\lambda}(\hat{w}^2_\lambda)$.
The desired assertions are proved. \hfill$\Box$

\section{Proofs of Theorems~\ref{th:bif-ness-orbitLagrMan}, \ref{th:bif-suffict1-orbitLagrMan}, \ref{th:bif-existence-orbitLagrMan},
\ref{th:bif-suffict-orbitLagrMan}}\label{sec:Lagr4}
\setcounter{equation}{0}

\subsection{Reduction of the problem (\ref{e:Lagr10*}) to one on open subsets of $\mathbb{R}^n$}\label{sec:AutoLagr1}

Since $\mathbb{I}_g(\bar\gamma(t))=\bar\gamma(t+\tau)\;\forall t\in\mathbb{R}$ and
$\mathbb{I}_g$ is an isometry, the closure $Cl({\bar\gamma}(\mathbb{R}))$ of $\bar\gamma(\mathbb{R})$ is compact.
As in Section~3.1.1 in \cite{Lu12-} or \cite{Lu12}
 we may choose  a number $\iota>0$ such that the following holds:
\begin{description}
\item[($\clubsuit6$)] the closure $\bar{\bf U}_{3\iota}(Cl({\bar\gamma}(\mathbb{R})))$ of
${\bf U}_{3\iota}(Cl({\bar\gamma}(\mathbb{R}))):=\{p\in M\,|\, d_g(p, Cl({\bar\gamma}(\mathbb{R})))< 3\iota\}$ is
a compact neighborhood of $\gamma_\mu([0, \tau])$ in $M$, and
$\bar{\bf U}_{3\iota}(Cl({\bar\gamma}(\mathbb{R})))\times\bar{\bf U}_{3\iota}(Cl({\bar\gamma}(\mathbb{R})))$
is contained in the image of $\mathbb{F}|_{\mathcal{W}(0_{TM})}$, where $\mathbb{F}$ is given by the following (\ref{e:exp}).
\item[($\spadesuit6$)] $\{(q,v)\in TM\,|\,q\in \bar{\bf U}_{3\iota}(Cl({\bar\gamma}(\mathbb{R}))),\;|v|_q\le 3\iota\}\subset \mathcal{W}(0_{TM})$.
\end{description}
Then $3\iota$ is less than the injectivity
radius of $g$ at each point on $\bar{\bf U}_{3\iota}(Cl({\bar\gamma}(\mathbb{R})))$.

Let us choose the $C^6$ Riemannian metric $g$ on $M$ so that
$S_0$ (resp. $S_1$) is totally geodesic near $\gamma_\mu(0)$ (resp. $\gamma_\mu(\tau)$).
There exists a fibrewise convex open neighborhood $\mathcal{U}(0_{TM})$ of the zero section of $TM$
such that the exponential map of $g$ gives rise to $C^5$ immersion
\begin{equation}\label{e:exp}
\mathbb{F}:\mathcal{U}(0_{TM})\to M\times M,\;(q,v)\mapsto (q,\exp_q(v)),
\end{equation}
(cf.~Appendix~A in \cite{Lu12-} or \cite{Lu12}).
 By (A.2) in \cite{Lu12-} or \cite{Lu12},
$d\mathbb{F}(q,0_q):T_{(q,0_q)}\mathcal{U}(0_{TM})\to T_{(q,q)}(M\times M)=T_qM\times T_qM$ is an isomorphism for each $q\in M$.
Since $\mathbb{F}$ is injective on the closed subset $0_{TM}\subset TM$, it follows from Exercise~7 in
\cite[page 41]{Hir} that $\mathbb{F}|_{\mathcal{W}(0_{TM})}$ is a $C^5$ embedding for some smaller open neighborhood
$\mathcal{W}(0_{TM})\subset \mathcal{U}(0_{TM})$ of $0_{TM}$.
Note that $\mathbb{F}(0_{TM})$ is equal to the diagonal $\Delta_M$ in
$M\times M$, and that $\gamma_\mu([0, \tau])$ is compact.

Since $\bar\gamma$ is $C^6$, as in \cite[\S3]{Lu5}, starting with a unit orthogonal
 frame at $T_{{\bar\gamma}(0)}M$ and using the parallel
transport along ${\bar\gamma}$ with respect to the Levi-Civita
connection of the Riemannian metric $g$ we get a unit orthogonal
parallel $C^5$ frame field $\mathbb{R}\to {\bar\gamma}^\ast TM,\;t\mapsto
(e_1(t),\cdots, e_n(t))$,
such that
$$
(e_1(t+\tau),\cdots, e_n(t+\tau))=(\mathbb{I}_{g\ast}(e_1(t)),\cdots, \mathbb{I}_{g\ast}(e_n(t)))E_{{\bar\gamma}}\quad\forall t\in\mathbb{R},
$$
where $E_{{\bar\gamma}}$ is an orthogonal matrix of order $n$.

By \cite[Corollary~2.5.11]{HoJo13}
there exists an orthogonal matrix $\Xi$ such that
\begin{equation}\label{e:orth}
\Xi^{-1}E_{\bar\gamma}\Xi={\rm
diag}(S_1,\cdots,S_{\sigma})\in\R^{n\times n},
\end{equation}
where each $S_j$ is either $1$, or $-1$, or
${\scriptscriptstyle\left(\begin{array}{cc}
\cos\theta_j& \sin\theta_j\\
-\sin\theta_j& \cos\theta_j\end{array}\right)}$, $0<\theta_j<\pi$,
and their orders satisfy: ${\rm ord}(S_1)\ge\cdots\ge{\rm ord}(S_{\sigma})$.
Replacing $(e_1,\cdots, e_n)$ by  $(e_1,\cdots, e_n)\Xi$  we may assume
\begin{equation}\label{e:standard}
E_{\bar\gamma}={\rm diag}(S_1,\cdots,S_{\sigma})\in\R^{n\times n}.
\end{equation}

Let $B^n_{r}(0):=\{x\in\R^n\,|\, |x|<r\}$ and $\bar{B}^n_{r}(0):=\{x\in\R^n\,|\, |x|\le r\}$ for $r>0$.
  Then
 \begin{eqnarray}\label{e:Lagr4*}
\phi_{{\bar\gamma}}:\mathbb{R}\times B^n_{3\iota}(0)\to
M,\;(t,x)\mapsto\exp_{{\bar\gamma}(t)}\left(\sum^n_{i=1}x_i
e_i(t)\right)
 \end{eqnarray}
 is a $C^5$ map and  satisfies
\begin{eqnarray*}
&&\phi_{{{\bar\gamma}}}(t+\tau,x)=\mathbb{I}_g\left(\phi_{{\bar\gamma}}(t, E_{{\bar\gamma}}
x)\right)\quad\hbox{and}\\
&&d\phi_{{\bar\gamma}}(t+\tau,x)[(1,v)]=d\mathbb{I}_g\left(\phi_{{\bar\gamma}}(t, E_{{\bar\gamma}}
x)\right)\circ d\phi_{{\bar\gamma}}(t,
E_{{\bar\gamma}} x)[(1, E_{{\bar\gamma}} v)]
\end{eqnarray*}
for any $(t,x, v)\in \mathbb{R}\times B^n_{3\iota}(0)\times\mathbb{R}^n$.
By \cite[Theorem~4.3]{PiTa01}, from $\phi_{{\bar\gamma}}$ we get a $C^2$ coordinate chart around ${\bar\gamma}$ on the $C^4$ Banach manifold
$\mathcal{X}_{\tau}(M, \mathbb{I}_g)$
 modeled on the Banach space
 $$
 \mathcal{X}^1_{\tau}(\mathbb{R}^n, {E}_{{\bar\gamma}})=\{\xi\in C^{1}(\mathbb{R}; \mathbb{R}^n)\,|\,
 {E}^T_{{\bar\gamma}}\xi(t)=\xi(t+\tau)\;\forall t\in\mathbb{R}\}
 $$
 with the induced norm $\|\xi\|_{C^1}$ from $C^1([0,\tau],\mathbb{R}^n)$,
 \begin{eqnarray}\label{e:Lagr5*}
\Phi_{{\bar\gamma}}:\mathcal{X}^1_{\tau}(B^n_{2\iota}(0), {E}_{{\bar\gamma}})=\{\xi\in
\mathcal{X}^1_{\tau}(\mathbb{R}^n, {E}_{{\bar\gamma}}) \,|\,\|\xi\|_{C^0}<2\iota\} \to \mathcal{X}^1_{\tau}(M, \mathbb{I}_g)
\end{eqnarray}
given by $\Phi_{{\bar\gamma}}(\xi)(t)=\phi_{{\bar\gamma}}(t,\xi(t))$.
  Moreover
$$
d\Phi_{{\bar\gamma}}(0): \mathcal{X}^1_{\tau}(\mathbb{R}^n, {E}_{{\bar\gamma}})\to T_{\bar\gamma}\mathcal{X}_{\tau}(M, \mathbb{I}_g),\;
\xi\mapsto\sum^n_{j=1}\xi_je_j
$$
is a Banach space isomorphism. (Actually, we have
\begin{eqnarray}\label{e:Lagr5+}
|\xi(t)|^2_{\mathbb{R}^n}=\sum^n_{j=1}(\xi_j(t))^2=g\left(\sum^n_{j=1}\xi_j(t)e_j(t), \sum^n_{j=1}\xi_j(t)e_j(t)\right)=
|d\Phi_{{\bar\gamma}}(0)[\xi](t)|^2_g
\end{eqnarray}
for any $\xi:\mathbb{R}\to \mathbb{R}^n$ and $t\in\mathbb{R}$.) Therefore there exists a
unique $\zeta_0\in \mathcal{X}^1_{\tau}(\mathbb{R}^n, {E}_{{\bar\gamma}})$ satisfying
$d\Phi_{{\bar\gamma}}(0)[\zeta_0]=\dot{\bar\gamma}$,
  that is, $\dot{\bar\gamma}(t)=\sum^n_{j=1}\zeta_{0j}(t)e_j(t),\;\forall t\in\R$,
where
$$
\zeta_{0j}(t)=g(\dot{\bar\gamma}(t), e_j(t))\;\forall t\in \R,\quad
j=1,\cdots,n.
$$
 Clearly, $\zeta_0\in \mathcal{X}^5_{\tau}(\mathbb{R}^n, {E}_{{\bar\gamma}})$ and $\zeta_0\ne 0$ because $\bar\gamma$ is nonconstant.
If $\gamma\in \mathcal{X}^1_{\tau}(M, \mathbb{I}_g)$ is nonconstant and $C^l$ ($2\le l\le 5$),
by the arguments above \cite[Proposition~4.1]{Lu10} the orbit $\mathcal{O}:=\mathbb{R}\cdot\gamma$
is either a one-to-one  $C^{l-1}$ immersion submanifold of dimension one or a $C^{l-1}$-embedded circle; moreover
$T_{\gamma}{\cal O}=\dot\gamma\R\subset T_{\gamma}\mathcal{X}^1_{\tau}(M, \mathbb{I}_g)$.
Then for any reals $a<b$, $[a,b]\cdot\bar\gamma$ is a $C^4$ embedded submanifold of dimension one.
Take $a>0$ such that $[-a,a]\cdot\bar\gamma\subset{\rm Im}(\Phi_{\bar\gamma})$. Then
\begin{equation}\label{e:S_0}
S_0:=\Phi^{-1}_{{\bar\gamma}}\bigl([-a,a]\cdot\bar\gamma\cap{\rm Im}(\Phi_{\bar\gamma})\bigr)
\end{equation}
is a one-dimensional compact $C^2$ submanifold of $\mathcal{X}^1_{\tau}(B^n_{2\iota}(0), {E}_{{\bar\gamma}})$
containing $0$ as an interior point, and
$$
T_0S_0=(d\Phi_{\bar\gamma}(0))^{-1}(T_{\bar\gamma}S_0)=(d\Phi_{\bar\gamma}(0))^{-1}(\mathbb{R}\dot{\bar\gamma})=\R\zeta_0.
$$

 Define the function $L^\star: \Lambda\times \mathbb{R}\times B^n_{3\iota}(0)\times\R^{n}\to \R$  by
\begin{eqnarray*}
  L^\star(\lambda, t, x, v)=L\bigr(\lambda,  \phi_{{\bar\gamma}}(t,x),
d\phi_{{\bar\gamma}}(t,x)[(1,v)]\bigl).
\end{eqnarray*}
It is continuous and satisfies
\begin{eqnarray}\label{e:E-invariant1Lagr}
&&L^\star(\lambda, t+\tau, x, v)=L\bigr(\lambda, \phi_{{\bar\gamma}}(t+\tau,x),
d\phi_{{\bar\gamma}}(t+\tau,x)[(1,v)]\bigl)\nonumber\\
&=&L\left(\lambda, \mathbb{I}_g\left(\phi_{{\bar\gamma}}(t, E_{{\bar\gamma}}
x)\right), d\mathbb{I}_g\left(\phi_{{\bar\gamma}}(t, E_{{\bar\gamma}}
x)\right)\circ d\phi_{{\bar\gamma}}(t,
E_{{\bar\gamma}} x)[(1, E_{{\bar\gamma}} v)]\right)\nonumber \\
&=&L\left(\lambda, \phi_{{\bar\gamma}}(t, E_{{\bar\gamma}}
x),  d\phi_{{\bar\gamma}}(t, E_{{\bar\gamma}} x)[(1, E_{{\bar\gamma}} v)]\right) \nonumber\\
&=&L^\star(\lambda, t, E_{{\bar\gamma}}x, E_{{\bar\gamma}} v)
\end{eqnarray}
for all $(\lambda,t,x,v)\in\Lambda\times \mathbb{R}\times B^n_{3\iota}(0)\times\R^{n}$, and thus
\begin{eqnarray*}
\partial_x L^\star(\lambda, t+\tau, x, v)&=&E_{{\bar\gamma}}\partial_x L^\star\bigl(\lambda, t,
E_{{\bar\gamma}}x, E_{{\bar\gamma}}v),\\
\partial_v L^\star(\lambda, t+\tau, x, v)&=&E_{{\bar\gamma}}\partial_v L^\star\bigl(\lambda, t,
E_{{\bar\gamma}}x, E_{{\bar\gamma}}v)
\end{eqnarray*}
for all $(\lambda,t,x,v)\in\Lambda\times \mathbb{R}\times B^n_{3\iota}(0)\times\R^{n}$.
(Here $\partial_x L^\star$ and $\partial_v L^\star$ denote the gradients of $L^\star$ with respect to $x$ and $v$,
respectively. Recall that all vectors in $\mathbb{R}^n$ in this paper are understood as column vectors.)
 Each $L^\star(\lambda, \cdot)$   is $C^4$ and all its partial derivatives of order no more than two depend continuously on
 $(\lambda,t,x,v)\in\Lambda\times \mathbb{R}\times B^n_{3\iota}(0)\times\R^{n}$.
 Moreover, $d\phi_{{\bar\gamma}}(t,x)[(1,v)]=\partial_x\phi_{{\bar\gamma}}(t,x)[v]
+\partial_t\phi_{{\bar\gamma}}(t,x)$ implies that
$$
\mathbb{R}^n\ni v\mapsto L^\star(\lambda, t, x, v)=L\bigr(\lambda, \phi_{{\bar\gamma}}(t,x),
d\phi_{{\bar\gamma}}(t,x)[(1,v)]\bigl)\in\mathbb{R}
$$
is strictly convex.

The $C^2$ functional  $\mathfrak{E}_\lambda$ in (\ref{e:Lagr9}) gives rise to a $C^2$
functional $\mathfrak{E}^\star_\lambda:\mathcal{X}^1_\tau(B^n_{2\iota}(0), {E}_{\bar\gamma})
\to\mathbb{R}$,
\begin{equation}\label{e:two-functionals}
\mathfrak{E}^\star_\lambda(\xi)=\int^\tau_0L^\star(\lambda, t, \xi(t), \dot{\xi}(t))dt=\mathfrak{E}_\lambda\circ\Phi_{{\bar\gamma}}(\xi).
\end{equation}
 Since  $\mathfrak{E}_\lambda$ is invariant under the continuous $\mathbb{R}$-action on
$\mathcal{X}^1_{\tau}(M, \mathbb{I}_g)$ given by
\begin{equation}\label{e:R-action}
\chi:\mathcal{X}^1_{\tau}(M, \mathbb{I}_g)\times \mathbb{R}\to \mathcal{X}^1_{\tau}(M, \mathbb{I}_g),\;(s,\gamma)\mapsto s\cdot\gamma,
\end{equation}
where $(s\cdot\gamma)(t)=\gamma(s+t)\;\forall s, t\in\mathbb{R}$,
 $S_0$ is a critical submanifold of each $\mathfrak{E}^\star_\lambda$,
  and there holds
   \begin{equation*}
D^2\mathfrak{E}^\star_\lambda(0)[\xi,\eta]=
D^2\mathfrak{E}_\lambda(\bar\gamma)\big[d\Phi_{{\bar\gamma}}(0)[\xi],
d\Phi_{{\bar\gamma}}(0)[\eta]\big]\quad\forall \xi,\eta\in
\mathcal{X}^1_\tau(B^n_{2\iota}(0), {E}_{\bar\gamma}),
\end{equation*}
which implies
\begin{equation}\label{e:Lagr7Morse*}
m^-_\tau(\mathfrak{E}^\star_\lambda, 0)=m^-_\tau(\mathfrak{E}_\lambda, \bar\gamma)\quad
\hbox{and}\quad
m^0_\tau(\mathfrak{L}^\star_\lambda, 0)=m^0_\tau(\mathfrak{E}_\lambda, \bar\gamma).
\end{equation}

 Since $S_0$ is compact, there exists $\rho_0>3\iota$ such that
 $\sup_t|\dot{x}(t)|<\rho_0$ for all $x\in S_0$.
For the Lagrangian $L^\star$, as in Lemma~3.8 in \cite{Lu12-} or \cite{Lu12}
we can modify it to obtain:

\begin{lemma}\label{lem:modif}
For a given subset $\hat\Lambda\subset\Lambda$ which is either compact or sequential compact,
there exists a continuous function $\check{L}:\hat\Lambda\times \mathbb{R}\times \bar{B}^n_{2\iota}(0)\times\mathbb{R}^n\to\R$
  satisfying the following conditions
 for some constants $\kappa>0$ and $0<c<C$:
\begin{enumerate}
\item[\rm (L0)] $\check{L}(\lambda, t+\tau, x, v)=\check{L}(\lambda, t, E_{{\bar\gamma}}x, E_{{\bar\gamma}} v)$
for all $(\lambda,t,x,v)$,   each function
$\check{L}_\lambda(\cdot)=\check{L}(\lambda,\cdot)$ ($\lambda\in\hat\Lambda$) is $C^4$ and  partial derivatives
\begin{eqnarray*}
&&\partial_t\check{L}_\lambda(\cdot),\quad\partial_q\check{L}_\lambda(\cdot),\quad\partial_v\check{L}_\lambda(\cdot),
\quad\partial_{qv}\check{L}_\lambda(\cdot),\quad
\partial_{qq}\check{L}_\lambda(\cdot),\quad\partial_{vv}\check{L}_\lambda(\cdot)\\
&&\hbox{and}\quad \partial_{tt}\check{L}_\lambda(\cdot),\quad \partial_{tq}\check{L}_\lambda(\cdot),\quad
\partial_{tv}\check{L}_\lambda(\cdot)
\end{eqnarray*}
depend continuously on  $(\lambda, t, q, v)\in \hat\Lambda\times \mathbb{R}\times \bar{B}^n_{2\iota}(0)\times\mathbb{R}^n$.

\item[\rm (L1)] $\check{L}$ and ${L}^\star$ are same on
$\hat\Lambda\times\mathbb{R}\times B^n_{3\iota/2}(0)\times B^n_{\rho_0}(0)$;
\item[\rm (L2)] $\partial_{vv}\check{L}_\lambda(t, q,v)\ge cI_n,\quad \forall (\lambda, t, q,v)\in \hat\Lambda\times [0,\tau]\times B^n_{3\iota/2}(0)\times\mathbb{R}^n$.

\item[\rm (L3)] $\Bigl| \frac{\partial^2}{\partial q_i\partial
q_j}\check{L}_\lambda(t, q,v)\Bigr|\le C(1+ |v|^2),\quad \Bigl|
\frac{\partial^2}{\partial q_i\partial v_j}\check{L}_\lambda(t, q,v)\Bigr|\le C(1+
|v|),\quad\hbox{and}\\
 \Bigl| \frac{\partial^2}{\partial v_i\partial v_j}\check{L}_\lambda(t, q,v)\Bigr|\le C,
 \quad\forall (\lambda, t, q,v)\in\hat\Lambda\times [0,\tau]\times B^n_{3\iota/2}(0)\times\mathbb{R}^n$.

\item[\rm (L4)]  $\check{L}(\lambda,t, q, v)\ge \kappa|v|^2-C,\quad\forall (\lambda, t, q, v)\in\hat\Lambda\times \mathbb{R}\times B^n_{3\iota/2}(0)\times\R^{n}$.
\item[\rm (L5)] $|\partial_{q}\check{L}(\lambda,t,q,v)|\le C(1+ |v|^2)$ and $|\partial_{q}\check{L}(\lambda,t,q,v)|\le C(1+ |v|)$
for all $(\lambda, t, q, v)\in \hat\Lambda\times \mathbb{R}\times B^n_{3\iota/2}(0)\times\R^{n}$.
\item[\rm (L6)] $|\check{L}_\lambda(t, q,v)|\le C(1+|v|^2),\quad \forall (\lambda, t, q,v)\in \hat\Lambda\times \mathbb{R}\times \bar{B}^n_{3\iota/2}(0)\times\mathbb{R}^n$.
\end{enumerate}
\end{lemma}

 Consider the Banach space
 $$
 {\bf X}:=\mathcal{X}^1_\tau(\mathbb{R}^n, {E}_{{\bar\gamma}})
 $$
 with the induced norm $\|\cdot\|_{C^1}$ from $C^1([0,\tau],\mathbb{R}^n)$, and
the Hilbert space
$$
{\bf H}:=\{\xi\in W^{1,2}_{\rm loc}(\mathbb{R}; \mathbb{R}^n)\,|\, {E}^T_{\bar\gamma}(\xi(t))=\xi(t+\tau)\;\forall t\in\mathbb{R}\}
$$
equipped with $W^{1,2}$-inner product as in (\ref{e:innerP2}).
Both carry a natural $\mathbb{R}$-action given by $(\theta\cdot x)(t)=x(t+\theta)$
for $\theta\in\mathbb{R}$. The spaces ${\bf H}$ and ${\bf X}$
have the following $\mathbb{R}$-invariant open subsets
\begin{eqnarray*}
\mathcal{U}:&=&\{\xi\in W^{1,2}_{\rm loc}(\mathbb{R}; B^n_{3\iota/2}(0))\,|\, {E}^T_{\bar\gamma}(\xi(t))=\xi(t+\tau)\;\forall t\in\mathbb{R}\},\\
\mathcal{U}^X:&=&\mathcal{U}\cap{\bf X}=\mathcal{X}^1_\tau(B^n_{3\iota/2}(0), {E}_{{\bar\gamma}})
\end{eqnarray*}
respectively. Define a family of functionals $\check{\mathscr{L}}_\lambda:\mathcal{U}\to\R$ by
\begin{equation}\label{e:check*}
\check{\mathscr{L}}_{\lambda}(x)=\int^{\tau}_0\check{L}_\lambda(t, x(t),\dot{x}(t))dt,\quad\lambda\in\hat\Lambda,
\end{equation}
and  put ${\bf H}^\bot:=\{x\in {\bf H}\,|\,(\zeta_0, x)_{1,2}=0\}$ and
\begin{eqnarray}\label{e:orth-spaceLagr}
&&
{\bf X}^\bot:=\{x\in {\bf X}\,|\,(\zeta_0, x)_{1,2}=0\}={\bf X}\cap {\bf H}^\bot,\\
&&{\check{\mathscr{L}}}_{\lambda}^\bot:\mathcal{U}\cap {\bf H}^\bot\to\R,\;x\mapsto\check{\mathscr{L}}_\lambda(x).\label{e:orth-functLagr}
\end{eqnarray}

\begin{remark}\label{rm:regularity}
{\rm Because of Lemma~\ref{lem:modif}, by \cite[Theorem~4.5]{BuGiHi} we deduce that
every critical point of $\check{\mathscr{L}}_{\lambda}$ is $C^4$.}
\end{remark}

\subsection{Properties of functionals $\check{\mathscr{L}}$ and $\check{\mathscr{L}}^\bot_{\lambda}$}\label{sec:AutoLagr2}

In order to use the abstract theorems developed in \cite{Lu8,Lu11,Lu10} we need to study some
properties of the functionals $\check{\mathscr{L}}$ and $\check{\mathscr{L}}^\bot_{\lambda}$ near $0$.
By \cite[\S3]{Lu5} and \cite{Lu1}, we have (i)-(iii) of the following corresponding result of
Proposition~3.9 in \cite{Lu12-} or \cite{Lu12}.

\begin{proposition}\label{prop:reduction}
\begin{enumerate}
\item[\rm (i)] $\check{\mathscr{L}}_{\lambda}$ is $C^{2-0}$ and the gradient map $\nabla\check{\mathscr{L}}_{\lambda}:\mathcal{U}\to {\bf H}$
has the G\^ateaux derivative ${B}_\lambda(x)\in{\mathscr{L}}_s({\bf H})$ at $x\in \mathcal{U}$.

\item[\rm (ii)] $\nabla\check{\mathscr{L}}_{\lambda}$ restricts to a $C^1$ map $A_\lambda:\mathcal{U}^X\to{\bf X}$.\\
\item[\rm (iii)] (D1) of \cite[Hypothesis~1.1]{Lu8} and (C) of \cite[Hypothesis~1.3]{Lu8} hold near the origin $0\in {\bf H}$, i.e.,
\begin{equation}\label{e:L-C-D-D1*}
\{u\in {\bf H}\,|\,{B}_\lambda(0)u=su,\, s\le 0\}\subset {\bf X}\quad\hbox{and}\quad
\{u\in {\bf H}\,|\,{B}_\lambda(0)u\in X\}\subset {\bf X}.
\end{equation}
\item[\rm (iv)]
Since $\check{L}=\tilde{L}=L^\star$ on $\Lambda\times \mathbb{R}\times {B}^n_{3\iota/2}(0)\times B^n_{\rho_0}(0)$, it holds that
\begin{equation}\label{e:two-functionals+}
\check{\mathscr{L}}_\lambda(x)=\int^{\tau}_0\check{L}_\lambda(t, x(t),\dot{x}(t))dt=\mathscr{L}^\star_\lambda(x)
\end{equation}
for each $x$ in an open subset $\{x\in \mathcal{X}^1_{\tau}(B^n_{3\iota/2}(0), {E}_{{\bar\gamma}})\,|\,\sup_t|\dot{x}(t)|<\rho_0\}$
of $\mathcal{U}^X$. Clearly, $0\in S_0$ has an open neighborhood
 $S_{00}$ in $S_0$  contained in the open subset.
It follows that $d\check{L}_\lambda(x)=0\;\forall x\in S_{00}$ and that
\begin{equation}\label{e:Lagr7Morse**}
m^-_\tau(\check{\mathscr{L}}_\lambda, 0)=m^-_\tau(\mathfrak{E}^\star_\lambda, 0)\quad
\hbox{and}\quad
m^0_\tau(\check{\mathscr{L}}_\lambda, 0)=m^0_\tau(\mathfrak{E}^\star_\lambda, 0).
\end{equation}
\end{enumerate}
\end{proposition}

By (iv) and (\ref{e:two-functionals}),
a point $x\in \mathcal{X}^1_{\tau}(B^n_{3\iota/2}(0), {E}_{{\bar\gamma}})$ near $0$
is a critical point of $\mathfrak{E}^\star_\lambda$ if and only if
$\gamma=\Phi_{{\bar\gamma}}(x)\in \mathcal{X}_{\tau}(M, \mathbb{I}_g)$
is a critical point of $\mathfrak{E}_\lambda$ near $\bar\gamma$. However,
if $\gamma\in{\rm Im}(\Phi_{\bar\gamma})$ is a critical point of $\mathfrak{E}_\lambda$,
so is each point in $\mathbb{R}\cdot\gamma$. Therefore for $|s|$ small enough $(\Phi_{\bar\gamma})^{-1}(s\cdot\gamma)$
is also a critical point of $\mathfrak{E}^\star_\lambda$.
Such a critical point is said to be the \textsf{{$\mathbb{R}$-same as $\gamma$}}.
We need to study behavior of $\mathbb{R}$-distinct critical
 points of $\mathfrak{E}^\star_\lambda$ near $0$.
Clearly, $d\check{\mathscr{L}}_{\lambda}(0)=0$ and
 $d\check{\mathscr{L}}_{\lambda}^\bot(0)=0\;\forall\lambda$.

 Let $\Pi:{\bf H}\to {\bf H}^\bot$ be the orthogonal projection.
Then $\Pi(x)=x-\frac{(x,\zeta_0)_{1,2}}{\|\zeta_0\|_{1,2}}\zeta_0$ for $x\in {\bf H}$, and
\begin{equation}\label{e:gradient2Lagr}
\nabla{\check{\mathscr{L}}}_{\lambda}^\bot(x)=
\nabla\check{\mathscr{L}}_\lambda(x)-
\frac{(\nabla\check{\mathscr{L}}_\lambda(x),\zeta_0)_{1,2}}{\|\zeta_0\|_{1,2}}\zeta_0
\quad \forall x\in \mathcal{U}\cap {\bf H}^\bot,
\end{equation}
where by \cite[(3.10)-(3.11)]{Lu5} (all vectors
in $\mathbb{R}^n$ are understood as column vectors now)
\begin{eqnarray}\label{e:3.10}
\nabla\check{\mathscr{L}}_\lambda(\xi)(t)
 &=&\frac{1}{2}\int^\infty_{t}
e^{t-s}\left[\partial_q\check{L}_{\lambda}\bigl(s, \xi(s),\dot{\xi}(s)\bigr)-\mathfrak{R}^{\xi}_\lambda(s)\right]\,ds
\nonumber\\
&+&\frac{1}{2}\int_{-\infty}^{t} e^{s-t}\left[\partial_q
\check{L}_{{\lambda}}\bigl(s, \xi(s),\dot{\xi}(s)\bigr)-\mathfrak{R}^{\xi}(s)\right]\,ds+\mathfrak{R}^{\xi}_\lambda(t)
\end{eqnarray}
where  $\mathfrak{R}^{\xi}_\lambda$, (provided $2={\rm ord}(S_p)>{\rm ord}(S_{p+1})$ for some
$p\in\{0,\cdots,\sigma\}$ in (\ref{e:standard})), is given by
\begin{eqnarray}\label{e:grad1}
&\mathfrak{R}_\lambda^\xi(t)= \int^t_0\partial_v
\check{L}_{\lambda}\bigl(s, \xi(s),\dot{\xi}(s)\bigr)ds+
 \nonumber\\
&\left[\left(\oplus_{l\le p}\frac{\sin\theta_l}{2-2\cos\theta_l}{\scriptscriptstyle\left(\begin{array}{cc}
0& -1\\
1& 0\end{array}\right)}-\frac{1}{2}I_{2p}\right)\oplus{\rm
diag}(a_{p+1}(t), \cdots, a_{\sigma}(t))\right]\int^\tau_0\partial_v \check{L}_{\lambda}\bigl(s, \xi(s),\dot{\xi}(s)\bigr)ds\nonumber\\
\end{eqnarray}
with $a_j(t)=\frac{2t{S}_j+ 2t+1-{S}_j}{4}$,
$j=p+1,\cdots,{\sigma}=n-p$. (As usual $p=0$ (resp. $p={\sigma}$) means that
${\rm ord}(S_1)=\cdots={\rm ord}(S_\sigma)=1$ (resp. ${\rm ord}(S_1)=\cdots={\rm ord}(S_\sigma)=2$)
and hence there is no the first (resp. second) term in the square brackets in (\ref{e:grad1}).) Note that
\begin{eqnarray}\label{e:grad1+}
\frac{d}{dt}\mathfrak{R}_\lambda^\xi(t)=\partial_v
\check{L}_{\lambda}\bigl(t,\xi(t),\dot{\xi}(t)\bigr)+
\mathfrak{M}\int^\tau_0\partial_v \check{L}_{\lambda}\bigl(s,\xi(s),\dot{\xi}(s)\bigr)ds,
\end{eqnarray}
where $\mathfrak{M}\in\mathbb{R}^{n\times n}$ is a matrix only depending on $E$.

Since $\zeta_0$ is $C^1$,  we derive from (\ref{e:gradient2Lagr}) that $\nabla{\check{\mathscr{L}}}_{\lambda}^\bot(x)\in {\bf X}^\bot$
for any $x\in {\bf X}^\bot\cap\mathcal{U}$, and
\begin{equation}\label{e:gradient3Lagr}
\mathbb{A}_\lambda:\mathcal{U}^X\cap {\bf X}^\bot\to {\bf X}^\bot,\;x\mapsto\nabla{\check{\mathscr{L}}}_{\lambda}^\bot(x)
\end{equation}
is $C^1$. By \cite[\S3]{Lu5} $\nabla{\check{\mathscr{L}}}_{\lambda}^\bot$
 has the G\^ateaux derivative $\mathbb{B}_\lambda(x)\in{\mathscr{L}}_s({\bf H}^\bot)$ at $x\in \mathcal{U}\cap{\bf H}^\bot$
 given by
\begin{equation}\label{e:gradient4Lagr}
\mathbb{B}_\lambda(x)v=B_\lambda(x)v-\frac{(B_\lambda(x)v,\zeta_0)_2}{\|\zeta_0\|_2}\zeta_0
\quad \forall v\in {\bf H}^\bot,
\end{equation}
where $B_\lambda(\zeta)\in{\mathscr{L}}_s({\bf H})$ for $\zeta\in\mathcal{U}$ is a self-adjoint Fredholm
operator defined by (\ref{e:gradient4Lagr+}).
Let ${B}_{\lambda}=\textsl{{P}}_{{\lambda}}+\textsl{{Q}}_{{\lambda}}$
be the decomposition as in Proposition~\ref{prop:Fix-Pfunct-analy}(v).
 Then
\begin{equation}\label{e:gradient5Lagr}
\mathbb{P}_\lambda(x):=\Pi\circ \textsl{P}_\lambda(x)|_{{\bf H}^\bot}\quad\hbox{and}\quad
\mathbb{Q}_\lambda(x):=\Pi\circ \textsl{Q}_\lambda(x)|_{{\bf H}^\bot}
\end{equation}
are positive definite and compact, respectively, and $\mathbb{B}_\lambda(x)=\mathbb{P}_\lambda(x)+\mathbb{Q}_\lambda(x)$
by (\ref{e:gradient4Lagr}).
Since $\R\zeta_0\subset{\rm Ker}(B_\lambda(0))$ we get
 \begin{eqnarray}\label{e:MorseIndexLagr}
 m^-_\tau({\check{\mathscr{L}}}_{\lambda}^\bot, 0)=  m^-_\tau({\check{\mathscr{L}}}_{\lambda}, 0)\quad\hbox{and}\quad
  m^0_\tau({\check{\mathscr{L}}}_{\lambda}^\bot, 0)= m^0_\tau({\check{\mathscr{L}}}_{\lambda}, 0)-1.
 \end{eqnarray}
Clearly, (L1) and (\ref{e:3.17}) yield
\begin{eqnarray}\label{e:uniformPost0}
(\textsl{P}_{{\lambda}}(\zeta)\xi, \xi)_{1,2}\ge\min\{c,1\}\|\xi\|_{1,2}^2,\quad\forall x\in\mathcal{U},\;\forall \xi\in{\bf H},
\end{eqnarray}
 and hence
 \begin{eqnarray}\label{e:uniformPost}
 (\mathbb{P}_{{\lambda}}(x)\xi, \xi)_{1,2}\ge\min\{c,1\}\|\xi\|_{1,2}^2,\quad\forall\xi\in{\bf H}^\bot,\quad\forall x\in
 \mathcal{U}\cap{\bf H}^\bot.
 \end{eqnarray}

\begin{proposition}\label{prop:PPP}
Let $(\lambda_k)\subset\Lambda$ and $(\zeta_k)\subset\mathcal{U}$
converge to $\mu\in\Lambda$ and $0\in\mathcal{U}$, respectively.
 Then $\|\textsl{P}_{{\lambda_k}}(\zeta_k)\xi-\textsl{P}_{\mu}(0)\xi\|_{1,2}\to 0$ for each $\xi\in H$.
In particular, if $(\zeta_k)\subset \mathcal{U}\cap{\bf H}^\bot$ converges to zero, then
$\|\mathbb{P}_{\lambda_k}(\zeta_k)\xi-\mathbb{P}_\mu(0)\xi\|_{1,2}\to 0$ for any $\xi\in{\bf H}^\bot$.
\end{proposition}

\begin{proof}[\bf Proof]
By (\ref{e:3.17}) we have
\begin{eqnarray*}
\|[\textsl{P}_{{\lambda_k}}(\zeta_k)-\textsl{P}_{{\mu}}(0)]\xi\|^2_{1,2}
  \le \int_0^\tau \big|[\partial_{vv}\check{L}_{{\lambda_k}}\bigl(t,\zeta_k(t),
  \dot\zeta_k(t)\bigr)-\partial_{vv}\check{L}_{\mu}\bigl(t, 0,0\bigr)]
\dot\xi(t)\big|^2_{\mathbb{R}^n}\, dt.
\end{eqnarray*}
Note that $\|\zeta_k\|_{1,2}\to 0$ implies $\|\zeta_k\|_{C^0}\to 0$.
Since $(\lambda, t, x,v)\mapsto\partial_{vv}\check{L}_{{\lambda}}(t, x,v)$ is continuous, by
the third inequality in (L2) in Lemma~\ref{lem:modif} we may apply 
\cite[Prop.~C.1]{Lu9} to
$$
f(t,\eta;\lambda)=\partial_{vv}\check{L}(\lambda, t, \zeta_k(t),\dot\zeta_k(t))\eta
$$
 to get that
$$
\int_0^\tau \big|[\partial_{vv}\check{L}_{{\lambda_k}}\bigl(t,\zeta_k(t),\dot\zeta_k(t)\bigr)-
\partial_{vv}\check{L}_{{\mu}}\bigl(t, 0, 0\bigr)]
\dot\xi(t)\big|^2_{\mathbb{R}^n}\, dt\to 0.
$$
Moreover, the Lebesgue dominated convergence theorem also leads to
$$
\int_0^\tau \big|[\partial_{vv}\check{L}_{{\lambda_k}}\bigl(t, 0, 0\bigr)-
\partial_{vv}\check{L}_{{\mu}}\bigl(t, 0, 0\bigr)]
\dot\xi(t)\big|^2_{\mathbb{R}^n}\, dt\to 0.
$$
Hence $\|[\textsl{P}_{{\lambda_k}}(\zeta_k)-\textsl{P}_{\mu}(0)]\xi\|_{1,2}\to 0$.
The final claim follows from this and (\ref{e:gradient5Lagr}).
\end{proof}

\begin{proposition}\label{prop:UniformQ}
$\mathcal{U}\ni\zeta\mapsto \textsl{Q}_{{\lambda}}(\zeta)\in\mathscr{L}_s({\bf H})$
is uniformly continuous at $0$ with respect to $\lambda\in\Lambda$.
Moreover, if $(\lambda_k)\subset\Lambda$  converges to $0\in\Lambda$ then
$\|\textsl{Q}_{\lambda_k}(0)-\textsl{Q}_{\mu}(0)\|\to 0$.
\end{proposition}

\begin{proof}[\bf Proof]
Write $\textsl{Q}_\lambda(\zeta):=\textsl{Q}_{\lambda,1}(\zeta)+ \textsl{Q}_{\lambda,2}(\zeta)+
\textsl{Q}_{\lambda,3}(\zeta)$, where
\begin{eqnarray*}
&&(\textsl{Q}_{\lambda,1}(\zeta)\xi,  \eta)_{1,2}
   = \int_0^{\tau}\partial_{vq}   \check{L}_\lambda\big(t, \zeta(t),\dot{\zeta}(t)\big)
[\dot{\xi}(t), \eta(t)]dt,\\
 &&(\textsl{Q}_{\lambda,2}(\zeta)\xi,
  \eta)_{1,2}= \int_0^{\tau}\partial_{qv}\check{L}_\lambda\big(t, \zeta(t), \dot{\zeta}(t)\big)
[\xi(t), \dot{\eta}(t)]dt, \\
&&(\textsl{Q}_{\lambda,3}(\zeta)\xi,
  \eta)_{1,2} = \int_0^{\tau} \bigl(\partial_{qq}\check{L}_\lambda\big(t, \zeta(t), \dot{\zeta}(t)\big) [\xi(t), \eta(t)]-
  \bigl(\xi(t), \eta(t)\bigr)_{\mathbb{R}^n}\bigr) \, dt.
\end{eqnarray*}
As above the first claim follows from (L2) in Lemma~\ref{lem:modif} and 
\cite[Prop.~C.1]{Lu9} directly.

In order to prove the second claim, as in the proof of \cite[page 571]{Lu1} we have
\begin{eqnarray*}
&&\|\textsl{Q}_{\lambda_k,1}(0)-\textsl{Q}_{\mu,1}(0)\|_{\mathscr{L}({\bf H})}\\
&\le&2(e^{\tau}+ 1)\left(\int^{\tau}_{0}\big|\partial_{vq}
  \check{L}_{\lambda_k}(s, 0, 0)-\partial_{vq}
  \check{L}_\mu(s, 0, 0)\big|^2  ds\right)^{1/2}.
\end{eqnarray*}
Because of the second inequality in ({\bf L2}), it follows from the Lebesgue dominated convergence theorem
that $\|\textsl{Q}_{\lambda_k,1}(0)-\textsl{Q}_{\mu,1}(0)\|_{\mathscr{L}({\bf H})}\to 0$.
Observe that
$(\textsl{Q}_{\lambda,2}(\zeta)\xi, \eta)_{1,2}
   =\bigr(\xi,  (\textsl{Q}_{\lambda,1}(\zeta))^\ast\eta\bigl)_{1,2}$.
Hence $\|\textsl{Q}_{\lambda_k,2}(0)-\textsl{Q}_{\mu,2}(0)\|_{\mathscr{L}({\bf H})}\to 0$.
Finally, it is easy to deduce that
\begin{eqnarray*}
\|\textsl{Q}_{\lambda_k,3}(0)-\textsl{Q}_{\mu,3}(0)\|^2_{\mathscr{L}({\bf H})}\le
  \int_0^{\tau} \bigl|\partial_{qq}\check{L}_{\lambda_k}\big(t, 0, 0\big)
-\partial_{qq}\check{L}_\mu\big(t, 0, 0\big)\bigr|^2\, dt.
\end{eqnarray*}
By the Lebesgue dominated convergence theorem the right side converges to zero.
Then $\|\textsl{Q}_{\lambda_k,3}(0)-\textsl{Q}_{\mu,3}(0)\|_{\mathscr{L}({\bf H})}\to 0$ and therefore
 $\|\textsl{Q}_{\lambda_k}(0)-\textsl{Q}_{\mu}(0)\|\to 0$.
\end{proof}

By (\ref{e:gradient5Lagr}) and Proposition~\ref{prop:UniformQ}, the map
$\mathcal{U}\cap{\bf H}^\bot\ni\zeta\mapsto \mathbb{Q}_{{\lambda}}(\zeta)\in\mathscr{L}_s({\bf H}^\bot)$
is uniformly continuous at $0$ with respect to $\lambda\in\Lambda$.
Moreover, if $(\lambda_k)\subset\Lambda$  converges to $0\in\Lambda$ then
$\|\mathbb{Q}_{\lambda_k}(0)-\mathbb{Q}_{\mu}(0)\|\to 0$.

\begin{proposition}\label{prop:reg}
 $({\bf H}^\bot, {\bf X}^\bot, {\check{\mathscr{L}}}_{\lambda}^\bot, \mathbb{A}_\lambda=\nabla{\check{\mathscr{L}}}_{\lambda}^\bot,\mathbb{B}_\lambda)$
satisfies (C) of \cite[Hypothesis~C.3]{Lu10} (\cite[Hypothesis~1.3]{Lu8}) and (D1) of \cite[Hypothesis~C.3]{Lu10} (\cite[Hypothesis~1.1]{Lu8})
 at the origin $0\in {\bf H}^\bot$, namely
\begin{enumerate}
\item[\rm (C)] $\{u\in{\bf H}^\bot\,|\, \mathbb{B}_\lambda(0)u\in{\bf X}^\bot\}\subset{\bf X}^\bot$,
\item[\rm (D1)] $\{u\in{\bf H}^\bot\,|\, \mathbb{B}_\lambda(0)u=s u,\;s\le 0\}\subset{\bf X}^\bot$.
\end{enumerate}
\end{proposition}
\begin{proof}[\bf Proof]
In order to prove (C) let $u\in{\bf H}^\bot$ be such that $v:=\mathbb{B}_\lambda(0)u\in{\bf X}^\bot$. By (\ref{e:gradient4Lagr})
\begin{equation*}
B_\lambda(0)u-\frac{(B_\lambda(0)u,\zeta_0)_2}{\|\zeta_0\|_2}\zeta_0=\mathbb{B}_\lambda(0)u=v\in{\bf X}^\bot.
\end{equation*}
Since $\zeta_0\in{\bf X}$, it follows from this and (\ref{e:L-C-D-D1*})
  that $u\in{\bf X}$ and hence $u\in{\bf X}^\bot$.

Next, let $u\in{\bf H}^\bot$ satisfy $\mathbb{B}_\lambda(0)u=s u$ for some $s\le 0$. Then (\ref{e:gradient4Lagr}) leads to
\begin{equation*}
B_\lambda(0)u=s u+\frac{(B_\lambda(0)u,\zeta_0)_2}{\|\zeta_0\|_2}\zeta_0.
\end{equation*}
Since $T_0S_0=\mathbb{R}\zeta_0$ and $\nabla{\check{\mathscr{L}}}_{\lambda}(x)=0\,\forall x\in S_0$,
we deduce  $B_\lambda(0)\zeta_0=0$ and therefore
 $B_\lambda(0)u=s u$. By (\ref{e:L-C-D-D1*}) this implies $u\in{\bf X}$ and so $u\in{\bf X}^\bot$.
\end{proof}

We also need the following two corresponding results of Propositions~3.10,~3.11 in \cite{Lu12-}
or \cite{Lu12}.

\begin{proposition}\label{prop:continA0}
Both $\hat\Lambda\times \mathcal{U}\ni (\lambda, x)\mapsto
\check{\mathscr{L}}_{\lambda}(x)\in\R$ and
$\hat\Lambda\times \mathcal{U}^X\ni (\lambda, x)\mapsto A_\lambda(x)\in{\bf X}$ are continuous.
\end{proposition}

Proofs of this proposition and the following key result are similar to
 those of Propositions~3.10,~3.11 in \cite{Lu12-} or \cite{Lu12},
respectively. For completeness their proof  are put off  until Section~\ref{sec:AutoLagr4}
because they are rather long.

\begin{proposition}\label{prop:solutionLagr1}
Let $\bar{\varepsilon}>0$ be such that $B_{{\bf H}}(0, \bar{\varepsilon})\subset\mathcal{U}$.
For any given $\epsilon>0$ there exists $0<\varepsilon\le\bar{\varepsilon}$ such that
if $x\in B_{{\bf H}^\bot}(0,\varepsilon):=\{x\,|\, x\in{\bf H}^\bot,\;
\|x\|_{1,2}<\varepsilon\}$ is a critical point of
$\check{\mathscr{L}}^\bot_{\lambda}$  with some $\lambda\in\hat\Lambda$
then $x$ is a critical point of  $\check{\mathscr{L}}_\lambda$, belongs to $C^4(\mathbb{R};\mathbb{R}^n)$
and also satisfies  $\|x\|_{C^2}<\epsilon$. In particular,
for $x\in B_{{\bf H}^\bot}(0,\varepsilon)$,
$d{\check{\mathscr{L}}}_{\lambda}^\bot(x)=0$ if and only if $d{\check{\mathscr{L}}}_{\lambda}(x)=0$.
\end{proposition}

\subsection{Completing the proofs of Theorems~\ref{th:bif-ness-orbitLagrMan}, \ref{th:bif-suffict1-orbitLagrMan},
\ref{th:bif-existence-orbitLagrMan}, \ref{th:bif-suffict-orbitLagrMan}}
\label{sec:autoLagr3}

The ideas are the same as the proofs of \cite[Theorems~1.18, 1.19, 1.21]{Lu10}.
But the corresponding checks and computations are much more complex and difficult.

\begin{proposition}[\hbox{\cite[Proposition~3.5]{BetPS2}}]\label{prop:BetPS1}
 Let $P$ be a finite-dimensional manifold, $N$ a (possibly infinite dimensional)
Banach manifold, $Q\subset N$ a Banach submanifold, and $A$ a topological
space. Assume that $\chi:A\times P\to N$ is a continuous function such that there exist
$a_0\in A$ and $m_0\in P$ with:
\begin{enumerate}
\item[\rm (a)]  $\chi(a_0,m_0)\in Q$;
\item[\rm (b)] $\chi(a_0,\cdot): P\to N$ is of class $C^1$;
\item[\rm (c)] $\partial_2\chi(a_0,m_0)(T_{m_0}P)+ T_{\chi(a_0,m_0)}Q=T_{\chi(a_0,m_0)}N$.
\end{enumerate}
Then, for $a\in A$ near $a_0$, $\chi(a,P)\cap Q\ne\emptyset$.
\end{proposition}

For $0<\delta\le 3\iota/2$ put
$$
B_{{\bf X}^\bot}(0,\delta):=\{\xi\in C^1_{{E}_{{\bar\gamma}},\tau}(\mathbb{R},B^n_{3\iota/2}(0))\,|\,
\|\xi\|_{C^1}<\delta\}\quad\hbox{and}\quad
\Omega_\delta:=\Phi_{\bar\gamma}\left(B_{{\bf X}^\bot}(0,\delta)\right).
$$
Clearly, $\Omega_\delta$
is a $C^2$ Banach submanifold of $\mathcal{X}_{\tau}(M, \mathbb{I}_g)$.
For the action $\chi$ in (\ref{e:R-action}),
since $\bar\gamma$ is nonconstant and $C^6$,
$$
\chi(\bar\gamma,\cdot):\mathbb{R}\to \mathcal{X}_{\tau}(M, \mathbb{I}_g),\;s\mapsto\chi(\bar\gamma,s)
$$
is a $C^4$ one-to-one immersion, and
$$
\partial_2\chi(\bar\gamma, 0)(T_{0}\mathbb{R})=\mathbb{R}\dot{\bar\gamma}=d\Phi_{\bar\gamma}(0)(\mathbb{R}\zeta_0)\quad\hbox{and}\quad
T_{\chi(\bar\gamma,0)}\Omega_\delta=T_{\bar\gamma}\Omega_\delta=d\Phi_{\bar\gamma}(0)({\bf X}^\bot),
$$
we have $\partial_2\chi(\bar\gamma,0)(T_{0}\mathbb{R})+ T_{\chi(\bar\gamma,0)}\Omega_\delta=T_{\chi(\bar\gamma,)}\mathcal{X}_{\tau}(M, \mathbb{I}_g)$.
Applying Proposition~\ref{prop:BetPS1} to $A=N=\mathcal{X}_{\tau}(M, \mathbb{I}_g)$, $P=\R$, $Q=\Omega_\delta$,
$a_0=\bar\gamma$, $m_0=0$ we get:

\begin{proposition}\label{prop:O-neigh}
For any given $0<\delta\le 2\iota$, if $\gamma\in \mathcal{X}_{\tau}(M, \mathbb{I}_g)$
is close to $\bar\gamma$, then
$(\R\cdot \gamma)\cap \Omega_\delta\ne\emptyset$, that is,
 $\R\cdot\Omega_\delta$  is a neighborhood of the orbit $\mathcal{O}=\R\cdot \bar\gamma$ in $\mathcal{X}_{\tau}(M, \mathbb{I}_g)$.
 \end{proposition}

\begin{proof}[\bf Proof of Theorem~\ref{th:bif-ness-orbitLagrMan}]
By the assumptions there exists a sequence $(\lambda_k)\subset\Lambda$
 converging to $\mu\in\Lambda$ such that the  problem (\ref{e:Lagr10*})
 with  $\lambda=\lambda_k$ has  solutions $\gamma_k$, $k=1,2,\cdots$,
which are $\R$-distinct each other and satisfy
$\gamma_k|_{[0,\tau]}\to \bar\gamma|_{[0,\tau]}$ in $C^1([0,\tau];M)$.
Then $\hat\Lambda=\{\mu,\lambda_k\,|\,k\in\mathbb{N}\}$ is compact and sequential compact.
Take a decreasing sequence of positive numbers $\delta_m\le 2\iota$ such that $\delta_m\to 0$.
For each $\delta_m$, by  Proposition~\ref{prop:O-neigh} we have
  $\gamma_{k_m}\in\R\cdot \Omega_{\delta_m}$ and thus
  $\beta_m:=s_m\cdot\gamma_{k_m}\in \Omega_{\delta_m}$ for some $s_m\in\R$.
Note that each $\beta_m$ is a critical point of ${{\mathfrak{E}}}_{\lambda_{k_m}}$ on
$\Phi_{\bar\gamma}(\mathcal{U}^X)$. Since any two of $(\gamma_k)$ are $\R$-distinct,
so are any two of $(\beta_{m})$.

Note that $\Omega_\delta\subset\Phi_{\bar\gamma}\left(\mathcal{U}^X\right)\subset \Phi_{\bar\gamma}\big(C^1_{{E}_{{\bar\gamma}},\tau}(\mathbb{R},B^n_{3\iota/2}(0))\big)$,
 and by (\ref{e:two-functionals}) and (\ref{e:two-functionals+}) we have
\begin{equation}\label{e:two-functionals++}
\check{\mathscr{L}}_{\lambda}(x)=\mathfrak{E}_\lambda\left(
\Phi_{\bar\gamma}(x)\right)\quad\forall x\in B_{{\bf X}^\bot}(0,\delta)
\end{equation}
because $\delta\le 2\iota<\rho_0$ and $x\in B_{{\bf X}^\bot}(0,\delta)$ imply that
$(t, x(t), \dot{x}(t))\in\mathbb{R}\times B^n_{2\iota}(0)\times B^n_{\rho_0}(0)$
and so $\check{L}(\lambda, t, x(t), \dot{x}(t))=L^\star(\lambda, t, x(t), \dot{x}(t))$
for all $t\in\mathbb{R}$.

It follows that each
$$
x_m:=(\Phi_{\bar\gamma})^{-1}(\beta_m)\in\{\xi\in{\bf X}^\bot\,|\,\|\xi\|_{C^1}<\delta_m\}
$$
 is a critical point of $\check{\mathscr{L}}_{\lambda_{k_m}}$ in ${\bf H}$ (and hence that
 of $\check{\mathscr{L}}_{\lambda_{k_m}}^\bot$ in ${\bf H}^\bot$)
  and they are distinct each other. Moreover $\|x_m\|_{1,2}\le\sqrt{\tau}\|x_m\|_{C^1}\to 0$.
 Hence $(\mu, 0)\in \Lambda\times (\mathcal{U}\cap{\bf H}^\bot)$ is a bifurcation point of $\nabla{\check{\mathscr{L}}}_{\lambda}^\bot=0$
in $\Lambda\times (\mathcal{U}\cap{\bf H}^\bot)$.

Since $\check{\mathscr{L}}_{\lambda}$ is $C^{2-0}$ and
$\nabla\check{\mathscr{L}}_{\lambda}:\mathcal{U}\to {\bf H}$
has the G\^ateaux derivative ${B}_\lambda(x)\in{\mathscr{L}}_s({\bf H})$ at $x\in \mathcal{U}$,
 ${\check{\mathscr{L}}}_{\lambda}^\bot$  is $C^{2-0}$ and  $\nabla{\check{\mathscr{L}}}_{\lambda}^\bot$
 has a G\^ateaux derivative $\mathbb{B}_\lambda(x)\in{\mathscr{L}}_s({\bf H}^\bot)$
 at $x\in \mathcal{U}\cap{\bf H}^\bot$
 given by (\ref{e:gradient4Lagr}). From (\ref{e:gradient5Lagr}), (\ref{e:uniformPost})
 and Propositions~\ref{prop:PPP},~\ref{prop:UniformQ} it easily follows that
the conditions (i)-(iv) of \cite[Theorem~3.1]{Lu8} (\cite[Theorem~C.6]{Lu10}) are satisfied with $\mathcal{F}_\lambda= {\check{\mathscr{L}}}_{\lambda}^\bot$ and
$H=X={\bf H}^\bot$ and $U=\mathcal{U}\cap{\bf H}^\bot$.
 Therefore $m^0_\tau(\check{\mathscr{L}}_{\mu}^\bot, 0)\ge 1$.
 This and (\ref{e:Lagr7Morse*}), (\ref{e:Lagr7Morse**}) and (\ref{e:MorseIndexLagr}) lead to
$m^0_\tau(\mathfrak{E}_{\mu}, \bar\gamma)=m^0_\tau(\check{\mathscr{L}}_{\mu}, 0)\ge 2$.
\end{proof}

Note that Propositions~\ref{prop:solutionLagr1},~\ref{prop:continA0} are not used
 in the proof of Theorem~\ref{th:bif-ness-orbitLagrMan}. However, they are necessary
for proofs of Theorems~\ref{th:bif-suffict1-orbitLagrMan},\ref{th:bif-suffict-orbitLagrMan}.

\begin{proof}[\bf Proof of Theorem~\ref{th:bif-suffict1-orbitLagrMan}]
The original $\Lambda$ can be replaced by the compact and sequential compact subset
 $\hat{\Lambda}=\{\mu,\lambda^+_k, \lambda^-_k\,|\,k\in\mathbb{N}\}$.
Follow the notations above. By Propositions~\ref{prop:continA0},~\ref{prop:PPP},~\ref{prop:UniformQ},
and (\ref{e:L-C-D-D1*}) and (\ref{e:uniformPost0}),
the conditions of \cite[Theorem~3.3]{Lu11} (or \cite[Theorem A.3]{Lu8}) are satisfied with
$\mathcal{L}_\lambda=\check{\mathscr{L}}_{\lambda}$, $H={\bf H}$, $X={\bf X}$, $U=\mathcal{U}$ and $\lambda^\ast=\mu$.
From these, (\ref{e:gradient3Lagr})-(\ref{e:gradient4Lagr})
and (\ref{e:gradient5Lagr}), (\ref{e:uniformPost}) and
Proposition~\ref{prop:reg} it follows that
$\mathcal{L}_\lambda=\check{\mathscr{L}}_{\lambda}^\bot$, $H={\bf H}^\bot$, $X={\bf X}^\bot$, $U=\mathcal{U}\cap{\bf H}^\bot$
and $\lambda^\ast=\mu$ satisfy the conditions of  \cite[Theorem~3.3]{Lu11} (or \cite[Theorem A.3]{Lu8}).

By the assumptions (a)--(b) of Theorem~\ref{th:bif-suffict1-orbitLagrMan} we may use
(\ref{e:Lagr7Morse*}), (\ref{e:Lagr7Morse**}) and (\ref{e:MorseIndexLagr}) to deduce
that $m^0_\tau(\check{\mathscr{L}}_{\mu}^\bot, 0)\ge 1$
 and that for each $k\in\mathbb{N}$,
 $$
 [m^-_\tau(\check{\mathscr{L}}^\bot_{\lambda_k^-}, 0), m^-_\tau(\check{\mathscr{L}}^\bot_{\lambda_k^-}, 0)+
 m^0_\tau(\check{\mathscr{L}}^\bot_{\lambda_k^-}, 0)]\cap[m^-_\tau(\check{\mathscr{L}}^\bot_{\lambda_k^+}, 0),
 m^-_\tau(\check{\mathscr{L}}^\bot_{\lambda_k^+}, 0)+m^0_\tau(\check{\mathscr{L}}^\bot_{\lambda_k^+}, 0)]=\emptyset
 $$
 and either $m^0_\tau(\check{\mathscr{L}}^\bot_{\lambda_k^-}, 0)=0$ or $m^0_\tau(\check{\mathscr{L}}^\bot_{\lambda_k^+}, 0)=0$.

Thus from \cite[Theorem~C.4]{Lu10} 
we conclude that  there exists an infinite sequence
$(\lambda_k, x_k)\subset\hat\Lambda\times{\bf H}^\bot\setminus\{(\mu,0)\}$
converging to $(\mu,0)$ such that each $x_k\ne 0$ and satisfies
 $\check{\mathscr{L}}^\bot_{\lambda_k}(x_k)=0$ for all $k\in\mathbb{N}$.

Fix $0<\delta\le 2\iota$.  Let $\Omega_\delta$ be as in  Proposition~\ref{prop:O-neigh}.
 By Proposition~\ref{prop:solutionLagr1}, passing to a subsequence (if necessary)
 we may assume:  each $x_k$ is $C^4$ and a critical point of
 $\check{\mathscr{L}}_{\lambda_k}$, $\|x_k\|_{C^2}<\delta\;\forall k$ and $\|x_k\|_{C^2}\to 0$.
 Then each $\gamma_k:=\Phi_{\bar\gamma}(x_k)\in\Omega_\delta$ is
 a  $C^6$ solution of the corresponding problem (\ref{e:Lagr10*}) with $\lambda=\lambda_k$, $k=1,2,\cdots$,
 and  $(\gamma_k)$ converges to  $\bar\gamma$ on any compact interval $I\subset\R$ in $C^2$-topology as $k\to\infty$.

 Since $d\Phi_{{\bar\gamma}}(0)[\zeta_0]=\dot{\bar\gamma}$ and
 $T_{\bar\gamma}\Omega_\delta=d\Phi_{\bar\gamma}(0)({\bf X}^\bot)$,
  $\mathbb{R}\zeta_0+{\bf X}^\bot={\bf X}$ implies $\mathbb{R}\dot{\bar\gamma}+T_{\bar\gamma}\Omega_\delta=
  T_{\bar\gamma}\mathcal{X}^1_{\tau}(M, \mathbb{I}_g)$, that is,
  the $C^4$ embedded circle $\mathcal{O}=\mathbb{R}\cdot\bar\gamma$ (\textsf{because of periodicity of} $\bar\gamma$) and $\Omega_\delta$ are  transversely intersecting at $\bar\gamma$.
 It follows that there exists a neighborhood $\mathscr{V}$ of $\bar\gamma$ in $\mathcal{X}^1_{\tau}(M, \mathbb{I}_g)$ such that
 $\mathscr{V}\cap\mathcal{O}\cap\Omega_\delta=\{\bar\gamma\}$.
Because $\|x_k\|_{C^2}\to 0$, there exists  $k_0>0$ such that  for each $k>k_0$,
 $\gamma_k=\Phi_{\bar\gamma}(x_k)\in \mathscr{V}\cap\Omega_\delta\setminus\{\bar\gamma\}$ and
 the $C^4$ immersed submanifold $\mathbb{R}\cdot\gamma_k$
 transversely intersect with $\Omega_\delta$ at $\gamma_k$. Hence
 \begin{equation}\label{e:diffOrbit}
 \mathbb{R}\cdot\gamma_k\ne\mathcal{O}\quad\hbox{for any $k>k_0$.}
 \end{equation}
(Otherwise, $\mathcal{O}$ and $\Omega_\delta$ have
at least two distinct intersecting points $\gamma_k$ and $\bar\gamma$ in $\mathscr{V}$.)

 We conclude that $\{\mathbb{R}\cdot\gamma_k\,|\, k\in\mathbb{N}\}$ is an infinite set.
 (Thus $(\gamma_k)$ has a subsequence which only consists of $\R$-distinct elements.
 The proof is completed.)
 Otherwise,  passing to a subsequence we may assume that all $\gamma_k$
 are $\R$-same, i.e., $\gamma_k=s_k\cdot \gamma^\ast$ for some $s_k\in\mathbb{R}$,
 where $\gamma^\ast:\mathbb{R}\to M$ satisfies (\ref{e:Lagr10*}) with $\lambda=\lambda_k$, $k=1,2,\cdots$.
Since all partial derivatives of $L(\lambda,\cdot)$ of order no more than two depend continuously on
 $(\lambda, x, v)\in\Lambda\times TM$, it easily follows that $\gamma^\ast$ satisfies (\ref{e:Lagr10*}) with $\lambda=\mu$.
 Clearly, $\bar\gamma$ sits in the intersection of $\mathcal{O}$ and the closure of $\mathbb{R}\cdot\gamma^\ast$. By the assumption
 (c) of Theorem~\ref{th:bif-suffict1-orbitLagrMan}
 $\mathbb{R}\cdot\gamma^\ast$ is closed and so equal to $\mathcal{O}$,
namely  $\mathbb{R}\cdot\gamma_k=\mathcal{O}\;\forall k$,
  which contradicts (\ref{e:diffOrbit}).
\end{proof}

In order to prove Theorem~\ref{th:bif-suffict-orbitLagrMan}
we also need some preparations. Consider the $C^4$ Hilbert--Riemannian manifold
$$
\Lambda_{\tau}(M, \mathbb{I}_g)=\{\gamma\in
W^{1,2}_{loc}(\R,M)\,|\,\gamma(t+\tau)=\mathbb{I}_g(\gamma(t))\;\forall
t\}
$$
with the natural Riemannian metric given by (\ref{e:1.1}); see  \cite[Theorem~4.2]{PiTa01} (or \cite[Theorem(8)]{Pa63}).
 Let $\|\cdot\|_1=\sqrt{\langle\cdot,\cdot\rangle_1}$ be the induced norm.

Since $\R_{\bar\gamma}$ is an infinite cyclic subgroup of $\R$ with generator $p>0$,
i.e., $\bar\gamma$ has the least period $p$,
the orbit ${\cal O}:=\R\cdot\bar\gamma$ is  an $\R$-invariant
compact connected $C^{3}$ submanifold of $\Lambda_{\tau}(M, \mathbb{I}_g)$,
precisely an $C^{3}$ embedded circle $S^1(p):=\R/p\Z$.
Let  $\pi:N{\cal O}\to{\cal O}$ be the normal bundle
of ${\cal O}$ in $\Lambda_{\tau}(M, \mathbb{I}_g)$. It is a $C^{2}$
Hilbert vector bundle over ${\cal O}$ (because $T{\cal O}$ is a $C^{2}$ subbundle of
    $T_{\cal O}\Lambda_{\tau}(M, \mathbb{I}_g)$), and
    $$
    XN{\cal O}:=T_{\cal O}\mathcal{X}_{\tau}(M, \mathbb{I}_g)\cap N{\cal O}
    $$
    is a $C^{2}$ Banach vector subbundle
    of $T_{\cal O}\mathcal{X}_{\tau}(M, \mathbb{I}_g)$ by \cite[Proposition~5.1]{Lu5}.
  Recall that  $3\iota$ is less than the injectivity
radius of $g$ at each point on $\bar\gamma(\mathbb{R})$.
     For $0<\nu\le 3\iota$  we define
$$
N{\cal O}(\nu):=\{(\gamma,v)\in N{\cal O}\,|\,\|v\|_{1,2}<\nu\} \quad\hbox{and}\quad
 XN{\cal O}(\nu):= \{(\gamma,v)\in XN{\cal
O}\,|\,\|v\|_{C^1}<\nu\}.
$$
Clearly, $XN{\cal O}(\nu)\subset N{\cal O}(\sqrt{\tau}\nu)$ and
there exist natural induced $\R$-actions on these bundles given by
$$
(\gamma,v)\mapsto (s\cdot\gamma, s\cdot v)\quad\forall s\in\mathbb{R}.
$$
 Using the exponential map $\exp$ of $g$ we define  the map
\begin{equation}\label{e:1.17}
{\rm EXP}:T\Lambda_{\tau}(M, \mathbb{I}_g)(\sqrt{\tau}\nu)=\{(\gamma,v)\in
T\Lambda_{\tau}(M, \mathbb{I}_g)\,|\,\|v\|_{1,2}<\sqrt{\tau}\nu\}\to
\Lambda_{\tau}(M, \mathbb{I}_g)
\end{equation}
 by ${\rm EXP}(\gamma,v)(t)=\exp_{\gamma(t)}v(t)\;\forall
t\in\R$. Clearly, ${\rm EXP}$ is equivariant, i.e.,
$$
s\cdot({\rm EXP}(\gamma,v))={\rm EXP}(s\cdot\gamma, s\cdot v)\quad\forall s\in\mathbb{R}.
$$
It follows from  \cite[Lemma~5.2]{Lu5} that ${\rm EXP}$ is $C^{2}$.
 For sufficiently small  $\nu>0$,  ${\rm EXP}$ gives rise to a $C^{2}$
diffeomorphism $\digamma: N{\cal O}(\sqrt{\tau}\nu)\to {\cal N}({\cal O}, \sqrt{\tau}\nu)$,
where ${\cal N}({\cal O}, \sqrt{\tau}\nu)$ is  an open
neighborhood of ${\cal O}$ in $\Lambda_{\tau}(M, \mathbb{I}_g)$.
Let ${\cal X}({\cal O},\nu):= \digamma(XN{\cal O}(\nu))$, which is contained in ${\cal N}({\cal O},\sqrt{\tau}\nu)$.

\begin{lemma}\label{lem:R-distinct}
Suppose that $(\mathbb{I}_g)^l=id_M$ for some $l\in\mathbb{N}$,  and that
 $0<\delta\le 2\iota$ is so small that
$$
\Omega_\delta=\Phi_{\bar\gamma}\left(B_{{\bf X}^\bot}(0,\delta)\right)\subset {\cal X}({\cal O},\nu).
$$
Then for any two different points $\gamma_i\in\Omega_\delta$, $i=1,2$,
either they are $\mathbb{R}$-distinct,
or there exists an integer $0<m\le l$ such that $\gamma_2=(m\tau)\cdot\gamma_2$ or $\gamma_1=(m\tau)\cdot\gamma_1$.
In particular, if $l=1$ and  $\tau$ is equal to the minimal period $p$ of $\bar\gamma$,
then any two different points in $\Omega_\delta$ are $\mathbb{R}$-distinct.
\end{lemma}
\begin{proof}[\bf Proof]
Let  different points $\xi_1, \xi_2\in B_{{\bf X}^\bot}(0,\delta)$ be such that
$\gamma_1=\Phi_{\bar\gamma}(\xi_1)$ and $\gamma_2=\Phi_{\bar\gamma}(\xi_2)$ are $\mathbb{R}$-same.
Then we have $s\ge 0$ such that $s\cdot\gamma_1=\gamma_2$.
Since $\mathbb{I}^l_g=id_M$ implies $(kl\tau)\cdot\gamma_1=\gamma_1$ and $(s-kl\tau)\cdot\gamma_1=s\cdot\gamma_1$ for any $k\in\mathbb{Z}$
we can assume $0\le s<l\tau$. By  (\ref{e:Lagr5+}),
$(\bar\gamma,\sum^n_{i=1}\xi_{2}^ie_i)$  and
 $(s\cdot\bar\gamma,\sum^n_{i=1}(s\cdot\xi_{1}^i)(s\cdot e_i))$
belong to $N{\cal O}(\sqrt{\tau}\nu)$.
 Note that
\begin{eqnarray*}
(s\cdot\gamma_1)(t)&=&\gamma_1(s+t)=\exp_{{\bar\gamma}(s+t)}\bigg(\sum^n_{i=1}\xi_{1}^i(s+t)
e_i(s+t)\bigg)\\
&=&{\rm EXP}\bigg(s\cdot\bar\gamma,\sum^n_{i=1}(s\cdot\xi_{1}^i)(s\cdot e_i)\bigg)(t)
\end{eqnarray*}
and so
$$
s\cdot\gamma_1=\digamma\bigg(s\cdot\bar\gamma,\sum^n_{i=1}(s\cdot\xi_{1}^i)(s\cdot e_i)\bigg).
$$
Similarity, we have
$$
\gamma_2=\digamma\bigg(\bar\gamma,\sum^n_{i=1}\xi_{2}^ie_i\bigg).
$$
Then $s\cdot\bar\gamma=\bar\gamma$ and
$$
\sum^n_{i=1}(s\cdot\xi_{1}^i)(s\cdot e_i)=\sum^n_{i=1}\xi_{2}^ie_i.
$$
The former implies $s\in\mathbb{R}_{\bar\gamma}\subset\{[0],\cdots,[(l-1)\tau]\}$,
where $[q\tau]=q\tau+ l\mathbb{Z}$.
 Hence $s\in\{0,\cdots,l-1\}$. Combining with the latter we obtain $\xi_1=\xi_2$,
and therefore a contradiction.

When $l=1$ and  $\tau$ is equal to the minimal period $p$ of
$\bar\gamma$, $\mathbb{R}_{\bar\gamma}=\{0\}$ and so
$\xi_1=\xi_2$. A contradiction is obtained.
\end{proof}

 \begin{proof}[\bf Proof of Theorem~\ref{th:bif-suffict-orbitLagrMan}]
The original $\Lambda$ may be replaced by $\hat\Lambda=[\mu-\varepsilon, \mu+\varepsilon]$.
By the first paragraph in the proof of Theorem~\ref{th:bif-suffict1-orbitLagrMan}
we have checked  that
$\mathcal{L}_\lambda=\check{\mathscr{L}}_{\lambda}^\bot$, $H={\bf H}^\bot$, $X={\bf X}^\bot$,
 $U=\mathcal{U}\cap{\bf H}^\bot$  and $\lambda^\ast=\mu$
  satisfy the conditions of \cite[Theorem~3.6]{Lu11}
  except for the condition (f).

By the assumptions of Theorem~\ref{th:bif-suffict-orbitLagrMan} and
(\ref{e:Lagr7Morse*}), (\ref{e:Lagr7Morse**}) and (\ref{e:MorseIndexLagr})
we get that $m^0_\tau(\check{\mathscr{L}}_{\mu}^\bot, 0)\ge 1$ and
$m^0_\tau(\check{\mathscr{L}}^\bot_{\lambda}, 0)=0$
 for each $\lambda\in\hat\Lambda\setminus\{\mu\}$ near $\mu$, and that
 $m^-_\tau(\mathscr{L}^\bot_{\lambda}, 0)$  takes, respectively, values $m^0_\tau(\check{\mathscr{L}}^\bot_{\mu}, 0)$
  and $m^-_\tau(\check{\mathscr{L}}^\bot_{\mu}, 0)+ m^0_\tau(\check{\mathscr{L}}^\bot_{\mu}, 0)-1$
 as $\lambda\in\hat\Lambda$ varies in two deleted half neighborhoods  of $\mu$.
These mean that the condition (f) of \cite[Theorem~C.7]{Lu10} (\cite[Theorem~3.6]{Lu11})  is satisfied.
Therefore  one of the following alternatives occurs:
\begin{enumerate}
\item[(i)] There exists a sequence $(x_k)\subset{\bf H}^\bot\setminus\{0\}$ converging to $0$ in  ${\bf H}^\bot$
such that $\nabla\check{\mathscr{L}}^\bot_\mu(x_k)=0$ for all $k$.

\item[(ii)]  For each $\lambda\in\Lambda\setminus\{\mu\}$ near $\mu$,
$\nabla\check{\mathscr{L}}^\bot_\lambda(w)=0$ has a  solution $x_\lambda\in {\bf X}^\bot$ different from $0$,
 which  converges to $0$ in ${\bf X}^\bot$ as $\lambda\to\mu$.

\item[(iii)] Given a neighborhood $\mathfrak{W}$ of $0$ in ${\bf X}^\bot$,
 there is a one-sided  neighborhood $\Lambda^0$ of $\mu$ such that
for any $\lambda\in\Lambda^0\setminus\{\mu\}$, $\nabla\check{\mathscr{L}}^\bot_\lambda(w)=0$
has at least two nonzero solutions in $\mathfrak{W}$, $x_\lambda^1$ and $x_\lambda^2$,
which can also be required to satisfy $\check{\mathscr{L}}^\bot_\lambda(x_\lambda^1)\ne\check{\mathscr{L}}^\bot_\lambda(x^2_\lambda)$
provided that $m^0_\tau(\check{\mathscr{L}}_{\mu}^\bot, 0)\ge 2$
and $\nabla\check{\mathscr{L}}^\bot_\lambda(w)=0$ has only
finitely many nonzero solutions in $\mathfrak{W}$.
\end{enumerate}

Let $\delta>0$ satisfy  Proposition~\ref{prop:O-neigh} and Lemma~\ref{lem:R-distinct}.
 By Proposition~\ref{prop:solutionLagr1},
 we obtain:
 \begin{enumerate}
\item[$\bullet$] In case (i),  passing to a subsequence (if necessary) all $x_k$ are $C^4$ and
 satisfy: $\nabla\check{\mathscr{L}}_\mu(x_k)=0$, $0<\|x_k\|_{C^2}<\delta$ and $\|x_k\|_{C^2}\to 0$.
 Therefore each $\gamma_k:=\Phi_{\bar\gamma}(x_k)\in\Omega_\delta$ is
 a  $C^6$ solution of the corresponding problem (\ref{e:Lagr10*}) with $\lambda=\mu$,
 and  $(\gamma_k)$ converges to  $\bar\gamma$ on any compact interval $I\subset\R$ in $C^2$-topology as $k\to\infty$.
 \item[$\bullet$] In case (ii), when $\lambda\in\Lambda\setminus\{\mu\}$ is close to $\mu$, all $x_\lambda$ are $C^4$ and  satisfy:
      $\nabla\check{\mathscr{L}}_\lambda(x_\lambda)=0$, $0<\|x_\lambda\|_{C^2}<\delta$ and $\|x_\lambda\|_{C^2}\to 0$ as $\lambda\to\mu$.
 Hence each $\gamma_\lambda:=\Phi_{\bar\gamma}(x_\lambda)\in\Omega_\delta\setminus\{\bar\gamma\}$ is
 a  $C^6$ solution of the corresponding problem (\ref{e:Lagr10*}),
   $\gamma_\lambda$ converges to  $\bar\gamma$ on any compact interval $I\subset\R$ in $C^2$-topology as $\lambda\to \mu$,
 and  $\mathbb{R}\cdot\gamma_\lambda\ne\mathcal{O}$ by Lemma~\ref{lem:R-distinct}.
\item[$\bullet$] In case (iii), we can require that the neighborhood $\mathfrak{W}$ so small that
$\Phi_{\bar\gamma}(\mathfrak{W})\subset\mathcal{W}$ and $\mathfrak{W}\subset B_{{\bf X}^\bot}(0,\delta)$.
The latter implies $\Phi_{\bar\gamma}(\mathfrak{W})\subset\Omega_\delta$.
Then all  $x_\lambda^1$ and $x_\lambda^2$ are $C^4$ and critical points of $\check{\mathscr{L}}_{\lambda}$,
and satisfy:
$0<\|x^i_\lambda\|_{C^2}<\delta$ and
$\|x^i_\lambda\|_{C^2}\to 0$, $i=1,2$. Consequently,
$\gamma_\lambda^1:=\Phi_{\bar\gamma}(x_\lambda^1)$ and $\gamma_\lambda^2:=\Phi_{\bar\gamma}(x_\lambda^2)$ belong to $\Omega_\delta\cap \mathcal{W}\setminus \mathbb{R}\cdot\gamma_0$, are
  $C^6$ solutions of the corresponding problem (\ref{e:Lagr10*}).
  When $m^0_\tau(\mathfrak{E}_{\mu}, \bar\gamma)= m^0_\tau(\check{\mathscr{L}}_{\mu}^\bot, 0)+1\ge 3$,
  and (\ref{e:Lagr10*}) with parameter value $\lambda$ has
only finitely many $\mathbb{R}$-distinct solutions in $\mathcal{W}$ which  are $\mathbb{R}$-distinct from $\bar\gamma$,
it is clear that  $\nabla\check{\mathscr{L}}^\bot_\lambda(w)=0$ has only finitely many nonzero solutions in $\mathfrak{W}$,
 and therefore we can require that $x_\lambda^1$ and $x_\lambda^2$ satisfy $\check{\mathscr{L}}^\bot_\lambda(x_\lambda^1)\ne\check{\mathscr{L}}^\bot_\lambda(x^2_\lambda)$,
  which implies $\mathfrak{E}_\lambda(\gamma^1_\lambda)\ne\mathfrak{E}_\lambda(\gamma^2_\lambda)$,
   \end{enumerate}
\end{proof}

\begin{proof}[\bf Proof of Theorem~\ref{th:bif-existence-orbitLagrMan}]
The original $\Lambda$ may be replaced by $\hat\Lambda=\alpha([0,1])$.
By the first paragraph in the proof of Theorem~\ref{th:bif-suffict1-orbitLagrMan}
we have checked  that $\mathcal{L}_\lambda=\check{\mathscr{L}}_{\lambda}^\bot$, $H={\bf H}^\bot$,
 $X={\bf X}^\bot$,
 $U=\mathcal{U}\cap{\bf H}^\bot$  satisfy the assumptions a)-c) and (i)-(v) of \cite[Theorem~C.5]{Lu10}
 for any $\lambda^\ast\in\hat\Lambda$. The condition (d) of Theorem~\ref{th:bif-existence-orbitLagrMan}
 can be translated into:
  $$
 [m^-_\tau(\check{\mathscr{L}}^\bot_{\lambda^-}, 0), m^-_\tau(\check{\mathscr{L}}^\bot_{\lambda^-}, 0)+
 m^0_\tau(\check{\mathscr{L}}^\bot_{\lambda^-}, 0)]\cap[m^-_\tau(\check{\mathscr{L}}^\bot_{\lambda^+}, 0),
 m^-_\tau(\check{\mathscr{L}}^\bot_{\lambda^+}, 0)+
 m^0_\tau(\check{\mathscr{L}}^\bot_{\lambda^+}, 0)]=\emptyset
 $$
 and either $m^0_\tau(\check{\mathscr{L}}^\bot_{\lambda^-}, 0)=0$ or $m^0_\tau(\check{\mathscr{L}}^\bot_{\lambda^+}, 0)=0$.

Hence the condition (e.3) of \cite[Theorem~C.5]{Lu10} is satisfied.
  Thus \cite[Theorem~C.4]{Lu10}
concludes that  there exists $\mu\in\alpha([0,1])$ and an infinite sequence $(\lambda_k, x_k)\subset\hat\Lambda\times{\bf H}^\bot\setminus\{(\mu,0)\}$
 converging to $(\mu,0)$ such that each $x_k\ne 0$ and satisfies
  $\check{\mathscr{L}}^\bot_{\lambda_k}(x_k)=0$ for all $k\in\mathbb{N}$.
Moreover,  $\mu$ is not equal to $\lambda^+$ (resp. $\lambda^-$) if
   $m^0_\tau(\check{\mathscr{L}}^\bot_{\lambda^+}, 0)=0$
    (resp. $m^0_\tau(\check{\mathscr{L}}^\bot_{\lambda^-}, 0)=0$).
We can assume $\lambda_k=\alpha(t_k)$ for some $(t_k)\subset [0,1]$
converging to $\bar{t}\in [0, 1]$.

Fix $0<\delta\le 2\iota$.  Let $\Omega_\delta$ satisfy
  Proposition~\ref{prop:O-neigh} and Lemma~\ref{lem:R-distinct}.
 By Proposition~\ref{prop:solutionLagr1}, passing to a subsequence (if necessary)
 we may assume:  each $x_k$ is $C^4$ and a critical point of
 $\check{\mathscr{L}}_{\lambda_k}$, $\|x_k\|_{C^2}<\delta\;\forall k$ and $\|x_k\|_{C^2}\to 0$.
 Then each $\gamma_k:=\Phi_{\bar\gamma}(x_k)\in\Omega_\delta$ is
 a  $C^6$ solution of the corresponding problem (\ref{e:Lagr10*}) with $\lambda=\lambda_k$, $k=1,2,\cdots$,
 and  $(\gamma_k)$ converges to  $\bar\gamma$ on any compact interval $I\subset\R$ in $C^2$-topology as $k\to\infty$.
 We can assume that all $x_k$ are distinct each other.
 By Lemma~\ref{lem:R-distinct} each $\gamma_k$ has at most $l$ $\mathbb{R}$-same points in $\{\gamma_k\,|\,k\in\mathbb{N}\}$.
 Hence $(\gamma_k)$ has a subsequence $(\gamma_{k_i})$ consisting of
 completely $\mathbb{R}$-distinct points.
 The required assertions are proved.
\end{proof}

\subsection{Proofs of Propositions~\ref{prop:continA0},\ref{prop:solutionLagr1}}\label{sec:AutoLagr4}

\begin{proof}[\bf Proof of Proposition~~\ref{prop:continA0}]
{\bf Step 1}({\it Prove that $\hat\Lambda\times \mathcal{U}\ni (\lambda, x)\mapsto
\check{\mathscr{L}}_{\lambda}(x)\in\R$ is continuous}).
   Indeed, for any two points $(\lambda, x)$ and $(\lambda_0, x_0)$
   in $\hat\Lambda\times \mathcal{U}$ we can write
\begin{eqnarray*}
\check{\mathscr{L}}_{\lambda}(x)-\check{\mathscr{L}}_{\lambda_0}(x_0)&=&
\left[\int^{\tau}_0\check{L}_\lambda(t, x(t),\dot{x}(t))dt-
\int^{\tau}_0\check{L}_{\lambda}(t, x_0(t),\dot{x}_0(t))dt\right]\\
&&+\left[\int^{\tau}_0\check{L}_\lambda(t, x_0(t),\dot{x}_0(t))dt-
\int^{\tau}_0\check{L}_{\lambda_0}(t, x_0(t),\dot{x}_0(t))dt\right].
\end{eqnarray*}
As $(\lambda, x)\to (\lambda_0, x_0)$, we derive from (L6) in Lemma~\ref{lem:modif}  and 
\cite[Proposition C.1]{Lu9} (resp.
(L6) in Lemma~\ref{lem:modif} and the Lebesgue dominated convergence theorem)  that
the first (resp. second) bracket on the right side converges to the zero.

(\textsf{{Actually}}, we only need that $\hat\Lambda\times \mathcal{U}^X\ni (\lambda, x)\mapsto
\check{\mathscr{L}}_{\lambda}(x)\in\R$ is continuous. This can easily be proved as follows.
For any fixed point $x_0\in\mathcal{U}^X$, we can take a positive $\rho>0$ such that $\rho>\sup_t|\dot{x}_0(t)|$.
Since $E_{\bar\gamma}$ is an orthogonal matrix,  and $\check{L}:\hat\Lambda\times\mathbb{R}\times \bar{B}^n_{2\iota}(0)\times\mathbb{R}^n\to\R$ is continuous,
 we deduce that $\check{L}$ is uniformly continuous in $\hat\Lambda\times\mathbb{R}\times \bar{B}^n_{2\iota}(0)\times \bar{B}^n_{\rho}(0)$.
This and (\ref{e:check*}) lead to the desired claim.)

{\bf Step 2}({\it Prove that $\hat\Lambda\times \mathcal{U}^X\ni (\lambda, x)\mapsto
A_{\lambda}(x)\in{\bf X}$ is continuous}).
For $(\lambda_1, x), (\lambda_2, y)\in\hat\Lambda\times \mathcal{U}^X$, and $\xi\in{\bf H}$, since
\begin{eqnarray*}
d\check{\mathscr{L}}_{\lambda_1}(x)[\xi]-d\check{\mathscr{L}}_{\lambda_2}(y)[\xi]
&=&\int_0^{\tau}  \left( \partial_{q}\check{L}(\lambda_1, t, x(t),\dot{x}(t))-\partial_{q}\check{L}(\lambda_2, t, y(t),\dot{y}(t))\right)\cdot\xi(t)dt\\
+&&\hspace{-6mm} \int_0^{\tau}\left(\partial_{v}\check{L}\left(\lambda_1, t, x(t),\dot{x}(t))-\partial_{v}\check{L}(\lambda_2, t, y(t),\dot{y}(t))\right)\cdot\dot{\xi}(t)
\right) \, dt,
\end{eqnarray*}
we have
\begin{eqnarray*}
\|\nabla\check{\mathscr{L}}_{\lambda_1}(x)-\nabla\check{\mathscr{L}}_{\lambda_2}(y)\|_{1,2}
\hspace{-2mm}&\le&\hspace{-2mm}\left(\int_0^{\tau}  \left| \partial_{q}\check{L}(\lambda_1, t, x(t),\dot{x}(t))-\partial_{q}\check{L}(\lambda_2, t, y(t),\dot{y}(t))\right|^2dt\right)^{1/2}\\
+&&\hspace{-6mm}  \left(\int_0^{\tau}\left|\partial_{v}\check{L}(\lambda_1, t, x(t),\dot{x}(t))-\partial_{v}\check{L}(\lambda_2, t,  y(t),\dot{y}(t))\right|^2dt\right)^{1/2}.
\end{eqnarray*}
Fix a point $(\lambda_1, x)\in\hat\Lambda\times\mathcal{U}^X$.
Then $\{(\lambda_1, t, x(t), \dot{x}(t))\,|\, t\in [0,\tau]\}$
is a compact subset of $\hat\Lambda\times [0, \tau]\times B^n_{2\iota}(0)\times \mathbb{R}^n$.
Since $\partial_{q}\check{L}$ and $\partial_{v}\check{L}$ are uniformly
 continuous in any compact neighborhood of this compact subset we deduce that
\begin{equation}\label{e:C^0-shoulian}
\|\nabla\check{\mathscr{L}}_{\lambda_1}(x)-\nabla\check{\mathscr{L}}_{\lambda_2}(y)\|_{C^0}
\le C_\tau\|\nabla\check{\mathscr{L}}_{\lambda_1}(x)-\nabla\check{\mathscr{L}}_{\lambda_2}(y)\|_{1,2}
\to 0
\end{equation}
provided  $(\lambda_2, y)\in\hat\Lambda\times\mathcal{U}^X$ converges to $(\lambda_1, x)$ in $\hat\Lambda\times\mathcal{U}^X$.

By (\ref{e:3.10}) and (\ref{e:grad1+}),  we have
\begin{eqnarray}\label{e:grad*}
\frac{d}{dt}\nabla\check{\mathscr{L}}_{\lambda}(x)(t)&=&\frac{e^{t}}{2}\int^\infty_te^{-s}\left(\partial_q
\check{L}_{{\lambda}}\bigl(s, x(s),\dot{x}(s)\bigr)-\mathfrak{R}_\lambda^{x}(s) \right)\, ds\nonumber\\
&&-\frac{e^{-t}}{2}\int^t_{-\infty}e^{s}\left(\partial_q
\check{L}_{{\lambda}}\bigl(s, x(s),\dot{x}(s)\bigr)-\mathfrak{R}_\lambda^{x})(s)
\right)\, ds\nonumber\\
&&+\partial_v \check{L}_{{\lambda}}\bigl(t, x(t),\dot{x}(t)\bigr)+\mathfrak{M}\int^{\tau}_0\partial_v
\check{L}_{{\lambda}}(s, x(s),\dot{x}(s))ds,
\end{eqnarray}
where $\mathfrak{R}_\lambda^{x}$ is given by (\ref{e:grad1}).
Let $\mathfrak{T}_\lambda^x(t)$ denote a column vector
 \begin{eqnarray*}
\left[\left(\oplus_{l\le p}\frac{\sin\theta_l}{2-2\cos\theta_l}{\scriptscriptstyle\left(\begin{array}{cc}
0& -1\\
1& 0\end{array}\right)}-\frac{1}{2}I_{2p}\right)\oplus{\rm
diag}(a_{p+1}(t), \cdots, a_{\sigma}(t))\right]\int^\tau_0\partial_v \check{L}_{\lambda}\bigl(s,x(s),\dot{x}(s)\bigr)ds.
\end{eqnarray*}
Then we have a constant $C(E_{\bar\gamma})>0$ only depending on
$E_{\bar\gamma}$ such that for all $t$,
 \begin{equation}\label{e:R1}
\big| \mathfrak{T}_{\lambda_1}^x(t)- \mathfrak{T}_{\lambda_2}^{y}(t)\big|\le C(E_{\gamma_0})(1+|t|)\int^\tau_0\Big|\partial_v \check{L}_{\lambda_1}\bigl(s,
x(s),\dot{x}(s)\bigr)-\partial_v \check{L}_{\lambda_2}\bigl(s, y(s),\dot{y}(s)\bigr)\Big|ds.
  \end{equation}
 By (\ref{e:grad1}) and  (\ref{e:grad1+}) we observe
 \begin{eqnarray}\label{e:R2}
&&\mathfrak{R}_\lambda^x(t)= \int^t_0\partial_v
\check{L}_{\lambda}\bigl(s, x(s),\dot{x}(s)\bigr)ds+ \mathfrak{T}_\lambda^x(t),\\
&&\frac{d}{dt}\mathfrak{T}_\lambda^x(t)=\mathfrak{M}\int^\tau_0\partial_v \check{L}_{\lambda}\bigl(s,x(s),\dot{x}(s)\bigr)ds.\label{e:R3}
\end{eqnarray}
  For $0\le t\le\tau$, let
\begin{eqnarray}\label{e:Gamma+1}
\Gamma^+_\lambda(x)(t):&=&\frac{e^{t}}{2}\int^\infty_te^{-s}\left(
\partial_q\check{L}_{{\lambda}}\bigl(s, x(s),\dot{x}(s)\bigr)-\mathfrak{R}_\lambda^{x}(s) \right) ds,\\
\Gamma^-_\lambda(x)(t):&=&\frac{e^{-t}}{2}\int^t_{-\infty}e^{s}
\left(\partial_q\check{L}_{{\lambda}}\bigl(s, x(s),\dot{x}(s)\bigr)-\mathfrak{R}_\lambda^{x}(s) \right)ds.\label{e:Gamma-1}
\end{eqnarray}
It follows from (\ref{e:grad*}) that
\begin{eqnarray}\label{e:grad+}
&&\left|\frac{d}{dt}\nabla\check{\mathscr{L}}_{\lambda_1}(x)(t)-
\frac{d}{dt}\nabla\check{\mathscr{L}}_{\lambda_2}(y)(t)\right| \nonumber\\
&\le&|\Gamma^+_{\lambda_1}(x)(t)-\Gamma^+_{\lambda_2}(y)(t)|+
|\Gamma^-_{\lambda_1}(x)(t)-\Gamma^-_{\lambda_2}(y)(t)| \nonumber\\
&&+|\partial_v \check{L}_{{\lambda_1}}\bigl(t, x(t),\dot{x}(t)\bigr)-
\partial_v \check{L}_{{\lambda_2}}\bigl(t, y(t),\dot{y}(t)\bigr)| \nonumber\\
&&+|\mathfrak{M}|\int^{\tau}_0|\partial_v\check{L}_{{\lambda_1}}(s, x(s),\dot{x}(s))- \partial_v\check{L}_{{\lambda_2}}(s,y(s),\dot{y}(s))|ds.
\end{eqnarray}
Let us estimate terms in the right side.

Suppose $k\tau\le t\le (k+1)\tau$ for some integer $k\ge 0$.
By  $\check{L}(\lambda, t+\tau, x, v)=\check{L}(\lambda, t, E_{{\bar\gamma}}x, E_{{\bar\gamma}} v)$
and (\ref{e:R1}) we deduce
\begin{eqnarray}\label{e:lv5}
&&|\mathfrak{R}_{\lambda_1}^x(t)-\mathfrak{R}_{\lambda_2}^{y}(t)|\nonumber\\
&\le& \int^{(k+1)\tau}_0\left|\partial_v\check{L}_{\lambda_1}\bigl(s, x(s),\dot{x}(s)\bigr)-\partial_v \check{L}_{\lambda_2}\bigl(s, y(s),\dot{y}(s)\bigr)
\right|ds+\big| \mathfrak{T}_{\lambda_1}^x(t)- \mathfrak{T}_{\lambda_2}^{y}(t)\big|\nonumber\\
&\le&(k+1)\int^{\tau}_0\left|\partial_v\check{L}_{\lambda_1}\bigl(s, x(s),\dot{x}(s)\bigr)-\partial_v \check{L}_{\lambda_2}\bigl(s, y(s),\dot{y}(s)\bigr)
\right|ds\nonumber\\
&+&C(E_{\bar\gamma})(1+ (k+1)\tau)\int^\tau_0\Big|\partial_v \check{L}_{\lambda_1}\bigl(s, x(s),\dot{x}(s)\bigr)-\partial_v \check{L}_{\lambda_2}\bigl(s, y(s),\dot{y}(s)\bigr)\Big|ds.
\end{eqnarray}
Similarly, if $(-k-1)\tau\le t\le -k\tau$ for some  integer $k\ge 0$, we have also
\begin{eqnarray}\label{e:lv6}
&&|\mathfrak{R}_{\lambda_1}^x(t)-\mathfrak{R}_{\lambda_2}^{y}(t)|\nonumber\\
&\le& \int_{(-k-1)\tau}^0\left|\partial_v\check{L}_{\lambda_1}\bigl(s, x(s),\dot{x}(s)\bigr)-\partial_v \check{L}_{\lambda_2}\bigl(s, y(s),\dot{y}(s)\bigr)
\right|ds+\big| \mathfrak{T}_{\lambda_1}^x(t)- \mathfrak{T}_{\lambda_2}^{y}(t)\big|\nonumber\\
&\le&(k+1)\int^{\tau}_0\left|\partial_v\check{L}_{\lambda_1}\bigl(s, x(s),\dot{x}(s)\bigr)-\partial_v \check{L}_{\lambda_2}\bigl(s, y(s),\dot{y}(s)\bigr)
\right|ds\nonumber\\
&+&C(E_{\bar\gamma})(1+ (k+1)\tau)\int^\tau_0\Big|\partial_v \check{L}_{\lambda_1}\bigl(s, x(s),\dot{x}(s)\bigr)-\partial_v \check{L}_{\lambda_2}\bigl(s, y(s),\dot{y}(s)\bigr)\Big|ds.
\end{eqnarray}

For $0\le t\le\tau$, it follows from (\ref{e:Gamma+1}) and (\ref{e:lv5}) that
\begin{eqnarray}\label{e:lv6.1}
&&|\Gamma^+_{\lambda_1}(x)(t)-\Gamma^+_{\lambda_2}(y)(t)|\\
&\le&\frac{e^{\tau}}{2}\sum^\infty_{k=0}\int^{(k+1)\tau}_{k\tau}e^{-s}
\left|\partial_q\check{L}_{{\lambda_1}}\bigl(s, x(s),\dot{x}(s)\bigr)-
\partial_q\check{L}_{{\lambda_2}}\bigl(s, y(s),\dot{y}(s)\bigr)\right|ds\nonumber\\
&&+\frac{e^{\tau}}{2}\sum^\infty_{k=0}\int^{(k+1)\tau}_{k\tau}e^{-s}\left|\mathfrak{R}_{\lambda_1}^{x}(s)- \mathfrak{R}_{\lambda_2}^{y}(s)\right| ds\nonumber\\
&\le&\frac{e^{\tau}}{2}\sum^\infty_{k=0}e^{-k\tau}\int^{\tau}_{0}
\left|\partial_q\check{L}_{{\lambda_1}}\bigl(s, x(s),\dot{x}(s)\bigr)-
\partial_q\check{L}_{{\lambda_2}}\bigl(s, y(s),\dot{y}(s)\bigr)\right|ds\nonumber\\
&&+\frac{\tau e^{\tau}}{2}\sum^\infty_{k=0}e^{-k\tau}(k+1)\int^{\tau}_0\left|\partial_v
\check{L}_{\lambda_1}\bigl(s, x(s),\dot{x}(s)\bigr)-\partial_v \check{L}_{\lambda_2}\bigl(s,y(s),\dot{y}(s)\bigr)\right|ds\nonumber\\
&&+C(E_{\bar\gamma})\frac{\tau e^{\tau}}{2}\sum^\infty_{k=0}
e^{-k\tau}(1+(k+1)\tau)\int^\tau_0\Big|\partial_v \check{L}_{\lambda_1}\bigl(s,x(s),\dot{x}(s)\bigr)-
\partial_v \check{L}_{\lambda_2}\bigl(s,y(s),\dot{y}(s)\bigr)\Big|ds.\nonumber
\end{eqnarray}
Similarly, for $0\le t\le\tau$, (\ref{e:Gamma-1}) and (\ref{e:lv6}) lead to
\begin{eqnarray}\label{e:lv6.2}
&&|\Gamma^-_{\lambda_1}(x)(t)-\Gamma^-_{\lambda_2}(y)(t)|\\
&\le&\frac{1}{2}\sum^\infty_{k=0}\int^{(-k+1)\tau}_{-k\tau}e^{s}
\left|\partial_q\check{L}_{{\lambda_1}}\bigl(s,x(s),\dot{x}(s)\bigr)-
\partial_q\check{L}_{{\lambda_2}}\bigl(s,
y(s),\dot{y}(s)\bigr)\right|ds\nonumber\\
&&+\frac{1}{2}\sum^\infty_{k=0}\int^{(-k+1)\tau}_{-k\tau}e^{s}\left|\mathfrak{R}_{\lambda_1}^{x}(s)- \mathfrak{R}_{\lambda_2}^{y}(s)\right| ds\nonumber\\
&\le&\frac{e^{\tau}}{2}\sum^\infty_{k=0}e^{-k\tau}\int^{\tau}_{0}
\left|\partial_q\check{L}_{{\lambda_1}}\bigl(s, x(s),\dot{x}(s)\bigr)-
\partial_q\check{L}_{{\lambda_2}}\bigl(s, y(s),\dot{y}(s)\bigr)\right|ds\nonumber\\
&&+\frac{\tau e^{\tau}}{2}\sum^\infty_{k=0}e^{-k\tau}(k+1)\int^{\tau}_0\left|\partial_v
\check{L}_{\lambda_1}\bigl(s, x(s),\dot{x}(s)\bigr)-\partial_v \check{L}_{\lambda_2}\bigl(s, y(s),\dot{y}(s)\bigr)\right|ds\nonumber\\
&&+C(E_{\bar\gamma})\frac{\tau e^{\tau}}{2}\sum^\infty_{k=0}
e^{-k\tau}(1+(k+1)\tau)\int^\tau_0\Big|\partial_v
\check{L}_{\lambda_1}\bigl(s, x(s),\dot{x}(s)\bigr)-
\partial_v \check{L}_{\lambda_2}\bigl(s, y(s),\dot{y}(s)\bigr)\Big|ds.\nonumber
\end{eqnarray}
From these and (\ref{e:grad+}) we get
\begin{eqnarray}\label{e:grad++}
&&\left|\frac{d}{dt}\nabla\check{\mathscr{L}}_{\lambda_1}(x)(t)-
\frac{d}{dt}\nabla\check{\mathscr{L}}_{\lambda_2}(y)(t)\right| \nonumber\\
&\le&|\partial_v \check{L}_{{\lambda_1}}\bigl(t, x(t),\dot{x}(t)\bigr)-\partial_v \check{L}_{{\lambda_2}}\bigl(t, y(t),\dot{y}(t)\bigr)| \nonumber\\
&&+(C(E_{\bar\gamma})C_\tau+
|\mathfrak{M}|)\int^{\tau}_0|\partial_v\check{L}_{{\lambda_1}}(s, x(s),\dot{x}(s))- \partial_v\check{L}_{{\lambda_2}}(s,y(s),\dot{y}(s))|ds\nonumber\\
&&+C^\ast_\tau\int^{\tau}_0|\partial_q\check{L}_{{\lambda_1}}(s, x(s),\dot{x}(s))- \partial_q\check{L}_{{\lambda_2}}(s,y(s),\dot{y}(s))|ds
\end{eqnarray}
for some constant $C^\ast_\tau>0$ and for all  $0\le t\le\tau$.

As above, for a fixed $(\lambda_1, x)\in\hat\Lambda\times\mathcal{U}^X$, by uniform continuity of $\partial_{q}\check{L}$ and $\partial_{v}\check{L}$
on a compact neighborhood of a compact subset
$\{(\lambda_1,t, x(t), \dot{x}(t))\,|\, t\in [0,\tau]\}$ of $\hat\Lambda\times
[0,\tau]\times B^n_{2\iota}(0)\times \mathbb{R}^n$,
we can derive from these and (\ref{e:grad++}) that
\begin{eqnarray*}
\left\|\frac{d}{dt}\nabla\check{\mathscr{L}}_{\lambda_1}(x)-
\frac{d}{dt}\nabla\check{\mathscr{L}}_{\lambda_2}(y)\right\|_{C^0}\to 0
\end{eqnarray*}
provided  $(\lambda_2, y)\in\hat\Lambda\times\mathcal{U}^X$ converges to $(\lambda_1, x)$ in $\hat\Lambda\times\mathcal{U}^X$.
This and (\ref{e:C^0-shoulian}) lead to the second claim.
\end{proof}

\begin{proof}[\bf Proof of Proposition~\ref{prop:solutionLagr1}]
It is enough to prove sufficiency.
Since $x\in B_{{\bf H}^\bot}(0,\varepsilon)$ is a critical point
of  the restriction of $\check{\mathscr{L}}_{\lambda}$ to $B_{{\bf H}^\bot}(0,\varepsilon)$,
 \begin{equation}\label{e:orbit-criticalLagr}
 d\check{\mathscr{L}}_{\lambda}(x)[\xi]=0\;\;\forall \xi\in T_{x}B_{{\bf H}^\bot}(0,\varepsilon)={\bf H}^\bot.
 \end{equation}
We shall prove $d\check{\mathscr{L}}_{\lambda}(x)=0$ in four steps.

{\bf Step 1}(\textsf{Prove that  $x$ is $C^4$}).
By (\ref{e:orbit-criticalLagr}) we have  $\mu(\lambda, x)\in\R$ such that
\begin{equation}\label{e:gradientLgar}
 \nabla\check{\mathscr{L}}_{\lambda}(x)=\mu(\lambda, x)\zeta_0.
\end{equation}
That is, $d\check{\mathscr{L}}_{\lambda}(x)[\xi]-\mu(\lambda, x)(\zeta_0,\xi)_{1,2}=0$
for all $\xi\in{\bf H}$.
It follows that
\begin{eqnarray*}
&&\int^\tau_0[(\partial_q\check{L}_\lambda(x(t), \dot{x}(t)),\xi(t))_{\mathbb{R}^n}+
(\partial_v\check{L}_\lambda(x(t), \dot{x}(t),\dot{\xi}(t))_{\mathbb{R}^n}]dt\\
&&-\int^\tau_0[\mu(\lambda, x)(\zeta_0(t),\xi(t))_{\mathbb{R}^n}+\mu(\lambda, x)(\dot{\zeta}_0(t),\dot{\xi}(t))_{\mathbb{R}^n}]dt=0\quad\forall\xi\in{\bf H}.
\end{eqnarray*}
Define ${\bf L}_\lambda(t, q, v)=\check{L}_\lambda(t, q,v)-\mu(\lambda, x)(\zeta_0(t), q)_{\mathbb{R}^n}+\mu(\lambda, x)(\dot{\zeta}_0(t), v)_{\mathbb{R}^n}$. Then $x(t)$ satisfies
\begin{eqnarray}\label{e:P-EL}
\int^\tau_0[(\partial_q{\bf L}_\lambda(t, x(t), \dot{x}(t)),\xi(t))_{\mathbb{R}^n}
+(\partial_v{\bf L}_\lambda(t, x(t), \dot{x}(t)),\dot{\xi}(t))_{\mathbb{R}^n}]dt=0\quad\forall\xi\in{\bf H}.
\end{eqnarray}
Since $\zeta_0$ is $C^5$, ${\bf L}_\lambda$ is $C^4$ and  satisfies  the conditions in Lemma~\ref{lem:modif},
by Remark~\ref{rm:regularity}
we obtain that  $x$ is $C^4$.

{\bf Step 2}(\textsf{For any $\epsilon>0$ there exists $\delta>0$ such that $\|x\|_{1,2}<\delta$ implies $|\mu(\lambda,x)|<\epsilon$}).
 Since $\hat\Lambda\subset\R$ is compact and sequential compact,
 $\check{L}(\lambda, t+\tau, q, v)=\check{L}(\lambda, t, E_{{\bar\gamma}}q, E_{{\bar\gamma}} v)$,
  and  partial derivatives
$$
\partial_q\check{L}_\lambda(\cdot),\quad\partial_v\check{L}_\lambda(\cdot),
\quad\partial_{qv}\check{L}_\lambda(\cdot),\quad
\partial_{qq}\check{L}_\lambda(\cdot),\quad\partial_{vv}\check{L}_\lambda(\cdot),
\quad\partial_{vt}\check{L}_\lambda(\cdot)
$$
depend continuously on $(\lambda, t, q, v)\in\hat\Lambda\times\mathbb{R}\times \bar{B}^n_{2\iota}(0)\times \mathbb{R}^n$,
 by shrinking $\iota>0$ (if necessary) it follows from (L3)  in Lemma~\ref{lem:modif} that
\begin{eqnarray*}
&&|\partial_q\check{L}_\lambda(t, q,v)-\partial_q\check{L}_\lambda(t, 0, 0)|\\
&\le& |\partial_q\check{L}_\lambda(t, q,v)-\partial_q\check{L}_\lambda(t, q, 0)|+
|\partial_q\check{L}_\lambda(t, q,0)-\partial_q\check{L}_\lambda(t, 0, 0)|\\
&\le&\sup_{0\le s\le 1}|\partial_{qv}\check{L}_\lambda(t, q,sv)|\cdot|v|+\sup_{0\le s\le 1}|\partial_{qq}\check{L}_\lambda(t, sq,0)|\cdot|q|\\
&\le&C(|v|+|v|^2)+ C|q|.
\end{eqnarray*}
Hence we have a constant $C'>0$ such that
\begin{equation}\label{e:lq}
|\partial_q\check{L}_\lambda(t, q,v)|\le C'(1+|v|^2),\quad\forall (\lambda, t, q, v)\in \hat\Lambda\times\mathbb{R}\times \bar{B}^n_{2\iota}(0)\times \mathbb{R}^n.
\end{equation}
Similarly, we can increase the constant $C'>0$ so that
\begin{equation}\label{e:lv}
|\partial_v\check{L}_\lambda(t, q,v)|\le C'(1+|v|),\quad\forall (\lambda, t, q, v)\in \hat\Lambda\times\mathbb{R}\times \bar{B}^n_{2\iota}(0)\times \mathbb{R}^n.
\end{equation}

 Since $\nabla\check{\mathscr{L}}_{\lambda}(0)=0$, 
by (\ref{e:gradientLgar}) we have
 \begin{eqnarray*}
 \mu(\lambda, x)(\zeta_0,\xi)_{1,2}&=&d\check{\mathscr{L}}_{\lambda}(x)[\xi]-d\check{\mathscr{L}}_{\lambda}(0)[\xi]\\
 &=&\int^\tau_0[(\partial_q\check{L}_\lambda(t, x(t), \dot{x}(t)),\xi(t))_{\mathbb{R}^n}
-(\partial_q\check{L}_\lambda(t, 0, 0),\xi(t))_{\mathbb{R}^n}]dt\\
&&+\int^\tau_0[(\partial_v\check{L}_\lambda(t, x(t), \dot{x}(t),\dot{\xi}(t))_{\mathbb{R}^n}
-(\partial_v\check{L}_\lambda(t, 0, 0),\dot{\xi}(t))_{\mathbb{R}^n}]dt.
  \end{eqnarray*}
 For the first integral, by the mean value theorem, (L3) in Lemma~\ref{lem:modif}  we derive
   \begin{eqnarray}\label{e:lq2}
  &&\Big|\int^\tau_0[(\partial_q\check{L}_\lambda(t, x(t), \dot{x}(t))-
  \partial_q\check{L}_\lambda(t, 0, 0),\xi(t))_{\mathbb{R}^n}]dt\Big|\nonumber\\
&\le&\Big|\int^\tau_0\int^1_0[(\partial_{qq}\check{L}_\lambda(t, sx(t), s\dot{x}(t))x(t),\xi(t))_{\mathbb{R}^n}]dtds\Big|\nonumber\\
&&\qquad+\Big|\int^\tau_0\int^1_0[(\partial_{qv}\check{L}_\lambda(t, sx(t), s\dot{x}(t))\dot{x}(t),\xi(t))_{\mathbb{R}^n}]dtds\Big|\nonumber\\
&\le& C\int^\tau_0\int^1_0(1+ |s\dot{x}(t))|^2)|x(t)||\xi(t)|dtds
+C\int^\tau_0\int^1_0(1+ |s\dot{x}(t))|)|\dot{x}(t)||\xi(t)|dtds\nonumber\\
&\le& C(C_\tau)^2\|x\|_{1,2}\|\xi\|_{1,2}\big(\tau+  2\|x\|_{1,2}^2\big)
+CC_\tau\|\xi\|_{1,2}(\sqrt{\tau}\|x\|_{1,2}+\|x\|_{1,2}^2).
  \end{eqnarray}
  Here we use inequalities
   $\|u\|_{C^0}\le (\tau+1)\|u\|_{1,1}\;\forall u\in W^{1,1}([0,\tau];\mathbb{R}^{n})$ and
$\|u\|_{C^0}\le (\sqrt{\tau}+1/\sqrt{\tau})\|u\|_{1,2}\;\forall u\in W^{1,2}([0,\tau];\mathbb{R}^{n})$.
 Similarly,  we may estimate  the second integral as follows:
 \begin{eqnarray}\label{e:lv2}
&&\Big|\int^\tau_0[(\partial_v\check{L}_\lambda(t, x(t), \dot{x}(t),\dot{\xi}(t))_{\mathbb{R}^n}
-(\partial_v\check{L}_\lambda(t, 0, 0),\dot{\xi}(t))_{\mathbb{R}^n}]dt\Big|\nonumber\\
&\le&\Big|\int^\tau_0\int^1_0[(\partial_{vv}\check{L}_\lambda(t, sx(t), s\dot{x}(t))\dot{x}(t),\dot{\xi}(t))_{\mathbb{R}^n}]dtds\Big|\nonumber\\
&&\qquad+\Big|\int^\tau_0\int^1_0[(\partial_{qv}\check{L}_\lambda(t, sx(t), s\dot{x}(t)){x}(t),\dot{\xi}(t))_{\mathbb{R}^n}]dtds\Big|
\nonumber\\
&\le& C\|x\|_{1,2}\|\xi\|_{1,2}+ CC_\tau\|x\|_{1,2}(\sqrt{\tau}\|\xi\|_{1,2}+\|x\|_{1,2}\|\xi\|_{1,2}).
  \end{eqnarray}
 Hence we get
     \begin{eqnarray*}
 |\mu(\lambda, x)|\|\zeta_0\|_{1,2}&\le&C(C_\tau)^2\|x\|_{1,2}\big(\tau+ 2\|x\|_{1,2}^2\big)
+CC_\tau(\sqrt{\tau}\|x\|_{1,2}+\|x\|_{1,2}^2)  \\
&&+C\|x\|_{1,2}+ CC_\tau\|x\|_{1,2}(\sqrt{\tau}+\|x\|_{1,2}).
  \end{eqnarray*}
 The desired claim immediately follows because $\zeta_0\ne 0$.

 {\bf Step 3}(\textsf{For any $\epsilon>0$ there exists
 $\varepsilon>0$ such that $\|x\|_{1,2}\le\varepsilon$ implies $\|x\|_{C^2}<\epsilon$}).
 By (L2) in Lemma~\ref{lem:modif} and  the mean value theorem of integrals we deduce
 \begin{eqnarray*}
 c|v|^2&\le&\int^1_0\left(\partial_{vv}\check{L}_\lambda(t, q, sv)[v], v\right)_{\mathbb{R}^n}ds
 =\left(\partial_v\check{L}_\lambda(t, q, v)-\partial_v\check{L}_\lambda(t, q, 0), v\right)_{\mathbb{R}^n}
  \end{eqnarray*}
 and so
 \begin{equation}\label{e:lv2.1}
 c|v|\le \left|\partial_v\check{L}_\lambda(t, q, v)-\partial_v\check{L}_\lambda(t, q, 0)\right|
 \end{equation}
 for any $(\lambda, t, q, v)\in\hat\Lambda\times\mathbb{R}\times \bar{B}^n_{2\iota}(0)\times \mathbb{R}^n$.
 Since we have proved that $x$ is $C^4$ in Step~1, (\ref{e:gradientLgar})
  also holds in the sense of pointwise, i.e.,
  \begin{equation}\label{e:lv2.2}
 \nabla\check{\mathscr{L}}_{\lambda}(x)(t)=\mu(\lambda, x)\zeta_0(t)\;\forall t.
\end{equation}
 From this and (\ref{e:grad*}) it follows that
 \begin{eqnarray}\label{e:lv2.3}
\mu(\lambda, x)\dot{\zeta}_0(t)&=&\frac{e^{t}}{2}\int^\infty_te^{-s}\left(\partial_q
\check{L}_{{\lambda}}\bigl(s, x(s),\dot{x}(s)\bigr)-\mathfrak{R}_\lambda^{x}(s) \right)\, ds\nonumber\\
&&-\frac{e^{-t}}{2}\int^t_{-\infty}e^{s}\left(\partial_q
\check{L}_{{\lambda}}\bigl(s, x(s),\dot{x}(s)\bigr)-\mathfrak{R}_\lambda^{x})(s)
\right)\, ds\nonumber\\
&&+\partial_v \check{L}_{{\lambda}}\bigl(t, x(t),\dot{x}(t)\bigr)+\mathfrak{M}\int^{\tau}_0\partial_v
\check{L}_{{\lambda}}(s, x(s),\dot{x}(s))ds,
\end{eqnarray}
where $\mathfrak{R}_\lambda^{x}$ is given by (\ref{e:grad1}).
In particular, taking $x=0$ we get
  \begin{eqnarray}\label{e:lv2.4}
\mu(\lambda, 0)\dot{\zeta}_0(t)&=&\frac{e^{t}}{2}\int^\infty_te^{-s}\left(\partial_q
\check{L}_{{\lambda}}\bigl(s, 0, 0\bigr)-\mathfrak{R}_\lambda^{0}(s) \right)\, ds\nonumber\\
&&-\frac{e^{-t}}{2}\int^t_{-\infty}e^{s}\left(\partial_q
\check{L}_{{\lambda}}\bigl(s, 0, 0\bigr)-\mathfrak{R}_\lambda^{0})(s)
\right)\, ds\nonumber\\
&&+\partial_v \check{L}_{{\lambda}}\bigl(t, 0, 0\bigr)+\mathfrak{M}\int^{\tau}_0\partial_v
\check{L}_{{\lambda}}(s, 0, 0)ds.
\end{eqnarray}
 (\ref{e:lv2.3}) minus (\ref{e:lv2.4}) gives rise to
 \begin{eqnarray}\label{e:lv2.4+}
 &&\partial_v \check{L}_{{\lambda}}\bigl(t, 0, 0\bigr)-
 \partial_v \check{L}_{{\lambda}}\bigl(t, x(t),\dot{x}(t)\bigr)\nonumber\\
 &=&\frac{e^{t}}{2}\int^\infty_te^{-s}\left(\partial_q
\check{L}_{{\lambda}}\bigl(s, x(s),\dot{x}(s)\bigr)-\partial_q
\check{L}_{{\lambda}}\bigl(s, 0, 0\bigr)\right)\, ds
+\frac{e^{t}}{2}\int^\infty_te^{-s}\left(\mathfrak{R}_\lambda^{0}(s)-
\mathfrak{R}_\lambda^{x}(s) \right)\, ds\nonumber\\
&&-\frac{e^{-t}}{2}\int^t_{-\infty}e^{s}\left(\partial_q
\check{L}_{{\lambda}}\bigl(s, x(s),\dot{x}(s)\bigr)-\partial_q
\check{L}_{{\lambda}}\bigl(s, 0, 0\bigr)
\right)\, ds
+\frac{e^{-t}}{2}\int^t_{-\infty}e^{s}\left(\mathfrak{R}_\lambda^{x}(s)-
\mathfrak{R}_\lambda^{0}(s)\right)\, ds\nonumber\\
&&+\mathfrak{M}\int^{\tau}_0[\partial_v
\check{L}_{{\lambda}}(s, x(s),\dot{x}(s))- \partial_v
\check{L}_{{\lambda}}(s, 0, 0)]ds\nonumber\\
&&+\mu(\lambda, 0)\dot{\zeta}_0(t)-\mu(\lambda, x)\dot{\zeta}_0(t)
\end{eqnarray}
For $0\le t\le\tau$, it follows from (\ref{e:lv6.1}) and (\ref{e:lv6.2}) that
\begin{eqnarray}\label{e:lv2.4++}
&&\bigg|\frac{e^{t}}{2}\int^\infty_te^{-s}\left(\partial_q
\check{L}_{{\lambda}}\bigl(s, x(s),\dot{x}(s)\bigr)-\partial_q
\check{L}_{{\lambda}}\bigl(s, 0, 0\bigr)\right)\, ds
+\frac{e^{t}}{2}\int^\infty_te^{-s}\left(\mathfrak{R}_\lambda^{0}(s)-\mathfrak{R}_\lambda^{x}(s) \right)\, ds\bigg|\nonumber\\
&\le&\biggl[\frac{e^{\tau}}{2}\sum^\infty_{k=0}e^{-k\tau}
+\frac{\tau e^{\tau}}{2}\sum^\infty_{k=0}e^{-k\tau}(k+1)
+C(E_{\bar\gamma})\frac{\tau e^{\tau}}{2}\sum^\infty_{k=0}e^{-k\tau}(1+(k+1)\tau)\biggr]\times\nonumber\\
&&\times\Big|\int^\tau_0\big[\partial_v \check{L}_{\lambda}\bigl(s,x(s),\dot{x}(s)\bigr)-\partial_v \check{L}_{\lambda}\bigl(s, 0, 0\bigr)\big]ds\Big|
\end{eqnarray}
and
\begin{eqnarray}\label{e:lv2.5}
&&\bigg|-\frac{e^{-t}}{2}\int^t_{-\infty}e^{s}\left(\partial_q
\check{L}_{{\lambda}}\bigl(s, x(s),\dot{x}(s)\bigr)-\partial_q
\check{L}_{{\lambda}}\bigl(s, 0, 0\bigr)
\right)\, ds
+\frac{e^{-t}}{2}\int^t_{-\infty}e^{s}\left(\mathfrak{R}_\lambda^{x}(s)-\mathfrak{R}_\lambda^{0}(s)
\right)\, ds\bigg|\nonumber\\
&\le&\biggl[\frac{e^{\tau}}{2}\sum^\infty_{k=0}e^{-k\tau}
+\frac{\tau e^{\tau}}{2}\sum^\infty_{k=0}e^{-k\tau}(k+1)
+C(E_{\bar\gamma})\frac{\tau e^{\tau}}{2}\sum^\infty_{k=0}e^{-k\tau}(1+(k+1)\tau)\biggr]\times\nonumber\\
&&\times\Big|\int^\tau_0\big[
\partial_v \check{L}_{\lambda}\bigl(s,x(s),\dot{x}(s)\bigr)-\partial_v \check{L}_{\lambda}\bigl(s, 0, 0\bigr)\big]ds\Big|.
\end{eqnarray}
An elementary calculation yields
\begin{eqnarray*}
&&\biggl[\frac{e^{\tau}}{2}\sum^\infty_{k=0}e^{-k\tau}
+\frac{\tau e^{\tau}}{2}\sum^\infty_{k=0}e^{-k\tau}(k+1)
+C(E_{\bar\gamma})\frac{\tau e^{\tau}}{2}\sum^\infty_{k=0}e^{-k\tau}(1+(k+1)\tau)\biggr]\\
&<& C(E_{\bar\gamma},\tau):=(2+C(E_{\bar\gamma})\tau+
C(E_{\bar\gamma})\tau^2)\frac{e^{2\tau}}{e^\tau-1}+
(1+C(E_{\bar\gamma})\tau^2)\frac{e^{4\tau}}{(e^\tau-1)^2}.
\end{eqnarray*}
Hence (\ref{e:lv2.4+}), (\ref{e:lv2.4++}) and (\ref{e:lv2.5}) lead to
 \begin{eqnarray*}
  &&\left|\partial_v \check{L}_{{\lambda}}\bigl(t, x(t),\dot{x}(t)\bigr)-\partial_v \check{L}_{{\lambda}}\bigl(t, 0, 0\bigr)\right|
 \\
 &\le& 2C(E_{\bar\gamma},\tau)\Big|\int^\tau_0\big[
\partial_v \check{L}_{\lambda}\bigl(s,x(s),\dot{x}(s)\bigr)-
\partial_v \check{L}_{\lambda}\bigl(s, 0, 0\bigr)\big]ds\Big|\nonumber\\
&&+|\mathfrak{M}|\Big|\int^\tau_0\big[
\partial_v \check{L}_{\lambda}\bigl(s,x(s),\dot{x}(s)\bigr)-
\partial_v \check{L}_{\lambda}\bigl(s, 0, 0\bigr)\big]ds\Big|
+|\mu(\lambda, 0)-\mu(\lambda, x)|\cdot|\dot{\zeta}_0(t)|.
  \end{eqnarray*}
From this, (\ref{e:lv2.1}) and (\ref{e:lv2.4+}), (\ref{e:lv2.4++}) and (\ref{e:lv2.5}) we deduce
 \begin{eqnarray}\label{e:lv2.6}
 c|\dot{x}(t)|&\le& \left|\partial_v \check{L}_{{\lambda}}\bigl(t, x(t),\dot{x}(t)\bigr)-
 D_v\check{L}_\lambda(t, x(t), 0)\right|\nonumber\\
 &\le&\left|\partial_v \check{L}_{{\lambda}}\bigl(t, x(t),\dot{x}(t)\bigr)-\partial_v \check{L}_{{\lambda}}\bigl(t, 0, 0\bigr)\right|
 +\left|D_v\check{L}_\lambda(t, x(t), 0)-
 \partial_v \check{L}_{{\lambda}}\bigl(t, 0, 0\bigr)\right|\nonumber\\
 &\le& (2C(E_{\bar\gamma},\tau)+|\mathfrak{M}|)\Big|\int^\tau_0\big[
\partial_v \check{L}_{\lambda}\bigl(s,x(s),\dot{x}(s)\bigr)-\partial_v \check{L}_{\lambda}\bigl(s, 0, 0\bigr)\big]ds\Big|\nonumber\\
&&+|\mu(\lambda, 0)-\mu(\lambda, x)|\cdot|\dot{\zeta}_0(t)|+\left|D_v\check{L}_\lambda(t, x(t), 0)-
\partial_v \check{L}_{{\lambda}}\bigl(t, 0, 0\bigr)\right|
  \end{eqnarray}
for all $0\le t\le\tau$.

Since $|\dot{\zeta}_0(t)|$ are bounded on $[0,\tau]$,
by Step 2 we have $|\mu(\lambda, 0)-\mu(\lambda, x)|\cdot\sup_t|\dot{\zeta}_0(t)|\to 0$
as $\|x\|_{1,2}\to 0$. It follows from  (\ref{e:lv}) and
 \cite[Proposition C.1]{Lu9}  that
 $$
 \Big|\int^\tau_0\big[
\partial_v \check{L}_{\lambda}\bigl(s,x(s),\dot{x}(s)\bigr)-\partial_v \check{L}_{\lambda}\bigl(s, 0, 0\bigr)\big]ds\Big|
 \to 0
 $$
 uniformly in $\lambda$ as $\|x\|_{1,2}\to 0$. These and (\ref{e:lv2.6}) show:
  \begin{eqnarray}\label{e:lv2.7}
 \textsf{For any $\epsilon>0$ there exists $\varepsilon>0$ such that $\|x\|_{1,2}\le\varepsilon$ implies $\|x\|_{C^1}<\epsilon$.}
\end{eqnarray}

 By  (\ref{e:P-EL}) we have
\begin{eqnarray}\label{e:6P-EL1}
0&=&\frac{d}{dt}\Big(\partial_v{\bf L}_\lambda(t, x(t), \dot{x}(t))\Big)-
\partial_q{\bf L}_\lambda(t, x(t), \dot{x}(t))\nonumber\\
&=&\frac{d}{dt}\Big(\partial_v\check{L}_\lambda(t, x(t), \dot{x}(t))+\mu(\lambda, x)\dot{\zeta}_0(t)\Big)-\partial_q\check{L}_\lambda(t, x(t), \dot{x}(t))+
\mu(\lambda, x)\zeta_0(t)\nonumber\\
&=&\partial_{vv}\check{L}_\lambda(t, x(t), \dot{x}(t))\ddot{x}(t)+
\partial_{vq}\check{L}_\lambda(t, x(t), \dot{x}(t))\dot{x}(t)
+\partial_{vt}\check{L}_\lambda(t, x(t), \dot{x}(t))\nonumber\\
&&-\partial_q\check{L}_\lambda(t, x(t), \dot{x}(t))+
 \mu(\lambda, x)\ddot{\zeta}_0(t)+\mu(\lambda, x)\zeta_0(t).
\end{eqnarray}
In particular, taking $x=0$ we get
\begin{eqnarray}\label{e:6P-EL2}
0=\partial_{vt}\check{L}_\lambda(t, 0, 0)-\partial_q\check{L}_\lambda(t, 0, 0)+
\mu(\lambda, 0)\ddot{\zeta}_0(t)+\mu(\lambda, 0)\zeta_0(t).
\end{eqnarray}
(\ref{e:6P-EL1}) minus (\ref{e:6P-EL2}) gives rise to
\begin{eqnarray}\label{e:6P-EL3}
0&=&\partial_{vv}\check{L}_\lambda(t,  x(t),
\dot{x}(t))\ddot{x}(t)+\partial_{vq}\check{L}_\lambda(t,  x(t), \dot{x}(t))\dot{x}(t)
\nonumber\\
&&+\partial_{vt}\check{L}_\lambda(t, x(t), \dot{x}(t))-\partial_{vt}\check{L}_\lambda(t, 0, 0)\nonumber\\
&&-\partial_{q}\check{L}_\lambda(t, x(t), \dot{x}(t))+\partial_q\check{L}_\lambda(t, 0, 0)\nonumber\\
&&+\mu(\lambda, x)\ddot{\zeta}_0(t)-\mu(\lambda, 0)\ddot{\zeta}_0(t)+\mu(\lambda, x)\zeta_0(t)-\mu(\lambda, 0)\zeta_0(t).
\end{eqnarray}
Note that (L2) of Lemma~\ref{lem:modif} implies
 $|[\partial_{vv}\check{L}_\lambda(t,  x(t), \dot{x}(t))]^{-1}\xi|\le\frac{1}{c}|\xi|\;\forall \xi\in{\mathbb{R}^n}$.  (\ref{e:6P-EL3}) leads to
\begin{eqnarray}\label{e:6P-EL4-}
|\ddot{x}(t)|&\le&
\frac{1}{c}|\partial_{vq}\check{L}_\lambda(t,  x(t), \dot{x}(t))|\cdot|\dot{x}(t)|
+\frac{1}{c}|\partial_{q}\check{L}_\lambda(t, x(t), \dot{x}(t))-
\partial_q\check{L}_\lambda(t, 0, 0)|\nonumber\\
&&+\frac{1}{c}|\partial_{vt}\check{L}_\lambda(t, x(t), \dot{x}(t))-
\partial_{vt}\check{L}_\lambda(t, 0, 0)|+
\frac{1}{c}|\partial_{q}\check{L}_\lambda(t, x(t), \dot{x}(t))-
\partial_q\check{L}_\lambda(t, 0, 0)|\nonumber\\
&&+|\mu(\lambda, x)-\mu(\lambda, 0)||\ddot{\zeta}_0(t)|
+|\mu(\lambda, x)-\mu(\lambda, 0)||\zeta_0(t)|.
\end{eqnarray}
Since $|{\zeta}_0(t)|$ and $|\ddot{\zeta}_0(t)|$ are bounded on $[0,\tau]$,
the desired claim may follow from this,
 Step 2 and (\ref{e:lv2.7}).

{\bf Step 4}(\textsf{Prove $d\check{\mathscr{L}}_{\lambda}(x)=0$}).
Since  $\|\zeta_0\|_{C^0}>0$, by Step 3 we get $\varepsilon>0$
such that $\|x\|_{1,2}\le\varepsilon$ implies
$\|x\|_{C^2}<\rho_0$. In particular, $x\in\mathcal{U}^X$.
Since $\check{L}=L^\star$ on $\hat\Lambda\times \mathbb{R}\times B^n_{3\iota/2}(0)\times B^n_{\rho_0}(0)$,
$$
\check{\mathscr{L}}_\lambda(x)=\int^{\tau}_0\check{L}_\lambda(t, x(t),\dot{x}(t))dt=\mathscr{L}^\star_\lambda(x)=
\mathfrak{E}_\lambda(\Phi_{\bar\gamma}(x)).
$$
Note that $\Phi_{\bar\gamma}(x)\in \mathcal{X}^6_{\tau}(M, \mathbb{I}_g)$.
$[-a,a]\cdot \Phi_{\bar\gamma}(x)$ is a $C^4$ submanifold of dimension one.
Therefore by these and the arguments below (\ref{e:Lagr4*})
we get that
$$
S_x:=\Phi^{-1}_{{\bar\gamma}}\bigl([-a,a]\cdot \Phi_{\bar\gamma}(x)\cap{\rm Im}(\Phi_{\bar\gamma})\cap \mathcal{X}^1_{\tau}(M, \mathbb{I}_g)\big)
$$
is a $C^2$ submanifold of $\mathcal{X}^1_\tau(B^n_{2\iota}(0), {E}_{\bar\gamma})$
containing $0$ as an interior point.
Observe that
$$
T_{\Phi_{\bar\gamma}(x)}\big([-a,a]\cdot \Phi_{\bar\gamma}(x)\cap{\rm Im}(\Phi_{\bar\gamma})\cap \mathcal{X}^1_{\tau}(M, \mathbb{I}_g)\bigr)=
\mathbb{R}\frac{d}{ds}(s\cdot\Phi_{\bar\gamma}(x))\Big|_{s=0}=\mathbb{R}(\Phi_{\bar\gamma}(x))^\cdot
$$
and that $\mathfrak{E}_\lambda$ is constant on $[-a,a]\cdot \Phi_{\bar\gamma}(x)$. Hence
$$
d\mathfrak{E}_\lambda(\Phi_{\bar\gamma}(x))[(\Phi_{\bar\gamma}(x))^\cdot]=0.
$$
Let $\zeta_x:=(d\Phi_{\bar\gamma}(x)))^{-1}((\Phi_{\bar\gamma}(x))^\cdot)\in T_xS_x$. Then
\begin{equation}\label{e:tangentzero}
d\check{\mathscr{L}}_\lambda(x)[\zeta_x]=
d\mathfrak{E}_\lambda(\Phi_{\bar\gamma}(x))[(\Phi_{\bar\gamma}(x))^\cdot]=0.
\end{equation}
Since $\Phi_{{\bar\gamma}}(\xi)(t)=\phi_{{\bar\gamma}}(t,\xi(t))$ and
$d\phi_{{\bar\gamma}}(t,p)[(1,v)]=
\partial_2\phi_{{\bar\gamma}}(t,p)[v]+\partial_1\phi_{{\bar\gamma}}(t,p)$ we get
\begin{eqnarray*}
&&\partial_2\phi_{{\bar\gamma}}(t,x(t))[\dot{x}(t)]+\partial_1\phi_{{\bar\gamma}}(t,x(t))
=(\Phi_{\bar\gamma}(x))^\cdot(t)=
(d\Phi_{\bar\gamma}(x)[\zeta_x])(t)\\
&=&\frac{d}{ds}\Big|_{s=0}\Phi_{\bar\gamma}(x+s\zeta_x)(t)
=\frac{d}{ds}\Big|_{s=0}\phi_{\bar\gamma}(t, x(t)+s\zeta_x(t))\\
&=&\partial_2\phi_{\bar\gamma}(t, x(t))[\zeta_x(t)]
\end{eqnarray*}
and thus
\begin{equation}\label{e:tangentzero+1}
\partial_1\phi_{{\bar\gamma}}(t,x(t))=\partial_2\phi_{\bar\gamma}(t, x(t))[\zeta_x(t)-\dot{x}(t)].
\end{equation}
As $\|x\|_{C^1}\to 0$ we deduce that $\|\zeta_x-\zeta_0\|_{C^0}\to 0$, that is,
\begin{equation}\label{e:tangentzero+2}
\zeta_x(t)=\dot{x}(t)+(\partial_2\phi_{\bar\gamma}(t, x(t)))^{-1}\bigl(\partial_1\phi_{{\bar\gamma}}(t,x(t))\bigr)
\to (\partial_2\phi_{\bar\gamma}(t, 0))^{-1}\bigl(\dot{\bar\gamma}(t)\bigr)=\zeta_0(t)
\end{equation}
uniformly on $[0, \tau]$.

We hope to prove that $\|\zeta_x-\zeta_0\|_{C^1}\to 0$ as $\|x\|_{1,2}\to 0$.

Fix $\bar{t}\in [0,\tau]$ and a $C^7$ coordinate chart $(\Theta, W)$ around $\bar\gamma(\bar{t})$ on $M$,
where $W$ is an open neighborhood of $\bar\gamma(\bar{t})$ and $\Theta$ is a $C^7$ diffeomorphism
from $W$ to an open subset in $\mathbb{R}^n$. Then there exists a
closed neighborhood $J$ of $\bar{t}$ in $[0, \tau]$
such that $\bar\gamma(t)\in W$ for any $t\in J$.

Shrinking $J$ we have a positive number $\nu<2\iota$ such that $x([0,\tau])\subset B^n_\nu(0)$ and that
\begin{equation}\label{e:tangentzero+3}
\phi_{\bar\gamma,t}:=\phi_{\bar\gamma}(t,\cdot)\;\hbox{maps $B^n_\nu(0)$ into $W$ for each $t\in J$}.
\end{equation}
Define $\Upsilon:J\times B^n_\nu(0)\to\mathbb{R}^n,\;(t,p)\mapsto \Theta(\phi_{\bar\gamma}(t, p))$.
It is $C^5$, and for each $t\in J$, $\Upsilon(t,\cdot)$ is a
$C^5$ diffeomorphism from $B^n_\nu(0)$ onto an open subset in $\mathbb{R}^n$.
Let $\partial_1\Upsilon$ and $\partial_2\Upsilon$ denote differentials of
 $\Upsilon$ with respect to $t$ and $p$, respectively.
For $t\in J$, since
\begin{eqnarray*}
&&d\Theta(\phi_{\bar\gamma}(t, x(t)))\circ(\partial_2\phi_{\bar\gamma}(t, x(t)))=d(\Theta\circ\phi_{\bar\gamma,t})(x(t))=\partial_2\Upsilon(t,x(t)),\\
&&d\Theta(\phi_{\bar\gamma}(t, x(t)))[\partial_1\phi_{{\bar\gamma}}(t,x(t))]=
\frac{d}{ds}\Theta(\phi_{\bar\gamma}(s,p))\Big|_{(s,p)=(t,x(t))}=
\partial_1\Upsilon(t,x(t)),
\end{eqnarray*}
composing $d\Theta(\phi_{\bar\gamma}(t, x(t)))$ in two sides of (\ref{e:tangentzero+1}) we obtain
\begin{equation}\label{e:tangentzero+4}
\partial_1\Upsilon(t,x(t))=\partial_2\Upsilon(t, x(t))[\zeta_x(t)-\dot{x}(t)],\quad\forall t\in J.
\end{equation}
Differentiating this equality with respect to $t$ gives rise to
\begin{eqnarray*}
 &&\partial_1\partial_1\Upsilon(t,x(t))+ \partial_2\partial_1\Upsilon(t,x(t))[\dot{x}(t)]\\
 &=&\partial_1\partial_2\Upsilon(t, x(t))[\zeta_x(t)-\dot{x}(t)]+
 \partial_2\partial_2\Upsilon(t, x(t))[\dot{x}(t), \zeta_x(t)-\dot{x}(t)]\\
&&+\partial_2\Upsilon(t, x(t))[\dot{\zeta}_x(t)-\ddot{x}(t)]
\end{eqnarray*}
and so
\begin{eqnarray*}
\dot{\zeta}_x(t)&=&\ddot{x}(t)-
\big(\partial_2\Upsilon(t, x(t))\big)^{-1}\partial_1\partial_2\Upsilon(t, x(t))[\zeta_x(t)-\dot{x}(t)]\\
&-&\big(\partial_2\Upsilon(t, x(t))\big)^{-1}\partial_2\partial_2\Upsilon(t, x(t))[\dot{x}(t), \zeta_x(t)-\dot{x}(t)]\\
 &+&\big(\partial_2\Upsilon(t, x(t))\big)^{-1}\partial_1\partial_1\Upsilon(t,x(t))+
 \big(\partial_2\Upsilon(t, x(t))\big)^{-1}\partial_2\partial_1\Upsilon(t,x(t))[\dot{x}(t)],
\end{eqnarray*}
where $\partial_2\Upsilon(t, x(t))$ may be understand as matrixes.
As $\|x\|_{C^2}\to 0$ it follows that
\begin{eqnarray*}
\dot{\zeta}_x(t)&\to&
 \big(\partial_2\Upsilon(t, 0))\big)^{-1}\partial_1\partial_1\Upsilon(t,0)=\dot{\zeta}_0(t)
\end{eqnarray*}
uniformly on $J$, and hence that $\|\zeta_x-\zeta_0\|_{C^1(J)}\to 0$ by (\ref{e:tangentzero+2}).

Note that $[0, \tau]$ can be covered by finitely many neighborhoods of form $J$.
We arrive at $\|\zeta_x-\zeta_0\|_{C^1}\to 0$, in particular, $\|\zeta_x-\zeta_0\|_{1,2}\to 0$.
Then for $\varepsilon>0$  small enough,
 the orthogonal decomposition ${\bf H}=(\R\zeta_0)\oplus{\bf H}^\bot$ of Hilbert spaces
implies   a direct sum decomposition of Banach spaces
 ${\bf H}=(\mathbb{R}\zeta_x)\dot{+} {\bf H}^\bot$.
  From this, (\ref{e:tangentzero}) and
(\ref{e:orbit-criticalLagr}) we may deduce that $d\check{\mathscr{L}}_\lambda(x)=0$
because $T_{x}B_{{\bf H}^\bot}(0,\varepsilon)=T_{0}B_{{\bf H}^\bot}(0,\varepsilon)={\bf H}^\bot$.
\end{proof}

\subsection{Relaxing the hypotheses in  Theorems~\ref{th:bif-ness-orbitLagrMan},~\ref{th:bif-suffict1-orbitLagrMan},
~\ref{th:bif-existence-orbitLagrMan},~\ref{th:bif-suffict-orbitLagrMan}
for Lagrangian systems on $\mathbb{R}^n$}\label{sec:AutoLagr5}

When $M$ is an open subset $U$ of $\mathbb{R}^n$ and $\mathbb{I}_g$ is
an  orthogonal matrix $E$ of order $n$
  which maintains $U$ invariant, Assumption~\ref{ass:Lagr10} in Theorems~\ref{th:bif-ness-orbitLagrMan},~\ref{th:bif-suffict1-orbitLagrMan},~\ref{th:bif-existence-orbitLagrMan},~\ref{th:bif-suffict-orbitLagrMan}
   can be replaced by the following weaker:

  \begin{assumption}\label{ass:BasiAssLagrS}
{\rm For an  orthogonal matrix $E$ of order $n$, and an $E$-invariant
 open subset $U\subset \mathbb{R}^n$
let $L:\Lambda\times U\times \mathbb{R}^n\to\R$ be a continuous function such that
for each $\lambda\in\Lambda$ the function
$L_\lambda(\cdot)=L(\lambda,\cdot)$   is $C^4$ and  partial derivatives
$$
\partial_qL_\lambda(\cdot),\quad\partial_vL_\lambda(\cdot),\quad\partial_{qv}L_\lambda(\cdot),\quad
\partial_{qq}L_\lambda(\cdot),\quad\partial_{vv}L_\lambda(\cdot)
$$
depend continuously on
 $(\lambda, q, v)\in\Lambda\times U\times \mathbb{R}^n$.
 Moreover, for each $(\lambda,  q)\in \Lambda\times U$, ${L}(\lambda, q, v)$ is
convex in $v$,   and satisfies
\begin{eqnarray}\label{e:M-invariant1LagrS}
L(\lambda, Eq, Ev)=L(\lambda, q,v)\quad\forall (\lambda, q,v)\in\Lambda\times U\times\mathbb{R}^{n}.
\end{eqnarray}
A nonconstant $C^2$ map $\bar\gamma: \R\to U$
 satisfies (\ref{e:PPerLagrorbit})  for each $\lambda\in\Lambda$
(so $\bar\gamma$ is $C^4$), and the closure of $\bar\gamma(\mathbb{R})$
has an $E$-invariant compact neighborhood $U_0$ contained in $U$
(thus there exists $\nu_0>0$ such that $Cl(\bar\gamma(\mathbb{R}))+\bar{B}^n_{\nu_0}(0)\subset U_0)$.
 For some real $\rho>\sup_t|\dot{\bar\gamma}(t)|$ and
  each $(\lambda, q)\in \Lambda\times U$, ${L}(\lambda,  q, v)$ is
strictly convex in $v$ in $B^n_{\rho}(0)$.
}
\end{assumption}

In fact, taking an orthogonal matrix $\Xi$ such that $\Xi^{-1}E\Xi$
is equal to the right side of (\ref{e:orth})
and replacing $L_\lambda$ and $\bar\gamma$ by
$$
(\Xi^{-1}U)\times\mathbb{R}^n\to\mathbb{R},\;(x,v)\mapsto L_\lambda(\Xi x,\Xi v)
$$
and $\Xi^{-1}\bar\gamma$, respectively,
 we may assume $E={\rm diag}(S_1,\cdots,S_{\sigma})\in\R^{n\times n}$ as in (\ref{e:standard}).

Take $\iota=\nu_0/3$ and define $\phi_{{\bar\gamma}}:\mathbb{R}\times B^n_{3\iota}(0)\to
U$ by $\phi_{{\bar\gamma}}(t,x)=\bar\gamma(t)+x$, which is clearly $C^4$.
Then $\phi_{{\bar\gamma}}$ gives rise to a $C^\infty$ coordinate chart around ${\bar\gamma}$ on the $C^\infty$ Banach manifold
$\mathcal{X}_{\tau}(U, E)$,
 \begin{eqnarray*}
\Phi_{{\bar\gamma}}:\mathcal{X}^1(B^n_{2\iota}(0),  {E}^T) \to \mathcal{X}^1_{\tau}(U, E)
\end{eqnarray*}
given by $\Phi_{{\bar\gamma}}(\xi)(t)=\bar\gamma(t)+\xi(t)$.
 For any $\xi\in \mathcal{X}^1(B^n_{2\iota}(0),  {E}^T)$ and
$\gamma\in \mathcal{X}^1_{\tau}(U, E)$ there holds
$$
T_{\xi}\mathcal{X}^1(B^n_{2\iota}(0),  {E}^T)=T_{\gamma}\mathcal{X}^1_{\tau}(U, E)={\bf X}:=\mathcal{X}^1(\mathbb{R}^n, E^T).
$$
Clearly, $d\Phi_{{\bar\gamma}}(0)={id}_{\bf X}$.
Since $d\phi_{{\bar\gamma}}(t,x)[(1,v)]=\dot{\bar\gamma}(t)+v$, we define
\begin{eqnarray*}
L^\star: \Lambda\times \mathbb{R}\times B^n_{3\iota}(0)\times\R^{n}\to \R,\;
  (\lambda, t, x, v)\mapsto L\bigr(\lambda, \bar\gamma(t)+ x, \dot{\bar\gamma}(t)+ v\bigl).
\end{eqnarray*}
It is continuous and satisfies (\ref{e:E-invariant1Lagr}) with $E_{\bar\gamma}=E$. Moreover, each
$L^\star_{\lambda}$ is $C^4$. Take $\rho_0>3\iota$ such that  $\rho>\rho_0>\sup_t|\dot{\bar\gamma}(t)|$.
Then for a given subset $\hat\Lambda\subset\Lambda$ which is either compact or sequential compact,
Lemma~\ref{lem:modif} yields a continuous function $\check{L}:\hat\Lambda\times \mathbb{R}\times \bar{B}^n_{2\iota}(0)\times\mathbb{R}^n\to\R$.
Put
  \begin{eqnarray*}
  &&{\bf H}:=\{\xi\in W^{1,2}_{\rm loc}(\mathbb{R}; \mathbb{R}^n)\,|\, {E}(\xi(t))=\xi(t+\tau)\;\forall t\in\mathbb{R}\}\quad\hbox{and}\\
&&\mathcal{U}:=\{\xi\in W^{1,2}_{\rm loc}(\mathbb{R}; B^n_{2\iota}(0))\,|\, {E}(\xi(t))=\xi(t+\tau)\;\forall t\in\mathbb{R}\},\quad
\mathcal{U}^X:=\mathcal{U}\cap{\bf X}.
\end{eqnarray*}
Define $\check{\mathscr{L}}_\lambda:\mathcal{U}\to\R$ as in (\ref{e:check*}),
and ${\check{\mathscr{L}}}_{\lambda}^\bot:\mathcal{U}\cap {\bf H}^\bot\to\R$
as in (\ref{e:orth-functLagr}), where
 $$
 {\bf H}^\bot:=\{x\in {\bf H}\,|\,(\dot{\bar\gamma}, x)_{1,2}=0\}\quad\hbox{and}\quad
{\bf X}^\bot:={\bf X}\cap{\bf H}^\bot.
$$
The other arguments are same,
except that ``$C^6$'' is replaced by ``$C^4$'', $\mathbb{I}_g$ by $E$, etc.
Actually, the proof of Proposition~\ref{prop:solutionLagr1} are much simpler in this situation.

\part{Bifurcations of geodesics}\label{part:II}

Geodesics and their generalizations serve not only as a crucial approach to understanding the geometric and topological properties of manifolds, but also possess profound physical significance. Research in this area directly led to the development of Morse theory (\cite{Mor1, Mor2}) and Ljusternik--Schnirelmann theory (\cite{LjSc29, LjSc34}), which marked the emergence of the field of the calculus of variations in the large. The generalizations of both theories and their applications have profoundly influenced the development of many areas in both mathematics and theoretical physics.


Let us recall two celebrated classical theorems from Riemannian geometry.\\

\noindent\textbf{Gauss lemma}: \emph{For a point $p$ on a Riemannian manifold $(M, g)$ and $u\in T_pM$,
suppose the geodesic $\gamma(t) = \gamma_u(t)$ is defined for $0\le t\le \tau$.
Then the exponential map $\exp_p$ is defined on an open neighborhood of
the segment $\{tu \mid t\in [0, \tau]\}$ in $T_pM$.
Under the canonical identification of $T_pM$ with its tangent space at
any point $tu$ (i.e., $T_pM \cong T_{tu}(T_pM)$), the differential of the
exponential map satisfies:
$$
D\exp_p(tu)[u] = \dot{\gamma}(t)\quad\text{and}\quad
\bigl\langle (D\exp_p(tu)[\xi],\, \dot{\gamma}(t) \bigr\rangle_g
           = \langle u,\, \xi \rangle_g\;\;\text{for any $\xi \in T_pM$}.
$$
Consequently, for $t>0$, $D\exp_p(tu)$ is not regular (i.e., ${\rm rank}D\exp_p(tu)<\dim M$)
if and only if $\gamma(t)$ is {\rm conjugate} to $p$ along $\gamma$ (i.e., there exists
a nonzero Jacobi field $Y$ along $\gamma$ such that $Y(0)=0$ and $Y(t)=0$).
$n(t):=\dim{\rm Ker}D\exp_p(tu)$ is called the {\rm multiplicity} of the {\rm conjugate point} $\gamma(t)$.}\\

 \noindent\textbf{Morse index theorem}(\cite{Mor2}): \emph{
 For two points $p$ and $q$ on a Riemannian manifold $(M, g)$,
      let $\gamma\colon[0, \tau]\to M$ be a geodesic joining $p=\gamma(0)$ and $q=\gamma(\tau)$,
      and assume $q$ is not conjugate to $p$ along $\gamma$.
      Then the Morse index $\operatorname{ind}(\gamma)$ equals the number of conjugate points to $p$ along $\gamma$ in the open interval $(0, \tau)$, counted with multiplicity.
  Consequently, the  geodesic segment $\gamma:[0, \tau]\to M$ can
contain only finitely many points which are conjugate
to $\gamma(0)$ along $\gamma$.}\\

 This fundamental result in Riemannian geometry and the calculus of variations
   establishes a deep connection between the variational stability
   ``of a geodesic and the number of conjugate points'' along it.

 Morse-Littauer \cite{MorLit32} first extended the Gauss lemma to analytic Finsler spaces. Later, Savage \cite{Sav43} proved it for $C^\infty$ Finsler spaces $(M, F)$.
 Subsequently, Warner \cite{Wa65} reestablished it using a new method.
 Therefore,
 if $v\in T_pM\setminus\{0\}$ is a critical point of  the exponential  map $\exp_p^F$ of a $C^\infty$
 Finsler space $(M, F)$,   then there exist
two sequences $(v^1_k), (v^2_k)\subset T_pM\setminus\{v\}$ converging to $v$ such that
$v^1_k\ne v^2_k$ and $\exp^{F}_p(v^1_k)=\exp^{F}_p(v^2_k)$ for each $k\in\N$.
That is, we have always at least two distinct geodesics from $p$ to some point of any neighborhood of $\exp^F_p(v)$
near the geodesic $[0,1]\ni t\mapsto\exp^F_p(tv)$.
This bifurcation result of geodesics
is absolutely not trivial because
 the $C^\infty$ map $\mathbb{R}^2\ni (x,y)\to (x^3,y)\in\mathbb{R}^2$
is a bijection and has singularity at each point of the $y$ axis.

For two-dimensional Riemannian manifolds, Klingenberg's geometrical approach addressed
bifurcation at conjugate points along geodesics (\cite[complement 2.1.13]{Kl95})
and geodesic bifurcations under smooth metric variations (\cite[section 3.4]{Kl95}).

These results have been extended to semi-Riemannian and Lorentzian
settings (cf. \cite{GiGiPi04, GiJa07, JaPi06, PiPoTa04} and the references therein).
For example, if there exists a nondegenerate conjugate instant $t_0\in (0, 1)$
along a geodesic $\gamma:[0, 1]\to M$
in a semi-Riemannian manifold $(M, g)$  with ${\rm sgn}(t_0)\ne 0$,
Piccione, Portaluri and Tausk \cite[Corollaries~5.5 and 5.7]{PiPoTa04} concluded that
 $\gamma(t_0)$ is a bifurcation point along $\gamma$ and that
the exponential map $\exp_{\gamma(0)}$ is not injective on any
neighborhood of $t_0\dot\gamma(0)$.  For a lightlike geodesic $z:[0, 1]\to M$ in
 a Lorentzian manifold $(M, g)$,
Javaloyes and Piccione \cite[Corollary~11]{JaPi06} showed that
  $z(t_0)$ with $t_0\in (0, 1)$ is conjugate to $z(0)$ along $z$ if and only if
the exponential map
$$\exp: \mathcal{A}\cap(\cup_{s\in [0,1]}\{v\in T_{z(s)}M\,|\, g(v,v)\ne 0\})
\to M
$$
 is not locally injective around $t_0\dot{z}(0)$.

Currently, the existing literature reveals limited advances in geodesic bifurcation theory.

In this part, we aim to investigate geodesic bifurcations in Finsler manifolds, with emphasis on their relation to conjugate points.
Using the author's specialized technique developed in \cite{Lu4}, we extend bifurcation results from \cite{Lu12-} and Part A to geodesic bifurcations in Finsler manifolds.
For example, by applying a bifurcation result (Theorem~1.9 in \cite{Lu12-} or \cite{Lu12}) concerning Euler-Lagrange curves of Lagrangian systems, we establish a novel bifurcation theorem for Finsler geodesics (Theorem~\ref{th:MorseLittauer}). Its special case (Theorem~\ref{th:MorseLittauer1}) significantly strengthens the Gauss lemma and its generalizations by Morse-Littauer \cite{MorLit32} and Savage \cite{Sav43}.
Together with the Morse index theorem on Finsler manifolds (see \cite{Lu14, Pe06} for
the most general form currently available), we obtain the following picture:
\begin{center}
\textsf{ For a geodesic $\gamma\colon[0, \tau]\to M$  joining $p=\gamma(0)$ and $q=\gamma(\tau)$
  on a Finsler manifold $(M, F)$, there exist only finitely many conjugate points of $p$ along $\gamma$.
  At each conjugate point, the bifurcation of geodesics follows
  a bifurcation alternative of Rabinowitz's type.}
  \end{center}
 (see Theorem~\ref{th:MorseBifFin} for a general precise form).
 This refines the classical conclusion.
Consider the $n$-sphere $M=\mathbb{S}^n$  with the round metric. Then the geodesics are great circles,
and the cut locus of the south pole is the north pole. Suppose that $p$ is the south pole
and the norm of $v\in T_p\mathbb{S}^n$ is equal to $\pi$ (the length of semi-great circle).
Then $\exp_p(v)$ is the north pole. It is easily seen that only (i) in Theorem~\ref{th:MorseLittauer1} occurs.
Counterexamples confirm the impossibility of cases (ii)-(iii) in Theorem~\ref{th:MorseLittauer1},
establish the sharpness of the above result.

Our results have significant practical implications. For example, Shen \cite{Shen03} solved the classical Zermelo navigation problem by illuminating the connection between time-optimal navigation on Riemannian manifolds and the geodesics of a Randers-type Finsler metric, and by finding the associated geodesics on the configuration space. Russell and Stepney \cite{RuSt14} later introduced the quantum Zermelo navigation problem, showing that Randers geodesics similarly provide its solutions. Caponio et al. \cite{CaJaMa1} employed Morse theory for Finsler geodesics to analyze causal geodesics in stationary spacetimes.
See Section~\ref{sec:Randers} for potential applications of our results to such problems.


Our approach involves refining the techniques in \cite{Lu4} to directly derive
bifurcation results for geodesics on Finsler manifolds from theorems
in \cite{Lu12-} and Part~\ref{part:A}.\\

\noindent\textbf{Organization of the part}.
In Section~\ref{sec:Fin},  we  review some necessary definitions and preliminary
results on Finsler geometry.
Section~\ref{sec:geodesics4} studies bifurcation points along a Finsler geodesic.
In Section~\ref{sec:geodesics1}, we investigate bifurcations of geodesics with two
types of special boundary conditions. Section~\ref{sec:geodesics2} concerns bifurcations of $\mathbb{I}_g$-invariant geodesics.
In Section~\ref{sec:geodesics3}, we treat bifurcations of reversible geodesics.
In Section~\ref{sec:Randers}, we show how our results lead to bifurcation results for future-pointing lightlike geodesics in conformally standard stationary spacetimes, and for time-minimizing paths in the classical Zermelo navigation problem with Randers metrics.

\section{Preliminaries for Finsler geometry}\label{sec:Fin}

Without special statements, let $(M, g)$ be as in ``Basic assumptions and conventions'' in Introduction
and let $\mathbb{I}_g$ be a $C^7$ isometry on $(M, g)$.
Let $P$ and $Q$ be two connected $C^7$ submanifolds in $M$ of dimension
less than $n=\dim M$ and without boundary. For an integer $2\le\ell\le 6$,
 a $C^\ell$ \textsf{ Finsler metric} on $M$ is a
continuous function $F: TM\to\R$  satisfying the following properties:
\begin{description}
\item[(i)]  $F$ is $C^\ell$ and positive in $TM\setminus 0_{TM}$,
where $0_{TM}$ is the zero section of $TM$.

\item[(ii)] $F(x,tv)=tF(x,v)$ for every $t>0$ and any $(x,v)\in TM$.

\item[(iii)] $L:=F^2$ is fiberwise strongly convex, that is, for any $(x,v)\in TM\setminus 0_{TM}$
the symmetric bilinear form (the fiberwise Hessian operator)
\begin{equation}\label{e:fundTensor}
g^F_v: T_xM\times T_xM\to\R,\,(u, w)\mapsto\frac{1}{2}\frac{\partial^2}{\partial s\partial
t}\left[L(x, v+ su+ tw)\right]\Bigl|_{s=t=0}
\end{equation}
is positive definite. ($g^F_v$ is called the \textsf{fundamental tensor} of $F$ at $v$.)
\end{description}
A Finsler metric $F$ is said to be  \textsf{reversible} (or \textsf{absolute homogeneous})
 if $F(x,-v)=F(x,v)$ for
all $(x,v)\in TM$.
We say a differentiable curve $\gamma:[a, b]\to M$
to be \textsf{admissible} (or \textsf{regular}) if $\dot\gamma(t)\in TM\setminus 0_{TM}$ for all $t$.
Such an admissible curve $\gamma=\gamma(t)$ in $(M, F)$ is said to have  \textsf{constant speed}
if $F(\gamma(t), \dot\gamma(t))$ is constant along $\gamma$.
 The length of an admissible piecewise $C^1$  curve $\gamma:[a,b]\to M$ with respect to
 $F$ is defined by $l_F(\gamma)=\int^b_aF(\gamma(t),\dot
\gamma(t))dt$.
 According to \cite[Proposition~5.1.1(a)]{BaChSh},
 an admissible piecewise $C^1$ curve $\gamma$ is
called a \textsf{F-geodesic}
in $(M, F)$ if it minimizes the length between two sufficiently
close points on the curve (hence $C^1$).
The distance between any pair of points $p,q\in M$ is defined by
$$
d_F(p,q)=\inf\{l_F(\gamma)\,|\, \hbox{$\gamma:[a,b]\to M$ is a piecewise $C^1$ curve from $p$ to $q$}\}.
$$

Let $W^{1,2}([0,\tau], M)$ denote  the space of absolutely continuous
curves $\gamma$ from $[0,\tau]$ to $M$ such that
$\int^\tau_0\langle\dot\gamma(t),\dot\gamma(t)\rangle_g dt<\infty$.
By  \cite[Theorem~4.3]{PiTa01},
$W^{1,2}([0,\tau], M)$ is a $C^4$ Riemannian--Hilbert manifold
with Riemannian--Hilbert metric given by (\ref{e:1.1}).
A $C^7$ submanifold ${\bf N}\subset M\times M$ determines
a Riemannian--Hilbert  submanifold of $W^{1,2}([0,\tau], M)$,
$$
\Lambda_{\bf N}(M):=\{\gamma\in W^{1,2}([0,\tau], M)\,|\,(\gamma(0),\gamma(\tau))\in
{\bf N}\}
$$
with tangent space $T_\gamma\Lambda_{\bf N}(M)=W^{1,2}_{\bf N}(\gamma^\ast TM)=\{
\xi\in W^{1,2}(\gamma^\ast TM)\,|\,(\xi(0),\xi(\tau))\in T_{(\gamma(0),\gamma(\tau))}{\bf N}\}$.
We also consider the $C^4$ Banach manifold
$$
{\cal C}_{\tau,\bf N}(M)=\{\gamma\in C^1([0,\tau],M)\,|\, (\gamma(0), \gamma(\tau))\in
{\bf N}\},
$$
which is equal to $C^{1}_{P\times Q}([0,\tau]; M)$ [resp. $C^{1}_{\mathbb{I}_g}([0, \tau]; M)$]
if ${\bf N}$
is the product $P\times Q$  (resp. the graph of the isometry $\mathbb{I}_g$).
It has the following open subset
\begin{equation}\label{e:regularCurve}
{\cal C}_{\tau,\bf N}(M)_{\rm reg}=\{\gamma\in {\cal C}_{\tau,\bf N}(M)\,|\, \hbox{$\gamma$ is admissible}, i.e., \dot\gamma(t)\ne 0\;\forall t\in [0, \tau]\}.
\end{equation}

\begin{claim}[\hbox{\cite{CaJaMa3, KoKrVa, Me}}]\label{cl:regularity}
For a $C^\ell$ Finsler metric $F$ on $M$ with $3\le\ell\le 6$,
 a curve $\gamma\in\Lambda_{\bf N}(M)$ is a  constant  (non-zero)
speed $F$-geodesic satisfying the boundary condition
\begin{equation}\label{e:Fin1-}
g^F_{\dot\gamma(0)}(u,\dot\gamma(0))=
g^F_{\dot\gamma(\tau)}(v,\dot\gamma(\tau))\quad\forall (u,v)\in
T_{(\gamma(0),\gamma(\tau))}{\bf N}
\end{equation}
if and only if it is a (nontrivial) critical point of
the $C^{2-0}$ energy functional of  $F$ given by
\begin{equation}\label{e:Fin1}
{\cal L}:\Lambda_{\bf N}(M)\to\R,\;\gamma\mapsto\int^\tau_0F^2(\gamma(t),
\dot\gamma(t))dt.
\end{equation}
(In this case $\gamma$ must be $C^\ell$.)
\end{claim}

Therefore under Claim~\ref{cl:regularity} a curve $\gamma\in\Lambda_{\bf N}(M)$ is a  constant  (non-zero)
speed $F$-geodesic satisfying the boundary condition (\ref{e:Fin1-})
if and only if it belongs to ${\cal C}_{\tau,\bf N}(M)_{\rm reg}$ and is a critical point of
the following $C^2$ functional
\begin{equation*}\label{e:Fin1+}
{\cal E}_{\bf N}:{\cal C}_{\tau,\bf N}(M)_{\rm reg}\to\R,\;\gamma\mapsto\int^\tau_0F^2(\gamma(t),
\dot\gamma(t))dt.
\end{equation*}
 We may denote  the Morse index and nullity of $\mathcal{E}_{\bf N}$ at a critical point $\gamma\in {\cal C}_{\tau,\bf N}(M)_{\rm reg}$ by
\begin{equation}\label{e:FinMorseIndex}
m^-(\mathcal{E}_{\bf N},\gamma)\quad\hbox{and}\quad m^0(\mathcal{E}_{\bf N},\gamma),
\end{equation}
respectively.  (See the explanations above  Assumption~1.2 in \cite{Lu12-} or \cite{Lu12}
 [resp. (1.16) in \cite{Lu12-} or (1.17) in \cite{Lu12}]
  for ${\bf N}=P\times Q$
[resp. ${\bf N}={\rm Graph}(\mathbb{I}_g)$].) In particular,
for ${\bf N}=P\times Q$ we write ${\cal E}_{\bf N}$ as
\begin{eqnarray}\label{e:Fenergy}
\mathcal{E}_{P,Q}:C^1_{P\times Q}([0,\tau];M)_{\rm reg}\to\R,\;\gamma\mapsto \int^\tau_0[F(\gamma(t), \dot{\gamma}(t))]^2dt,
\end{eqnarray}
whose critical point $\gamma$  corresponds to  a $C^\ell$
constant (non-zero) speed $F$-geodesic with boundary condition
\begin{equation}\label{e:1.4}
g^F_{\dot\gamma(0)}(u,\dot\gamma(0))=0\;\forall u\in
T_{\gamma(0)}P\quad\hbox{and}\quad
g^F_{\dot\gamma(\tau)}(v,\dot\gamma(\tau))=0\;\forall v\in
T_{\gamma(\tau)}Q
\end{equation}
(cf. \cite[Chap.1, \S1]{BuGiHi}, \cite[Proposition~2.1]{CaJaMa3} and
\cite[Prop.~3.1, Cor.3.7]{Jav15}).
(Such geodesics are said to be \textsf{$g_{\dot\gamma}$-orthogonal}
(or \textsf{perpendicular}) \textsf{to $P$ and $Q$}.)
 When $\ell=6$, the geodesic $\gamma$, $m^-(\mathcal{E}_{P,Q},\gamma)$
  and $m^0(\mathcal{E}_{P,Q},\gamma)$ have
direct geometric explanations. See the second half of this section.

\begin{assumption}\label{ass:Fin1}
{\rm $\{F_\lambda\,|\,\lambda\in\Lambda\}$ is a family of $C^\ell$
Finsler metrics on $M$ with $3\le\ell\le 6$ parameterized by a topological space $\Lambda$,
such that $\Lambda\times TM\ni (\lambda, x,v)\to F_\lambda(x,v)\in\R$ is a continuous,
and that all partial derivatives of each $F_\lambda$ of order less than three
depend continuously on $(\lambda, x, v)\in\Lambda\times (TM\setminus 0_{TM})$.}
 \end{assumption}

 \begin{assumption}\label{ass:Fin2}
{\rm  Under Assumption~\ref{ass:Fin1} with an integer $4\le \ell\le 6$,
for each $\lambda\in\Lambda$ let $\gamma_\lambda:[0, \tau]\to M$ be a
 constant (non-zero) speed $F_\lambda$-geodesic satisfying the boundary condition
\begin{equation}\label{e:FinBoundNlambda}
g^{F_\lambda}_{\dot\gamma_\lambda(0)}(u,\dot\gamma_\lambda(0))=
g^{F_\lambda}_{\dot\gamma_\lambda(\tau)}(v,\dot\gamma_\lambda(\tau))\quad\forall (u,v)\in
T_{(\gamma_\lambda(0),\gamma_\lambda(\tau))}{\bf N},
\end{equation}
where ${\bf N}\subset M\times M$ is a $C^7$ submanifold.
(Therefore $\gamma_\lambda$ is $C^\ell$ by Claim~\ref{cl:regularity}.)
It is also required  that the maps $\Lambda\times [0,\tau]\ni(\lambda,t)\to \gamma_\lambda(t)\in M$
and
$\Lambda\times [0,\tau]\ni(\lambda, t)\mapsto \dot{\gamma}_\lambda(t)\in TM$ are continuous.}
 \end{assumption}

Let $m^-(\mathcal{E}_{\lambda,\bf N},\gamma_\lambda)$ and
 $m^0(\mathcal{E}_{\lambda,\bf N},\gamma_\lambda)$
denote  the Morse index and nullity  at $\gamma_\lambda$ of  the $C^{2}$ functional
\begin{equation}\label{e:FinEnergy}
{\cal E}_{\lambda, \bf N}: {\cal C}_{\tau,\bf N}(M)_{\rm reg}\to\mathbb{R},\;  \gamma \mapsto\int^\tau_0[F_\lambda(\gamma(t), \dot{\gamma}(t))]^2dt.
\end{equation}

For conveniences,  a constant (non-zero) speed $F_\lambda$-geodesic
 satisfying the boundary condition
(\ref{e:FinBoundNlambda}) is called a \textsf{constant (non-zero)
speed $(F_\lambda, \bf N)$-geodesic}.

\begin{definition}\label{def:BifurFin}
{\rm
Under  Assumptions~\ref{ass:Fin1},~\ref{ass:Fin2},
constant (non-zero) speed $(F_\lambda, \bf N)$-geodesics with a parameter $\lambda\in\Lambda$ are said \textsf{to bifurcate at $\mu\in\Lambda$ along sequences} (with respect to the branch $\{\gamma_\lambda\,|\,\lambda\in\Lambda\}$)
  if  there exists an infinite sequence $\{(\lambda_k, \gamma^k)\}^\infty_{k=1}$
  in $\Lambda\times C^1([0,\tau], M)\setminus\{(\mu,\gamma_\mu)\}$
  converging to $(\mu,\gamma_\mu)$, such that each $\gamma^k\ne\gamma_{\lambda_k}$
  is a constant (non-zero) speed $(F_{\lambda_k}, \bf N)$-geodesic, $k=1,2,\cdots$.
  (Actually it is not hard to prove that
   $\gamma^k\to\gamma_\mu$ in $C^2([0,\tau], M)$.)}
  \end{definition}

Here are the problems we study and answers:
 \begin{enumerate}
\item[$\bullet$] For a  constant (nonzero) speed $F$-geodesic $\gamma$ which is perpendicular to $P$ at $\gamma(0)$,
we shall provide where $\gamma$ bifurcates  and depict a rough bifurcation diagram near it.

\item[$\bullet$] Under Assumptions~\ref{ass:Fin1},~\ref{ass:Fin2}, using  the Morse index $m^-(\mathcal{E}_{\lambda,\bf N},\gamma_\lambda)$,
the nullity $m^0(\mathcal{E}_{\lambda,\bf N},\gamma_\lambda)$ and critical groups $C_\ast(\mathcal{E}_{\lambda,\bf N},\gamma_\lambda; {\bf K})$,
we give the conditions under which constant (non-zero)
speed $(F_\lambda, \bf N)$-geodesics with a parameter $\lambda\in\Lambda$
bifurcate at some $\mu\in\Lambda$ along sequences with respect to the branch $\{\gamma_\lambda\,|\,\lambda\in\Lambda\}$,
characterize the location of such a parameter $\mu$, and depict the bifurcation diagram near $\mu$.
 \end{enumerate}

Our ideas are suitably modifying $F_\lambda$ and converting the above questions into those studied in Part~\ref{sec:LgrResults}.
Firstly, suitably modifying the proof of \cite[Proposition~2.2]{Lu4} we have:

\begin{proposition}\label{prop:Fin2.2}
Under Assumption~\ref{ass:Fin1} let
$L_\lambda:=(F_\lambda)^2$.
Suppose that
\begin{eqnarray*}
\alpha_g:=\inf_{\lambda\in\Lambda}\inf_{(x,v)\in TM,\, |v|_x=1}\inf_{u\ne
0}\frac{g^{F_\lambda}_{v}(u,u)}{g_x(u,u)}\quad\hbox{and}\quad
\beta_g:=\sup_{\lambda\in\Lambda}\sup_{(x,v)\in TM,\, |v|_x=1}\sup_{u\ne
0}\frac{g^{F_\lambda}_v(u,u)}{g_x(u,u)}
\end{eqnarray*}
are positive numbers, and that for some constant $C_1>0$,
\begin{equation}\label{e:Fin1.1}
|v|^2_x\le L_\lambda(x, v)\le C_1|v|^2_x\quad\forall (\lambda, x,v)\in \Lambda\times TM.
\end{equation}
Hereafter $|v|_x=\sqrt{g_x(v,v)}$. For each $\lambda\in\Lambda$ let us define $C^\ell$ functions $L^\star_\lambda:TM\to\R$ by
\begin{eqnarray}\label{e:Fin1.2}
 L^\star_\lambda(x,v)=\psi_{\varepsilon,\delta}(L_\lambda(x,v))+
\phi_{\mu,b}(|v|^2_x)-b
\end{eqnarray}
 and by $L^\ast_\lambda(x,v)=(L^\star_\lambda(x,v)-\varrho_0)/\kappa$, where
 $\psi_{\varepsilon,\delta}$, $\phi_{\mu,b}$ and $\kappa, \varrho, \varrho_0, b$
 are as in Lemma~1.2 in \cite{Lu12-} or \cite{Lu12}.
Then for a given $c>0$  we can choose $\kappa>0$ so large that
these $L^\ast_\lambda$ satisfy the following:
 \begin{enumerate}
\item[\rm (i)]
 $L^\ast_\lambda(x,v)= L_\lambda(x,v)\quad\hbox{if}\;L_\lambda(x,v)\ge\frac{2c}{3C_1}$.

\item[\rm (ii)] $L^\ast_\lambda$ attains the minimum, and $L^\ast_\lambda(x,v)=\min L^\ast_\lambda\Longleftrightarrow
v=0$.

\item[\rm (iii)]  $L^\ast_\lambda(x,v)\le L_\lambda(x,v)$ for all
$(x,v)\in TM$.

 \item[\rm (iv)] $\partial_{vv}L^\ast_\lambda(x,v)[u,u]\ge \min\{\frac{2\mu}{\kappa},
\frac{1}{2}\alpha_g\}|u|_x^2$.

\item[\rm (v)]  If $F_\lambda$ is reversible, i.e. $F_\lambda(x,-v)=F_\lambda(x,v)\;\forall
(x,v)\in TM$, so is $L^\ast_\lambda$.

\item[\rm (vi)]  If $F_\lambda$ is $\mathbb{I}_g$-invariant for a $g$-isometry $\mathbb{I}_g:M\to M$,
(i.e., it satisfies $F_\lambda(\mathbb{I}_{g}(x),\mathbb{I}_{g\ast}(u))=F_\lambda(x,u)$ for all $(x,u)\in TM$),
 so is $L^\ast_\lambda$.
 \end{enumerate}
Moreover, $\Lambda\times TM\ni (\lambda, x,v)\to L_\lambda^\ast(x,v)\in\R$ is continuous
and all partial derivatives of each $L^\ast_\lambda$ of order less than three
depend continuously on $(\lambda, x, v)\in\Lambda\times TM$.
\end{proposition}
\begin{proof}[\bf Proof]
By the assumptions, for any $(x,v)\in
TM\setminus\{0\}$ and $(x,u)\in TM$ we have
\begin{equation}\label{e:Fin1.3}
\alpha_g |u|_x^2\le g^{F_\lambda}_v(u,u)\le\beta_g |u|_x^2,\quad\forall\lambda\in\Lambda.
\end{equation}
 Suppose that (2.4) in \cite{Lu12-} or \cite{Lu12}
 is satisfied and
that $\kappa\ge\mu$.
Since $\phi''_{\mu,b}\le 0$,  $\phi''_{\mu,b}(|v|_x^2)=0$ for
$|v|_x^2\ge\frac{2c}{3C_1}$, and $\phi''_{\mu,b}(|v|_x^2)$
is bounded for $|v|_x^2\in [\frac{\delta}{3C_1}, \frac{2c}{3C_1}]$,
we may choose $\kappa>0$ so large that
$$
2\kappa\alpha_g + \frac{8c}{3C_1}\phi''_{\mu,b}(|v|_x^2)\ge
\frac{1}{2}\kappa\alpha_g.
$$
By the proof of \cite[Proposition~2.2]{Lu4}, $L^\star_\lambda$ satisfies \cite[Proposition~2.2]{Lu4}
and therefore $L^\ast_\lambda$ meets conditions (i)-(iv) in
Proposition~\ref{prop:Fin2.2}. Clearly, (\ref{e:Fin1.2}) implies (v)-(vi).

Since $\Lambda\times TM\ni (\lambda, x,v)\to F_\lambda(x,v)\in\R$ is  continuous,
by (\ref{e:Fin1.2})
we see that $\Lambda\times TM\ni (\lambda, x,v)\to L_\lambda^\ast(x,v)\in\R$ is continuous.
Note that $\{(\lambda, x,v)\in\Lambda\times TM\,|\, L_\lambda(x,v)<\varepsilon\}$
is an open neighborhood of $\Lambda\times 0_{TM}$ in $\Lambda\times TM$ and that
$\psi_{\varepsilon,\delta}(L_\lambda(x,v))=0$
for all $(\lambda, x,v)$ in this neighborhood.
It follows from this and (\ref{e:Fin1.2}) that
all partial derivatives of each $L^\ast_\lambda$ of order less than three
depend continuously on $(\lambda, x, v)\in\Lambda\times TM$
because all partial derivatives of each $F_\lambda$ of order less than three
depend continuously on $(\lambda, x, v)\in\Lambda\times (TM\setminus 0_{TM})$.
(Actually, $\Lambda\times TM\ni (\lambda, x,v)\to L_\lambda^\ast(x,v)\in\R$ is $C^\ell$ in
$\{(\lambda, x,v)\in\Lambda\times TM\,|\, L_\lambda(x,v)<\varepsilon\}$.)
\end{proof}

(\textsf{{Note}}: If $M$ and $\Lambda$ are compact, for any Riemannian metric $g$ on $M$ both
$\alpha_g$ and $\beta_g$ are positive numbers, and
(\ref{e:Fin1.1}) always holds if $g$ is replaced
by a small scalar multiple of $g$.) 

 Under Assumption~\ref{ass:Fin2}, let $\hat\Lambda\subset \Lambda$ be
 either compact or sequential compact.  Since
 the map $\Lambda\times [0,\tau]\ni(\lambda,t)\to \gamma_\lambda(t)\in M$ is continuous,
 the image of $\hat{\Lambda}\times [0,\tau]$ under it is a compact subset of $M$ and therefore
 there exists an open subset $\hat{M}$ of $M$ with compact closure such that
 $\gamma_\lambda([0,\tau])\subset\hat{M}$ for all $\lambda\in\hat{\Lambda}$.
 Then the conditions in Proposition~\ref{prop:Fin2.2} can be satisfied.
By  Assumption~\ref{ass:Fin2} $\hat{\Lambda}\times [0,\tau]\ni (\lambda,t)\mapsto
F_\lambda(\gamma_\lambda(t), \dot{\gamma}_\lambda(t))$
is continuous and positive. Therefore
we have $c>0$ such that
\begin{equation}\label{e:Fin1.9}
[F_\lambda(\gamma_\lambda(t), \dot{\gamma}_\lambda(t))]^2>\frac{2c}{C_1},\quad
\forall (\lambda,t)\in\hat{\Lambda}\times [0,\tau],
 \end{equation}
 where $C_1>0$ is as in Proposition~\ref{prop:Fin2.2}.
Let $L^\ast_\lambda:T\hat{M}\to\R$,  $\lambda\in\hat{\Lambda}$,
 be given by Proposition~\ref{prop:Fin2.2} with $(M, \Lambda)=(\hat{M},\hat{\Lambda})$.
Then the $C^{2}$ functional
\begin{equation}\label{e:FinEnergy*}
\mathcal{E}^\ast_{\lambda,\bf N}: {\cal C}_{\tau, \bf N}(\hat{M})\to\mathbb{R},\;  \gamma \mapsto\int^\tau_0L^\ast_\lambda(\gamma(t), \dot{\gamma}(t))dt
\end{equation}
and the functional $\mathcal{E}_{\lambda,\bf N}$ in (\ref{e:FinEnergy})
coincide in the following open subset of ${\cal C}_{\tau,\bf N}(\hat{M})_{\rm reg}$,
\begin{equation}\label{e:regularSet}
{\cal C}_{\tau,\bf N}(\hat{M}, \{F_\lambda\,|\,\lambda\in\hat{\Lambda}\}, {c/C_1}):=\bigg\{\alpha\in {\cal C}_{\tau,\bf N}(\hat{M})\,\bigg|\,
\min_{(\lambda,t)\in\hat{\Lambda}\times [0,\tau]}[F_\lambda(\alpha(t), \dot{\alpha}(t))]^2>2c/C_1\bigg\}.
\end{equation}
Since $\{\gamma_\lambda\,|\,\lambda\in\hat{\Lambda}\}\subset
{\cal C}_{\tau,\bf N}(\hat{M}, \{F_\lambda\,|\,\lambda\in\hat{\Lambda}\}, {c/C_1})$ by (\ref{e:Fin1.9}),
they are critical points of
$\mathcal{E}^\ast_{\lambda,\bf N}$  and
\begin{eqnarray}\label{e:IndexSame}
&&m^-(\mathcal{E}_{\lambda,\bf N},\gamma_\lambda)=m^-(\mathcal{E}^\ast_{\lambda,\bf N}\gamma_\lambda)
\quad\hbox{and}\quad m^0(\mathcal{E}_{\lambda,\bf N},\gamma_\lambda)=m^0(\mathcal{E}^\ast_{\lambda,\bf N},\gamma_\lambda),\\
&&C_m(\mathcal{E}_{\lambda,\bf N},\gamma_\lambda; {\bf K})=C_m(\mathcal{E}^\ast_{\lambda,\bf N}\gamma_\lambda;{\bf K})\quad\forall m\in\mathbb{Z}
\label{e:IndexSame+}
\end{eqnarray}
for any Abel group ${\bf K}$.

\begin{claim}\label{cl:Fin1}
Under Assumption~\ref{ass:Fin2}, if $\gamma:[0, \tau]\to M$ is a   constant (non-zero)
speed $F_\lambda$-geodesic (hence $C^\ell$)  with boundary condition (\ref{e:FinBoundNlambda})
and is close to $\gamma_\lambda$
in $C^1([0,\tau]; M)$ then it is a critical point of
\begin{equation}\label{e:FinEnergy+}
{\cal C}_{\tau,\bf N}(\hat{M})\ni\gamma \mapsto\mathcal{E}^\ast_{\lambda,\bf N}(\gamma)=\int^\tau_0L^\ast_\lambda(\gamma(t), \dot{\gamma}(t))dt.
\end{equation}
Conversely, if $\gamma\in{\cal C}_{\tau, \bf N}(\hat{M})$ near $\gamma_\lambda$ is a critical point of
$\mathcal{E}^\ast_{\lambda,\bf N}$ then it is a $C^\ell$   constant (non-zero)
speed $F_\lambda$-geodesic  with boundary condition (\ref{e:FinBoundNlambda})
and is near $\gamma_\lambda$
in $C^2([0,\tau]; M)$.
\end{claim}
\begin{proof}[\bf Proof]
If $d\mathcal{E}_{\lambda,\bf N}(\gamma)=0$ and $\gamma$ is close to $\gamma_\lambda$ in $C^1$-topology
then $\gamma\in{\cal C}_{\tau,\bf N}(\hat{M}, \{F_\lambda\,|\,\lambda\in\hat{\Lambda}\}, {c/C_1})$ and therefore $d\mathcal{E}^\ast_{\lambda,\bf N}(\gamma)=0$.
Conversely, we only need to prove that $\gamma$ is also near $\gamma_\lambda$
in $C^2([0,\tau]; M)$. This may follow from Lemma~2.6(ii) in \cite{Lu12-} or \cite{Lu12}
by localization arguments.
\end{proof}

\noindent{\bf Geometric characteristics of $F$-geodesics and their Morse indexes and nullities.}
Without special statements, from now on we always assume $\ell=6$, i.e.,
 $F$ is a $C^6$-Finsler metric on $M$.
In this case the Christoffel symbols of the Chern connection  $\nabla$
 (on the pulled-back tangent bundle $\pi^\ast TM$) with respect to
a coordinate chart $(\Omega, x^i)$ on $M$ are $C^3$ functions
$\Gamma^i_{jm}: T\Omega\setminus 0_{T\Omega}\to\R$ such that
\begin{equation}\label{e:inducedConn-}
\nabla_{\partial_{x^i}}\partial_{x^j}(v)=\sum_m\Gamma^i_{jm}(v)\partial_{x^m},\quad i,j\in\{1,\cdots,n\}
\end{equation}
(cf. \cite{Lu14}); and there exist $C^2$ functions
$R^i_{jkl}: T\Omega\setminus 0_{T\Omega}\to\R$, $1\le i,j,k,l\le n$,
such that the trilinear map $R_v$
from $T_{\pi(v)}M\times T_{\pi(v)}M\times T_{\pi(v)}M$ to $T_{\pi(v)}M$ given by
\begin{equation}\label{e:ChernC}
R_v(\xi,\eta)\zeta=\sum_{i,j,k,l}\xi^k \eta^l \zeta^j R_{j\,\,kl}^{\,\,\,i}(v) \partial_{x^i}|_{\pi(v)}
\end{equation}
defines  the \textsf{Chern curvature tensor} (or
\textsf{$hh$-curvature tensor}
\cite[(3.3.2) \& Exercise 3.9.6]{BaChSh})  $R_V$
on $\Omega\subset M$ (cf. \cite{Lu14}).

For a curve $c\in W^{1,2}([a,b], M)$ and $r\in\{0,1\}$ let
$W^{r,2}(c^\ast TM)$
denote the space of all   $W^{r,2}$ vector fields along $c$.
Then  $\dot{c}\in L^2(c^\ast TM):=W^{0,2}(c^\ast TM)$. Let
$(x^i, y^i)$ be the canonical coordinates around $\dot{c}(t)\in TM$.
Write
$\dot c(t)=\dot{c}^i(t) \partial_{x^i}|_{c(t)}$ and
$\zeta(t)=\zeta^i(t)\partial_{x^i}|_{c(t)}$
for $\zeta\in W^{1,2}(c^\ast TM)$.
Call $\xi\in C^0(c^\ast TM)$ \textsf{admissible}
if $\xi(t)\in TM\setminus 0_{TM}$ for all $t\in [a,b]$.
The Chern connection induces a \textsf{covariant derivative} of $\zeta$ along
 $c$ (with this admissible $\xi$ as reference vector)  is defined by
\begin{equation}\label{e:covariant-derivative}
D^\xi_{\dot{c}}\zeta(t):= \sum_m\bigl(\dot{\zeta}^m(t)
+ \sum_{i,j} \zeta^i(t)\dot{c}^j(t)\Gamma_{ij}^m(c(t), \xi(t))\bigr)\partial_{x^m}|_{c(t)}.
\end{equation}
Clearly, $D^\xi_{\dot{c}}\zeta$ belongs to  $L^{2}(c^\ast TM)$, and sits in $C^{\min\{1,r\}}(c^\ast TM)$ provided that
 $c$ is of class $C^{r+1}$, $\zeta\in C^{r+1}(c^\ast TM)$ and $\xi\in C^r(c^\ast TM)$ for some $0\le r\le 6$;  $D^\xi_{\dot{c}}\zeta(t)$ depends only on $\xi(t)$, $\dot{c}(t)$ and behavior of $\zeta$ near $t$.
 If $c$, $\xi$ and $\zeta$ are $C^3$, $C^1$ and $C^2$, respectively,
then $D^\xi_{\dot{c}}\zeta$ is $C^1$ and $D^\xi_{\dot{c}}D^\xi_{\dot{c}}\zeta$ is well-defined and is $C^0$.
It may be proved that \textsf{a $C^2$ admissible curve $\gamma$ in $(M, F)$
 is a $F$-geodesic of constant speed if and only if
 $D^{\dot \gamma}_{\dot{\gamma}}\dot{\gamma}(t)\equiv 0$.} In this case $\gamma$ must be $C^6$.

For the above $C^7$ submanifolds $P$ and $Q$, we define the
 \textsf{normal bundle} of $P$ in $(M, F)$ by
$$
 TP^\bot:=\{v\in TM\setminus 0_{TM}~|~\pi(v)\in P,\;g^F_v(v,w)=0\;\forall w\in T_{\pi(v)}P\}
$$
(though it is not a vector bundle over $P$).
In fact, it is only an $n$-dimensional $C^6$ submanifold
of $TM$ and the restriction $\pi:TP^\bot\to P$ is a submersion (\cite[Lemma~3.3]{Jav15}).
 For $v\in TP^\bot$ with $\pi(v)=p$,  there exists a splitting $T_pM=T_pP\oplus (T_pP)^\perp_v$,
where $(T_pP)^\perp_v$ is the subspace of $T_pM$ consisting of $g_v$-orthogonal vectors to $T_pP$.
Notice that $v\in (T_pP)^\perp_v$ and that each $u\in T_pM$ has a decomposition ${\rm tan}^P_v(u)+{\rm nor}^P_v(u)$,
where ${\rm tan}^P_v(u)\in T_pP$ and ${\rm nor}^P_v(u)\in (T_pP)^\perp_v$.
Let  $\tilde{S}^P_{v}:T_pP\to T_pP$
be the normal second fundamental form (or shape operator) of $P$ in the direction $v$.
Then (\ref{e:1.4}) implies  $\dot\gamma(0)\in TP^\bot$ and  $\dot\gamma(\tau)\in TQ^\bot$.
In terms of these  the Hessian of $\mathcal{E}_{P,Q}$ at  $\gamma$ is given by
\begin{eqnarray}\label{e:secondDiff}
D^2\mathcal{E}_{P,Q}(\gamma)[V, W]&=&\int_0^\tau \left(g_{\dot\gamma}(R_{\dot\gamma}(\dot\gamma,V)\dot\gamma, W)+g_{\dot\gamma}(D_{\dot\gamma}^{\dot\gamma}V,D_{\dot\gamma}^{\dot\gamma}W)\right)dt
\nonumber\\
&&+g_{\dot\gamma(0)}\left(\tilde{S}^P_{\dot\gamma(0)}(V(0)),W(0)\right)
-g_{\dot\gamma(\tau)}\left(\tilde{S}^Q_{\dot\gamma(\tau)}(V(\tau)),W(\tau)\right)
 \end{eqnarray}
for $V,W\in C^1_{P\times Q}(\gamma^\ast M)=T_\gamma C^1([0,\tau];M, P, Q)$.
(Here we use the equality $R^{\gamma}(\dot\gamma,V)\dot\gamma=R_{\dot\gamma}(\dot\gamma,V)\dot\gamma$
 in \cite[page 66]{Jav15}.)
 The right side of (\ref{e:secondDiff})
 can be extended into  a continuous  symmetric bilinear form
  ${\bf I}^\gamma_{P,Q}$ on $W^{1,2}_{P\times Q}(\gamma^\ast TM)$,
 called as the  \textsf{$(P,Q)$-index form of $\gamma$}.
Since all $R^i_{jkl}$ are $C^{2}$ it can be proved that
 $V\in W^{1,2}_{P\times Q}(\gamma^\ast TM)$ belongs
  to ${\rm Ker}({\bf I}^\gamma_{P,Q})$ if and only if
it is $C^4$ and satisfies
\begin{eqnarray}\label{e:kernel}
\left.\begin{array}{cc}
D_{\dot\gamma}^{\dot\gamma}D_{\dot\gamma}^{\dot\gamma}V-R_{\dot\gamma}(\dot\gamma,V)\dot\gamma=0,\\
{\rm tan}^P_{\dot\gamma(0)}\left((D_{\dot\gamma}^{\dot\gamma}V)(0)\right)=\tilde{S}^P_{\dot\gamma(0)}(V(0)),\quad
{\rm tan}^Q_{\dot\gamma}\left((D_{\dot\gamma}^{\dot\gamma}V)(\tau)\right)=\tilde{S}^Q_{\dot\gamma(\tau)}(V(\tau)).
\end{array}\right\}
 \end{eqnarray}

Let $\gamma:[0,\tau]\to M$ be a $F$-geodesic of (nonzero) constant speed. (It is $C^6$.)
A $C^2$ vector field $J$ along $\gamma$ is said to be a  \textsf{Jacobi field}
if it satisfies the so-called  \textsf{Jacobi equation}
\begin{equation}\label{e:JacobiEq}
D_{\dot{\gamma}}^{\dot{\gamma}}D_{\dot{\gamma}}^{\dot{\gamma}}J-R_{\dot\gamma}(\dot\gamma,J)\dot\gamma=0.
\end{equation}
(Jacobi fields along $\gamma$ must be $C^4$ because each $R^i_{jkl}$ is $C^{2}$.)
The set $\mathscr{J}_\gamma$ of all Jacobi fields along $\gamma$  is a $2n$-dimensional vector space.
For $0\le t_1<t_2\le\tau$ if there exists a nonzero Jacobi field $J$ along $\gamma|_{[t_1,t_2]}$ such that $J$ vanishes at $\gamma(t_1)$ and $\gamma(t_2)$, then $\gamma(t_1)$ and $\gamma(t_2)$ are said to be mutually \textsf{conjugate} along $\gamma|_{[t_1,t_2]}$.
Suppose that the  geodesic $\gamma$ orthogonally starts at $P$. That is,
$\gamma(0)\in P$ and $\dot\gamma(0)$ is $g^F_{\dot\gamma(0)}$-orthogonal to $P$.
 A Jacobi field $J$ along $\gamma$ is called a  \textsf{$P$-Jacobi} if
\begin{equation}\label{e:P-JacobiField}
J(0)\in T_{\gamma(0)}P\quad\hbox{and}\quad
{\rm tan}^P_{\dot\gamma(0)}\left((D_{\dot{\gamma}}^{\dot{\gamma}}J)(0)\right)=
\tilde{S}^P_{\dot\gamma(0)}(J(0)).
\end{equation}
An instant $t_0\in (0,\tau]$ is  called \textsf{$P$-focal}
if there exists a non-null $P$-Jacobi field $J$ such that $J(t_0)=0$;
and $\gamma(t_0)$ is said to be a \textsf{$P$-focal point} along $\gamma$.
The dimension of the space $\mathscr{J}^P_\gamma$ of all
$P$-Jacobi fields along $\gamma$  is equal to $n=\dim M$.
The dimension $\mu^P_\gamma(t_0)$ of
$$
\mathscr{J}^P_\gamma(t_0):=\left\{J\in \mathscr{J}^P_\gamma\,\big|\, J(t_0)=0\right\}
$$
is called the (geometrical)  \textsf{multiplicity} of $\gamma(t_0)$.
For convenience we understand $\mu^{P}_\gamma(t_{0})=0$ if $\gamma(t_0)$
is not a $P$-focal point along $\gamma$.
Then the claim above (\ref{e:kernel})  implies that
\begin{equation}\label{e:P-JacobiField1-}
{\rm Ker}({\bf I}^{\gamma_t}_{P,\gamma(t)})=
\mathscr{J}^P_{\gamma_t}(t)\quad\hbox{with $\gamma_t=\gamma|_{[0,t]}$},\quad\forall t\in (0, \tau].
\end{equation}
In particular, ${\rm Ker}({\bf I}^\gamma_{P,q})=
\mathscr{J}^P_\gamma(\tau)$ with $q=\gamma(\tau)$.
If $\gamma$ is $g^F_{\dot\gamma}$-orthogonal  to $Q$ at $\gamma(\tau)$, elements in $\mathscr{J}^{P,Q}_\gamma:={\rm Ker}({\bf I}^\gamma_{P,Q})$
are called
\textsf{$(P,Q)$-Jacobi fields} along $\gamma$.
Then with $q=\gamma(\tau)$ we have
\begin{equation}\label{e:MS0}
m^0(\mathcal{E}_{P,q},\gamma)=\dim \mathscr{J}^P_\gamma(\tau)\quad\hbox{and}\quad
m^0(\mathcal{E}_{P,q},\gamma)=\dim \mathscr{J}^{P,q}_\gamma.
\end{equation}

For a  constant (nonzero) speed $F$-geodesic $\gamma:[0,\tau]\to M$  orthogonally starting at $P$,
(\ref{e:P-JacobiField1-}) shows that an instant $t_0\in (0,\tau]$ is $P$-focal if and only if
it is a $P$-focal point along the Euler-Lagrange curve $\gamma$ of $L=F^2$
and their multiplicities are same,
i.e., $\nu^P_{\gamma}(t_0)=\mu^P_{\gamma}(t_0)$.
Therefore  Theorem~3.14 in \cite{Lu12-} or \cite{Lu12}
with $q=\gamma(\tau)$ gives:
  \begin{corollary}\label{cor:FinMorseIndex}
 Under the above assumptions it holds that
 \begin{equation}\label{e:MS1}
{\rm Index}({\bf I}^\gamma_{P,q})=m^-(\mathcal{E}_{P,q},\gamma)=
\sum\limits_{t_{0}\in(0,\tau)}\nu^{P}_\gamma(t_{0})=
\sum\limits_{t_{0}\in(0,\tau)}\mu^{P}_\gamma(t_{0}).
\end{equation}
\end{corollary}
Moreover, if $\gamma$ is also perpendicular to  $Q$ at $q=\gamma(\tau)$,
and $\{X(\tau)\,|\,X\in\mathscr{J}^P_\gamma\}\supseteq T_{\dot\gamma(\tau)}Q$
(the latter may be satisfied if $\gamma(\tau)$ is not a $P$-focal point),
 \cite[Theorem~1.1(iii)]{Lu14} gives
\begin{eqnarray}\label{e:MS2}
{\rm Index}({\bf I}^\gamma_{P,Q})={\rm Index}({\bf I}^\gamma_{P,q})+ {\rm Index}(\mathcal{A}_\gamma)
\end{eqnarray}
with $q=\gamma(\tau)$,
where $\mathcal{A}_\gamma$ is the bilinear symmetric form on $\mathscr{J}^{P}_\gamma$ defined by
$$
\mathcal{A}_\gamma(J_1,J_2)=g_{\dot\gamma(\tau)}\Big(D_{\dot\gamma}^{\dot\gamma}J_1(\tau)+
\tilde{S}^Q_{\dot\gamma(\tau)}(J_1(\tau)), J_2(\tau)\Big).
$$
\begin{remark}\label{rm:FinMorseIndex}
{\rm When $M, F, P$ and $Q$ are smooth Ioan Radu Peter \cite{Pe06} proved
(\ref{e:MS1}) and (\ref{e:MS2}) if the Morse index form,
$P$-Jacobi field and the  shape operator
are introduced by the Cartan connection.
Recently, in \cite{Lu14} the author proved the Morse index theorem in the case of two variable
 endpoints in conic Finsler manifolds by
 employing the Chern connection to introduce the Morse index form,
 $P$-Jacobi field and the  shape operator. (\ref{e:MS1}) and (\ref{e:MS2})
 are included in \cite[Theorem~1.1(iii)]{Lu14}.}
\end{remark}

Recall that the \textsf{exponential map} of a $C^6$ Finsler matric $F$ on $M$
 is $\exp^F:\mathcal{D}\subset TM\to M$, where
$\mathcal{D}$ is the set of vectors $v$ in $TM$
such that the unique geodesic $\gamma_v$ satisfying
$\gamma_v(0)=\pi(v)$ and $\dot\gamma_v(0)=v$ is defined at least in $[0, b)\supset [0,1]$,
and $\exp^F(v)=\gamma_v(1)$.
$\mathcal{D}$ is a starlike open neighborhood of the zero section $0_{TM}$ of $TM$,
$\exp^F$ is $C^1$, $C^{3}$ in $TM\setminus 0_{TM}$,
 and $D(\exp^F_p)(0_p):T_pM\to T_pM$ is the identity map at the origin $0_p\in T_pM$ for any $p\in M$,
 where $\exp^F_p$ is the restriction of $\exp^F$ to $\mathcal{D}_p:=\mathcal{D}\cap T_pM$.
 For $v\in \mathcal{D}_p$ and $w\in T_pM$, by
 \cite[Proposition~3.15]{Jav15} or \cite[Lemma~11.2.2]{Shen})
 we have $D\exp_p^F(v)[w]=J(1)$, where $J$ is
 the unique Jacobi field on $\gamma_v$ such that $J(0)=0$ and $J'(0)=w$.
It follows that $D\exp_p(v):T_v(T_pM)\equiv T_pM\to T_pM$
is singular if and only if $\gamma_v(1)$ is a conjugate point
 of $p$ along $\gamma_v$ (cf. \cite[Proposition~7.1.1]{BaChSh}).
Moreover, the multiplicity (or order) of the conjugate point $\gamma_v(1)$
 is $\dim {\rm Ker}\left(D\textmd{exp}^{F}(v)\right)$.

More generally, let $P\subset M$ and $TP^\bot$ be as above, and let $\exp^{FN}$ be the restriction
of $\exp^F$ to $\mathcal{D}\cap (TP^\bot\cup 0_{TM}|_P)$, where
$0_{TM}|_P$ is the restriction of the zero section $0_{TM}$ of $TM$ to $P$.
 We say $\exp^{FN}$ to be the \textsf{normal exponential map}.
Note that $v\in \mathcal{D}\cap(TP^\bot\cup 0_{TM}|_P)$ if and only if $tv\in \mathcal{D}\cap (TP^\bot\cup 0_{TM}|_P)$ for all $0\le t\le 1$.
A point $q=\exp^{FN}(v)$ with $v\in \mathcal{D}\cap (TP^\bot\cup 0_{TM}|_P)$
 is a $P$-focal point along $[0,1]\ni t\mapsto \gamma_{v}(t)=\exp^{FN}_{\pi(v)}(tv)$
 if and only if it is a critical value of ${\rm exp}^{FN}$, and in this case
 the multiplicity (or order) of the focal point $q$
 is equal to $\dim {\rm Ker}\big(D\textmd{exp}^{FN}(v)\big)$
 (by the definitions above (\ref{e:P-JacobiField1-}) and
 \cite[Proposition~3.4]{Lu14}).
 (See also Lemma~4.8 in \cite[page~59]{Sak96} for the case in Riemannian geometry.)
 Therefore, if $v\in \mathcal{D}\cap (TP^\bot\cup 0_{TM}|_P)$
  is such that $\exp^{FN}(v)$ is not a focal point
along $[0,1]\ni t\mapsto\exp^{FN}(tv)$, and $u\in\mathcal{D}\cap (TP^\bot\cup 0_{TM}|_P)$
 is sufficiently close to $v$, then $\exp^{FN}(u)$ can not be
 a focal point along $[0,1]\ni t\mapsto\exp^{FN}(tu)$
 either.  Moreover, applying Sard theorem to
 the $C^3$ map $\exp^{FN}$ between $n$-dimensional $C^6$ manifolds
  $\mathcal{D}\cap (TP^\bot\cup 0_{TM}|_P)$ and $M$
 we obtain that the focal set of $P$ (i.e.,
 the set of all $P$-focal points) has measure zero in $M$.

\section{Bifurcations points along a Finsler geodesic}\label{sec:geodesics4}

\begin{assumption}\label{ass:Fin5}
{\rm Let $M$ be a $n$-dimensional, connected $C^{7}$ submanifold of $\R^N$, and
let $P$ be a $C^{7}$ submanifold in $M$ of dimension less than $n$.
For a $C^\ell$ Finsler metric $F$ on $M$ with $3\le\ell\le 6$ let $\gamma:[0,\tau]\to M$
be a  constant (nonzero) speed $F$-geodesic which is perpendicular to $P$ at $\gamma(0)$, i.e., $g^F_{\dot\gamma(0)}(\dot\gamma(0), u)=0\;\forall u\in T_{\gamma(0)}P$. (Note that $\gamma$ is $C^\ell$.) }
 \end{assumption}

Because of Definition~1.8 in \cite{Lu12-} or \cite{Lu12}
and  \cite[Definition~6.1]{PiPoTa04} we introduce:

\begin{definition}\label{def:bifur}
{\rm Under Assumption~\ref{ass:Fin5},  $\gamma(\mu)$ with $\mu\in (0, \tau]$ is called
a \textsf{bifurcation point on $\gamma$ relative to $P$}
if there exists a sequence $(t_k)\subset (0,\tau]$ converging to $\mu$
and a sequence   constant (non-zero)  speed $F$-geodesics $\gamma_k:[0, t_k]\to M$
emanating perpendicularly from $P$  such that
\begin{eqnarray}\label{e:geobifu1}
&&\gamma_k(t_k)=\gamma(t_k)\;\hbox{for all $k\in\N$},\\
&&0<\|\gamma_k-\gamma|_{[0, t_k]}\|_{C^1([0, t_k],\mathbb{R}^N)}\to 0\;\hbox{ as $k\to\infty$}.
\label{e:geobifu2}
\end{eqnarray}}
\end{definition}

\begin{remark}\label{rm:geobifu1}
{\rm
As pointed out below Definition~1.8 in \cite{Lu12-} or \cite{Lu12},
using Lemma~2.6 in \cite{Lu12-} or \cite{Lu12}
we can prove that the limit in (\ref{e:geobifu2}) is equivalent to
 each of the following two conditions:
\begin{enumerate}
\item[\rm (1)]
 $\gamma_k(0)\to\gamma(0)$ and $\dot{\gamma}_k(0)\to\dot{\gamma}(0)$,
 \item[(2)]
$\|\gamma_k-\gamma|_{[0, t_k]}\|_{C^2([0, t_k],\mathbb{R}^N)}\to 0$ as $k\to\infty$.
\end{enumerate}
}
\end{remark}

\begin{theorem}\label{th:MorseBifFin}
 Under Assumption~\ref{ass:Fin5}, the following are true.
 \begin{enumerate}
\item[\rm (i)] There exists only finitely many $P$-focal points  along $\gamma$.
\item[\rm (ii)] If $\gamma(\mu)$ with $\mu\in (0, \tau]$ is a
bifurcation point on $\gamma$ relatively to $P$,
 then  it is also a $P$-focal point along $\gamma$.
\item[\rm (iii)] If $\gamma(\mu)$ with $\mu\in (0, \tau)$ is a $P$-focal
point along $\gamma$, then it is
a bifurcation point on $\gamma$ relative to $P$ and  one of the following alternatives occurs:
\begin{enumerate}
\item[\rm (iii-1)] There exists a sequence of distinct $C^\ell$
  constant {\rm (}non-zero{\rm )}  speed $F$-geodesics emanating perpendicularly from $P$
and ending at $\gamma(\mu)$, $\alpha_k:[0,\mu]\to M$,
 $\alpha_k\ne\gamma|_{[0,\mu]}$, $k=1,2,\cdots$, such that
 $\alpha_k\to\gamma|_{[0,\mu]}$ in $C^2([0,\mu],\mathbb{R}^N)$ as $k\to\infty$.

\item[\rm (iii-2)]  For every $\lambda\in (0, \tau)\setminus\{\mu\}$ near $\mu$ there
exists a $C^\ell$   constant {\rm (}non-zero{\rm )}
speed $F$-geodesics emanating perpendicularly from $P$ and ending at $\gamma(\lambda)$,
$\alpha_\lambda:[0,\lambda]\to M$,  $\alpha_\lambda\ne\gamma|_{[0,\lambda]}$, such that
 $\|\alpha_\lambda-\gamma|_{[0,\lambda]}\|_{C^2([0,\lambda],\mathbb{R}^N)}\to 0$ as $\lambda\to\mu$.

\item[\rm (iii-3)] For a given small $\epsilon>0$ there is a one-sided  neighborhood $\Lambda^\ast$ of $\mu$ such that
for any $\lambda\in\Lambda^\ast\setminus\{\mu\}$, there exist at least two $C^\ell$ constant {\rm (}non-zero{\rm )}  speed
 $F$-geodesics  emanating perpendicularly from $P$ and ending at $\gamma(\lambda)$, $\beta^i_\lambda:[0,\lambda]\to M$, $\beta_\lambda^i\ne\gamma|_{[0,\lambda]}$, $i=1,2$,  to satisfy the condition that  $\|\beta^i_\lambda-\gamma|_{[0,\lambda]}\|_{C^1([0,\lambda],\mathbb{R}^N)}<\epsilon$, $i=1,2$.
Moreover, the geodesics $\beta^1_\lambda$ and $\beta^2_\lambda$ can also be chosen to have
distinct speeds {\rm (}or lengths{\rm )} if the multiplicity of $\gamma(\mu)$ as a $P$-focal point along $\gamma$ is greater than one and
there exist only finitely many $C^\ell$ constant {\rm (}non-zero{\rm )}  speed  $F$-geodesics
 emanating perpendicularly from $P$ and ending at $\gamma(\lambda)$,
$\alpha_1,\cdots,\alpha_m$, such that $\|\alpha_i-\gamma|_{[0,\lambda]}\|_{C^1([0,\lambda],\mathbb{R}^N)}<\epsilon$, $i=1,\cdots, m$.
\end{enumerate}
\end{enumerate}
\end{theorem}
\begin{proof}[\bf Proof]
For any $\lambda\in (0, \tau]$, $\gamma_\lambda:=\gamma|_{[0,\lambda]}$ is a critical point of  the $C^{2}$-functional
\begin{equation*}
\mathcal{E}_{P,\gamma(\lambda)}: C^{1}_{P\times\{\gamma(\lambda)\}}([0,\lambda]; M)_{\rm reg}\to\mathbb{R},\;\alpha\mapsto
 \int^\lambda_0[F(\alpha(t), \dot{\alpha}(t))]^2dt.
\end{equation*}

Since $\gamma([0,\tau])$ is compact in $M$ and the geodesics involved are
near this compact subset \textbf{we may assume
that $M$ is compact}. Therefore there exists an Riemannian metric $g$ on $M$
and an constant $C_1>0$ such that
$|v|^2_x\le [F(x, v)]^2\le C_1|v|^2_x$ for all $(x,v)\in TM$,
where $|v|_x=\sqrt{g_x(v,v)}$.
Clearly, there exists a constant $c>0$ such that
$[F(\gamma(t), \dot{\gamma}(t))]^2>2c/C_1$ for all $t\in [0,\tau]$.
 As in Proposition~\ref{prop:Fin2.2} we define
 \begin{eqnarray*}
 L^\ast:TM\to\R,\;(x,v)\mapsto\psi_{\varepsilon,\delta}([F(x,v)]^2)+
\phi_{\mu,b}(|v|^2_x)-b,
\end{eqnarray*}
 which is $C^\ell$ and gives a family of $C^2$-functionals
\begin{equation*}
\mathcal{E}^\ast_{P,\gamma(\lambda)}: C^{1}_{P\times\{\gamma(\lambda)\}}([0,\lambda]; M)
\to\mathbb{R},\;
\alpha\mapsto\int^\lambda_0L^\ast(\alpha(t), \dot{\alpha}(t))dt,\qquad\lambda\in (0, \tau].
\end{equation*}
By Proposition~\ref{prop:Fin2.2}(i),  $L^\ast(x,v)= L(x,v)$ if $L(x,v)\ge\frac{2c}{3C_1}$.
Hence the functionals $\mathcal{E}_{P,\gamma(\lambda)}$ and $\mathcal{E}^\ast_{P,\gamma(\lambda)}$
agree on the following open subset of
$C^{1}_{P\times\{\gamma(\lambda)\}}([0,\lambda]; M)_{\rm reg}$ containing $\gamma_\lambda$,
$$
C^{1}_{P\times\{\gamma(\lambda)\}}([0,\lambda]; M, F, c/C_1):=\left\{\alpha\in C^{1}_{P\times\{\gamma(\lambda)\}}([0,\lambda]; M)\,\big|\,
\min_{t}[F(\alpha(t), \dot{\alpha}(t))]^2>2c/C_1\right\}.
$$
Then each $\gamma_\lambda$ is also a critical point of $\mathcal{E}^\ast_{P,\gamma(\lambda)}$
 and the Hessians
 \begin{equation}\label{e:Hessians}
 D^2\mathcal{E}_{P,\gamma(\lambda)}(\gamma_\lambda)=D^2\mathcal{E}^\ast_{P,\gamma(\lambda)}(\gamma_\lambda),
\quad\forall\lambda\in (0, \tau].
\end{equation}

By Assumption~\ref{ass:Fin5} we see that $L^\ast$ satisfies
Assumption~1.7 in \cite{Lu12-} or \cite{Lu12}
with $S_0=P$
and $\gamma$ is a $L^\ast$-curve emanating perpendicularly from $P$.
From Theorem~1.9(i) in \cite{Lu12-} (or \cite{Lu12})
and (\ref{e:Hessians})  there only exist finitely many
numbers $0<s_1<\cdots<s_m\le\tau$ such that
$$
\dim{\rm Ker}(D^2\mathcal{E}_{P,\gamma(s_i)}(\gamma_{s_i}))=
\dim{\rm Ker}(D^2\mathcal{E}^\ast_{P,\gamma(s_i)}(\gamma_{s_i}))>0,\quad i=1,\cdots,m.
$$
These and (\ref{e:MS0}) lead to (i).

Suppose that $\gamma(\mu)$ with $\mu\in (0, \tau]$ is a bifurcation point on $\gamma$ relatively to $P$.
By Definition~\ref{def:bifur} it is easy to see that $\mu$  is a bifurcation instant for $(P, \gamma)$
where $\gamma$ is as a $L^\ast$-curve. Therefore Theorem~1.9(ii) in \cite{Lu12-} or \cite{Lu12}
tells us
$\dim{\rm Ker}(D^2\mathcal{E}^\ast_{P,\gamma(\mu)}(\gamma_{\mu}))>0$.
Combing the latter with (\ref{e:Hessians}) and (\ref{e:MS0}) we arrive at (ii).

Finally, let us assume that $\gamma(\mu)$ with $\mu\in (0, \tau)$ is a $P$-focal point along $\gamma$.
Then by (\ref{e:Hessians}) we see that $\mu$ is a
 $P$-focal point along $\gamma$ (as a $L^\ast$-curve relative $P$).
It follows from this result and Theorem~1.9(iii) in \cite{Lu12-} or \cite{Lu12}
that $\mu$ is a bifurcation instant for $(P, \gamma)$
and that one of Theorem~1.9(iii-k) in \cite{Lu12-} or \cite{Lu12},
$k=1,2,3$, holds after $L$ is replaced by $L^\ast$.
Combining these facts with (\ref{e:Hessians}), we obtain the desired result. The argument proceeds in two steps.

(1) By choosing $\epsilon > 0$ sufficiently small, we can ensure that
any $\alpha \in C^{1}_{P \times \{\gamma(\lambda)\}}([0,\lambda]; M)$ satisfying $\|\alpha - \gamma|_{[0,\lambda]}\|_{C^1([0,\lambda],\mathbb{R}^N)} < \epsilon$ belongs to $C^{1}_{P \times \{\gamma(\lambda)\}}([0,\lambda]; M, F, c/C_1)$
automatically.

(2) Since both $\beta_\lambda^1$ and $\beta_\lambda^2$ are constant (nonzero) speed geodesics, relations
\begin{eqnarray*}
\int^\lambda_0L(\beta_\lambda^1(t), \dot{\beta}_\lambda^1(t))dt&=&
\int^\lambda_0L^\ast(\beta_\lambda^1(t), \dot{\beta}_\lambda^1(t))dt\\
&\ne& \int^\lambda_0
L^\ast(\beta_\lambda^2(t), \dot{\beta}_\lambda^2(t))dt=\int^\lambda_0
L(\beta_\lambda^2(t), \dot{\beta}_\lambda^2(t))dt
\end{eqnarray*}
imply that $\beta_\lambda^1$ and $\beta_\lambda^2$ have different speeds.
\end{proof}

Theorem~\ref{th:MorseBifFin} has the following deep geometrical consequence.

\begin{theorem}\label{th:MorseLittauer}
Let $M$ and $P$ be as in Assumption~\ref{ass:Fin5}, and let $F$ be a $C^6$ Finsler metric on $M$.
 Suppose that $v\in \mathcal{D}\cap (TP^\bot\cup 0_{TM}|_P)$ is a critical point of ${\exp}^{FN}$.
 Then
 $\exp^{FN}$ is not injective near $v$, precisely  one of the following alternatives occurs:
\begin{enumerate}
\item[\rm (i)] There exists a sequence $(v_k)$ of distinct points in $\mathcal{D}\cap (TP^\bot\cup 0_{TM}|_P)\setminus\{v\}$ converging to $v$,
such that  $\exp^{FN}(v_k)=\exp^{FN}(v)$ for each $k=1,2,\cdots$.

\item[\rm (ii)]  For every $\lambda\in \R\setminus\{1\}$ near $1$ there
exists $v_\lambda\in \mathcal{D}\cap (TP^\bot\cup 0_{TM}|_P)\setminus\{v\}$  such that
$\exp^{FN}(\lambda v_\lambda)=\exp^{FN}(\lambda v)$ and $v_\lambda\to v$ as $\lambda\to 1$.

\item[\rm (iii)] Given a small neighborhood
$\mathcal{O}$ of $v$ in $\mathcal{D}\cap (TP^\bot\cup 0_{TM}|_P)$
there is a one-sided  neighborhood $\Lambda^\ast$ of $1$ in $\R$ such that
for any $\lambda\in\Lambda^\ast\setminus\{1\}$,  there exist at least two points $v_\lambda^1$ and $v_\lambda^2$ in $\mathcal{O}\setminus\{v\}$
such that  $\exp^{FN}(\lambda v_\lambda^k)=\exp^{FN}(\lambda v)$ for each $k=1,2$. Moreover the points $v_\lambda^1$ and $v_\lambda^2$ above
can also be chosen to satisfy $F(v_\lambda^1)\ne F(v_\lambda^2)$ if $\dim {\rm Ker}\big(D{\rm exp}^{FN}(v)\big)>1$ and
$\mathcal{O}\setminus\{v\}$ only contains finitely many points, $v_1,\cdots,v_m$, such that
${\rm exp}^{FN}(\lambda v_i)={\rm exp}^{FN}(\lambda v)$, $i=1,\cdots,m$.
 \end{enumerate}
 \end{theorem}
\begin{proof}[\bf Proof]
Let $\gamma_v(t)=\exp^{FN}(tv)$. It is well-defined on $[0,\tau]$ for some $\tau>1$
because $\mathcal{D}$ is a starlike open neighborhood of the zero section $0_{TM}$ of $TM$.
By the last paragraph of Section~\ref{sec:Fin} the assumption about $v$
implies that $q=\exp^{FN}(v)$ is a $P$-focal point along $\gamma_v$.
Therefore $q=\gamma_v(1)$  is a bifurcation point on $\gamma_v$ relatively to $P$
by Theorem~\ref{th:MorseBifFin}(iii).

Let $(\alpha_k)$ be as in (iii-1) of Theorem~\ref{th:MorseBifFin} with $\mu=1$ and $\gamma=\gamma_v$.
Then $v_k:=(\alpha_k(0), \dot\alpha_k(0))\to (\gamma_v(0),\dot\gamma_v(0))$
and all $v_k$ sit in $\mathcal{D}\cap (TP^\bot\cup 0_{TM}|_P)$ by the definition of $\mathcal{D}$.
Since $\alpha_k(t)={\rm exp}^{FN}(tv_k)$ for $t\in [0,1]$, we have ${\rm exp}^{FN}(v_k)=
\alpha_k(1)=\gamma_v(1)={\rm exp}^{FN}(v)$ and $v_k\ne v$ for all $k$.
 (i) is proved.

Similarly, let $\alpha_\lambda$ be as in (iii-2) of Theorem~\ref{th:MorseBifFin} with $\mu=1$ and $\gamma=\gamma_v$.
Then as $\lambda\to 1$ we have $v_\lambda:=(\alpha_\lambda(0), \dot\alpha_\lambda(0))\to (\gamma_v(0),\dot\gamma_v(0))=v$
because $0<\|\alpha_\lambda-\gamma_v|_{[0,\lambda]}\|_{C^2([0,\lambda],\mathbb{R}^N)}\to 0$
as $\lambda\to 1$.
Hence shrinking $\Lambda^\ast$ towards $1$ (if necessary) we may assume that
all geodesics $\alpha_\lambda$ are well-defined over $[0,1]$
(because $\gamma_v$ is well-defined on $[0,\tau]$ for some $\tau>1$).
 Therefore
$v_\lambda\in\mathcal{D}\cap (TP^\bot\cup 0_{TM}|_P)\setminus\{v\}$  and
$\exp^{FN}(\lambda v_\lambda)=\alpha_\lambda(\lambda)=\gamma_v(\lambda)=\exp^{FN}(\lambda v)$
for all $\lambda\in\Lambda^\ast$.
 (ii) is proved.

Finally, let us show how  (iii-3) of Theorem~\ref{th:MorseBifFin} leads to
 (iii) of Theorem~\ref{th:MorseLittauer}.
 Since $\gamma_v$ is well-defined on $[0,\tau']$ for some $\tau'>\tau$,
we may take a neighborhood $\mathcal{O}$ of $v$ in $\mathcal{D}\cap (TP^\bot\cup 0_{TM}|_P)$
such that for each $u\in \mathcal{O}$ the geodesic
 $t\mapsto \gamma_u(t):=\exp^{FN}(tu)$  is well-defined on $[0,\tau]$.
  By Remark~\ref{rm:geobifu1}, for a given small number $\delta\in (0,1)$
   there exists a small $\epsilon>0$ such that
 $(\alpha(0), \dot{\alpha}(0))\in\mathcal{O}$ for any $C^\ell$ constant (non-zero)
  speed $F$-geodesic  $\alpha:[0, \lambda]\to M$ emanating perpendicularly from $P$
  and ending at $\gamma_v(\lambda)$
 with $\lambda\in [1-\delta, 1+\delta]$ and satisfying  $\|\alpha-\gamma_v|_{[0,\lambda]}\|_{C^1([0,\lambda],\mathbb{R}^N)}<\epsilon$.
Let $\Lambda^\ast$ and $\beta^i_\lambda$ with $\lambda\in\Lambda^\ast$ be as in (iii-3) of Theorem~\ref{th:MorseBifFin} with $\mu=1$ and $\gamma=\gamma_v$.
We may shrink $\Lambda^\ast$ towards $1$ so that $\Lambda^\ast\subset[1-\delta,1+\delta]$.
Then the choice of $\epsilon$ implies that
$$
\hbox{$v_\lambda^j:=(\beta_\lambda^j(0), \dot\beta_\lambda^j(0))\in\mathcal{O}\setminus\{v\}$
and  $\exp^{FN}(\lambda v_\lambda^j)=\exp^{FN}(\lambda v)$ for $j=1,2$, and $v_\lambda^1\ne v_\lambda^2$.}
$$
Suppose that $\dim {\rm Ker}\big(D{\rm exp}^{FN}(v)\big)>1$, i.e.,
the multiplicity of $\gamma_v(1)$ as a $P$-focal point along $\gamma_v$ is greater than one
(by the last paragraph of Section~\ref{sec:Fin}), and that
$\mathcal{O}\setminus\{v\}$ only contains finitely many points, $v_1,\cdots,v_m$, such that
${\rm exp}^{FN}(\lambda v_i)={\rm exp}^{FN}(\lambda v)$, $i=1,\cdots,m$.
The second assumption implies that there are no infinitely many $C^\ell$ constant (non-zero)  speed  $F$-geodesics
 emanating perpendicularly from $P$ and ending at $\gamma_v(\lambda)$,
$\alpha_i:[0,\lambda]\to M$, $i=1,2,\cdots$, such that $\|\alpha_i-{\gamma_v}|_{[0,\lambda]}\|_{C^1([0,\lambda],\mathbb{R}^N)}<\epsilon$, $i=1,2,\cdots$.
(Otherwise, by the choice of $\epsilon$ we have $v_i:=(\alpha_i(0), \dot\alpha_i(0))\in\mathcal{O}\setminus\{v\}$
and  $\exp^{FN}(\lambda v_i)=\exp^{FN}(\lambda v)$ for each $i=1,2,\cdots$.)
Therefore there exist only finitely many such $\alpha$, saying $\alpha_1,\cdots,\alpha_k$.
In this case, by (iii-3) of Theorem~\ref{th:MorseBifFin}
the geodesics $\beta^1_\lambda$ and $\beta^2_\lambda$ above
can also be chosen to have distinct speeds, i.e., $F(v_\lambda^1)\ne F(v_\lambda^2)$.
The proof of (iii) of Theorem~\ref{th:MorseLittauer} is complete.
\end{proof}

Clearly, the following special case of Theorem~\ref{th:MorseLittauer} also significantly
strengthens a generalization of the Gauss lemma  by
 Morse-Littauer \cite{MorLit32} and Savage \cite{Sav43}:
\textsf{The exponential  map $\exp^F$ of a $C^\infty$ Finsler space $(M, F)$
is not locally injective near any critical point.}

\begin{theorem}\label{th:MorseLittauer1}
Let $M\subset\R^N$ be a $C^7$ manifold and let $F:TM\to\R$ be a $C^6$ Finsler metric.
If $v$ is a critical point of the restriction ${\exp}^{F}_p$ of
 the exponential map $TM\supseteq\mathcal{D}\ni u\mapsto\exp^F(u)\in M$ to
$\mathcal{D}_p:=\mathcal{D}\cap T_pM$, then
 one of the following alternatives occurs:
 \begin{enumerate}
\item[\rm (i)] There exists a sequence $(v_k)$ of distinct points
 in $\mathcal{D}_p\setminus\{v\}$ converging to $v$,
such that  $\exp^F_p(v_k)=\exp^F_p(v)$ for each $k=1,2,\cdots$.

\item[\rm (ii)]  For every $\lambda\in \R\setminus\{1\}$ near $1$ there
exists $v_\lambda\in \mathcal{D}_p\setminus\{v\}$  such that
$\exp^F_p(\lambda v_\lambda)=\exp^F_p(\lambda v)$ and $v_\lambda\to v$ as $\lambda\to 1$.

\item[\rm (iii)] Given a small neighborhood $\mathcal{O}$ of $v$ in $\mathcal{D}_p$
there is a one-sided  neighborhood $\Lambda^\ast$ of $1$ in $\R$ such that
for any $\lambda\in\Lambda^\ast\setminus\{1\}$,  there exist at least two points $v_\lambda^1$ and $v_\lambda^2$ in $\mathcal{O}\setminus\{v\}$
such that  $\exp^F_p(\lambda v_\lambda^k)=\exp^F_p(\lambda v)$ for each $k=1,2$. Moreover the points $v_\lambda^1$ and $v_\lambda^2$ above
can also be chosen to satisfy $F(v_\lambda^1)\ne F(v_\lambda^2)$ if $\dim {\rm Ker}\left(D{\rm exp}^F_p(v)\right)>1$ and
$\mathcal{O}\setminus\{v\}$ only contains finitely many points, $v_1,\cdots,v_m$, such that
${\rm exp}^F_p(\lambda v_i)={\rm exp}^F_p(\lambda v)$, $i=1,\cdots,m$.
 \end{enumerate}
 \end{theorem}

\section{Bifurcations of geodesics with two kinds of special boundary conditions}\label{sec:geodesics1}

Let $(M, g)$, $\mathbb{I}_g$ and submanifolds $P$, $Q$  be as at the beginning of Section~\ref{sec:Fin}.
In order to use the results in Section~1 and \cite{Lu12-} (or \cite{Lu12})
conveniently,\textbf{ we write $P$ and $Q$
as $S_0$ and $S_1$}, respectively.
Then the boundary condition (\ref{e:FinBoundNlambda}) becomes
\begin{equation}\label{e:FinBoundNlambda1}
g^{F_\lambda}_{\dot\gamma(0)}(u,\dot\gamma(0))=0\;\forall u\in
T_{\gamma(0)}S_0\quad\hbox{and}\quad
g^{F_\lambda}_{\dot\gamma(1)}(v,\dot\gamma(\tau))=0\;\forall v\in
T_{\gamma(\tau)}S_1
\end{equation}
if ${\bf N}=S_0\times S_1$, and
\begin{equation}\label{e:FinBoundNlambda2}
g^{F_\lambda}_{\dot\gamma(0)}(u,\dot\gamma(0))=
g^{F_\lambda}_{\dot\gamma(\tau)}(\mathbb{I}_{g\ast}u,\dot\gamma(\tau))
\quad\forall u\in T_{\gamma(0)}M
\end{equation}
if ${\bf N}={\rm Graph}(\mathbb{I}_{g})$. In these two cases the Morse index and nullity in
(\ref{e:FinMorseIndex}) have more precise explanations.
See (\ref{e:MS0}), (\ref{e:MS1}) and (\ref{e:MS2}) and \cite[\S6]{Lu5}.

\begin{theorem}\label{th:bif-nessFin}
Under Assumptions~\ref{ass:Fin1},~\ref{ass:Fin2}
with ${\bf N}=S_0\times S_1$ or ${\rm Graph}(\mathbb{I}_{g})$,
for $\mu\in\Lambda$  such that $\gamma_\mu(0)\ne\gamma_\mu(\tau)$
in the case $\dim S_0>0$ and $\dim S_1>0$, there holds:
\begin{enumerate}
\item[\rm (I)]{\rm (\textsf{Necessary condition}):}
Suppose  that constant (non-zero) speed $(F_\lambda, \bf N)$-geodesics
 with a parameter $\lambda\in\Lambda$
bifurcate at some $\mu\in\Lambda$ along sequences with respect to the branch $\{\gamma_\lambda\,|\,\lambda\in\Lambda\}$.
Then $m^0_\tau(\mathcal{E}_{\mu,\bf N}, \gamma_\mu)\ne 0$.
\item[\rm (II)]{\rm (\textsf{Sufficient condition}):}
Suppose that $\Lambda$ is first countable and that
there exist two sequences in  $\Lambda$ converging to $\mu$, $(\lambda_k^-)$ and
$(\lambda_k^+)$,  such that one of the following conditions is satisfied:
 \begin{enumerate}
 \item[\rm (II.1)] For each $k\in\mathbb{N}$, either $\gamma_{\lambda^+_k}$
  is not an isolated critical point of  ${\mathcal{E}}_{\lambda^+_k,\bf N}$,
 or $\gamma_{\lambda^-_k}$ is not an isolated critical point of ${\mathcal{E}}_{\lambda^-_k,\bf N}$,
 or $\gamma_{\lambda^+_k}$ (resp. $\gamma_{\lambda^-_k}$) is an isolated critical point of $\mathcal{E}_{\lambda^+_k,\bf N}$ (resp. $\mathcal{E}_{\lambda^-_k,\bf N}$) and
  $C_m(\mathcal{E}_{\lambda^+_k,\bf N}, \gamma_{\lambda^+_k};{\bf K})$
   and $C_m(\mathcal{E}_{\lambda^-_k,\bf N}, \gamma_{\lambda^-_k};{\bf K})$
  are not isomorphic for some Abel group ${\bf K}$ and some $m\in\mathbb{Z}$.
\item[\rm (II.2)] For each $k\in\mathbb{N}$, there exists $\lambda\in\{\lambda^+_k, \lambda^-_k\}$
such that $\gamma_{\lambda}$  is an either non-isolated or homological visible critical point of
$\mathcal{E}_{\lambda,\bf N}$ , and
$$
\left.\begin{array}{ll}
&[m^-(\mathcal{E}_{\lambda_k^-,\bf N}, \gamma_{\lambda^-_k}), m^-(\mathcal{E}_{\lambda_k^-,\bf N}, \gamma_{\lambda^-_k})+
m^0(\mathcal{E}_{\lambda_k^-,\bf N}, \gamma_{\lambda^-_k})]\\
&\cap[m^-(\mathcal{E}_{\lambda_k^+,\bf N}, \gamma_{\lambda^+_k}),
m^-(\mathcal{E}_{\lambda_k^+,\bf N}, \gamma_{\lambda^+_k})+m^0(\mathcal{E}_{\lambda_k^+,\bf N}, \gamma_{\lambda^+_k})]=\emptyset.
\end{array}\right\}\eqno(\hbox{$\spadesuit_k$})
$$
\item[\rm (II.3)] For each $k\in\mathbb{N}$, (\hbox{$\spadesuit_k$}) holds true,
and either $m^0(\mathcal{E}_{\lambda_k^-,\bf N}, \gamma_{\lambda^-_k})=0$ or $m^0(\mathcal{E}_{\lambda_k^+,\bf N}, \gamma_{\lambda^+_k})=0$.
 \end{enumerate}
  Then  there exists a sequence $(\lambda_k)\subset\hat{\Lambda}:=\{\mu,\lambda^+_k, \lambda^-_k\,|\,k\in\mathbb{N}\}$ converging to $\mu$
and constant speed $F_{\lambda_k}$-geodesic $\gamma^k:[0,\tau]\to M$ satisfying the boundary condition
(\ref{e:FinBoundNlambda})  with $\lambda=\lambda_k$, $k=1,2,\cdots$,
   such that $\gamma^k\to \gamma_\mu$ in $C^2([0, \tau];M)$.
 In particular, constant (non-zero) speed $(F_\lambda, \bf N)$-geodesics
  with a parameter $\lambda\in\Lambda$
bifurcate at $\mu\in\Lambda$ along sequences with respect to the branch $\{\gamma_\lambda\,|\,\lambda\in\Lambda\}$.
  \end{enumerate}
  \end{theorem}
\begin{proof}[\bf Proof]
{\it Step 1}[\textsf{Prove} (I)].
By Definition~\ref{def:BifurFin}
there exists an infinite sequence $\{(\lambda_k, \gamma^k)\}^\infty_{k=1}$
  in $\Lambda\times C^1([0,\tau], M)\setminus\{(\mu,\gamma_\mu)\}$
  converging to $(\mu,\gamma_\mu)$, such that each $\gamma^k\ne\gamma_{\lambda_k}$
  is a $F_{\lambda_k}$-geodesic satisfying the boundary condition
(\ref{e:FinBoundNlambda})  with $\lambda=\lambda_k$, $k=1,2,\cdots$.
  Let $\hat{\Lambda}=\{\mu\}\cup\{\lambda_k\,|\,k\in\mathbb{N}\}$. It is compact and sequential compact.
(Note that all $\gamma^m$ and $\gamma_{\lambda}$ are $C^\ell$, $4\le\ell\le 6$.)
It is easy to find an open subset $\hat{M}$ of $M$ with compact closure such that the closure
 $$
 Cl\left(\cup_{(\lambda,m)
 \in\hat{\Lambda}\times\mathbb{N}}\gamma_\lambda([0,\tau])\cup\gamma^m([0,\tau])\right)\subset\hat{M}.
 $$
Then the conditions in Proposition~\ref{prop:Fin2.2} can be satisfied with $(M, \Lambda)=(\hat{M}, \hat{\Lambda})$.
Therefore  for the constant $C_1>0$ as in Proposition~\ref{prop:Fin2.2} with $(M, \Lambda)=(\hat{M}, \hat{\Lambda})$ we have $c>0$ such that
for all $(m, \lambda,t)\in\mathbb{N}\times\hat{\Lambda}\times [0,\tau]$,
\begin{equation*}
[F_\lambda(\gamma_\lambda(t), \dot{\gamma}_\lambda(t))]^2>\frac{2c}{C_1}\quad\hbox{and}\quad
[F_{\lambda_m}(\gamma^m(t), \dot{\gamma}^m(t))]^2>\frac{2c}{C_1}.
 \end{equation*}
Let $L^\ast_\lambda:T\hat{M}\to\R$,  $\lambda\in\hat{\Lambda}$, be given by Proposition~\ref{prop:Fin2.2} with $(M, \Lambda)=(\hat{M}, \hat{\Lambda})$.
Then the corresponding $C^{2}$ functional $\mathcal{E}^\ast_{\lambda,\bf N}$ given by (\ref{e:FinEnergy*})
and the $C^{2-0}$ functional $\mathcal{E}_{\lambda,\bf N}$ in (\ref{e:FinEnergy}) coincide in the
following open subset of ${\cal C}_{\tau,\bf N}(\hat{M})_{\rm reg}$
as in (\ref{e:regularSet}),
\begin{equation}\label{e:regularSet11}
{\cal C}_{\tau,\bf N}(\hat{M}, \{F_\lambda, F_{\lambda_m}\,|\,(\lambda,m)\in\hat{\Lambda}\times\mathbb{N}\}, {c/C_1})
\end{equation}
 consisting of all $\alpha\in {\cal C}_{\tau,\bf N}(\hat{M})$ such that
\begin{equation*}
\min_{(\lambda,t)\in\hat{\Lambda}\times [0,\tau]}[F_\lambda(\alpha(t), \dot{\alpha}(t))]^2>2c/C_1\quad\hbox{and}\quad
\min_{(m,t)\in\mathbb{N}\times [0,\tau]}[F_{\lambda_m}(\alpha(t), \dot{\alpha}(t))]^2>2c/C_1.
 \end{equation*}
For any $(m, \lambda)\in\mathbb{N}\times\hat{\Lambda}$, since $\gamma_\lambda$ and $\gamma^m$ belong
to ${\cal C}_{\tau,\bf N}(\hat{M},
\{F_\lambda, F_{\lambda_m}\,|\,(\lambda,m)\in\hat{\Lambda}\times\mathbb{N}\}, {c/C_1})$,
each $\gamma_\lambda$ (resp. $\gamma^m$) is a critical point of
$\mathcal{E}^\ast_{\lambda,\bf N}$ (resp. $\mathcal{E}^\ast_{\lambda_m,\bf N}$) and we have also (\ref{e:IndexSame}).
Hence when  ${\bf N}={\rm Graph}(\mathbb{I}_{g})$ (resp. ${\bf N}=S_0\times S_1$),
$(\mu, \gamma_\mu)$  is a  bifurcation point of the problem
\cite[(1.14)]{Lu12-} or \cite[(1.13)]{Lu12}
[resp. (1.5)-(1.6) in \cite{Lu12-} or \cite{Lu12}] with respect
to the  trivial branch $\{(\lambda, \gamma_\lambda)\,|\,\lambda\in\hat{\Lambda}\}$
in $C^{1}([0,\tau]; M)$. It follows from  Theorem~1.13(I) in \cite{Lu12-} or \cite{Lu12}
 [resp. Theorem~1.4 in \cite{Lu12-} or \cite{Lu12}
  ] that $m^0(\mathcal{E}^\ast_{\mu,\bf N}, \gamma_\mu)>0$ and
  so $m^0(\mathcal{E}_{\mu,\bf N}, \gamma_\mu)>0$ by (\ref{e:IndexSame}).

{\it Step 2}[\textsf{Prove} (II)].
Since $\hat{\Lambda}=\{\mu,\lambda^+_k, \lambda^-_k\,|\,k\in\mathbb{N}\}$
is compact and sequential compact,
as above we can  find an open subset $\hat{M}$ of $M$ with compact closure such that the closure
 $$
Cl\left(\cup_{(\lambda,m)\in\hat{\Lambda}\times\mathbb{N}}\gamma_\lambda([0,\tau])
\cup\gamma_{\lambda^+_m}([0,\tau])\cup\gamma_{\lambda^-_m}([0,\tau])\right)
\subset\hat{M}.
$$
For the constant $C_1>0$ as in Proposition~\ref{prop:Fin2.2}
with $(M, \Lambda)=(\hat{M}, \hat{\Lambda})$ we have $c>0$ such that
for all $(m, \lambda,t)\in\mathbb{N}\times\hat{\Lambda}\times [0,\tau]$,
\begin{equation*}
[F_\lambda(\gamma_\lambda(t), \dot{\gamma}_\lambda(t))]^2>\frac{2c}{C_1}\quad\hbox{and}\quad
[F_{\lambda_m^\pm}(\gamma_{\lambda^\pm_m}(t), \dot{\gamma}_{\lambda^\pm_m}(t))]^2>\frac{2c}{C_1}.
 \end{equation*}
Let $L^\ast_\lambda:T\hat{M}\to\R$,  $\lambda\in\hat{\Lambda}$,
be given by Proposition~\ref{prop:Fin2.2} with $(M, \Lambda)=(\hat{M}, \hat{\Lambda})$.
As above, for all $(m, \lambda)\in\mathbb{N}\times\hat{\Lambda}$ we have
$$
\gamma_\lambda, \gamma_{\lambda^+_m}, \gamma_{\lambda^-_m}\in {\cal C}_{\tau,\bf N}(\hat{M}, \{F_\lambda, F_{\lambda^+_m}, F_{\lambda^-_m}\,|\,(\lambda,m)\in\hat{\Lambda}\times\mathbb{N}\}, {c/C_1}),
$$
and (\ref{e:IndexSame}) and (\ref{e:IndexSame+}) lead to
\begin{eqnarray*}
&&m^-(\mathcal{E}_{\lambda^\pm_m,\bf N},\gamma_{\lambda^\pm_m})=m^-(\mathcal{E}^\ast_{\lambda^\pm_m,\bf N},\gamma_{\lambda^\pm_m})
\quad\hbox{and}\quad m^0(\mathcal{E}_{\lambda^\pm_m,\bf N},\gamma_{\lambda^\pm_m})=m^0(\mathcal{E}^\ast_{\lambda^\pm_m,\bf N},\gamma_{\lambda^\pm_m}),\\
&&C_k(\mathcal{E}_{\lambda^\pm_m,\bf N},\gamma_{\lambda^\pm_m};{\bf K})=
C_k(\mathcal{E}^\ast_{\lambda^\pm_m,\bf N},\gamma_{\lambda^\pm_m};
{\bf K})\quad\forall (k,m)\in\mathbb{Z}\times\mathbb{N}
\end{eqnarray*}
for any Abel group ${\bf K}$. By these we see that for ${\bf N}=S_0\times S_1$ [resp. ${\bf N}={\rm Graph}(\mathbb{I}_{g})$]
the conditions (II.1), (II.2) and (II.3) in Theorem~\ref{th:bif-nessFin}, respectively,
give rise to the corresponding conditions (II.1), (II.2) and (II.3)
 in  Theorem~1.4 in \cite{Lu12-} or \cite{Lu12}
 (resp. Theorem~1.13 in \cite{Lu12-} or \cite{Lu12}).
Hence there exists an infinite sequence $\{(\lambda_k, \gamma^k)\}^\infty_{k=1}$
  in $\hat{\Lambda}\times C^2([0,\tau], \hat{M})\setminus\{(\mu,\gamma_\mu)\}$
  converging to $(\mu,\gamma_\mu)$, such that each $\gamma^k\ne\gamma_{\lambda_k}$
  satisfies
\begin{equation}\label{e:Lagr1FinSler}
\frac{d}{dt}\left(\partial_vL^\ast_\lambda(t, \gamma(t), \dot{\gamma}(t))\right)-
\partial_x L^\ast_\lambda(t, \gamma(t), \dot{\gamma}(t))=0
\end{equation}
with $\lambda=\lambda_k$ and the boundary condition \cite[(1.7)]{Lu12-} or \cite[(1.6)]{Lu12}
 [resp. \cite[(1.14)]{Lu12-} or \cite[(1.13)]{Lu12})
 ] with $\lambda=\lambda_k$
if ${\bf N}=S_0\times S_1$ [resp. ${\bf N}={\rm Graph}(\mathbb{I}_{g})$], $k=1,2,\cdots$.
From these and Claim~\ref{cl:Fin1} we conclude that for each $k$ large enough, $\gamma^k$ is a constant (nonzero) speed $F_{\lambda_k}$-geodesic
satisfying the boundary condition (\ref{e:FinBoundNlambda1}) [resp. (\ref{e:FinBoundNlambda2})] with $\lambda=\lambda_k$
if ${\bf N}=S_0\times S_1$ [resp. ${\bf N}={\rm Graph}(\mathbb{I}_{g})$], $k=1,2,\cdots$.
\end{proof}

\begin{theorem}[\textsf{Existence for bifurcations}]\label{th:bif-existFin1}
Under Assumptions~\ref{ass:Fin1},~\ref{ass:Fin2},
let ${\bf N}=S_0\times S_1$ or ${\rm Graph}(\mathbb{I}_{g})$.
Suppose that $\Lambda$ is path-connected and  there exist two
points $\lambda^+, \lambda^-\in\Lambda$ such that
  one of the following conditions is satisfied:
 \begin{enumerate}
 \item[\rm (i)] Either $\gamma_{\lambda^+}$  is not an isolated
  critical point of $\mathcal{E}_{\lambda^+,\bf N}$,
 or $\gamma_{\lambda^-,\bf N}$ is not an isolated critical point of $\mathcal{E}_{\lambda^-,\bf N}$,
 or $\gamma_{\lambda^+,\bf N}$ {\rm (}resp. $\gamma_{\lambda^-,\bf N}${\rm )}
 is an isolated critical point of
  $\mathcal{E}_{\lambda^+,\bf N}$ {\rm (}resp. $\mathcal{E}_{\lambda^-,\bf N}${\rm )} and
  $C_m(\mathcal{E}_{\lambda^+,\bf N}, \gamma_{\lambda^+};{\bf K})$ and $C_m(\mathcal{E}_{\lambda^-,\bf N}, \gamma_{\lambda^-};{\bf K})$ are not isomorphic for some Abel group ${\bf K}$ and some $m\in\mathbb{Z}$.

\item[\rm (ii)] The intervals $[m^-({\mathcal{E}}_{\lambda^-,\bf N}, \gamma_{\lambda^-}),
m^-({\mathcal{E}}_{\lambda^-,\bf N}, \gamma_{\lambda^-})+ m^0({\mathcal{E}}_{\lambda^-,\bf N}, \gamma_{\lambda^-})]$ and
$$
[m^-({\mathcal{E}}_{\lambda^+,\bf N}, \gamma_{\lambda^+}),
m^-({\mathcal{E}}_{\lambda^+,\bf N}, \gamma_{\lambda^+})+m^0({\mathcal{E}}_{\lambda^+,\bf N}, \gamma_{\lambda^+})]
$$
are disjoint, and there exists $\lambda\in\{\lambda^+, \lambda^-\}$ such that $\gamma_{\lambda}$
 is an either non-isolated or homological visible critical point of $\mathcal{E}_{\lambda,\bf N}$.

\item[\rm (iii)] The intervals $[m^-({\mathcal{E}}_{\lambda^-,\bf N}, \gamma_{\lambda^-}),
m^-({\mathcal{E}}_{\lambda^-,\bf N}, \gamma_{\lambda^-})+ m^0({\mathcal{E}}_{\lambda^-,\bf N}, \gamma_{\lambda^-})]$ and
$$
[m^-({\mathcal{E}}_{\lambda^+,\bf N}, \gamma_{\lambda^+}),
m^-({\mathcal{E}}_{\lambda^+,\bf N}, \gamma_{\lambda^+})+m^0({\mathcal{E}}_{\lambda^+,\bf N}, \gamma_{\lambda^+})]
$$
are disjoint, and either $m^0(\mathcal{E}_{\lambda^+,\bf N}, \gamma_{\lambda^+})=0$ or $m^0(\mathcal{E}_{\lambda^-,\bf N}, \gamma_{\lambda^-})=0$.
 \end{enumerate}
  Then for any path $\alpha:[0,1]\to\Lambda$ connecting $\lambda^+$ to $\lambda^-$
  such that $\gamma_{\alpha(s)}(0)\ne \gamma_{\alpha(s)}(\tau)$
  for any $s\in [0,1]$ in the case ${\bf N}=S_0\times S_1$ and $\dim S_0\dim S_1>0$,
  there exists  a sequence $(\lambda_k)\subset\alpha([0,1])$ converging to some $\mu\in\alpha([0,1])$, and
  constant {\rm (}non-zero{\rm )} speed $F_{\lambda_k}$-geodesics $\gamma^k:[0, \tau]\to M$
 satisfying the boundary condition (\ref{e:FinBoundNlambda}), $k=1,2,\cdots$,
  such that $0<\|\gamma^k-\gamma_{\lambda_k}\|_{C^2([0,\tau];\mathbb{R}^N)}\to 0$
  as $k\to\infty$.
Moreover, $\mu$ is not equal to $\lambda^+$ {\rm (}resp. $\lambda^-${\rm )} if $m^0_\tau(\mathcal{E}_{\lambda^+,\bf N}, \gamma_{\lambda^+})=0$
 {\rm (}resp. $m^0_\tau(\mathcal{E}_{\lambda^-,\bf N}, \gamma_{\lambda^-})=0${\rm )}.
  \end{theorem}
\begin{proof}[\bf Proof]
As above, since $\hat{\Lambda}:=\alpha([0,1])$ is a compact and sequential compact subset in $\Lambda$
 we can  find an open subset $\hat{M}$ of $M$ with compact closure such that the closure
 $Cl\left(\cup_{\lambda\in\hat{\Lambda}\times\mathbb{N}}\gamma_\lambda([0,\tau])
\right)\subset\hat{M}$.
For the constant $C_1>0$ as in Proposition~\ref{prop:Fin2.2} with $(M, \Lambda)=(\hat{M}, \hat{\Lambda})$ we have $c>0$ such that
\begin{equation*}
[F_\lambda(\gamma_\lambda(t), \dot{\gamma}_\lambda(t))]^2>\frac{2c}{C_1}\quad\hbox{and}\quad
[F_{\lambda^\pm}(\gamma_{\lambda^\pm}(t), \dot{\gamma}_{\lambda^\pm}(t))]^2>\frac{2c}{C_1}\quad\hbox{for all $(\lambda,t)\in\hat{\Lambda}\times [0,\tau]$}.
 \end{equation*}
Let $L^\ast_\lambda:T\hat{M}\to\R$,  $\lambda\in\hat{\Lambda}$,
 be given by Proposition~\ref{prop:Fin2.2} with $(M, \Lambda)=(\hat{M}, \hat{\Lambda})$.
As above,  we have  $\gamma_\lambda, \gamma_{\lambda^+}, \gamma_{\lambda^-}\in
{\cal C}_{\tau,\bf N}(\hat{M}, \{F_\lambda, F_{\lambda^+}, F_{\lambda^-}\,|\,\lambda\in\hat{\Lambda}\}, {c/C_1})$ for all $\lambda\in\hat{\Lambda}$,
and (\ref{e:IndexSame}) and (\ref{e:IndexSame+}) lead to
\begin{eqnarray*}
&&m^-(\mathcal{E}_{\lambda^\pm,\bf N},\gamma_{\lambda^\pm})=m^-(\mathcal{E}^\ast_{\lambda^\pm,\bf N},\gamma_{\lambda^\pm})
\quad\hbox{and}\quad m^0(\mathcal{E}_{\lambda^\pm,\bf N},\gamma_{\lambda^\pm})=m^0(\mathcal{E}^\ast_{\lambda^\pm,\bf N},\gamma_{\lambda^\pm}),\\
&&C_k(\mathcal{E}_{\lambda^\pm,\bf N},\gamma_{\lambda^\pm};{\bf K})=C_k(\mathcal{E}^\ast_{\lambda^\pm,\bf N},\gamma_{\lambda^\pm};
{\bf K})\quad\forall k\in\mathbb{Z}
\end{eqnarray*}
for any Abel group ${\bf K}$. As above, for ${\bf N}=S_0\times S_1$ [resp. ${\bf N}={\rm Graph}(\mathbb{I}_{g})$]
the corresponding results may follow from these and Theorem~1.5 in \cite{Lu12-} or \cite{Lu12}
(resp. Theorem~1.14 in \cite{Lu12-} or \cite{Lu12}).
\end{proof}

\begin{theorem}[\textsf{Alternative bifurcations of Rabinowitz's type}]\label{th:bif-suffFin}
Under Assumptions~\ref{ass:Fin1},~\ref{ass:Fin2} with $\Lambda$ being a real interval, suppose that
${\bf N}=S_0\times S_1$ or ${\rm Graph}(\mathbb{I}_{g})$,
and that $\mu\in{\rm Int}(\Lambda)$ satisfies conditions:
 $\gamma_\mu(0)\ne\gamma_\mu(\tau)$ {\rm (}in the case $\dim S_0>0$ and $\dim S_1>0${\rm )},
$m^0(\mathcal{E}_{\mu,\bf N}, \gamma_\mu)>0$,
    $m^0(\mathcal{E}_{\lambda,\bf N}, \gamma_\lambda)=0$  for each $\lambda\in\Lambda\setminus\{\mu\}$ near $\mu$, and
  $m^-(\mathcal{E}_{\lambda,\bf N}, \gamma_\lambda)$ take, respectively, values $m^-(\mathcal{E}_{\mu,\bf N}, \gamma_\mu)$ and
  $m^-(\mathcal{E}_{\mu,\bf N}, \gamma_\mu)+ m^0(\mathcal{E}_{\mu,\bf N}, \gamma_\mu)$
 as $\lambda\in\Lambda$ varies in two deleted half neighborhoods  of $\mu$.
Then one of the following alternatives occurs:
\begin{enumerate}
\item[\rm (i)] There exists a sequence  $C^\ell$  constant (non-zero) speed $F_{\mu}$-geodesics
$\gamma^m\ne\gamma_{\mu}$  satisfying the boundary condition (\ref{e:FinBoundNlambda})
such that $\gamma^m\to\gamma_\mu$ in $C^2([0,\tau], M)$.

\item[\rm (ii)]  For every $\lambda\in\Lambda\setminus\{\mu\}$ near $\mu$ there is a
$C^\ell$  constant (non-zero) speed $F_{\lambda}$-geodesic
$\gamma'_\lambda\ne \gamma_\lambda$  satisfying the boundary condition (\ref{e:FinBoundNlambda})
 such that $\gamma'_\lambda-\gamma_\gamma$ converges to zero in
   $C^2([0,\tau], \R^N)$ as $\lambda\to \mu$.

\item[\rm (iii)] For a given neighborhood $\mathcal{W}$ of $\gamma_\mu$ in $C^2([0,\tau], \R^N)$,
there is a one-sided  neighborhood $\Lambda^0$ of $\mu$ such that
for any $\lambda\in\Lambda^0\setminus\{\mu\}$,
$\mathcal{W}$ contains  at least two distinct $C^\ell$  constant (non-zero) speed $F_{\lambda}$-geodesics
satisfying the boundary condition (\ref{e:FinBoundNlambda}),
 $\gamma_\lambda^1\ne \gamma_\lambda$ and $\gamma_\lambda^2\ne \gamma_\lambda$,
 which can also be chosen to satisfy  $F_\lambda(\gamma^1_\lambda(t), \dot{\gamma}^1_\lambda(t))\ne F_\lambda(\gamma^2_\lambda(t), \dot{\gamma}^2_\lambda(t))\;\forall t$  provided that  $m^0_\tau(\mathcal{E}_{\mu,\bf N}, \gamma_\mu)>1$ and
$\mathcal{W}$ only contains finitely many distinct constant (non-zero) speed $F_{\lambda}$-geodesics
satisfying the boundary condition (\ref{e:FinBoundNlambda}).
\end{enumerate}
\end{theorem}
\begin{proof}[\bf Proof]
Since $\Lambda$ is a real interval and $\mu\in{\rm Int}(\Lambda)$, for some real $\rho>0$ the compact set $\hat{\Lambda}:=[\mu-\rho, \mu+\rho]$
is contained in $\Lambda$. The continuous map $\hat{\Lambda}\times [0,\tau]\ni(\lambda,t)\to \gamma_\lambda(t)\in M$ has compact image set  and therefore
 the latter is contained in an open subset $\hat{M}$ of $M$ with compact closure.
For the constant $C_1>0$ as in Proposition~\ref{prop:Fin2.2} with $(M, \Lambda)=(\hat{M}, \hat{\Lambda})$ we have $c>0$ such that
\begin{equation*}
[F_\lambda(\gamma_\lambda(t), \dot{\gamma}_\lambda(t))]^2>\frac{2c}{C_1},\quad
\forall (\lambda,t)\in\hat{\Lambda}\times [0,\tau].
 \end{equation*}
Let $L^\ast_\lambda:T\hat{M}\to\R$,  $\lambda\in\hat{\Lambda}$,
be given by Proposition~\ref{prop:Fin2.2} with $(M, \Lambda)=(\hat{M}, \hat{\Lambda})$.
Then for all $\lambda\in\hat{\Lambda}$ we have  $\gamma_\lambda\in {\cal C}_{\tau,\bf N}(\hat{M}, \{F_\lambda\,|\,\lambda\in\hat\Lambda\}, {c/C_1})$,
and
\begin{eqnarray*}
m^-(\mathcal{E}_{\lambda,\bf N},\gamma_{\lambda})=m^-(\mathcal{E}^\ast_{\lambda,\bf N},\gamma_{\lambda})
\quad\hbox{and}\quad m^0(\mathcal{E}_{\lambda,\bf N},\gamma_{\lambda})=m^0(\mathcal{E}^\ast_{\lambda,\bf N},\gamma_{\lambda}).
\end{eqnarray*}
By these and the assumptions of Theorem~\ref{th:bif-suffFin} we obtain that
 $m^0(\mathcal{E}^\ast_{\mu,\bf N}, \gamma_\mu)>0$,
 $m^0(\mathcal{E}^\ast_{\lambda,\bf N}, \gamma_\lambda)=0$
  for each $\lambda\in\hat{\Lambda}\setminus\{\mu\}$ near $\mu$, and
  $m^-(\mathcal{E}^\ast_{\lambda,\bf N}, \gamma_\lambda)$ take, respectively, values $m^-(\mathcal{E}^\ast_{\mu,\bf N}, \gamma_\mu)$ and
  $m^-(\mathcal{E}^\ast_{\mu,\bf N}, \gamma_\mu)+ m^0(\mathcal{E}^\ast_{\mu,\bf N}, \gamma_\mu)$
 as $\lambda\in\hat{\Lambda}$ varies in two deleted half neighborhoods  of $\mu$.
 Hence the desired results may follow from Theorems~1.15, 1.6 in \cite{Lu12-}
  (or \cite{Lu12})
and Claim~\ref{cl:Fin1} as above.
\end{proof}

\noindent{\bf A modest strengthening of the above results in the case of
a family of Riemannian matrics}.
The following is a special case of Assumptions~1.1, ~1.2 and 1.11 in \cite{Lu12-} or \cite{Lu12}.

\begin{assumption}\label{ass:Rieman1}
	{\rm $\{h_\lambda\,|\,\lambda\in\Lambda\}$ is a family of $C^\ell$ Riemannian metrics on $M$ with $4\le\ell\le 6$ parameterized by a topological space $\Lambda$,
such that $\Lambda\times TM\ni (\lambda, x,v)\to h_\lambda(x,v)\in\R$ is a continuous,
		and that all partial derivatives of each $h_\lambda$ of order less than three
		depend continuously on $(\lambda, x, v)\in\Lambda\times TM$.
		For each $\lambda\in\Lambda$ let $\gamma_\lambda:[0, \tau]\to M$ be a
		$h_\lambda$-geodesic satisfying the boundary condition
		\begin{equation}\label{e:FinBoundNlambdaRiem}
		h_\lambda(u,\dot\gamma_\lambda(0))=
		h_\lambda(v,\dot\gamma_\lambda(\tau))\quad\forall (u,v)\in
		T_{(\gamma_\lambda(0),\gamma_\lambda(\tau))}{\bf N},
		\end{equation}
		where ${\bf N}\subset M\times M$ is a $C^7$ submanifold.
(Therefore $\gamma_\lambda$ is $C^\ell$ by Claim~\ref{cl:regularity}.)
		It is also required  that the maps
$\Lambda\times [0,\tau]\ni(\lambda,t)\to \gamma_\lambda(t)\in M$ and
		$\Lambda\times [0,\tau]\ni(\lambda, t)\mapsto \dot{\gamma}_\lambda(t)\in TM$ are continuous.}
\end{assumption}

If there is no confusion, we use $\mathcal{E}_{\lambda, \bf N}$
to denote the functional
\begin{equation}\label{e:FinEnergyRiem}
{\cal C}_{\tau,\bf N}(M)\to\mathbb{R},\;   \gamma \mapsto\int^\tau_0h_\lambda(\gamma(t), \dot{\gamma}(t))dt,
\end{equation}
and $m^-(\mathcal{E}_{\lambda,\bf N},\gamma_\lambda)$ and $m^0(\mathcal{E}_{\lambda,\bf N},\gamma_\lambda)$
to denote  the Morse index and nullity  at $\gamma_\lambda$ of  it.
For  ${\bf N}=S_0\times S_1$ [resp. ${\rm Graph}(\mathbb{I}_{g})$], the following theorem
may directly follow from  Theorem~1.4 in \cite{Lu12-} or \cite{Lu12}
(resp. Theorem~1.13 in \cite{Lu12-} or \cite{Lu12}).

\begin{theorem}\label{th:bif-nessFinRiem}
	Under Assumption~\ref{ass:Rieman1} with
${\bf N}=S_0\times S_1$ or ${\rm Graph}(\mathbb{I}_{g})$,
the conclusions of Theorem~\ref{th:bif-nessFin} hold for $\mu\in\Lambda$ with the following
notational changes:
	\begin{enumerate}
		\item[\rm (I)] 
In statement (I) of Theorem~\ref{th:bif-nessFin}, replace the phrase
        ``constant (non-zero) speed $(F_\lambda, \bf N)$-geodesics''
        with ``$(h_\lambda, \bf N)$-geodesics''.

		\item[\rm (II)] 
In statement (II) of Theorem~\ref{th:bif-nessFin}, replace the phrases
        \begin{itemize}
            \item ``constant speed $F_{\lambda_k}$-geodesic $\gamma^k:[0,\tau]\to M$
                  satisfying the boundary condition (\ref{e:FinBoundNlambda}) with $\lambda=\lambda_k$'', and
            \item ``constant (non-zero) speed $(F_\lambda, \bf N)$-geodesics'',
        \end{itemize}
        with ``$h_{\lambda_k}$-geodesic $\gamma^k:[0,\tau]\to M$ satisfying the boundary condition
        (\ref{e:FinBoundNlambdaRiem}) with $\lambda=\lambda_k$'' and ``$(h_\lambda, \bf N)$-geodesics'',
        respectively.
	\end{enumerate}
\end{theorem}

For ${\bf N}=S_0\times S_1$ [resp. ${\bf N}={\rm Graph}(\mathbb{I}_{g})$]
Theorem~1.5 in \cite{Lu12-} or \cite{Lu12}
 (resp. Theorem~1.14 in \cite{Lu12-} or \cite{Lu12})
 directly leads to:

\begin{theorem}[\textsf{Existence for bifurcations}]\label{th:bif-existFin1Riem}
Under Assumption~\ref{ass:Rieman1} with $\mathbf{N} = S_0 \times S_1$ or $\mathrm{Graph}(\mathbb{I}_g)$, the conclusion of Theorem~\ref{th:bif-existFin1} holds with the following modifications: the geodesics are $h_{\lambda_k}$-geodesics $\gamma^k: [0, \tau] \to M$ satisfying the boundary condition (\ref{e:FinBoundNlambdaRiem}).
\end{theorem}


Similarly, for ${\bf N}=S_0\times S_1$ [resp. ${\bf N}={\rm Graph}(\mathbb{I}_{g})$]
from Theorem~1.6 in \cite{Lu12-} or \cite{Lu12}
(resp. Theorem~1.15 in \cite{Lu12-} or \cite{Lu12})
we directly derive:

\begin{theorem}[\textsf{Alternative bifurcations of Rabinowitz's type}]\label{th:bif-suffFinRiem}
    Under Assumption~\ref{ass:Rieman1} with $\mathbf{N}=S_0\times S_1$ or $\mathrm{Graph}(\mathbb{I}_{g})$, the conclusions of Theorem~\ref{th:bif-suffFin} hold with the following modifications:
    \begin{enumerate}
        \item[\rm (i)] In statement (i) of Theorem~\ref{th:bif-suffFin},
                the phrase ``constant (non-zero) speed $F_{\mu}$-geodesics $\gamma^m \neq \gamma_{\mu}$ satisfying the boundary condition (\ref{e:FinBoundNlambda})'' is replaced by ``$h_{\mu}$-geodesics $\gamma^m \neq \gamma_{\mu}$ satisfying the boundary condition (\ref{e:FinBoundNlambdaRiem})''.

        \item[\rm (ii)] In statement (ii) of Theorem~\ref{th:bif-suffFin},
        the phrase ``constant (non-zero) speed $F_{\lambda}$-geodesic $\gamma'_{\lambda} \neq \gamma_{\lambda}$ satisfying the boundary condition (\ref{e:FinBoundNlambda})'' is replaced by ``$h_{\lambda}$-geodesic $\gamma'_{\lambda} \neq \gamma_{\lambda}$ satisfying the boundary condition (\ref{e:FinBoundNlambdaRiem})''.

        \item[\rm (iii)] In statement (iii) of Theorem~\ref{th:bif-suffFin},
       replace the phrases
        \begin{itemize}
            \item  ``constant (non-zero) speed $F_{\lambda}$-geodesics satisfying the boundary condition (\ref{e:FinBoundNlambda}), $\gamma_{\lambda}^1 \neq \gamma_{\lambda}$ and $\gamma_{\lambda}^2 \neq \gamma_{\lambda}$, which can also be chosen to satisfy $F_{\lambda}(\gamma^1_\lambda(t), \dot{\gamma}^1_\lambda(t)) \neq F_{\lambda}(\gamma^2_\lambda(t), \dot{\gamma}^2_\lambda(t)) \; \forall t$'',  and
            \item ``constant (non-zero) speed $F_{\lambda}$-geodesics satisfying the boundary condition (\ref{e:FinBoundNlambda})''
        \end{itemize}
     with  ``$h_{\lambda}$-geodesics satisfying the boundary condition (\ref{e:FinBoundNlambdaRiem}), $\gamma_{\lambda}^1 \neq \gamma_{\lambda}$ and $\gamma_{\lambda}^2 \neq \gamma_{\lambda}$, which can also be chosen to satisfy $h_{\lambda}(\gamma^1_\lambda(t), \dot{\gamma}^1_\lambda(t)) \neq h_{\lambda}(\gamma^2_\lambda(t), \dot{\gamma}^2_\lambda(t)) \; \forall t$''
           and ``$h_{\lambda}$-geodesics satisfying the boundary condition (\ref{e:FinBoundNlambdaRiem})'',
        respectively.
    \end{enumerate}
    In summary, all occurrences of ``Assumptions~\ref{ass:Fin1}, \ref{ass:Fin2}'' are replaced by ``Assumption~\ref{ass:Rieman1}'', and all geodesics mentioned are $h_{\lambda}$-geodesics (or $h_{\mu}$-geodesics) satisfying the boundary condition (\ref{e:FinBoundNlambdaRiem}).
\end{theorem}

\section{Bifurcations of $\mathbb{I}_{g}$-invariant geodesics}\label{sec:geodesics2}

For  an  $\mathbb{I}_{g}$-invariant Finsler metric $F$ on $M$,
a $F$-geodesic $\gamma:\mathbb{R}\to M$  said to be
\textsf{$\mathbb{I}_{g}$-invariant} if $\gamma(t+1)=\mathbb{I}_{g}(\gamma(t))\;\forall t\in\R$.
Clearly, $s\cdot\gamma$ is also an  $\mathbb{I}_{g}$-invariant $F$-geodesics for any $s\in\mathbb{R}$,
where $(s\cdot\gamma)(t)=\gamma(t+s)$ for $t\in\mathbb{R}$.
Two $\mathbb{I}_{g}$-invariant $F$-geodesics $\gamma_1$ and $\gamma_2$ are said to be
\textsf{$\mathbb{R}$-distinct} if there is no $s\in\R$ such that
 $s\cdot\gamma_1=\gamma_2$.
If a $F$-geodesic $\gamma:[0,1]\to M$ satisfies $\mathbb{I}_{g\ast}(\dot\gamma(0))=\dot\gamma(1)$, then it
may be extended into an $\mathbb{I}_{g}$-invariant $F$-geodesic
$\gamma^\star:\R\to M$ via
\begin{equation}\label{e:Fin1.6}
\gamma^\star(t)=\mathbb{I}^{[t]}_{g}(\gamma(t-[t]))\;\forall t\in\R,
\end{equation}
where $[s]$ denotes the greatest integer at most $s$,
 called the \textsf{corresponding (maximal)
$\mathbb{I}_{g}$-invariant $F$-geodesic} (determined by $\gamma$).

 \begin{assumption}\label{ass:Fin3}
{\rm
Under Assumption~\ref{ass:Fin1} with $\ell=6$, all $F_\lambda$ are also $\mathbb{I}_{g}$-invariant,
and $\bar\gamma:\mathbb{R}\to M$ is an $\mathbb{I}_{g}$-invariant constant (non-zero) speed $F_{\lambda}$-geodesic
 for each $\lambda\in\Lambda$.}
 \end{assumption}

Under this assumption,  each element in
$\R\cdot \bar\gamma:=\{\bar\gamma(\theta+\cdot)\,|\,\theta\in\R\}$ (\textsf{$\R$-orbit})
is also an $\mathbb{I}_{g}$-invariant constant (non-zero) speed $F_\lambda$-geodesic for each $\lambda\in\Lambda$.
Because of this reason, similar to Definition~\ref{def:orbitBifur} we have:


\begin{definition}\label{def:FinOrbitInvar}
{\rm
$\R$-orbits of $\mathbb{I}_{g}$-invariant constant (non-zero) speed $F_\lambda$-geodesics with a parameter $\lambda\in\Lambda$
are said to be \textsf{sequentially bifurcating
from the $\R$-orbit $\R\cdot \bar\gamma$ at $\mu$}
if there exists a sequence of parameters $\lambda_k\to\mu$ in $\Lambda$
 and, correspondingly, a sequence of
  $\mathbb{I}_{g}$-invariant constant (non-zero) speed $F_{\lambda_k}$-geodesics
 $\gamma^k$, $k=1,2,\cdots$,
   satisfying the following conditions:
 \begin{enumerate}
\item[\rm (i)] $\gamma^k\notin\R\cdot \bar\gamma$ for all $k$,
\item[\rm (ii)] the  $F_{\lambda_k}$-geodesics $\gamma^k$ are pairwise $\R$-distinct,
\item[\rm (iii)] $\gamma^k|_{[0,1]}\to \bar\gamma|_{[0,1]}$ in $C^1([0,1];M)$.
\end{enumerate}
 [Passing to a subsequence, (i) is implied in (ii).]}
\end{definition}

 By definition of the fundamental tensor $g^{F_\lambda}$ in (\ref{e:fundTensor}),
 $\mathbb{I}_{g}$-invariance of $F_\lambda$ implies that
 $$
 g^{F_\lambda}(\mathbb{I}_{g}(x),\mathbb{I}_{g\ast}(v))[\mathbb{I}_{g\ast}(u)), \mathbb{I}_{g\ast}(w))]=
 g^{F_\lambda}(x,v)[u,w]
 $$
 for all $(x,v)\in TM\setminus 0_{TM}$ and $u,w\in T_xM$.
 The following claim easily follows from these and the  $\mathbb{I}_{g}$-invariance of $g$.

 \begin{claim}\label{cl:Fin1.5}
 For a given sequential compact subset $\tilde{\Lambda}\subset\Lambda$ and
  a given open neighborhood $\mathcal{M}$ of $\bar\gamma([0,1])$
 with compact closure, on the open submanifold
 $\tilde{M}:=\cup_{k\in\mathbb{Z}}(\mathbb{I}_{g})^k(\mathcal{M})$ of $M$
 the corresponding numbers  defined  by Proposition~\ref{prop:Fin2.2} with $(\Lambda, M)=(\tilde{\Lambda},\tilde{M})$,
  \begin{eqnarray*}
&&\tilde{\alpha}_g:=\inf_{\lambda\in\Lambda}\inf_{(x,v)\in T\tilde{M},\, |v|_x=1}\inf_{u\ne
0}\frac{g^{F_\lambda}_{v}(u,u)}{g_x(u,u)}=\inf_{\lambda\in\Lambda}\inf_{(x,v)\in T\mathcal{M},\, |v|_x=1}\inf_{u\ne
0}\frac{g^{F_\lambda}_{v}(u,u)}{g_x(u,u)}\quad\hbox{and}\\
&&\tilde{\beta}_g:=\sup_{\lambda\in\Lambda}\sup_{(x,v)\in T\tilde{M},\, |v|_x=1}\sup_{u\ne
0}\frac{g^{F_\lambda}_v(u,u)}{g_x(u,u)}=\sup_{\lambda\in\Lambda}\sup_{(x,v)\in T\mathcal{M},\, |v|_x=1}\sup_{u\ne
0}\frac{g^{F_\lambda}_v(u,u)}{g_x(u,u)}
\end{eqnarray*}
are positive, and by scaling down or up $g$ (if necessary) it holds that for some constant $C_1>0$,
\begin{equation*}
|v|^2_x\le L_\lambda(x, v)\le C_1|v|^2_x\quad\forall (\lambda, x,v)\in \tilde{\Lambda}\times T\tilde{M}.
\end{equation*}
 Moreover (since $\bar\gamma$ is $\mathbb{I}_{g}$-invariant) there exists $c>0$ such that
\begin{equation}\label{e:Fin1.9+}
[F_\lambda(\bar\gamma(t), \dot{\bar\gamma}(t))]^2>\frac{2c}{C_1},\quad
\forall (\lambda,t)\in\tilde{\Lambda}\times \mathbb{R}.
 \end{equation}
 \end{claim}

\begin{claim}\label{cl:Fin2}
Under Claim~\ref{cl:Fin1.5}  let $L^\ast_\lambda:T\tilde{M}\to\R$,  $\lambda\in\tilde{\Lambda}$, be given by
 Proposition~\ref{prop:Fin2.2} with $(\Lambda, M)=(\tilde{\Lambda},\tilde{M})$.
Then there exists a neighborhood $\mathscr{U}$ of $\bar{\gamma}|_{[0,1]}$ in $C^1([0,1], M)$ such that
if $\gamma:\mathbb{R}\to M$ is a constant (non-zero) speed $\mathbb{I}_{g}$-invariant $F_\lambda$-geodesic
whose restriction to $[0,1]$ sits in $\mathscr{U}$, where $\lambda\in\tilde\Lambda$,
then it is $C^6$, sits in $\tilde{M}$ and satisfies
\begin{equation}\label{e:FinLagr10*}
\left.\begin{array}{ll}
\frac{d}{dt}\Big(\partial_vL^\ast_\lambda(\gamma(t), \dot{\gamma}(t))\Big)-\partial_x L^\ast_\lambda(\gamma(t), \dot{\gamma}(t))=0\;\forall t\in\mathbb{R},\\
\mathbb{I}_g(\gamma(t))=\gamma(t+1)\quad\forall t\in\mathbb{R}.
\end{array}\right\}
\end{equation}
Conversely, for a solution $\gamma$ of (\ref{e:FinLagr10*}), which must be $C^6$, if $\gamma|_{[0,1]}$
is in $\mathscr{U}$,
   then $\gamma$ is a constant (non-zero) speed $\mathbb{I}_{g}$-invariant $F_\lambda$-geodesic.
\end{claim}
\begin{proof}[\bf Proof]
 Since $\tilde{\Lambda}\subset\Lambda$ is sequential compact, it follows from
 (\ref{e:Fin1.9+}) that there exists a neighborhood $\mathscr{U}$
 of $\bar\gamma|_{[0,1]}$ in $C^1([0,1], M)$ such that
if the restriction of an $\mathbb{I}_{g}$-invariant $C^1$ curve $\gamma:\R\to M$
to $[0,1]$ belongs to $\mathscr{U}$ then
$[F_\lambda(\gamma(t), \dot{\gamma}(t))]^2>\frac{2c}{3C_1}$
for all $(\lambda,t)\in\tilde{\Lambda}\times \mathbb{R}$.
 Therefore for an $\mathbb{I}_{g}$-invariant $C^2$ curve $\gamma:\R\to M$, if $\gamma|_{[0,1]}$
is in $\mathscr{U}$  then $\gamma$ satisfies
 $$
 \frac{d}{dt}\Big(\partial_vL^\ast_\lambda(\gamma(t), \dot{\gamma}(t))\Big)-\partial_x L^\ast_\lambda(\gamma(t), \dot{\gamma}(t))=
\frac{d}{dt}\Big(\partial_vL_\lambda(\gamma(t), \dot{\gamma}(t))\Big)-\partial_x L_\lambda(\gamma(t), \dot{\gamma}(t))
 $$
 by Proposition~\ref{prop:Fin2.2}(i) with $(\Lambda, M)=(\tilde{\Lambda},\tilde{M})$.
 These imply the desired conclusions.
\end{proof}

 Let $\mathcal{X}^1_{\tau}(\tilde{M}, \mathbb{I}_g)$ be  the $C^4$ Banach manifold defined in (\ref{e:BanachPerMani}), which
 may be identified with  $C^{1}_{\mathbb{I}_g}([0, 1]; \tilde{M})$. Define functionals
\begin{eqnarray*}
&&{\cal E}_{\lambda, \mathbb{I}_g}: \mathcal{X}^1_{\tau}(\tilde{M}, \mathbb{I}_g)\to\mathbb{R},\;  \gamma \mapsto\int^1_0[F_\lambda(\gamma(t), \dot{\gamma}(t))]^2dt,\\
&&{\cal E}^\ast_{\lambda, \mathbb{I}_g}: \mathcal{X}^1_{\tau}(\tilde{M}, \mathbb{I}_g)\to\mathbb{R},\; \gamma \mapsto\int^1_0L^\ast_\lambda(\gamma(t), \dot{\gamma}(t))dt.
\end{eqnarray*}
Clearly, they agree on  the  open subset
\begin{equation*}
C^{1}(\mathbb{R}; \tilde{M}, \mathbb{I}_g, \{F_\lambda\,|\,\lambda\in\tilde{\Lambda}\}, {c/C_1}):=\left\{\alpha\in \mathcal{X}^1_{\tau}(\tilde{M}, \mathbb{I}_g)\,\Big|\,
\min_{(\lambda,t)\in\tilde{\Lambda}\times [0,1]}[F_\lambda(\alpha(t), \dot{\alpha}(t))]^2>2c/C_1\right\}
\end{equation*}
of $\mathcal{X}^1_{\tau}(\tilde{M}, \mathbb{I}_g)$ containing $\bar\gamma$,
and their critical points on $\mathcal{X}^1_{\tau}(\tilde{M}, \mathbb{I}_g)$ near $\bar\gamma$
correspond, respectively, to constant speed $\mathbb{I}_{g}$-invariant $F_\lambda$-geodesics near $\bar\gamma$ in $\tilde{M}$ and
solutions of (\ref{e:FinLagr10*}) near $\bar\gamma$ in $\tilde{M}$.
 For each $\lambda\in\tilde{\Lambda}$  we write
\begin{eqnarray*}
&&m^-(\mathcal{E}_{\lambda,\mathbb{I}_g}, \bar{\gamma}):=m^-(\mathcal{E}_{\lambda,\mathbb{I}_g}, \bar{\gamma}|_{[0,1]})\quad\hbox{and}\quad
m^0(\mathcal{E}_{\lambda,\mathbb{I}_g}, \bar{\gamma}):=m^0(\mathcal{E}_{\lambda,\mathbb{I}_g}\bar{\gamma}|_{[0,1]}),\\
&&m^-(\mathcal{E}^\ast_{\lambda,\mathbb{I}_g}, \bar{\gamma}):=m^-(\mathcal{E}^\ast_{\lambda,\mathbb{I}_g}, \bar{\gamma}|_{[0,1]})\quad\hbox{and}\quad
m^0(\mathcal{E}^\ast_{\lambda,\mathbb{I}_g}, \bar{\gamma}):=m^0(\mathcal{E}^\ast_{\lambda,\mathbb{I}_g},\bar{\gamma}|_{[0,1]}).
\end{eqnarray*}
Then $m^-(\mathcal{E}^\ast_{\lambda,\mathbb{I}_g}, \bar{\gamma})=m^-(\mathcal{E}_{\lambda,\mathbb{I}_g}, \bar{\gamma})$ and
$m^0(\mathcal{E}^\ast_{\lambda,\mathbb{I}_g}, \bar{\gamma})=m^0(\mathcal{E}_{\lambda,\mathbb{I}_g}, \bar{\gamma})$.

\begin{theorem}[\textsf{Necessary condition}]\label{th:bif-ness-orbitLagrManFin}
Under Assumptions~\ref{ass:Fin3}, if $\R$-orbits of $\mathbb{I}_{g}$-invariant constant (non-zero)
speed $F_\lambda$-geodesics with a parameter $\lambda\in\Lambda$
sequentially bifurcate
from the $\R$-orbit $\R\cdot \bar\gamma$ at $\mu$,
 then $m^0(\mathcal{E}_{\mu,\mathbb{I}_g}, \bar{\gamma})\ge 2$.
\end{theorem}
\begin{proof}[\bf Proof]
By the assumption there exists a sequence $(\lambda_k)\subset\Lambda$  converging to $\mu\in\Lambda$,
and constant (non-zero) speed $\mathbb{I}_{g}$-invariant $F_{\lambda_k}$-geodesics
$\gamma^k:\mathbb{R}\to M$ (which must be $C^6$), $k=1,2,\cdots$, such that
these $\gamma^k$ are pairwise $\R$-distinct and satisfy
$\gamma^k|_{[0,1]}\to \bar{\gamma}|_{[0,1]}$ in $C^1([0, 1];M)$.
Let $\tilde{\Lambda}=\{\mu,\lambda_k\,|\,k\in\mathbb{N}\}$, which is sequential compact.
Choose $\tilde{M}$ as in Claim~\ref{cl:Fin1.5}, and $L^\ast_\lambda:T\tilde{M}\to\R$,  $\lambda\in\tilde{\Lambda}$,
as in Claim~\ref{cl:Fin2}. Since $\gamma^k|_{[0,1]}\to \bar{\gamma}|_{[0,1]}$ in $C^1([0, 1];M)$,
for the neighborhood $\mathscr{U}$ of $\bar{\gamma}|_{[0,1]}$ in $C^1([0,1], M)$
in Claim~\ref{cl:Fin2}, we can assume that all $\gamma^k|_{[0,1]}$ belong to $\mathscr{U}$.
By Claim~\ref{cl:Fin2} each $\gamma^k$ is $C^6$, sits in $\tilde{M}$ and satisfies
(\ref{e:FinLagr10*}) with $\lambda=\lambda_k$.
Then Theorem~\ref{th:bif-ness-orbitLagrMan} concludes  $m^0(\mathcal{E}^\ast_{\lambda,\mathbb{I}_g}, \bar{\gamma})\ge 2$
and so $m^0(\mathcal{E}_{\lambda,\mathbb{I}_g}, \bar{\gamma})\ge 2$.
\end{proof}

\begin{theorem}[\textsf{Sufficient condition}]\label{th:bif-suffict1-orbitLagrManFin}
Under Assumption~\ref{ass:Fin3} suppose  the following conditions hold.
\begin{enumerate}
\item[\rm (a)]  $\bar\gamma$ is periodic, and $m^0(\mathcal{E}_{\mu,\mathbb{I}_g}, \bar\gamma)\ge 2$.
\item[\rm (b)]  There exist two sequences in  $\Lambda$ converging to $\mu$, $(\lambda_k^-)$ and
$(\lambda_k^+)$,  such that for each $k\in\mathbb{N}$,
 \begin{eqnarray*}
 &&[m^-_\tau(\mathcal{E}_{\lambda_k^-,\mathbb{I}_g}, \bar\gamma), m^-_\tau(\mathcal{E}_{\lambda_k^-,\mathbb{I}_g}, \bar\gamma)+
 m^0_\tau(\mathcal{E}_{\lambda_k^-,\mathbb{I}_g}, \bar\gamma)-1]\\
&& \cap[m^-_\tau(\mathcal{E}_{\lambda_k^+,\mathbb{I}_g}, \bar\gamma),
 m^-_\tau(\mathcal{E}_{\lambda_k^+,\mathbb{I}_g},
 \bar\gamma)+m^0_\tau(\mathcal{E}_{\lambda_k^+,\mathbb{I}_g}, \bar\gamma)-1]=\emptyset
 \end{eqnarray*}
 and either $m^0_\tau(\mathcal{E}_{\lambda_k^-,\mathbb{I}_g}, \bar\gamma)=1$ or $m^0_\tau(\mathcal{E}_{\lambda_k^+,\mathbb{I}_g}, \bar\gamma)=1$.
 \item[\rm (c)] For any constant (non-zero) speed $\mathbb{I}_{g}$-invariant $F_{\mu}$-geodesic
$\gamma:\mathbb{R}\to M$,  if there exists a sequence $(s_k)$ of reals such that  $s_k\cdot\gamma$
  converges to  $\bar\gamma$ on any compact interval $I\subset\R$ in $C^1$-topology, then $\gamma$ is periodic.
  (Clearly, this holds if $(\mathbb{I}_g)^l=id_M$ for some $l\in\mathbb{N}$.)
   \end{enumerate}
  Then there exists a sequence $(\lambda_k)\subset\tilde{\Lambda}:=\{\mu,\lambda^+_k, \lambda^-_k\,|\,k\in\mathbb{N}\}$
  converging to $\mu$ and $C^6$  constant (non-zero) speed
   $\mathbb{I}_{g}$-invariant $F_{\lambda_k}$-geodesics
$\gamma^k:\mathbb{R}\to M$, $k=1,2,\cdots$, such that
any two of these $\gamma_k$ are $\R$-distinct and
  that $\gamma^k|_{[0,1]}\to \bar{\gamma}|_{[0,1]}$ in $C^2([0, 1];M)$.
 In particular, $\R$-orbits of $\mathbb{I}_{g}$-invariant constant (non-zero)
speed $F_\lambda$-geodesics with a parameter $\lambda\in\Lambda$
sequentially bifurcate
from the $\R$-orbit $\R\cdot \bar\gamma$ at $\mu$.
  \end{theorem}
\begin{proof}[\bf Proof]
Since $\tilde{\Lambda}=\{\mu,\lambda^+_k, \lambda^-_k\,|\,k\in\mathbb{N}\}$
is compact and sequential compact,
we may choose $\tilde{M}$ as in Claim~\ref{cl:Fin1.5}, and $L^\ast_\lambda:T\tilde{M}\to\R$,  $\lambda\in\tilde{\Lambda}$,
as in Claim~\ref{cl:Fin2}. Clearly, the assumptions (a) and (b) lead to
the corresponding conditions (a) and (b) in Theorem~\ref{th:bif-suffict1-orbitLagrMan} with  $(M, L_\lambda)=(\tilde{M}, L^\ast_\lambda)$, respectively.

For any solution $\gamma$ of (\ref{e:Lagr10*}) with $(M, L_\lambda)=(\tilde{M}, L^\ast_\lambda)$ and $\lambda=\mu$,
that is, a solution of (\ref{e:FinLagr10*}) with $\lambda=\mu$,
 suppose that there exists a sequence $(s_k)$ of reals such that  $s_k\cdot\gamma$
  converges to  $\bar\gamma$ on any compact interval $I\subset\mathbb{R}$ in $C^1$-topology.
  Then $s_k\cdot\gamma|_{[0,1]}\to \bar{\gamma}|_{[0,1]}$ in $C^1([0, 1];M)$.
  By Claim~\ref{cl:Fin2}, for each $k$ large enough,
  $s_k\cdot\gamma$ and hence $\gamma$ is a constant (non-zero) speed $\mathbb{I}_{g}$-invariant $F_\mu$-geodesic.
  The assumption (c) assures that $\gamma$ is periodic.
Hence the condition (c) in Theorem~\ref{th:bif-suffict1-orbitLagrMan}
with  $(M, L_\lambda)=(\tilde{M}, L^\ast_\lambda)$
is satisfied.

By Theorem~~\ref{th:bif-suffict1-orbitLagrMan}
   there exists a sequence $(\lambda_k)\subset\tilde{\Lambda}$ converging to $\mu$ and
 $C^6$ solutions $\gamma^k$ of  (\ref{e:FinLagr10*}) with $\lambda=\lambda_k$, $k=1,2,\cdots$,
  such that any two of these $\gamma^k$ are $\mathbb{R}$-distinct and that
  $(\gamma^k)$ converges to  $\bar\gamma$ on any compact interval $I\subset\mathbb{R}$
  in $C^2$-topology as $k\to\infty$.
  By Claim~\ref{cl:Fin2}, for each $k$ large enough $\gamma^k$ is a constant (non-zero) speed
  $\mathbb{I}_{g}$-invariant $F_{\lambda_k}$-geodesic.
\end{proof}

\begin{theorem}[\hbox{\textsf{Existence for bifurcations}}]\label{th:bif-existence-orbitLagrManFin}
Under Assumption~\ref{ass:Fin3}, suppose that $\Lambda$ is path-connected,
$(\mathbb{I}_g)^l=id_M$ for some $l\in\mathbb{N}$,
and the following is satisfied:
\begin{enumerate}
\item[\rm (d)] There exist two  points $\lambda^+, \lambda^-\in\Lambda$ such that
\begin{eqnarray*}
&& [m^-_\tau(\mathcal{E}_{\lambda^-,\mathbb{I}_g}, \bar\gamma), m^-_\tau(\mathcal{E}_{\lambda^-,\mathbb{I}_g}, \bar\gamma)+
 m^0_\tau(\mathcal{E}_{\lambda^-,\mathbb{I}_g}, \bar\gamma)-1]\\
 &&\cap[m^-_\tau(\mathcal{E}_{\lambda^+,\mathbb{I}_g}, \bar\gamma),
 m^-_\tau(\mathcal{E}_{\lambda^+,\mathbb{I}_g}, \bar\gamma)+m^0_\tau(\mathcal{E}_{\lambda^+,\mathbb{I}_g}, \bar\gamma)-1]=\emptyset
 \end{eqnarray*}
 and either $m^0_\tau(\mathcal{E}_{\lambda^-,\mathbb{I}_g}, \bar\gamma)=1$ or $m^0_\tau(\mathcal{E}_{\lambda^+,\mathbb{I}_g}, \bar\gamma)=1$.
  \end{enumerate}
 Then for any path $\alpha:[0,1]\to\Lambda$ connecting $\lambda^+$ to $\lambda^-$ there exists a sequence $(\lambda_k)$ in $\alpha([0,1])$ converging to
 $\mu\in \alpha([0,1])\subset\Lambda$, and
 $C^6$ constant (non-zero) speed $\mathbb{I}_{g}$-invariant $F_{\lambda_k}$-geodesics,
  $k=1,2,\cdots$, such that any two of these $\gamma_k$ are $\R$-distinct and that
  $(\gamma_k)$ converges to  $\bar\gamma$ on any compact interval $I\subset\R$ in $C^2$-topology as $k\to\infty$.
   Moreover, this $\mu$ is not equal to $\lambda^+$ {\rm (}resp. $\lambda^-${\rm )} if
   $m^0_\tau(\mathcal{E}_{\lambda^+,\mathbb{I}_g}, \bar\gamma)=1$ {\rm (}resp. $m^0_\tau(\mathcal{E}_{\lambda^-,\mathbb{I}_g}, \bar\gamma)=1${\rm )}.
 \end{theorem}
\begin{proof}[\bf Proof]
Since  $\tilde{\Lambda}:=\alpha([0,1])$ is a compact and sequential compact subset in $\Lambda$
we may choose $\tilde{M}$ as in Claim~\ref{cl:Fin1.5}, and $L^\ast_\lambda:T\tilde{M}\to\R$,  $\lambda\in\tilde{\Lambda}$,
as in Claim~\ref{cl:Fin2}.  Then the assumption (d) yields the corresponding condition (d)  in Theorem~\ref{th:bif-existence-orbitLagrMan}
 with  $(M, L_\lambda)=(\tilde{M}, L^\ast_\lambda)$.
Hence  there exists  a sequence $(\lambda_k)$ in $\alpha([0,1])$ converging to $\mu\in \alpha([0,1])\subset\Lambda$,
and $C^6$ solutions $\gamma_k$ of the corresponding problem (\ref{e:Lagr10*}) on $(M, L_\lambda)=(\tilde{M}, L^\ast_\lambda)$
 with $\lambda=\lambda_k$, $k=1,2,\cdots$,
   such that any two of these $\gamma_k$ are $\R$-distinct and that
  $(\gamma_k)$ converges to  $\bar\gamma$ on any compact interval $I\subset\R$ in $C^2$-topology as $k\to\infty$.
   Moreover, this $\mu$ is not equal to $\lambda^+$ (resp. $\lambda^-$) if
   $m^0_\tau(\mathcal{E}_{\lambda^+,\mathbb{I}_g}, \bar\gamma)=1$ (resp. $m^0_\tau(\mathcal{E}_{\lambda^-,\mathbb{I}_g}, \bar\gamma)=1$).
The required results easily follow from these as before.
\end{proof}

\begin{theorem}[\textsf{Alternative bifurcations of Rabinowitz's type}]\label{th:bif-suffict-orbitLagrManFin}
Under Assumption~\ref{ass:Fin3} with $\Lambda$ being a real interval, suppose
\begin{enumerate}
\item[\rm (a)] $\mu\in{\rm Int}(\Lambda)$, $\mathbb{I}_g=id_M$ and $\bar\gamma$ have least period $1$;
\item[\rm (b)] $m^0(\mathcal{E}_{\mu,\mathbb{I}_g}, \bar\gamma)\ge 2$,  $m^0(\mathcal{E}_{\lambda,\mathbb{I}_g}, \bar\gamma)=1$
 for each $\lambda\in\Lambda\setminus\{\mu\}$ near $\mu$;
\item[\rm (c)] $m^-(\mathcal{E}_{\lambda,\mathbb{I}_g}, \bar\gamma)$  take, respectively, values $m^0(\mathcal{E}_{\mu,\mathbb{I}_g}, \bar\gamma)$ and
 $m^-(\mathcal{E}_{\mu,\mathbb{I}_g}, \bar\gamma)+ m^0(\mathcal{E}_{\mu,\mathbb{I}_g}, \bar\gamma)-1$
 as $\lambda\in\Lambda$ varies in two deleted half neighborhoods  of $\mu$.
\end{enumerate}
 Then  one of the following alternatives occurs:
\begin{enumerate}
\item[\rm (i)]
There exists a sequence   $C^6$  constant {\rm (}non-zero{\rm )} speed $\mathbb{I}_{g}$-invariant $F_{\mu}$-geodesic
$\gamma^k:\mathbb{R}\to M$, $k=1,2,\cdots$, such that these $\gamma^k$ are pairwise $\R$-distinct  and  converge to $\bar\gamma$ on any compact interval $I\subset\R$ in $C^2$-topology as $k\to\infty$.

\item[\rm (ii)]  For every $\lambda\in\Lambda\setminus\{\mu\}$ near $\mu$ there is a $C^6$
constant {\rm (}non-zero{\rm )} speed $\mathbb{I}_{g}$-invariant $F_{\lambda}$-geodesic  $\gamma^\lambda$,
which is $\R$-distinct with $\bar\gamma$ and
converges to $\bar\gamma$ on any compact interval $I\subset\R$ in $C^2$-topology as $\lambda\to \mu$.

\item[\rm (iii)] For a given neighborhood $\mathcal{W}$ of $\bar\gamma$ in $C^{1}(\mathbb{R}; {M}, \mathbb{I}_g)$, there exists
 a one-sided  neighborhood $\Lambda^0$ of $\mu$ such that
for any $\lambda\in\Lambda^0\setminus\{\mu\}$, there exist at least two $\R$-distinct $C^6$
constant {\rm (}non-zero{\rm )} speed $\mathbb{I}_{g}$-invariant $F_{\lambda}$-geodesics,
$\gamma_\lambda^1\notin\mathbb{R}\cdot\bar\gamma$ and $\gamma_\lambda^2\notin\mathbb{R}\cdot\bar\gamma$,
which can also be chosen to have different $F_\lambda$-speeds
{\rm (}i.e., $F_\lambda(\gamma^1_\lambda(t), \dot{\gamma}^1_\lambda(t))\ne F_\lambda(\gamma^2_\lambda(t),
\dot{\gamma}^2_\lambda(t))\;\forall t${\rm )}
provided that $m^0_\tau(\mathscr{E}_{\mu,\mathbb{I}_g}, \bar\gamma)\ge 3$ and
there exist only finitely many $\mathbb{R}$-distinct $C^6$
constant {\rm (}non-zero{\rm )} speed $\mathbb{I}_{g}$-invariant $F_{\lambda}$-geodesics
 in $\mathcal{W}$ which are  $\mathbb{R}$-distinct from $\bar\gamma$.
\end{enumerate}
\end{theorem}

\begin{proof}[\bf Proof]
As in the proof of Theorem~\ref{th:bif-suffFin} we have a number $\rho>0$ such that
the compact and sequential compact set $\tilde{\Lambda}:=[\mu-\rho, \mu+\rho]\subset\Lambda$.
 Choose $\tilde{M}$ as in Claim~\ref{cl:Fin1.5}, and $L^\ast_\lambda:T\tilde{M}\to\R$,  $\lambda\in\tilde{\Lambda}$,
as in Claim~\ref{cl:Fin2}. Since  $m^-(\mathcal{E}^\ast_{\lambda,\mathbb{I}_g},\bar{\gamma})=m^-(\mathcal{E}_{\lambda,\mathbb{I}_g},
\bar{\gamma})$ and
$m^0(\mathcal{E}^\ast_{\lambda,\mathbb{I}_g}, \bar{\gamma})=m^0(\mathcal{E}_{\lambda,\mathbb{I}_g}, \bar{\gamma})$ for each $\lambda\in\tilde\Lambda$,
 the assumptions (b) and (c) imply  that
 $m^0(\mathcal{E}^\ast_{\mu,\mathbb{I}_g}, \bar\gamma)\ge 2$,  $m^0(\mathcal{E}^\ast_{\lambda,\mathbb{I}_g}, \bar\gamma)=1$
 for each $\lambda\in\tilde\Lambda\setminus\{\mu\}$ near $\mu$ and that
 $m^-(\mathcal{E}^\ast_{\lambda,\mathbb{I}_g}, \bar\gamma)$  take, respectively, values $m^0(\mathcal{E}^\ast_{\mu,\mathbb{I}_g}, \bar\gamma)$ and
 $m^-(\mathcal{E}^\ast_{\mu,\mathbb{I}_g}, \bar\gamma)+ m^0(\mathcal{E}^\ast_{\mu,\mathbb{I}_g}, \bar\gamma)-1$
 as $\lambda\in\tilde\Lambda$ varies in two deleted half neighborhoods  of $\mu$.
 Hence the desired results may follow from Theorem~\ref{th:bif-suffict-orbitLagrMan} and Claim~\ref{cl:Fin2}
 as above.
\end{proof}

\begin{remark}\label{rm:Fin}
{\rm As noted below Theorem~\ref{th:bif-suffict-orbitLagrMan},
if $M$ is an open subset $U$ of $\mathbb{R}^n$ and $\mathbb{I}_g$ is an  orthogonal matrix $E$ of order $n$
  which maintain $U$ invariant, ``Assumption~\ref{ass:Fin3}''
  and  all ``$C^6$'' in  Theorem~\ref{th:bif-ness-orbitLagrManFin},~\ref{th:bif-suffict1-orbitLagrManFin},~\ref{th:bif-suffict-orbitLagrManFin}
can be replaced by ``Assumption~\ref{ass:Fin3}  with $\ell=4$''
  and  ``$C^4$'' respectively.
}
\end{remark}

\section{Bifurcations of reversible geodesics}\label{sec:geodesics3}

For a reversible Finsler metric $F$ on $M$,
and any geodesic of $F$,
 $\gamma:(-r, r)\to M$, the reverse curve $\gamma^{-}:(-r,r)\to M$ defined by
 $\gamma^{-}(t)=\gamma(-t)$  is a geodesic of $F$ that coincide pointwise with $\gamma$. (See \cite[Remark~3.1]{MaSaSh10}).
 The irreversibility of a Finsler metric  is a very strong restriction that excludes a lot of interesting examples, for instance Randers metrics (see Section~\ref{sec:Randers} for the definition).

\begin{assumption}\label{ass:Fin4}
{\rm
Under Assumption~\ref{ass:Fin1},
for each $\lambda\in\Lambda$ suppose that $F_\lambda$ is reversible and
that $\gamma_\lambda:\mathbb{R}\to M$ is a $1$-periodic
 $F_\lambda$-geodesic of constant (non-zero) speed, (which must be $C^\ell$
 and satisfies $\gamma_\lambda(-t)=\gamma_\lambda(t)$ for all $t\in\mathbb{R}$).
It is also required  that the maps $\Lambda\times \mathbb{R}\ni(\lambda,t)\to \gamma_\lambda(t)\in M$ and
$\Lambda\times \mathbb{R}\ni(\lambda, t)\mapsto \dot{\gamma}_\lambda(t)\in TM$ are continuous.}
 \end{assumption}

\begin{definition}\label{def:BifurFinRever}
{\rm
Under  Assumption~\ref{ass:Fin4}, $1$-periodic
$F_\lambda$-geodesics of constant (non-zero) speed  with a parameter $\lambda\in\Lambda$
are said \textsf{to bifurcate at $\mu\in\Lambda$ along sequences}
(with respect to the branch
$\{\gamma_\lambda\,|\,\lambda\in\Lambda\}$)
if  there exists an infinite sequence $\{(\lambda_k, \gamma^k)\}^\infty_{k=1}$
  in $\Lambda\times EC^1(S_1, M)\setminus\{(\mu,\gamma_\mu)\}$
  converging to $(\mu,\gamma_\mu)$, such that each $\gamma^k\ne\gamma_{\lambda_k}$
  is a $1$-periodic  $F_{\lambda_k}$-geodesic of constant (non-zero) speed, $k=1,2,\cdots$.
  (Actually it is not hard to prove that  $\gamma^k\to\gamma_\mu$ in $C^2(S_1, M)$.)}
  \end{definition}

Under Assumption~\ref{ass:Fin4} nonconstant critical points of the $C^{2-0}$-functional
\begin{equation}\label{e:ReFinEnergy}
\gamma \mapsto\mathcal{E}^E_\lambda(\gamma)=\int^1_0[F_\lambda(\gamma(t), \dot{\gamma}(t))]^2dt
\end{equation}
 on the Banach manifold
   \begin{equation*}
 EC^{1}(S_\tau; M):=\{\gamma\in C^1(\R;M)\,|\,\gamma(t+1)=\gamma(t)\,\&\, \gamma(-t)=\gamma(t)\;\forall t\in\R\}
 \end{equation*}
 correspond to $1$-periodic $F_\lambda$-geodesics of constant (non-zero) speed.
Since $\mathcal{E}^E_\lambda$ is $C^2$ near such a critical point $\gamma$,  the Morse index and nullity
$m^-(\mathcal{E}^E_\lambda,\gamma)$ and $m^0(\mathcal{E}^E_\lambda,\gamma)$ are well-defined as before.

\begin{theorem}\label{th:bif-nessLagrBrakeMFin}
Let Assumption~\ref{ass:Fin4}  be satisfied, $\mu\in\Lambda$.
\begin{enumerate}
\item[{\rm (I)}]{\rm (\textsf{Necessary condition}):} Suppose that
$1$-periodic
 $F_\lambda$-geodesics of constant (non-zero) speed  with a parameter $\lambda\in\Lambda$
bifurcate at $\mu\in\Lambda$ along sequences with respect to the branch
$\{\gamma_\lambda\,|\,\lambda\in\Lambda\}$.
Then $m^0(\mathcal{E}^{E}_{\mu}, \gamma_\mu)\ge 1$.
\item[{\rm (II)}]{\rm (\textsf{Sufficient condition}):}
Let $\Lambda$ be first countable.
Suppose  that there exist two sequences in  $\Lambda$ converging to $\mu$, $(\lambda_k^-)$ and
$(\lambda_k^+)$,  such that
one of the following conditions is satisfied:
 \begin{enumerate}
 \item[\rm (II.1)] For each $k\in\mathbb{N}$, either $\gamma_{\lambda^+_k}$  is not an isolated critical point of
 $\mathcal{E}^E_{\lambda^+_k}$,
 or $\gamma_{\lambda^-_k}$ is not an isolated critical point of $\mathcal{E}^E_{\lambda^-_k}$,
 or $\gamma_{\lambda^+_k}$ {\rm (}resp. $\gamma_{\lambda^-_k}${\rm )} is an isolated critical point of $\mathcal{E}^E_{\lambda^+_k}$
 {\rm (}resp. $\mathcal{E}^E_{\lambda^-_k}${\rm )} and
  $C_m(\mathcal{E}^E_{\lambda^+_k}, \gamma_{\lambda^+_k};{\bf K})$ and $C_m(\mathcal{E}^E_{\lambda^-_k}, \gamma_{\lambda^-_k};{\bf K})$
  are not isomorphic for some Abel group ${\bf K}$ and some $m\in\mathbb{Z}$.
\item[\rm (II.2)] For each $k\in\mathbb{N}$, there exists $\lambda\in\{\lambda^+_k, \lambda^-_k\}$
such that $\gamma_{\lambda}$  is an either non-isolated or homological visible critical point of
$\mathcal{E}^E_{\lambda}$ , and
$$
\left.\begin{array}{ll}
&[m^-(\mathcal{E}^E_{\lambda_k^-}, \gamma_{\lambda^-_k}), m^-(\mathcal{E}^E_{\lambda_k^-}, \gamma_{\lambda^-_k})+
m^0(\mathcal{E}^E_{\lambda_k^-}, \gamma_{\lambda^-_k})]\\
&\cap[m^-(\mathcal{E}^E_{\lambda_k^+}, \gamma_{\lambda^+_k}),
m^-(\mathcal{E}^E_{\lambda_k^+}, \gamma_{\lambda^+_k})+m^0(\mathcal{E}^E_{\lambda_k^+}, \gamma_{\lambda^+_k})]=\emptyset.
\end{array}\right\}\eqno(\hbox{$\clubsuit_k$})
$$
\item[\rm (II.3)] For each $k\in\mathbb{N}$, (\hbox{$\clubsuit_k$}) holds true,
and either $m^0(\mathcal{E}^E_{\lambda_k^-}, \gamma_{\lambda^-_k})=0$ or $m^0(\mathcal{E}^E_{\lambda_k^+}, \gamma_{\lambda^+_k})=0$.
 \end{enumerate}
Then  there exists a sequence $(\lambda_k)\subset\hat{\Lambda}:=\{\mu,\lambda^+_k, \lambda^-_k\,|\,k\in\mathbb{N}\}$ converging to $\mu$,
   $1$-periodic  $F_{\lambda_k}$-geodesics $\gamma^k\ne \gamma_{\lambda_k}$ of constant {\rm (}non-zero{\rm )} speed, $k=1,2,\cdots$,
   such that $\gamma^k\to \gamma_\mu$ in $C^2(S_1;M)$.
 In particular, $1$-periodic
 $F_\lambda$-geodesics of constant {\rm (}non-zero{\rm )}
  speed  with a parameter $\lambda\in\Lambda$
bifurcate at $\mu\in\Lambda$ along sequences with respect to the branch
$\{\gamma_\lambda\,|\,\lambda\in\Lambda\}$.
 \end{enumerate}
\end{theorem}
\begin{proof}[\bf Proof]
{\it Step 1}[\textsf{Prove} (I)].
By the assumptions there exists a sequence $(\lambda_k)\subset\Lambda$ converging  to $\mu\in\Lambda$
such that for each $k$ there exists a nonconstant  reversible and $1$-periodic
 $F_{\lambda_k}$-geodesic  $\gamma^k\ne\gamma_{\lambda_k}$ to satisfy $\gamma^k\to \gamma_\mu$ in $C^1(S_1;M)$.
  Let $\hat{\Lambda}=\{\mu\}\cup\{\lambda_k\,|\,k\in\mathbb{N}\}$. It is sequential compact.
  Then we can choose an open subset $\hat{M}$ of $M$ with compact closure such that
 $\gamma_\lambda(S_1)\cup\gamma^m(S_1)\subset\hat{M}$ for all $(\lambda,m)\in\hat{\Lambda}\times\mathbb{N}$,
 and therefore the conditions in Proposition~\ref{prop:Fin2.2} can be satisfied
 with $(M, \Lambda)=(\hat{M}, \hat{\Lambda})$.
For the constant $C_1>0$ as in Proposition~\ref{prop:Fin2.2} with $(M, \Lambda)=(\hat{M}, \hat{\Lambda})$ we have $c>0$ such that
for all $(m, \lambda,t)\in\mathbb{N}\times\hat{\Lambda}\times\R$,
\begin{equation*}
[F_\lambda(\gamma_\lambda(t), \dot{\gamma}_\lambda(t))]^2>\frac{2c}{C_1}\quad\hbox{and}\quad
[F_{\lambda_m}(\gamma^m(t), \dot{\gamma}^m(t))]^2>\frac{2c}{C_1}.
 \end{equation*}
Let $L^\ast_\lambda:T\hat{M}\to\R$,  $\lambda\in\hat{\Lambda}$, be given by Proposition~\ref{prop:Fin2.2} with $(M, \Lambda)=(\hat{M}, \hat{\Lambda})$. Then $L^\ast_\lambda(x,v)=L^\ast_\lambda(x,-v)$ for all $(x,v)\in T\hat{M}$,
and on $EC^1(S_1;M)$ the corresponding $C^{2}$ functional $\mathcal{E}^{E\ast}_\lambda$ given by (\ref{e:FinEnergy*}) with $\tau=1$
and the $C^{2-0}$ functional $\mathcal{E}^E_\lambda$ in (\ref{e:FinEnergy}) with $\tau=1$ coincide in the open subset
\begin{equation}\label{e:openSet13.2}
EC^1(S_1, \hat{M}, \{F_\lambda\}, {c/C_1}):=\{\alpha\in EC^1(S_1, \hat{M})\,|\,
\min_{(\lambda,t)\in\hat{\Lambda}\times S_1}[F_\lambda(\alpha(t), \dot{\alpha}(t))]^2>2c/C_1\}
\end{equation}
of $EC^1(S_1;M)$. Moreover, at a critical point $\gamma$ of them on
$EC^1(S_1, \hat{M}, \{F_\lambda\}, {c/C_1})$
it holds that their Morse indexes and nullities satisfy
$m^-(\mathcal{E}^{E\ast}_\lambda,\gamma)=m^-(\mathcal{E}^E_\lambda,\gamma)$ and $m^0(\mathcal{E}^{E\ast}_\lambda,\gamma)=
m^0(\mathcal{E}^E_\lambda,\gamma)$. As in Step 1 of proof of  Theorem~\ref{th:bif-nessFin}
the desired conclusion may follow from \cite[Th.1.22(I)]{Lu12-} or \cite[Th.1.29(I)]{Lu12}.

  {\it Step 2}[\textsf{Prove} (II)].
  Since $\hat{\Lambda}=\{\mu,\lambda^+_k, \lambda^-_k\,|\,k\in\mathbb{N}\}$ is sequential compact,
as above we can  find an open subset $\hat{M}$ of $M$ with compact closure such that
 $\gamma_\lambda(S_1)\cup\gamma_{\lambda^+_m}(S_1)\cup\gamma_{\lambda^-_m}(S_1)\subset\hat{M}$ for all $(\lambda,m)\in\hat{\Lambda}\times\mathbb{N}$.
For the constant $C_1>0$ as in Proposition~\ref{prop:Fin2.2} with $(M, \Lambda)=(\hat{M}, \hat{\Lambda})$ we have $c>0$ such that
for all $(m, \lambda,t)\in\mathbb{N}\times\hat{\Lambda}\times \R$,
\begin{equation*}
[F_\lambda(\gamma_\lambda(t), \dot{\gamma}_\lambda(t))]^2>\frac{2c}{C_1}\quad\hbox{and}\quad
[F_{\lambda_m}(\gamma_{\lambda^\pm_m}(t), \dot{\gamma}_{\lambda^\pm_m}(t))]^2>\frac{2c}{C_1}.
 \end{equation*}
Let $L^\ast_\lambda:T\hat{M}\to\R$,  $\lambda\in\hat{\Lambda}$, be given by Proposition~\ref{prop:Fin2.2} with $(M, \Lambda)=(\hat{M}, \hat{\Lambda})$.
As above, for all $(m, \lambda)\in\mathbb{N}\times\hat{\Lambda}$ we have
 $\gamma_\lambda, \gamma_{\lambda^+_m}, \gamma_{\lambda^-_m}\in EC^1(S_1, \hat{M}, \{F_\lambda\}, {c/C_1})$, and
  \begin{eqnarray*}
&&m^-(\mathcal{E}^E_{\lambda^\pm_m},\gamma_{\lambda^\pm_m})=
m^-(\mathcal{E}^{E\ast}_{\lambda^\pm_m},\gamma_{\lambda^\pm_m})
\quad\hbox{and}\quad m^0(\mathcal{E}^E_{\lambda^\pm_m},\gamma_{\lambda^\pm_m})=
m^0(\mathcal{E}^{E\ast}_{\lambda^\pm_m},\gamma_{\lambda^\pm_m}),\\
&&C_k(\mathcal{E}^E_{\lambda^\pm_m},\gamma_{\lambda^\pm_m};{\bf K})=C_k(\mathcal{E}^{E\ast}_{\lambda^\pm_m},\gamma_{\lambda^\pm_m};
{\bf K})\quad\forall (k,m)\in\mathbb{Z}\times\mathbb{N}
\end{eqnarray*}
for any Abel group ${\bf K}$.
The other reasoning may be derived from \cite[Th.1.22(II)]{Lu12-} or \cite[Th.1.29(II)]{Lu12}
as in Step 2 of proof of  Theorem~\ref{th:bif-nessFin}.
  \end{proof}

\begin{theorem}[\textsf{Existence for bifurcations}]\label{th:bif-existFin1Brake}
Under Assumption~\ref{ass:Fin4}, suppose that $\Lambda$ is path-connected and  there exist two  points $\lambda^+, \lambda^-\in\Lambda$ such that
  one of the following conditions is satisfied:
 \begin{enumerate}
 \item[\rm (i)] Either $\gamma_{\lambda^+}$  is not an isolated critical point of $\mathcal{E}^E_{\lambda^+}$,
 or $\gamma_{\lambda^-}$ is not an isolated critical point of $\mathcal{E}^E_{\lambda^-}$,
 or $\gamma_{\lambda^+}$ {\rm (}resp. $\gamma_{\lambda^-}${\rm )}  is an isolated critical point of
  $\mathcal{E}^E_{\lambda^+}$ {\rm (}resp. $\mathcal{E}^E_{\lambda^-}${\rm )} and
  $C_m(\mathcal{E}^E_{\lambda^+}, \gamma_{\lambda^+};{\bf K})$ and $C_m(\mathcal{E}^E_{\lambda^-}, \gamma_{\lambda^-};{\bf K})$ are not isomorphic for some Abel group ${\bf K}$ and some $m\in\mathbb{Z}$.

\item[\rm (ii)] The intervals $[m^-({\mathcal{E}}^E_{\lambda^-}, \gamma_{\lambda^-}),
m^-({\mathcal{E}}^E_{\lambda^-}, \gamma_{\lambda^-})+ m^0({\mathcal{E}}^E_{\lambda^-}, \gamma_{\lambda^-})]$ and
$$
[m^-({\mathcal{E}}^E_{\lambda^+}, \gamma_{\lambda^+}),
m^-({\mathcal{E}}^E_{\lambda^+}, \gamma_{\lambda^+})+m^0({\mathcal{E}}^E_{\lambda^+}, \gamma_{\lambda^+})]
$$
are disjoint, and there exists $\lambda\in\{\lambda^+, \lambda^-\}$ such that $\gamma_{\lambda}$
 is an either non-isolated or homological visible critical point of $\mathcal{E}^E_{\lambda}$.

\item[\rm (iii)] The intervals $[m^-({\mathcal{E}}^E_{\lambda^-}, \gamma_{\lambda^-}),
m^-({\mathcal{E}}^E_{\lambda^-}, \gamma_{\lambda^-})+ m^0({\mathcal{E}}^E_{\lambda^-}, \gamma_{\lambda^-})]$ and
$$
[m^-({\mathcal{E}}^E_{\lambda^+}, \gamma_{\lambda^+}),
m^-({\mathcal{E}}^E_{\lambda^+}, \gamma_{\lambda^+})+m^0({\mathcal{E}}^E_{\lambda^+}, \gamma_{\lambda^+})]
$$
are disjoint, and either $m^0(\mathcal{E}^E_{\lambda^+}, \gamma_{\lambda^+})=0$ or $m^0(\mathcal{E}^E_{\lambda^-}, \gamma_{\lambda^-})=0$.
 \end{enumerate}
  Then for any path $\alpha:[0,1]\to\Lambda$ connecting $\lambda^+$ to $\lambda^-$ there exists
 a sequence $(\lambda_k)$ in $\alpha([0,1])$ converging to some
 $\mu\in \alpha([0,1])$,
  and
  $1$-periodic $F_{\lambda_k}$-geodesics $\gamma^k$
 of constant {\rm (}non-zero{\rm )} speed, $k=1,2,\cdots$, such that $0<\|\gamma^k-\gamma_{\lambda_k}\|_{C^2(S_1;\mathbb{R}^N)}\to 0$
  as $k\to\infty$. Moreover, $\mu$ is not equal to $\lambda^+$ {\rm (}resp. $\lambda^-${\rm )}
   if $m^0_\tau(\mathcal{E}^E_{\lambda^+}, \gamma_{\lambda^+})=0$ {\rm (}resp. $m^0_\tau(\mathcal{E}^E_{\lambda^-}, \gamma_{\lambda^-})=0${\rm )}.
  \end{theorem}

  \begin{proof}[\bf Proof]
Since $\hat{\Lambda}:=\alpha([0,1])$ is  compact and sequential compact subset
 we can  find an open subset $\hat{M}$ of $M$ with compact closure such that the closure
 $Cl\left(\cup_{\lambda\in\hat{\Lambda}\times\mathbb{N}}\gamma_\lambda(S_1)
\right)\subset\hat{M}$.
For the constant $C_1>0$ as in Proposition~\ref{prop:Fin2.2} with $(M, \Lambda)=(\hat{M}, \hat{\Lambda})$ we have $c>0$ such that
\begin{equation*}
[F_\lambda(\gamma_\lambda(t), \dot{\gamma}_\lambda(t))]^2>\frac{2c}{C_1}\quad\hbox{for all $(\lambda,t)\in\hat{\Lambda}\times S_1$}.
 \end{equation*}
Let $L^\ast_\lambda:T\hat{M}\to\R$,  $\lambda\in\hat{\Lambda}$, be given by Proposition~\ref{prop:Fin2.2} with $(M, \Lambda)=(\hat{M}, \hat{\Lambda})$.
Then $L^\ast_\lambda(x,v)=L^\ast_\lambda(x,-v)$ for all $(x,v)\in T\hat{M}$,
and on $EC^1(S_1;M)$ the corresponding $C^{2}$ functional $\mathcal{E}^{E\ast}_\lambda$ given by (\ref{e:FinEnergy*}) with $\tau=1$
and the $C^{2-0}$ functional $\mathcal{E}^E_\lambda$ in (\ref{e:FinEnergy}) with $\tau=1$ coincide in the open subset
$EC^1(S_1, \hat{M}, \{F_\lambda\}, {c/C_1})$ in (\ref{e:openSet13.2}), which contains
$\{\gamma_\lambda\,|\,\lambda\in\hat{\Lambda}\}$. Therefore
\begin{eqnarray*}
&&m^-(\mathcal{E}^E_{\lambda^\pm},\gamma_{\lambda^\pm})=
m^-(\mathcal{E}^{E\ast}_{\lambda^\pm},\gamma_{\lambda^\pm})
\quad\hbox{and}\quad m^0(\mathcal{E}^E_{\lambda^\pm},\gamma_{\lambda^\pm})=
m^0(\mathcal{E}^{E\ast}_{\lambda^\pm},\gamma_{\lambda^\pm}),\\
&&C_k(\mathcal{E}^E_{\lambda^\pm},\gamma_{\lambda^\pm};{\bf K})=C_k(\mathcal{E}^{E\ast}_{\lambda^\pm},\gamma_{\lambda^\pm};
{\bf K})\;\forall k\in\mathbb{Z},\quad\hbox{for any Abel group ${\bf K}$}.
\end{eqnarray*}
The desired results may follow from these and
\cite[Th.1.23]{Lu12-} or \cite[Th.1.30]{Lu12}.
\end{proof}

\begin{theorem}[\textsf{Alternative bifurcations of Rabinowitz's type}]\label{th:bif-suffLagrBrakeMFin}
Under Assumptions~\ref{ass:Fin4}  with $\Lambda$ being a real interval,
let  $\mu\in{\rm Int}(\Lambda)$ satisfy  $m^0(\mathcal{E}^E_{\mu}, \gamma_\mu)\ne 0$.
Suppose that $m^0(\mathcal{E}^E_{\lambda}, \gamma_\lambda)=0$
 for each $\lambda\in\Lambda\setminus\{\mu\}$ near $\mu$, and
that $m^-(\mathcal{E}^E_{\lambda}, \gamma_\lambda)$ take,
respectively, values $m^-(\mathcal{E}^E_{\mu}, \gamma_\mu)$ and
  $m^-(\mathcal{E}^E_{\mu}, \gamma_\mu)+ m^0(\mathcal{E}^E_{\mu}, \gamma_\mu)$
 as $\lambda\in\Lambda$ varies in two deleted half neighborhoods  of $\mu$.
  Then  one of the following alternatives occurs:
\begin{enumerate}
\item[\rm (i)] There exists a sequence of  nonconstant
 $1$-periodic  $F_{\mu}$-geodesics $\gamma^k\ne \gamma_{\mu}$, $k=1,2,\cdots$,
   such that $\gamma^k\to \gamma_\mu$ in $C^2(S_1;M)$.

\item[\rm (ii)]  For every $\lambda\in\Lambda\setminus\{\mu\}$ near $\mu$ there is a
  nonconstant   $1$-periodic  $F_{\lambda}$-geodesic $\alpha_\lambda\ne \gamma_{\lambda}$,
 such that  $\alpha_\lambda-\gamma_\lambda$
 converges to zero in  $C^2(S_1, \R^N)$ as $\lambda\to \mu$.
(Recall $M\subset\R^N$.)

\item[\rm (iii)]  For a given neighborhood $\mathcal{W}$ of $\gamma_\mu$ in $C^2(S_1, M)$,
there exists a one-sided  neighborhood $\Lambda^0$ of $\mu$ such that
for any $\lambda\in\Lambda^0\setminus\{\mu\}$,
there exist at least two distinct
$1$-periodic  $F_{\lambda}$-geodesics of constant {\rm (}non-zero{\rm )} speed in $\mathcal{W}$, $\gamma^1_\lambda\ne\gamma_\lambda$ and
 $\gamma^2_\lambda\ne\gamma_\lambda$,
which can also be chosen to have different $F_\lambda$-speeds
{\rm (}i.e., $F_\lambda(\gamma^1_\lambda(t), \dot{\gamma}^1_\lambda(t))\ne F_\lambda(\gamma^2_\lambda(t),
\dot{\gamma}^2_\lambda(t))\;\forall t${\rm )}
provided that $m^0(\mathcal{E}^E_{\mu}, \gamma_\mu)>1$ and
there exist only finitely many
$1$-periodic  $F_{\lambda}$-geodesics of constant {\rm (}non-zero{\rm )} speed in $\mathcal{W}$.
\end{enumerate}
\end{theorem}
\begin{proof}[\bf Proof]
As in the proof of Theorem~\ref{th:bif-suffFin} we have a number $\rho>0$ such that
the compact set $\hat{\Lambda}:=[\mu-\rho, \mu+\rho]\subset\Lambda$.
Choose an open subset $\hat{M}$ of $M$ with compact closure such that
$\gamma_\lambda(S_1)\subset\hat{M}$ for all $\lambda\in\hat{\Lambda}$.
Then for the constant $C_1>0$ as in Proposition~\ref{prop:Fin2.2} with $(M, \Lambda)=(\hat{M}, \hat{\Lambda})$ we have $c>0$ such that
$[F_\lambda(\gamma_\lambda(t), \dot{\gamma}_\lambda(t))]^2>{2c}/{C_1}$
for all $(\lambda,t)\in\hat{\Lambda}\times \R$.
 Let $L^\ast_\lambda:T\hat{M}\to\R$,  $\lambda\in\hat{\Lambda}$,
 be given by Proposition~\ref{prop:Fin2.2} with $(M, \Lambda)=(\hat{M}, \hat{\Lambda})$.
  The other reasoning may be derived from
 \cite[Th.1.24]{Lu12-} or \cite[Th.1.31]{Lu12}
as in the proof of Theorem~\ref{th:bif-suffFin}.
\end{proof}

\section{Randers metrics and potential applications}\label{sec:Randers}

A class of special Finsler metrics (Randers metric) are often used in several places.
A \textsf{Randers metric} on manifold $M$ is a Finsler metric of form
\[
F(x,y)=\sqrt{h_x(y,y)}+\beta(x)[y]
\]
with $\|\beta(x)\|_{h_x}<1$ for all $x\in M$,
where $h$ is a Riemannian metric, $\beta$ is a one-form on $M$ and
 $\|\beta(x)\|_{h_x}$ is the norm of the linear functional
 $\beta(x):T_xM\to\mathbb{R}$ on the Hilbert space $(T_xM, h_x)$.
  It is the condition $\|\beta(x)\|_{h_x}<1$ for all $x\in M$ that
 assures  $F$ is positive and $F^2$ is fiberwise strongly
 convex (see \cite[Sect.~11.1]{BaChSh}).
If the Riemannian manifold $(M,h)$ is complete and
\[
\|\beta\|_h:=\sup_{x\in M}\|\beta(x)\|_{h_x}<1,
\]
the Randers manifold $(M,F)$ is forward and backward complete
(see \cite[Remark~4.1]{CaJaMa3}).

\subsection{Bifurcation of  the  future-pointing lightlike geodesics
on conformally standard stationary spacetime}\label{sec:Randers.1}

Consider the  conformally standard stationary Lorentzian manifold
 ${\mathcal M}={\mathcal M}_0\times \R$
 with the Lorentzian metric  given by
 \begin{equation}\label{e:Lorent}
\mathfrak{g}_{(x,t)}\bigl((y,\tau),(y,\tau)\bigl)=\tilde{g}_{0x}(y,y)+2\tilde{g}_{0x}(V(x),y)\tau
-\tau^2
\end{equation}
(\cite[(33)]{CaJaMa1}), where $g_0$ is a Riemannian metric,
  $V$ is a vector field, and $f$ is a positive function
on ${\mathcal M}_0$, with $\tilde{g}_0=g_0/f$.
In \cite{CaJaMa1}, the \textsf{Fermat metric} associated to the Lorentzian manifold
$(\mathcal M, \mathfrak{g})$
is defined as the Randers metrics on $\mathcal M_0$  given by
\begin{align}
F(x,y)&=\sqrt{(\tilde{g}_{0x}(V(x),y))^2+\tilde{g}_{0x}(y,y)}+\tilde{g}_{0x}[V(x),y],\\
F_-(x,y)&=\sqrt{(\tilde{g}_{0x}(V(x),y))^2+\tilde{g}_{0x}(y,y)}-\tilde{g}_{0x}[V(x),y],
\end{align}
(\cite[(35)]{CaJaMa1}). Clearly, the expression  $h_x(y,y)=(\tilde{g}_{0x}(V(x),y))^2+\tilde{g}_{0x}(y,y)$)
defines the associated  Riemannian metric of them.
Kovner's Fermat principle can be stated as follows (see \cite[Theorem~4.1]{CaJaMa1} and \cite[Proposition~12]{CaJaMa3}):

\begin{proposition}[Fermat's principle]\label{prop:fermatp}
Let $(\mathcal M, \mathfrak{g})$ be the standard stationary spacetime as above,
$p=(p_0, 0), q=(q_0, 1)\in\mathcal M$,  $[0,1]\ni s\to\gamma(s)=(q_0,s)\in \mathcal M$, $p_0, q_0\in\mathcal M_0$.  A curve $[0,1]\ni s\to z(s)=(x(s),t(s))\in\mathcal{M}$ is a future-pointing lightlike geodesic
 of $({\mathcal M}, \mathfrak{g})$ as in (\ref{e:Lorent}) if and only if $[0,1]\ni s\to x(s)\in \mathcal M_0$ is a geodesic of the Fermat metric $F$, parametrized  with constant Riemannian speed $h_{x}(\dot x,\dot x)=\mathrm{const.}$, and $t(s)$ is given by
\[t(s)=t_0+\!\int_0^sF(x,\dot x)\,d v.\]
\end{proposition}

This proposition implies that $x:[0, 1]\to \mathcal{M}_0$ is an $F$-geodesic provided that
 $z=(x,t):[0,1]\to\mathcal{M}$ is a future-pointing lightlike geodesic
(see \cite{CaJaMa3} for the precise definition).
Let ${\rm Index}({\bf I}^z_{p,q})$ (resp.~${\rm Index}({\bf I}^x_{p_0,q_0})$),
denote the \textsf{geometric index} of $z$  (resp. $x$),
that is the number of conjugate points  along $z$ (resp. $x$) counted with their multiplicity.
(In the notation of \cite{CaJaMa1}, these two indices are $\mu(z)$ and $\mu(x)$.)

\begin{thm}[\text{\cite[Theorem~13]{CaJaMa1}}]\label{th:Loen-FinSindex}
Let $(\mathcal M, \mathfrak{g})$ be a conformally standard stationary spacetime
as above, $[0,1]\ni s\to z(s)=(x(s),t(s))\in {\mathcal M}$ be a future-pointing lightlike geodesic.
 Let $F$ be the Fermat metric associated to $(\mathcal{M}, \mathfrak{g})$.
Then the points $x(0)$ and $x(1)$ are non-conjugate along the geodesic $x$
in $(\mathcal{M}_0,F)$ if and only if the points $z(0)$ and $z(1)$ are non-conjugate along the lightlike geodesic $z$ in $(\mathcal{M}, \mathfrak{g})$.  Moreover
\begin{equation}\label{e:Loen-FinSindex}
{\rm Index}({\bf I}^z_{p,q})={\rm Index}({\bf I}^x_{p_0,q_0}).
\end{equation}
\end{thm}
Since Corollary~\ref{cor:FinMorseIndex} gives
 ${\rm Index}({\bf I}^x_{p_0,q_0})=m^-(\mathcal{E}_{p_0,q_0},x)$,
  Proposition~\ref{prop:fermatp} and
Theorem~\ref{th:Loen-FinSindex}, together with the results in Section~\ref{sec:geodesics4} and
Section~\ref{sec:geodesics1} with ${\bf N}=\{p_0\}\times\{q_0\}$
 yield  the corresponding bifurcation results for the  future-pointing lightlike geodesics
 in $({\mathcal M}, \mathfrak{g})$.

\subsection{Bifurcation of  the Zermelo navigation problem }\label{sec:Randers.2}

Consider an object moving on a Riemannian manifold
 $(M, h)$ under an internal force, with its velocity
  $U_x$ satisfying $h_x(U_x, U_x) = 1$.
  In the absence of external forces, any time-minimizing path coincides with
  a shortest path of $h$.

   Now suppose the object is subject to an external force field $V$
   satisfying  $h_x(V_x, V_x) < 1$ for all $x\in M$.
   The classical Zermelo navigation problem asks for the curve between two given points
   on $M$ that minimizes the travel time. Zhongmin Shen solved this problem \cite{Shen03} by showing that its solutions are precisely the geodesics of a Randers metric on $M$. Define
       $$
    \tilde{h}_x(y,y)=\frac{\sqrt{h_x(V(x),y)+h_x(y,y)(1-h_x(V(x),V(x))}}{1-h_x(y,y)},\qquad
    \tilde{\beta}_x(y)=-\frac{h_x(V(x),y)}{1-h_x(V(x),V(x))}.
        $$
  Then $\tilde{h}$ is a Riemannian metric on $M$, $\tilde{\beta}$ is a $1$-form on $M$ satisfying
  $\|\tilde{\beta}_x\|_{\tilde{h}_x}<1$ for all $x\in M$. Consequently, $F:=\tilde{h}+\tilde{\beta}$
  is a Randers metric on $M$. Shen \cite{Shen03} proved that
the $\tilde{F}$-length of any piecewise smooth curve $C$ on $M$
 equals the travel time along $C$.
 Therefore, under the influence of the external force $V$,
 the time-minimizing paths between two points on $(M, h)$
 are no longer $h$-geodesics, but rather $F$-geodesics.

Applying Theorem~\ref{th:MorseBifFin}, we can analyze the bifurcation structure of such curves.
We note that Shen \cite{Shen03} also treated the more general case where $h$
itself is a Finsler metric, reducing the corresponding navigation problem to the geodesic problem for another Finsler metric determined by $h$ and $V$. Our results are not restricted to Randers-type
Finsler metrics and thus remain applicable
in this broader setting.

\vspace{3mm}


\noindent{\bf Acknowledgements.}\; I am  deeply grateful to the anonymous referees for their careful reading of the manuscript and insightful comments
 and suggestions, which have greatly improved the revised version.\\

\noindent{\bf Funding}\; This work was supported by National Natural Science Foundation of China [12371108].\\

\noindent{\bf Data Availability}\; We do not analyse or generate any datasets, because our work proceeds within a theoretical
and mathematical approach.\\

%

\noindent{\bf Declarations}\\

\noindent{\bf Conflict of interest}\; The author has no Conflict of interest to declare that are relevant to the content of this
article.

\renewcommand{\refname}{REFERENCES}

\medskip

\begin{tabular}{l}
 School of Mathematical Sciences, Beijing Normal University\\
 Laboratory of Mathematics and Complex Systems, Ministry of Education\\
 Beijing 100875, The People's Republic of China\\
 E-mail address: gclu@bnu.edu.cn\\
\end{tabular}


\begin{thebibliography}{99}


\bibitem{BaChSh} D. Bao, S. S. Chern, Z. Shen, \emph{An Introduction to Riemann-Finsler Geometry}, Springer, Berlin (2000).\vspace{-2mm}



\bibitem{BetPS2}
\newblock R. G. Bettiol, P. Piccione, G. Siciliano,
\newblock \emph{Delaunay-Type Hypersurfaces in Cohomogeneity One Manifolds},
\newblock {IMRN}, Vol. 2016, No. 10, pp. 3124--3162.\vspace{-2mm}


\bibitem{BuGiHi} G. Buttazzo, M. Giaquinta and S. Hildebrandt, \emph{One-dimensional Variational
Problems: An Introduction}.  Clarendon Press Oxford 1998.\vspace{-2mm}



\bibitem{CaJaMa1}E. Caponio, M.A. Javaloyes, A. Masiello,
\emph{Morse theory of causal geodesics in a stationary spacetime
via Morse theory of geodesics of a Finsler metric}, Ann. I. H. Poincar\'e--AN {\bf 27}(2010), 857-876.\vspace{-2mm}


\bibitem{CaJaMa3} E. Caponio, M.~{\'A}. Javaloyes,  A.
Masiello, \emph{On the energy functional on Finsler manifolds and
applications to stationary spacetimes},  Math.Ann., {\bf 351}(2011), no.2, 365--392.\vspace{-2mm}


\bibitem{Dav} M. N. Davletshin, {\it Hill's formula for g-periodic trajectories of Lagrangian systems},
 Trans. Moscow Math. Soc.  {\bf 74}(2013), 65--96.\vspace{-2mm}


\bibitem{Du} J.J. Duistermaat, \emph{On the Morse Index in Variational
Calculus}, {Advance in Mathematics}, {\bf 21}(1976), 173--195.\vspace{-2mm}



\bibitem{GiGiPi04} R. Giamb\'o, F. Giannoni, P. Piccione,
\emph{Gravitational lensing in general relativity via bifurcation theory}, {Nonlinearity}, {\bf 17}(2004), no. 1,
117--132. 
\vspace{-2mm}

\bibitem{GiJa07} R. Giamb\'o,  M. A. Javaloyes, \emph{A second order variational principle for the Lorentz force equation:
conjugacy and bifurcation}, Proc. R. Soc. Edinb.,{\bf 137A}(2007), 923--936.\vspace{-2mm}
%


\bibitem{Hir} M. W. Hirsch, \emph{Differential Topology}, Springer-Verlag, New York, 1976.\vspace{-2mm}


\bibitem{HoJo13}  Roger A. Horn,  Charles R. Johnson, {\it Matrix analysis}, Second edition. Cambridge University Press, Cambridge, 2013.
\vspace{-2mm}

\bibitem{JaPi06} M. A. Javaloyes, P. Piccione, \emph{On the singularities of the semi-Riemannian exponential map. Bifurcation of
geodesics and light rays}. { Variations on a century of relativity: theory and applications, 115--123, Lect. Notes
Semin. Interdiscip. Mat., V, Semin. Interdiscip. Mat. (S.I.M.), Potenza}, 2006.\vspace{-2mm}


\bibitem{Jav15}  M.~{\'A}. Javaloyes, B.~L.~Soares,
\emph{Geodesics and Jacobi fields of pseudo-Finsler manifolds},
{ Publ. Math. Debrecen}, {\bf 87}(2015), pp. 57--78.\vspace{-2mm}






\bibitem{Kl95} W. Klingenberg, \emph{Riemannian Geometry}, de Gruyter, New York, 1995.\vspace{-2mm}


\bibitem{KoKrVa} L. Kozma, A. Krist\'aly, C. Varga, \emph{Critical point
theorems on Finsler manifolds}, { Beitr\"age zur Algebra und}
({ Geometrie Contributions to Algebra and Geometry}), {\bf
45}(2004), no.1, 47--59.\vspace{-2mm}



\bibitem{LjSc29}
\newblock L. Ljusternik, L. Schnirelmann,
\newblock  \emph{Sur le probleme de trois geodesiques fermees sur les surfaces de genre $0$},
\newblock C. R. Acad. Sci Paris,  {\bf 189}(1929),  269-271.\vspace{-2mm}

\bibitem{LjSc34}
\newblock L. Ljusternik, L. Schnirelmann,
\emph{M\'ethodes Topologiques dans les Probl\'emes Variationnels},
Hermann, Paris, 1934.\vspace{-2mm}





\bibitem{Lu1} G. Lu, \emph{Corrigendum to ``The Conley conjecture for Hamiltonian
systems on the cotangent bundle and its analogue for Lagrangian
systems" [J. Funct. Anal. 256(9)(2009)2967--3034]}, { J. Funct.
Anal.} {\bf 261}(2011), 542--589. arXiv:0909.0609v2[math.SG], Thu, 10 Feb 2011.\vspace{-2mm}


%


\bibitem{Lu4} G. Lu, \emph{Methods of infinite dimensional Morse theory
for geodesics on Finsler manifolds}, { Nonlinear Analysis}, {\bf 113}(2015), 230--282.
arXiv:1212.2078 v3, Aug 22, 2013.\vspace{-2mm}




\bibitem{Lu5}
\newblock G. Lu,
\newblock \emph{Splitting lemmas for the Finsler energy functional on the space of $H^1$-curves},
\newblock {Proc. London Math. Soc.}, \textsf{113}(2016), no.3,24--76.\vspace{-2mm}



\bibitem{Lu6}
\newblock G. Lu,
\newblock \emph{Morse theory methods for quasi-linear elliptic systems of higher order},
\newblock  {arXiv:1702.06667}.\vspace{-2mm}



\bibitem{Lu8}
\newblock G. Lu,
\newblock \emph{Parameterized splitting theorems and bifurcations for potential operators, Part I:
Abstract  theory}, {Discrete \& Continuous Dynamical Systems}  {\bf 42}(2022), no.3, 1243--1316. 
\vspace{-2mm}

\bibitem{Lu9}
\newblock G. Lu,
\newblock \emph{Parameterized splitting theorems and bifurcations for potential operators, Part II:
Applications to quasi-linear elliptic equations and systems},
{Discrete \& Continuous Dynamical Systems}  {\bf 42}(2022), no.3, 1317--1368. 
\vspace{-2mm}

\bibitem{Lu11}
\newblock G. Lu,
\newblock \emph{A note on bifurcation theorems of Rabinowitz type},
\newblock  J. Math. Anal. Appl. {\bf 530}(2024), no. 2, Paper No. 127697, 24 pp.\vspace{-2mm}


	\bibitem{Lu10}
G.~Lu, \emph{Bifurcations for Hamiltonian systems},
{Adv.\ Nonlinear Stud.},
published online ahead of print, 2026.
doi:10.1515/ans-2023-0211;
arXiv:2112.10726v6 [math.DS].\vspace{-2mm}

%



\bibitem{Lu14}
\newblock G. Lu,
\newblock \emph{The Morse index theorem in the case of two variable end-points in conic Finsler manifolds},
\newblock    Annali di Matematica Pura ed Applicata (1923 -),
{\bf 203}(2024), 533--562. \vspace{-2mm}


%

\bibitem{Lu12-}
\newblock G. Lu,
\newblock \emph{Bifurcations in Lagrangian systems and geodesics I},
Calc. Var. Partial Differ. Equ., {\bf 58}(2026),  no. 5,  Art. no. 159, 72pp
\vspace{-2mm}



\bibitem{Lu12}
\newblock G. Lu,
\newblock \emph{Bifurcations for Lagrangian systems and geodesics},
\newblock   arXiv:2404.18815v2[math.DS], 14 Febrary 2025.\vspace{-2mm}



%
%
%
%




\bibitem{MaSaSh10}
 Ioana Monica Masca,  Vasile Sorin Sabau,  Hideo Shimada,  \emph{Reversible geodesics for $(\alpha,\beta)$-metrics},
 { Internat. J. Math.}, {\bf 21}(2010), no. 8, 1071--1094.\vspace{-2mm}


\bibitem{Me} F. Mercuri, \emph{The critical points theory for the closed
geodesics problem}, { Math. Z.}, {\bf 156}(1977), 231--245.\vspace{-2mm}

\bibitem{Mor1}
\newblock M. Morse,
\newblock \emph{Relations between the critical points of a real function of $n$ independent variables},
\newblock Trans. Amer. Math. Soc., {\bf 27}(1925),  345--396.
\vspace{-2mm}

\bibitem{Mor2}
\newblock M. Morse,
\newblock \emph{The calculus of variations in the large},
\newblock American Math. Soc. Colloquium Publications 18, Ann Arbor, Mich., 1934.
\vspace{-2mm}


\bibitem{MorLit32}  M. Morse and S. B. Littauer, \emph{A characterization of fields in the calculus of
 variations}, {Proceedings of the National Academy of Sciences, U. S. A.,}
  {\bf 18}(1932), pp. 724--730.\vspace{-2mm}





\bibitem{Pa63} R. S. Palais, \emph{Morse theory on Hilbert manifolds}, {Topology}, {\bf 2}(1963), 299--340.\vspace{-2mm}



\bibitem{Pe06}  Ioan Radu Peter, \emph{On the Morse index theorem where the ends are submanifolds in Finsler geometry},
{Houston J. Math.}, {\bf 32}(2006), no. 4, 995--1009.\vspace{-2mm}


\bibitem{PiTa01} P. Piccione, D. V. Tausk, \emph{On the Banach differential structure for sets of maps on non-compact domains}, {Nonlinear Analysis}, {\bf 46}(2001), 245--265.\vspace{-2mm}

\bibitem{PiPoTa04} P. Piccione, A. Portaluri, D. V. Tausk, \emph{Spectral flow, Maslov index and bifurcation
of semi-Riemannian geodesics}, {Ann. Global Anal. Geom.}, {\bf 25}(2004), no. 2, 121--149.\vspace{-2mm} 



\bibitem{RuSt14} B. Russell and S. Stepney, \emph{Zermelo navigation and
a speed limit to quantum information processing}, Phys.
Rev. {\bf A90}(2014), 012303.\vspace{-2mm}


\bibitem{Sak96} Takashi Sakai,  {\it Riemannian geometry}. Translated from the 1992 Japanese original by the author.
Translations of Mathematical Monographs, 149. American Mathematical Society, Providence, RI, 1996.\vspace{-2mm}



\bibitem{Sav43} L. J. Savage,  \emph{On the crossings of extremals at focal points}, { Bulletin of the
 Americant Mathematical Society}, {\bf 49}(1943), pp. 467--46.\vspace{-2mm}


\bibitem{Shen03}Z. Shen, \emph{Finsler metrics with $K=0$ and $S=0$},
Can. J. Math., {\bf 55}(2003), pp. 112-132.
\vspace{-2mm}


\bibitem{Shen} Z. Shen, \emph{Lectures on Finsler Geometry},
      World Scientific Publishing Co., New Jersey (2001).\vspace{-2mm}


%
\bibitem{Wa65} F. W. Warner, \emph{The Conjugate locus of a Riemannian manifold},
 { Amer. J. Math.,}
{\bf 87}(1965), 575--604.\vspace{-2mm}
%



\end{thebibliography}
\end{document}